\documentclass[12pt,a4paper,reqno]{amsart}
\usepackage{amsmath}
\usepackage{amsfonts}
\usepackage{amssymb}
\usepackage{diagrams}
\usepackage{bm} 

\newcommand\R{{\mathbf{R}}}
\newcommand\C{{\mathbf{C}}}
\newcommand\Z{{\mathbf{Z}}}

\renewcommand\H{{\mathbf{H}}}

\renewcommand\S{{\mathcal{S}}}

\newcommand\bigO{{\mathcal{O}}}
\newcommand\Sch{{\operatorname{Schwartz}}}

\newcommand\T{{\mathbf{T}}}
\newcommand\E{{\mathrm{E}}}
\newcommand\ESD{{\mathrm{ESD}}}

\newcommand\Dil{{\operatorname{Dil}}}
\newcommand\Trans{{\operatorname{Trans}}}
\newcommand\Time{{\operatorname{Time}}}
\newcommand\Rev{{\operatorname{Rev}}}
\newcommand\Rot{{\operatorname{Rot}}}
\newcommand\Frame{{\operatorname{Frame}}}
\newcommand\Energy{{\dot{\mathcal{H}^1}}}

\newcommand\loc{{\operatorname{loc}}}

\newcommand\strong{{+}}

\newcommand\eps{{\varepsilon}}

\newcommand\dist{{\operatorname{dist}}}

\theoremstyle{plain}
  \newtheorem{theorem}[subsection]{Theorem}
  \newtheorem{conjecture}[subsection]{Conjecture}
  
  \newtheorem{proposition}[subsection]{Proposition}
  \newtheorem{lemma}[subsection]{Lemma}
  \newtheorem{corollary}[subsection]{Corollary}

\theoremstyle{remark}
  \newtheorem{remark}[subsection]{Remark}
  
  \newtheorem{example}[subsection]{Example}

\theoremstyle{definition}
  \newtheorem{definition}[subsection]{Definition}

\include{psfig}

\begin{document}

\title[Global regularity of wave maps VII]{Global regularity of wave maps VII.  Control of delocalised or dispersed solutions}
\author{Terence Tao}
\address{Department of Mathematics, UCLA, Los Angeles CA 90095-1555}
\email{tao@math.ucla.edu}

\vspace{-0.3in}
\begin{abstract} This is the final paper in the series \cite{tao:heatwave}, \cite{tao:heatwave2}, \cite{tao:heatwave3}, \cite{tao:heatwave4} that establishes global regularity for two-dimensional wave maps into hyperbolic targets.  In this paper we establish the remaining claims required for this statement, namely a divisible perturbation theory, and a means of synthesising solutions for frequency-delocalised, spatially-dispersed, or spatially-delocalised data out of solutions of strictly smaller energy.  

As a consequence of the perturbation theory here and the results obtained earlier in the series, we also establish spacetime bounds and scattering properties of wave maps into hyperbolic space.
\end{abstract}

\maketitle

{\bf Note: This is a first draft only, and is NOT publication quality!  I hope to reorganise this series of papers in the future, hopefully with a somewhat more simplified and streamlined approach to the subject.}

\section{Introduction}

This paper is a technical component of a larger program \cite{tao:heatwave} to establish large data global regularity for the initial value problem for two-dimensional wave maps into hyperbolic spaces.   To explain this problem, let us set out some notation.

Let $\R^{1+2}$ be Minkowski space $\{ (t,x): t \in \R, x \in \R^2 \}$
with the usual metric $g_{\alpha \beta} x^\alpha x^\beta := -dt^2 + dx^2$.  We shall work primarily in $\R^{1+2}$, parameterised by Greek indices $\alpha, \beta = 0,1,2$, raised and lowered in the usual manner.  We also parameterise the spatial coordinates by Roman indices $i,j = 1,2$.  Fix $m \geq 1$, and let $\H = (\H^m,h)$ be the $m$-dimensional hyperbolic space of constant sectional curvature $-1$.  We observe that the Lorentz group $SO(m,1)$ acts isometrically on $\H$.

A \emph{classical wave map} $\phi = (\phi,I)$ from a slab $I \times \R^2$ in $1+2$-dimensional Minkowski space to $\H$ is defined to be any smooth map $\phi: I \times \R^2 \to \H$ which differs from a constant $\phi(\infty) \in \H$ by a function Schwartz in space (embedding $\H$ in $\R^{1+m}$ to define the Schwartz space), and is a (formal)  critical point of the Lagrangian
\begin{equation}\label{lagrangian}
 \int_{\R^{1+2}} \langle \partial^\alpha \phi(t,x), \partial_\alpha \phi(t,x) \rangle_{h(\phi(t,x))}\ dt dx
\end{equation}
in the sense that $\phi$ obeys the corresponding Euler-Lagrange equation
\begin{equation}\label{cov}
 (\phi^*\nabla)^\alpha \partial_\alpha \phi = 0,
\end{equation}
where $(\phi^*\nabla)^\alpha$ is covariant differentiation on the vector bundle $\phi^*(T\H)$ 
with respect to the pull-back $\phi^*\nabla$ via $\phi$ of the Levi-Civita connection $\nabla$ on the tangent bundle $T\H$ of $\H$.  If $I = \R$, we say that the wave map is \emph{global}.  We observe that wave maps have a conserved energy
\begin{equation}\label{energy-def}
 \E(\phi) = \E(\phi[t]) := \int_{\R^2} \T_{00}(t,x)\ dx = \int_{\R^2} \frac{1}{2} |\partial_t \phi|_{h(\phi)}^2
+ \frac{1}{2} |\nabla_x \phi|_{h(\phi)}^2\ dx
\end{equation}
for all $t \in I$, where $\phi[t] := (\phi(t), \partial_t \phi(t))$ is the data of $\phi$ at $t$.  

We record five important (and well known) symmetries of wave maps:

\begin{itemize}
\item For any $t_0 \in \R$, we have the time translation symmetry 
\begin{equation}\label{time-trans}
\Time_{t_0}: \phi(t,x) \mapsto \phi(t-t_0,x).
\end{equation}
\item For any $x_0 \in \R^2$, we have the space translation symmetry
\begin{equation}\label{space-trans}
\Trans_{x_0}: \phi(t,x) \mapsto \phi(t,x-x_0).
\end{equation}
\item Time reversal symmetry
\begin{equation}\label{time-reverse}
\Rev: \phi(t,x) \mapsto \phi(-t,x).
\end{equation}
\item For every $U \in SO(m,1)$, we have the target rotation symmetry
\begin{equation}\label{rotate}
\Rot_U: \phi(t,x) \mapsto U \circ \phi(t,x).
\end{equation}
\item For every $\lambda > 0$, we have the scaling symmetry
\begin{equation}\label{cov-scaling}
\Dil_\lambda: \phi(t,x) \mapsto \phi(\frac{t}{\lambda}, \frac{x}{\lambda})
\end{equation}
\end{itemize}
In each of these symmetries, the lifespan $I$ of the wave map transforms in the obvious manner (e.g. for \eqref{time-trans}, $I$ transforms to $I+t_0$).  Also observe that the energy \eqref{energy-def} is invariant with respect to all of these symmetries.

Define \emph{classical data} to be any pair $(\phi_0,\phi_1)$, where $\phi_0: \R^2 \to \H$ is map which is in the Schwartz space modulo constants (embedding $\H$ in $\R^{1+m}$ to define the Schwartz space), and $\phi_1: \R^2 \to T\H$ is a Schwartz map with $\phi_1(x) \in T_{\phi_0(x)} \H$ for all $x \in \R^2$; let $\S$ denote the space of all classical data, equipped with the Schwartz topology; one should view this space as a non-linear analogue of the linear Schwartz data space $\S(\R^2) \times \S(\R^2)$.  Note that if $(\phi,I)$ is a classical wave map, then $\phi[t]$ lies in $\S$, smoothly in $t$.

For future reference, we note that the symmetries \eqref{space-trans}, \eqref{time-reverse}, \eqref{rotate}, \eqref{cov-scaling} induce analogous symmetries on $\S$:
\begin{align}
\Trans_{x_0}: (\phi_0(x), \phi_1(x)) &\mapsto (\phi_0(x-x_0), \phi_1(x-x_0)) \label{space-trans-data}\\
\Rev: (\phi_0(x), \phi_1(x)) &\mapsto (\phi_0(x), -\phi_1(x)) \label{time-reverse-data}\\
\Rot_U: (\phi_0(x), \phi_1(x)) &\mapsto (U\phi_0(x), dU(\phi_0(x))(\phi_1(x))) \label{rotate-data}  \\
\Dil_\lambda: (\phi_0(x), \phi_1(x)) &\mapsto (\phi_0(\frac{x}{\lambda}), \frac{1}{\lambda} \phi_1(\frac{x}{\lambda})) \label{scaling-data}
\end{align}

The purpose of this paper and the three preceding papers \cite{tao:heatwave}, \cite{tao:heatwave2}, \cite{tao:heatwave3}, \cite{tao:heatwave4} is to prove the following conjecture, stated for instance in \cite{klainerman:unified}:

\begin{conjecture}[Global regularity for wave maps]\label{conj}  For every $(\phi_0,\phi_1) \in \S$ there exists a unique global classical wave map $\phi: \R \times \R^2 \to \H$ with $\phi[0] = (\phi_0,\phi_1)$.
\end{conjecture}

There has been extensive prior progress on this conjecture, see for instance \cite[Chapter 6]{tao:cbms} or \cite{krieger:survey} for a survey, or the first paper \cite{tao:heatwave} in this series for further discussion and references.  Recently, an alternate proof of Conjecture \ref{conj} (and more generally, global regularity for any smooth target manifold which is either compact, or has compact quotients) was given in \cite{sterbenz}, \cite{sterbenz2}.  A third proof, in the case when the target is the hyperbolic plane $\H^2$, has also been announced by Krieger and Schlag (private communication).

The preceding papers \cite{tao:heatwave}, \cite{tao:heatwave2}, \cite{tao:heatwave3}, \cite{tao:heatwave4} reduced matters to being able to define the concept of a \emph{$(A,\mu)$-wave map} for parameters $A, \mu > 0$, which are a class of classical wave maps $(\phi,I)$ on a compact time interval $I$ of energy at most $A$ which obey some additional estimates depending on $A, \mu$ which we will specify in Section \ref{aesec}.  Given a classical wave map $(\phi,I)$ and parameters $A,\mu >0$, define the \emph{$(A,\mu)$-entropy} of $(\phi,I)$ to be the minimum number of compact intervals $J$ needed to cover $I$, such that the restriction $(\phi\downharpoonright_J,J)$ of $(\phi,I)$ to $J$ is a $(A,\mu)$-wave map for each such $J$.

In Section \ref{aesec} we will prove

\begin{theorem}[Properties of $(A,\mu)$-wave maps]\label{symscat}  For every $A, \mu > 0$ there exists a class of classical wave maps $(\phi,I)$ on compact time intervals with the following properties:
\begin{itemize}
\item[(i)] (Monotonicity) If $(\phi,I)$ is an $(A,\mu)$-wave map for some $A > 0$ and $0 < \mu \leq 1$, then it is also a $(A',\mu')$-wave map for every $A' \geq A$ and $\mu' \geq \mu$.  Also, $(\phi\downharpoonright_J,J)$ is an $(A,\mu)$-wave map for every $J \subset I$.
\item[(ii)] (Symmetries) If $(\phi,I)$ is an $(A,\mu)$-wave map for some $A > 0$ and $0 < \mu \leq 1$, then any application of the symmetries \eqref{time-trans}-\eqref{cov-scaling} to $\phi$ (and $I$) is also a $(A,\mu)$-wave map.
\item[(iii)] (Divisibility)  If $(\phi,I)$ is an $(A,\mu)$-wave map, with $\mu$ sufficiently small depending on $A$, and $0 < \mu' \leq 1$, then $(\phi,I)$ has an $(O(A),\mu')$-entropy of $O_{A,\mu,\mu'}(1)$.
\item[(iv)] (Continuity)  Let $A > 0$, and let $0 < \mu < 1$ be sufficiently small depending on $A$.  If $I$ is a compact interval, $I_n$ is an increasing sequence of compact sub-intervals which exhaust the interior of $I$, and $(\phi^{(n)},I_n)$ is a sequence of $(A,\mu)$-wave maps which converge uniformly in $\Energy$ on every compact interval of $I$, then the limit extends continuously in $\Energy$ to all of $I$.  (For the definition of the energy space $\Energy$, see Section \ref{esd-sec} below.)
\item[(v)] (Small data scattering) For every $0 < \mu \leq 1$ there exists an $E > 0$ such that whenever $I$ is a interval and $(\phi,I)$ is a classical wave map with energy at most $E$, then $\phi$ is a $(\mu,\mu)$-wave map.
\item[(vi)] (Local well-posedness in the scattering size) For every $E > 0$, $0 < \mu \leq 1$, $\Phi \in \Energy$, $t_0\in \R$ with $\E(\Phi) \leq E$, there exists a compact interval $I$ with $t_0$ in the interior, such that for any $(\phi_0,\phi_1) \in \S$ with $(\phi_0,\phi_1)$ sufficiently close to $\Phi$ in $\Energy$, there is a $(O_E(1),\mu)$-wave map $(\phi,I)$ with $\phi[t_0] = (\phi_0,\phi_1)$.
\end{itemize}
\end{theorem}

Call an energy $E > 0$ \emph{good} if there exists $A,M > 0$, $0 < \mu \leq 1$ such that every classical wave map $(\phi,I)$ with $I$ compact and $\E(\phi) \leq E$ has an $(A,\mu)$-entropy of at most $M$.  Call $E$ \emph{bad} if it is not good. In \cite[Lemma 3.1]{tao:heatwave4}, the following consequence of the above theorem was shown:

\begin{lemma}[Scattering bound implies global well-posedness]\label{quant} Suppose $E$ is good.  Then for any classical initial data $(\phi_0,\phi_1) \in \S$ with energy at most $E$ and any $t_0 \in \R$, there is a global classical wave map $\phi$ with $\phi[t_0] = (\phi_0,\phi_1)$.
\end{lemma}

In Section \ref{esd-sec} we will recall the notion of the \emph{energy spectral distribution} $\ESD(\phi_0,\phi_1)$ of some classical data $(\phi_0,\phi_1) \in \S$.  We will then prove the following theorems:

\begin{theorem}[Frequency delocalisation implies spacetime bound]\label{freqbound}  Let $0 < E_0 < \infty$ be such that every energy strictly less than $E_0$ is good. Let $\phi^{(n)}[0] \in \S$ be a sequence of classical initial data with the energy bound\footnote{See Section \ref{notation-sec} below for the asymptotic notation used in this paper.}
\begin{equation}\label{enbound}
 \E( \phi^{(n)}[0] ) \leq E_0 + o_{n \to \infty}(1)
\end{equation}
and suppose that there exists $\eps > 0$ (independent of $n$) and $K_n \to \infty$ such that we have the frequency delocalisation property
\begin{equation}\label{fdl-1}
 \int_{1/K_n < s < K_n} \ESD(\phi^{(n)}[0])(s)\ ds = o_{n \to \infty}(1)
\end{equation}
and
\begin{equation}\label{fdl-2}
 \int_{s \geq K_n} \ESD(\phi^{(n)}[0])(s)\ ds, \int_{s \leq 1/K_n} \ESD(\phi^{(n)}[0])(s)\ ds \geq \eps
\end{equation}
for all $n$.  Then there exists $A, M > 0$ independent of $n$ such that for each sufficiently large $n$ and every compact time interval $I \ni 0$, one can extend $\phi^{(n)}[0]$ to a classical wave map $(\phi^{(n)}, I)$ with $(A,1)$-entropy at most $M$.
\end{theorem}

\begin{theorem}[Spatial dispersion implies spacetime bound]\label{spacbound} Let $0 < E_0 < \infty$ be such that every energy strictly less than $E_0$ is good. Let $\phi^{(n)}[0] \in \S$ be a sequence of initial data with the energy bound \eqref{enbound}, which is uniformly localised in frequency in the sense that for every $\eps > 0$ there exists a $C > 0$ (independent of $n$) such that
\begin{equation}\label{spacloc}
  \int_{s \geq C} \ESD(\phi^{(n)}[0])(s)\ ds, \int_{s \leq 1/C} \ESD(\phi^{(n)}[0])(s)\ ds \leq \eps
\end{equation}
for all $n$.  Suppose also that the data is asymptotically spatially dispersed in the sense that
\begin{equation}\label{disp}
 \sup_{x_0 \in \R^2} \int_{|x-x_0| \leq 1} \T_{00}(\phi^{(n)})(0,x)\ dx = o_{n \to \infty}(1),
\end{equation}
where the energy density $\T_{00}$ was defined in \eqref{energy-def}.
Then there exists $A, M > 0$ independent of $n$ such that for each sufficiently large $n$ and every compact time interval $I \ni 0$, one can extend $\phi^{(n)}[0]$ to a classical wave map $(\phi^{(n)}, I_\pm)$ to either the forward time interval $I_+ := I \cap [0,+\infty)$ or the backward time interval $I_- := I \cap (-\infty,0]$ with $(A,1)$-entropy at most $M$.
\end{theorem}

\begin{theorem}[Spatial delocalisation implies spacetime bound]\label{spacdeloc} Let $0 < E_0 < \infty$ be such that every energy strictly less than $E_0$ is good. Let $\phi^{(n)}[0] \in \S$ be a sequence of initial data with the energy bound \eqref{enbound}, which is uniformly localised in frequency in the sense of \eqref{spacloc}.  Suppose also that there exists $\eps > 0$ (independent of $n$) and $R_n \to \infty$ such that
$$ \int_{|x| \leq 1} \T_{00}(\phi^{(n)})(0,x)\ dx \geq \eps$$
and
$$ \int_{|x| \geq R_n} \T_{00}(\phi^{(n)})(0,x)\ dx \geq \eps.$$
Then there exists $A, M > 0$ independent of $n$ such that for each sufficiently large $n$ and every compact time interval $I \ni 0$, one can extend $\phi^{(n)}[0]$ to a classical wave map $(\phi^{(n)}, I)$ with $(A,1)$-entropy at most $M$.
\end{theorem}

Theorems \ref{freqbound}, \ref{spacbound}, \ref{spacdeloc} will be proven in Sections \ref{freqbound-sec}, \ref{spacbound-sec}, and \ref{spacdeloc-sec} respectively.  Theorems \ref{symscat}-\ref{spacdeloc}, combined with the arguments in \cite{tao:heatwave}, \cite{tao:heatwave2}, \cite{tao:heatwave3}, \cite{tao:heatwave4}, imply Conjecture \ref{conj}; see \cite{tao:heatwave4} (and \cite{tao:heatwave}) for a detailed discussion.  In Section \ref{scattering-sec} below we will also explain how these results also imply scattering and spacetime bounds for wave maps.

\subsection{Acknowledgements}

This project was started in 2001, while the author was a Clay Prize Fellow.  The author thanks Andrew Hassell and the Australian National University for their hospitality when a substantial portion of this work was initially conducted, and to Ben Andrews and Andrew Hassell for a crash course in Riemannian geometry and manifold embedding, and in particular to Ben Andrews for explaining the harmonic map heat flow.  The author also thanks Mark Keel for background material on wave maps, Daniel Tataru for sharing some valuable insights on multilinear estimates and function spaces, and to Sergiu Klainerman, Igor Rodnianski and Jacob Sterbenz for valuable discussions.  The author is supported by NSF grant DMS-0649473 and a grant from the Macarthur Foundation.

\section{Notation}\label{notation-sec}

In order to efficiently manage the rather large amount of technical computations, it is convenient to introduce a substantial amount of notation, as well as some basic estimates that we will use throughout the paper.

\subsection{Small exponents}

We will need to fix four small exponents
$$ 0 < \delta_{-1} \ll \delta_0 \ll \delta_1 \ll \delta_2 \ll 1$$
which are absolute constants, with each $\delta_i$ being sufficiently small depending on all higher $\delta_i$.  (The strange choice of subscripts here is needed for compatibility with \cite{tao:heatwave3} and \cite{tao:wavemap2}). The exact choice of constants is not important, but for sake of concreteness one could take $\delta_i := 10^{-10^{3-i}}$, for instance.  All the implied constants in the asymptotic notation below can depend on these exponents.  

\begin{remark}
The interpretation of these constants in our argument will be as follows. The largest constant $\delta_2$ is the exponent that quantifies certain useful exponential decays in frequency-localised linear, bilinear, and trilinear estimates when the ratio of two frequencies becomes favorable.  The intermediate constant $\delta_1$ is used to design the weakly frequency localised space $S_k$ (and its variant $S_{\mu,k}$, adapted to the large data theory) that we will encounter later in this paper.  The smaller constant $\delta_0$ is used to control the fluctuation of the frequency envelopes $c(s)$ that we will use to control solutions.  The smallest constant $\delta_{-1}$ is needed to define a certain technical refinement $N_k^\strong$ of the nonlinearity space $N_k$ in Appendix \ref{nk-refine}.
\end{remark}

It will be useful to have some standard cutoff functions that compare two frequency parameters $k, k'$:

\begin{definition}[Cutoff functions]  Given any integers $k, k'$, we define $\chi_{k \geq k'} = \chi_{k' \leq k} := \min( 1, 2^{-(k-k')} )$ and $\chi_{k = k'} := 2^{-|k-k'|}$.
\end{definition}

Thus $\chi_{k' \leq k}$ is weakly localised to the region $k' \leq k$, and similarly for the other cutoffs. In practice we shall usually raise these cutoffs to an exponent such as $\pm \delta_0$, $\pm \delta_1$, or $\pm \delta_2$.

\subsection{Asymptotic notation}

The dimension $m$ of the target hyperbolic space $\H^m$ is fixed throughout the paper, and all implied constants can depend on $m$.

We use $X = O(Y)$ or $X \lesssim Y$ to denote the estimate $|X| \leq CY$ for some absolute constant $C > 0$, that can depend on the $\delta_i$ and the dimension $m$ of the target hyperbolic space.  If we wish to permit $C$ to depend on some further parameters, we shall denote this by subscripts, e.g. $X = O_k(Y)$ or $X \lesssim_k Y$ denotes the estimate $|X| \leq C_k Y$ where $C_k > 0$ depends on $k$.  

Note that parameters can be other mathematical objects than numbers.  For instance, the statement that a function $u: \R^2 \to \R$ is Schwartz is equivalent to the assertion that one has a bound of the form $|\nabla_x^k u(x)| \lesssim_{j,k,u} \langle x \rangle^{-j}$ for all $j,k \geq 0$ and $x \in \R^2$, where $\langle x \rangle := (1+|x|^2)^{1/2}$.

\subsection{Schematic notation}

We use $\nabla_x$ as an abbreviation for $(\partial_1, \partial_2)$, thus for instance 
\begin{align*}
|\nabla_x \phi|_{\phi^* h}^2 &= \langle \partial_i \phi, \partial_i \phi \rangle_{\phi^* h}^2 \\
|(\phi^* \nabla)_x \nabla_x \phi| &= \langle (\phi^* \nabla)_i \partial_j \phi, (\phi^* \nabla)_i \partial_j \phi \rangle_{\phi^* h}^2
\end{align*}
with the usual summation conventions.

We use juxtaposition to denote tensor product; thus for instance if $\psi_x := (\psi_1, \psi_2)$, then $\psi_x^2 = \psi_x \psi_x$ denotes the rank $2$ tensor with the four components $\psi_i \psi_j$ for $i,j = 1,2$; similarly, $\nabla_x^2$ is the rank $2$ tensor operator with four components $\partial_i \partial_j$ for $i,j=1,2$, and so forth.

If $X$ is a tensor-valued quantity, we use $\bigO( X )$ to denote an expression which is \emph{schematically} of the form $X$, which means that it is a tensor whose components are linear combinations of those in $X$, with coefficients being absolute constants (depending only on $m$).  Thus for instance, if $v, w \in \R^m$, then the anti-symmetric matrix $v \wedge w := v w^\dagger - w v^\dagger$ has the schematic form $v \wedge w = \bigO( v w )$.  If the coefficients in the schematic representation depend on a parameter, we will denote this by subscripts.  Thus for instance we have the \emph{Leibniz rule}
\begin{equation}\label{leibnitz}
\nabla_x^j \bigO( \phi \psi ) = \sum_{j_1,j_2 \geq 0: j_1+j_2=j} \bigO_j( \nabla_x^{j_1} \phi \nabla_x^{j_2} \psi )
\end{equation}
and similarly for products of three or more functions.

\subsection{Difference notation}

Throughout this paper, we adopt the notational conventions $\delta f := f' - f$ and $f^* := (f,f')$ for any field $f$ for which the right-hand side makes sense.  For future reference we observe the discretised Leibniz rule
\begin{equation}\label{disc-leib-eq}
\delta(fg) = \bigO( (\delta f) g^* ) + \bigO( f^* \delta g )
\end{equation}
and similarly for products of three or more functions.

\subsection{Function spaces}  

We use the usual Lebesgue spaces $L^p_x(\R^2)$ and Sobolev spaces $H^s_x(\R^2)$, and create spacetime norms such as $L^q_t L^p_x(I \times \R^2)$ in the usual manner.  Later on we shall also use the more complicated spaces adapted to the wave maps problem from \cite{tao:wavemap2} (see also \cite{tataru:wave2}).

If $X$ is a normed space for scalar-valued functions, we also extend $X$ to functions $\phi := (\phi_1,\ldots,\phi_m)$ taking values in a standard finite-dimensional vector space such as $\R^m$ with the convention
$$ \|\phi\|_X := (\sum_{j=1}^m \|\phi_j\|_X^2)^{1/2}.$$
Note that if $X$ was already a Hilbert space on scalar functions, it continues to be a Hilbert space on vector-valued functions, and the orthogonal group $O(m)$ on that space acts isometrically on this space.  If $X$ is merely a normed vector space, then the orthogonal group is no longer isometric, but the action of an element of this group has operator norm bounded above and below by constants depending only on $m$.

\begin{definition}[Littlewood-Paley projections] 
Let $\varphi(\xi)$ be a radial bump function supported in the ball $\{ \xi \in \R^2: |\xi| \leq \tfrac {11}{10} \}$ and equal to $1$ on the ball $\{ \xi \in \R^2: |\xi| \leq 1 \}$.  For each real number $k$, we define the Fourier multipliers
\begin{align*}
\widehat{P_{\leq k} f}(\xi) &:= \varphi(\xi/2^k) \hat f(\xi)\\
\widehat{P_{>k} f}(\xi) &:= (1 - \varphi(\xi/2^k)) \hat f(\xi)\\
\widehat{P_k f}(\xi) &:= (\varphi(\xi/2^k) - \varphi(2\xi/2^k)) \hat f(\xi).
\end{align*}
We similarly define $P_{<k}$ and $P_{\geq k}$.
\end{definition}

We remark that while the Littlewood-Paley projections are traditionally defined for integer $k$, it will be convenient for various minor technical reasons to generalise these projections to real $k$.

It is useful to introduce the following frequency-weighted Sobolev norms.


\begin{definition}[Sobolev-type norms]\label{sobspace}  For every $m \geq 0$ and $s>0$, and any Schwartz function on $\R^2$, we define
$$ \| u \|_{A^m(s)} := \sum_{j=0}^m s^{j/2} \|\nabla_x^j u\|_{L^\infty_x(\R^2)} + s^{(j+1)/2} \| \nabla_x^{j+1} u \|_{L^2_x(\R^2)}.$$
If instead $u$ is a function on $\R^+ \times \R^2$ which is Schwartz in space, we define
$$ \| u \|_{A^m(\R^+ \times \R^2)} := \sup_{s>0} \|u(s)\|_{A^k(s)} + \int_0^\infty \|u(s)\|_{A^k(s)}^2 \frac{ds}{s}.$$
\end{definition}

We observe the useful algebra property
\begin{equation}\label{algebra-hks}
\| uv \|_{A^m(s)} \lesssim_m \|u\|_{A^m(s)} \|v\|_{A^m(s)}
\end{equation}
for any $m \geq 0$ and $s>0$, and any Schwartz functions $u,v$ on $\R^2$, which follow from many applications of the Leibniz rule, the triangle inequality, and H\"older's inequality.  This algebra property implies a similar one on $\R^+ \times \R^2$, namely
\begin{equation}\label{algebra-hks-2}
\| uv \|_{A^m(\R^+ \times \R^2)} \lesssim_m \|u\|_{A^m(\R^+ \times\R^2)} \|v\|_{A^m(\R^+ \times \R^2)}
\end{equation}

In a similar spirit we define the frequency-weighted Sobolev norms
\begin{equation}\label{ek-def}
\|u\|_{\dot H^m_k(\R^2)} := \sup_{k'} \chi_{k=k'}^{-\delta_1} 2^{mk'} \| P_{k'} u \|_{L^2_x(\R^2)}
\end{equation}
and
\begin{equation}\label{ek-def2}
\|u\|_{A^m_k(\R^2)} := \sum_{m'=0}^m 2^{-(m'-1)k} \| u \|_{\dot H^{m'}_k(\R^2)}
\end{equation}

\begin{lemma}[Product estimates]\label{ek-prod}  For any Schwartz $u,v$ and $k_1,k_2 \in \R$, one has
$$ \| uv\|_{\dot H^1_{\max(k_1,k_2)}(\R^2)} \lesssim \|u\|_{\dot H^1_{k_1}(\R^2)} \|v\|_{\dot H^1_{k_2}(\R^2)};$$
and
$$ \| uv\|_{\dot H^0_{\max(k_1,k_2)}(\R^2)} \lesssim \chi_{k_1 \geq k_2}^{\delta_1} \|u\|_{\dot H^0_{k_1}(\R^2)} \|v\|_{\dot H^1_{k_2}(\R^2)}.$$
If $m \geq 1$, one similarly has
$$ \| uv\|_{A^m_{\max(k_1,k_2)}(\R^2)} \lesssim_m \|u\|_{A^m_{k_1}(\R^2)} \|v\|_{A^m_{k_2}(\R^2)}.$$
\end{lemma}

\begin{proof}  For any given $k,k_1,k_2 \in \R$, consider the problem of obtaining the best constant $C_{k,k_1,k_2}$ in the dyadic estimate
$$ \| P_k( P_{k_1} u P_{k_2} v ) \|_{L^2_x(\R^2)} \leq C_{k,k_1,k_2} \|P_{k_1} u \|_{L^2_x(\R^2)} \|P_{k_2} v \|_{L^2_x(\R^2)}.$$
Discarding the $P_k$ projection and using Bernstein's inequality to estimate $P_{k_2}v$ in $L^\infty_x$, one obtains the bound
$$ C_{k,k_1,k_2} \lesssim 2^{k_2};$$
a similar argument gives $C_{k,k_1,k_2} \lesssim 2^{k_1}$.  Finally, using Bernstein to control the $L^2_x$ norm of $P_k(P_{k_1} u P_{k_2} v)$ by the $L^1$ norm gives $C_{k,k_1,k_2} \lesssim 2^{k}$.  In other words, we have shown
$$ \| P_k( P_{k_1} u P_{k_2} v ) \|_{L^2_x(\R^2)} \lesssim 2^{\min(k,k_1,k_2)} \|P_{k_1} u \|_{L^2_x(\R^2)} \|P_{k_2} v \|_{L^2_x(\R^2)}.$$
The claims now follow by decomposing $uv = \sum_{k,k'_1,k'_2} P_k( P_{k'_1} u P_{k'_2} v )$ and using the triangle inequality (noting from Fourier analysis that we can restrict the summation to the low-high cases $k'_1 \leq k'_2 + O(1)$, $k=k'_2+O(1)$, the high-low cases $k'_2 \leq k'_1 + O(1), k = k'_1+O(1)$, and the high-high cases $k'_1=k'_2+O(1)$, $k \leq k'_1+O(1)$).
\end{proof}

\subsection{Frequency envelopes}

We first recall a useful definition from \cite{tao:wavemap}.

\begin{definition}[Frequency envelope]\label{freqenv}\cite{tao:wavemap}  Let $E > 0$.  A \emph{frequency envelope} of energy $E$ is a map $c: \R^+ \to \R^+$ with
\begin{equation}\label{cse}
\int_0^\infty c(s)^2 \frac{ds}{s} = E
\end{equation}
such that
\begin{equation}\label{sm}
c(s') \leq \max( (s'/s)^{\delta_0}, (s/s')^{\delta_0} ) c(s)
\end{equation}
for all $s, s' > 0$.
\end{definition}

\begin{remark} The estimate \eqref{sm} is asserting that $c(s)$ can grow at most as fast as $s^{\delta_0}$, and decay at most as rapidly as $s^{-\delta_0}$.  In \cite{tao:wavemap} a discretised version of this concept was used, with the continuous variable $s$ being replaced by the discrete variable $2^{-2k}$ for integer $k$.  However, as the heat flow uses a continuous time variable $s$, it is more natural to use the continuous version of a frequency envelope.  In \cite{tataru:wave3} it was observed that one could take asymmetric envelopes, in which $c(s)$ is allowed to decay faster as $s$ decreases than when $s$ increases.  However, due to our use of heat flow (which does not have as good frequency damping properties near the frequency origin as Littlewood-Paley operators) it is not convenient to use these asymmetric envelopes in our arguments.
\end{remark}

\begin{remark} Observe from \eqref{cse}, \eqref{sm} that if $c$ is a frequency envelope of energy $E$, then we have the pointwise bounds
\begin{equation}\label{cse-sup}
c(s) \lesssim \sqrt{E}
\end{equation}
for all $s > 0$.  Another useful inequality, arising from Cauchy-Schwarz, is
$$
\int_{s_1}^{s_2} \frac{c(s)}{s}\ ds \leq \sqrt{E} \log(s_2/s_1)^{1/2}
$$
which implies in particular that
$$
\int_{s_- \leq s_1 \leq \ldots \leq s_j \leq s_+} \frac{c(s_1)}{s_1} \ldots \frac{c(s_j)}{s_j}\ ds_1 \ldots ds_j \leq \frac{E^{j/2}}{j!} \log(s_+/s_-)^{j/2}.$$
Using Stirling's formula $j! = O(1)^j j^j$ and the elementary calculus inequality $\log(x)^{j/2} \lesssim_\eps O(1)^j j^{j/2} x^{\delta_0}$ for any $x \geq 1$ (which comes from applying the estimate $e^t \geq t$ to $t := \frac{2\delta_0}{j} \log x$) we conclude that
\begin{equation}\label{sss}
\int_{s_- \leq s_1 \leq \ldots \leq s_j \leq s_+} \frac{c(s_1)}{s_1} \ldots \frac{c(s_j)}{s_j}\ ds_1 \ldots ds_j \lesssim_\eps O_E(1)^j j^{-j/2} (s_+/s_-)^{\delta_0}.
\end{equation}
\end{remark}

We also define the frequency $k(s) \in \Z$ of a heat-temporal variable $s>0$ by the formula
\begin{equation}\label{ks-def}
k(s) := \lfloor \log_2 s^{-1/2} \rfloor,
\end{equation}
thus $2^{-2k(s)} \sim s$.  

We recall some Gronwall-type inequalities from \cite{tao:heatwave3}:

\begin{lemma}[Gronwall-type inequality from $s=+\infty$]\label{gron-lem}\cite[Lemma 2.15]{tao:heatwave3} Let $c$ be a frequency envelope of energy at most $E$.  Let $f, g: \R^+ \to \R^+$ be locally integrable functions with $\lim_{s \to \infty} f(s) = 0$ such that $g$ obeys \eqref{sm}, and
$$ f(s) \leq g(s) + \int_s^\infty f(s') c(s') ((s'-s)/s')^{-\theta} (s'/s)^{-3\delta_0}\ \frac{ds'}{s'}$$
for all $s > 0$ and some $0 \leq \theta < 1$.  Then we have
$$ f(s) \lesssim_{E,\theta} g(s)$$
for all $s > 0$.
\end{lemma}

\begin{lemma}[Gronwall-type inequality from $s=0$]\label{gron-lem-2}\cite[Lemma 2.16]{tao:heatwave3} Let $c$ be a frequency envelope of energy at most $E$.  Let $f, g: \R^+ \to \R^+$ be locally integrable functions with $\lim_{s \to 0} f(s) = 0$ such that $g$ obeys \eqref{sm}, and
$$ f(s) \leq g(s) + \int_s^\infty f(s') c(s') ((s-s')/s)^{-\theta} (s'/s)^{3\delta_0}\ \frac{ds'}{s'}$$
for all $s > 0$ and some $0 \leq \theta < 1$.  Then we have
$$ f(s) \lesssim_{E,\theta} g(s)$$
for all $s > 0$.
\end{lemma}

\section{Resolutions and reconstruction of maps}

In several of the preceding papers in this series \cite{tao:heatwave2}, \cite{tao:heatwave3}, \cite{tao:heatwave4}, a classical wave map $\phi: I \times \R^2 \to \H$ was analysed by extending it to a map $\phi: \R^+ \times I \times \R^2 \to \H$ along the harmonic map heat flow, and then viewing that extended map in an orthonormal frame in the \emph{caloric gauge} to obtain differentiated fields $\Psi_{s,t,x} = (\psi_{s,t,x}, A_{s,t,x})$ (with $A_s=0$), which took values in vector spaces such as $\R^m$ rather than in the hyperbolic space $\H^m$.  In particular, the field $\psi_s: \R^+ \times I \times \R^2 \to \R^m$ was generated; one can view $\psi_s$ as a sort of non-linear Littlewood-Paley resolution of $\phi$, analogous to the classical linear Littlewood-Paley resolution $\psi_s = \partial_s e^{s\Delta} \phi$ of maps $\phi: I \times \R^2 \to \R^m$ into a Euclidean space rather than hyperbolic space.  Furthermore, by integrating the gauge condition $A_s=0$, one could reconstruct the original map $\phi$ from the resolution $\psi_s$ (and some additional boundary data $\phi(\infty)$, $e(\infty)$ of little importance).

In this paper we will need to generalise this type of \emph{resolution} of a map $\phi$ to a field $\psi_s$, and \emph{reconstruction} of $\phi$ from $\psi_s$, to more general types of maps $\phi$ (not necessarily wave maps), and more general fields $\psi_s$ (not necessarily obeying the harmonic map heat flow).  The additional freedom afforded by this generalisation is necessary because we will be performing various decompositions of $\psi_s$ (e.g. into high and low frequency components) in order to establish Theorems \ref{freqbound}, \ref{spacbound}, \ref{spacdeloc}, and such decompositions will only preserve the wave map and heat flow equations \emph{approximately} rather than exactly.  The $\psi_s$ thus plays the role of the \emph{dynamic field} (cf. the ``dynamic separation'' technique in \cite{krieger:2d}).

In this section we set up the framework for this generalisation.

\begin{definition}[Resolution and reconstruction]  Let $I$ be a compact time interval.  A \emph{map} on $I$ is a smooth map $\phi: I \times \R^2 \to \H$.  A \emph{dynamic field} on $I$ is a smooth map $\psi_s: \R^+ \times I \times \R^2 \to \R^m$ which obeys the qualitative estimates
\begin{equation}\label{quali}
|\partial_t^i \partial_s^j \nabla_x^k \psi_s(s,t,x)| \lesssim_{i,j,k,\psi_s} \langle s \rangle^{-(3+k+2j)/2}
\end{equation}
for all $i,j,k \ge 0$, $s \in \R^+$, $t \in I$, and $x \in \R^2$.

Let $\phi(\infty) \in \H$, and let $e(\infty) \in \Frame(T_{\phi(\infty)} \H )$ be an orthonormal frame at $\phi(\infty)$.
A map $\phi: I \times \R^2 \to \H$ is said to be \emph{reconstructed} from a dynamic field $\psi_s: \R^+ \times I \times \R^2 \to \R^m$ with data $(\phi(\infty),e(\infty))$ at infinity if $\phi$ can be extended to a smooth map $\phi: \R^+ \times I \times \R^2 \to \H$, with $\phi(s) \to \phi(\infty)$ as $s \to \infty$ uniformly in $I \times \R^2$, and there exists a smooth frame $e \in \Gamma(\phi^* T\H)$ with $e(s) \to e(\infty)$ as $s \to \infty$ uniformly as $s \to \infty$, such that $e$ is parallel transported by $\partial_s$ in the sense that
\begin{equation}\label{nasal}
 (\phi^* \nabla)_s e = 0,
\end{equation}
and $\psi_s$ controls the $s$ derivative of $\phi$ in the sense that
\begin{equation}\label{nasa}
\psi_s = e^* \partial_s \phi.
\end{equation}
(See \cite{tao:heatwave2} for the differential geometry conventions used here.)
If $\phi$ can be reconstructed from $\psi_s$ and $(\phi(\infty),e(\infty))$, we say that $\psi_s$ is a \emph{resolution} of $\phi$.
\end{definition}

An easy application of the Picard existence theorem shows that given a dynamic field $\psi_s$ and data $(\phi(\infty),e(\infty))$, there exists a unique field $\phi$ that can be reconstructed from $\psi_s$ and $(\phi(\infty),e(\infty))$.  Due to the symmetric nature of hyperbolic space $\H$; changing the data $(\phi(\infty),e(\infty))$ amounts to a rotation \eqref{rotate}; see \cite{tao:heatwave2} for further discussion.  On the other hand, a single map $\phi$ can have multiple resolutions $\psi_s$, even after holding $(\phi(\infty),e(\infty))$ fixed; roughly speaking, any sufficiently nice retraction of $\phi$ to $\phi(\infty)$ will generate a resolution.  (However, there is a canonical choice of resolution, the \emph{caloric gauge}; see below.)

Let $\psi_s$ be a dynamic field.  We select an arbitrary pair of data $(\phi(\infty),e(\infty))$ at infinity and reconstruct the map $\phi: \R^+ \times I \times \R^2 \to \H$, as well as the associated frame $e \in \Gamma(\phi^* T\H)$.  We can then define the \emph{derivative fields} $\psi_\alpha: I \times \R^2 \to \R^m$ for $\alpha=0,1,2$ by the formula
$$ \psi_\alpha := e^* \partial_\alpha \phi$$
and the \emph{connection fields} $A_\alpha: I \times \R^2 \to \mathfrak{so}(m)$ by the formula
$$ (A_\alpha)_{ab} := \langle (\phi^* \nabla)_\alpha e_a, e_b \rangle_{\phi^* h}.$$
Note that the notation here is compatible with the existing field $\psi_s$; one could also define $A_s$, but it automatically vanishes,
\begin{equation}\label{ass}
A_s = 0,
\end{equation}
thanks to \eqref{nasal}.  Note also that the exact choice of $(\phi(\infty),e(\infty))$ do not affect the value of either $\psi_\alpha$ or $A_\alpha$.
It is convenient to concatenate $\psi_x := (\psi_1, \psi_2)$, $A_x := (A_1, A_2)$, $\psi_{t,x} := (\psi_0,\psi_1,\psi_2)$, $A_{t,x} := (A_0,A_1,A_2)$, $\Psi_x := (\psi_x, A_x)$, $\Psi_{t,x} := (\psi_{t,x}, A_{t,x})$, and $\Psi_{s,t,x} := (\Psi_{t,x}, \psi_s)$.   We refer to $\Psi_{s,t,x}$ as the \emph{differentiated fields}.

The connection fields $A_\alpha$ define a covariant derivative $D_\alpha := \partial_\alpha + A_\alpha$ on vector fields such as $\psi_s$ or $\psi_\alpha$; we record the \emph{zero-torsion identity}
$$ D_\alpha \psi_\beta = D_\beta \psi_\alpha$$
or equivalently
\begin{equation}\label{zerotor}
\partial_\alpha \psi_\beta - \partial_\beta \psi_\alpha = A_\alpha \psi_\beta - A_\beta \psi_\alpha
\end{equation}
and the \emph{constant curvature identity}
$$ [D_\alpha,D_\beta] = - \psi_\alpha \wedge \psi_\beta$$
or equivalently
\begin{equation}\label{constcurv}
\partial_\alpha A_\beta - \partial_\beta A_\alpha = - \psi_\alpha \wedge \psi_\beta - [A_\alpha,A_\beta].
\end{equation}
Similar identities hold in the $s$ direction; in particular we have the schematic identities
\begin{equation}\label{psa}
 \partial_s \psi_\alpha = \nabla_{t,x} \psi_s + \bigO( A_\alpha \psi_s )
\end{equation}
and
\begin{equation}\label{asa}
 \partial_s A_\alpha = \bigO( \psi_\alpha \psi_s ).
\end{equation}
We can unify these estimates schematically as
\begin{equation}\label{PSA}
 \partial_s \Psi_\alpha = \bigO(\nabla_{t,x} \psi_s) + \bigO( \Psi_\alpha \psi_s ).
\end{equation}
In order to retain the null structure in later arguments, it is important to note that the coefficients in the $\bigO()$ notation here do not depend on $\alpha$. 

\begin{lemma}[Decay estimate]\label{dec} $\Psi_{t,x} = O_{\psi_s}( \langle s \rangle^{-1/2} )$.
\end{lemma}

\begin{proof}  From \eqref{PSA}, \eqref{quali} one has
$$ |\partial_s \Psi_{t,x}| \lesssim_{\psi_s}  \langle s \rangle^{-3/2} + \langle s \rangle^{-3/2} |\Psi_{t,x}|$$
From Gronwall's inequality, one concludes that
\begin{equation}\label{tsh}
 |\Psi_{t,x}(s,t,x) - \Psi_{t,x}(s',t,x)| \lesssim_{\psi_s} \langle \min(s,s') \rangle^{-1/2} 
\end{equation}
for all $s,s' > 0$ and $(t,x) \in I \times \R^2$.  Thus $\Psi_{t,x}(s)$ is a uniform Cauchy sequence in $s$ and therefore converges uniformly to a limit $\Psi_{t,x}(\infty)$ as $s \to \infty$.  Because $\phi(s)$ converges uniformly to $\phi(\infty)$, we see that $\Psi_{t,x}(\infty)$ must be zero, and the claim follows from \eqref{tsh}.
\end{proof}

From this lemma and \eqref{psa}, \eqref{asa} we conclude that
$$
\psi_\alpha(s_0) = - \int_{s_0}^\infty \partial_\alpha \psi_s(s_1)\ ds_1 + \int_{s_0}^\infty \bigO( A_\alpha(s_1) \psi_s(s_1) )\ ds_1
$$
and
$$
A_\alpha(s_0) = \int_{s_0}^\infty \bigO( \psi_\alpha(s_1) \psi_s(s_1) )\ ds_1.
$$
We can iterate these identities to obtain
$$ \psi_\alpha(s_0) = \sum_{j=0}^J X_{2j+1,\alpha}(s_0) + E_{2J+1,\alpha}(s_0)$$
and
$$ A_\alpha(s_0) = \sum_{j=0}^J X_{2j+2,\alpha}(s_0) + E_{2J+2,\alpha}(s_0)$$
for any integer $J \geq 0$, where $X_{j,\alpha}(s_0)$ for $j \geq 1$ takes the form
\begin{equation}\label{xjdef}
 X_{j,\alpha}(s_0) = \bigO_j( \int_{s_0<s_1<\ldots<s_j} \psi_s(s_1) \ldots \psi_s(s_{j-1}) \partial_\alpha \psi_s(s_j)\ ds_1 \ldots ds_j )
 \end{equation}
and the error term $E_{j,\alpha}(s_0)$ takes the form
$$ E_{j,\alpha}(s_0) = \bigO_j( \int_{s_0<s_1<\ldots<s_j} \psi_s(s_1) \ldots \psi_s(s_{j-1}) \psi_s(s_j) \Psi_{t,x}(s_j) \ ds_1 \ldots ds_j ).$$
The coefficients in the $\bigO_j$ notation grow at most exponentially in $j$.  Using \eqref{quali} and Lemma \ref{dec}, one has the pointwise bound
$$ |E_{j,\alpha}(s_0)| \leq O_{\psi_s}(\int_{s_0}^\infty \langle s \rangle^{-3/2}\ ds)^j / j!$$
and thus (by the absolute integrability of $\langle s \rangle^{-3/2}$) $E_{j,\alpha}(s_0)$ converges uniformly to zero as $j \to \infty$.  Thus we have the expansions
\begin{align}
\psi_\alpha &= \sum_{j \geq 1, \hbox{ odd}} X_{j,\alpha} \label{psi-odd} \\
A_\alpha &= \sum_{j \geq 2, \hbox{ even}} X_{j,\alpha}. \label{psi-even}
\end{align}
which describe the differentiated fields $\Psi_{s,t,x}$ explicitly in terms of the dynamic field $\psi_s$.  

Let $\psi_s$ be a dynamic field.  We introduce the \emph{wave-tension field} $w: \R^+ \times I \times \R^2$ and the \emph{heat-tension field} $h: \R^+ \times I \times \R^2$ by the formulae
\begin{equation}\label{wave-tension}
 w := D^\alpha \psi_\alpha
 \end{equation}
and
\begin{equation}\label{heat-tension}
 h := \psi_s - D_i \psi_i
\end{equation}
with the usual summation conventions. Observe that $\phi$ obeys the wave map equation \eqref{cov} if and only if the resolution $\psi_s$ obeys the equation 
\begin{equation}\label{cov-wave}
w(0)=0
\end{equation}
(i.e. $w$ vanishes at the boundary $s=0$).  By abuse of notation, we therefore say that $\psi_s$ is a \emph{wave map} if \eqref{cov-wave} holds.

We say that a dynamic field $\psi_s$ is a \emph{heat flow} if the heat-tension field $h$ vanishes identically on the domain $\R^+ \times I \times \R^2$, thus
\begin{equation}\label{heat-eq}
\psi_s = D_i \psi_i.
\end{equation}
The frame $e(0)$ at $s=0$ associated to that field is the \emph{caloric gauge} for the reconstruction $\phi(0)$.  We have the following theorem:

\begin{theorem}[Existence and uniqueness of caloric gauge]\label{exist}  Let $\phi: I \times \R^2 \to \S$ be a map which equals $\phi(\infty)$ at infinity, and let $e(\infty) \in T_{\phi(\infty)}\H$.  Then there exists a unique resolution $\psi_s$ of $\phi$ with data $(\phi(\infty),e(\infty))$ at infinity which is a heat flow.
\end{theorem} 

\begin{proof} See \cite[Theorem 3.16]{tao:heatwave2}.
\end{proof}

Heat flows obey the following equations of motion (see \cite[Proposition 5.4]{tao:heatwave3}):
\begin{align}
\partial_s \psi_s &= D_i D_i \psi_s - (\psi_s \wedge \psi_i) \psi_i \label{psis-eq}\\
\partial_s \psi_{t,x} &= D_i D_i \psi_{t,x} - (\psi_{t,x} \wedge \psi_i) \psi_i \label{psit-eq}\\
\partial_s w &= D_i D_i w - (w \wedge \psi_i) \psi_i + 2 (\psi_\alpha \wedge \psi_i) D_i \psi^\alpha	 \label{w-eq} \\
D^\alpha D_\alpha \psi_s &= \partial_s w - (\psi_\alpha \wedge \psi_s) \psi^\alpha.   \label{psis-box}
\end{align}
In particular, we see that the fields $\psi_{s,t,x}$ and $w$ evolve in $s$ by a covariant heat equation (with a forcing term of shape $\bigO( \psi_\alpha \psi_x D_x \psi^\alpha )$), while $\psi_s$ evolves in spacetime by a covariant wave equation (with $\partial_s w$ as the forcing term).

\subsection{Instantaneous analogues}

The above discussion was for maps and fields on a time interval $I$.  One can also consider instantaneous versions of these concepts:

\begin{definition}[Instantaneous resolution and reconstruction]  Let $t_0 \in \R$ be a time.  An \emph{instantaneous map} $(\phi(t_0),\partial_t \phi(t_0))$ at time $t_0$ is smooth map $(\phi(t_0),\partial_t \phi(t_0)): \R^2 \to T\H$.  If $\phi: I \times \R^2 \to \H$ is a map with $t_0 \in I$, we write $\phi[t_0]$ for the instantaneous map $\phi[t_0] := (\phi(t_0), \partial_t \phi(t_0))$, and refer to $\phi[t_0]$ as the \emph{data} of $\phi$ at time $t_0$ and $\phi$ as an \emph{extension} of $\phi[t_0]$.  Clearly, every instantaneous map can be viewed as the data of some map in a neighbourhood of $t_0$.

An \emph{instantaneous dynamic field} $(\psi_s(t_0), \partial_t \psi_s(t_0))$ at $t_0$ is a pair of smooth maps $\psi_s(t_0), \partial_t \psi_s(t_0): \R^+ \times \R^2 \to \R^m$ which obeys the qualitative estimates
\begin{equation}\label{quali-2}
|\partial_t^i \partial_s^j \nabla_x^k \psi_s(s,t,x)| \lesssim_{j,k,\psi_s} \langle s \rangle^{-(3+k+2j)/2}
\end{equation}
for all $i=0,1$, $j,k \ge 0$, $s \in \R^+$.  If $\psi_s: \R^+ \times I \times \R^2 \to \R^m$ is a dynamic field and $t_0 \in I$, we write $\psi_s[t_0]$ for the instantaneous dynamic field $\psi_s[t_0] := (\psi_s(t_0), \partial_t \psi_s(t_0))$; we refer to $\psi_s[t_0]$ as the \emph{data} of $\psi_s$ at time $t_0$ and $\psi_s$ as an \emph{extension} of $\psi_s[t_0]$.  Again, every instantaneous dynamic field can be viewed as the data of some dynamic field in a neighbourhood of $t_0$.

Let $\phi(\infty) \in \H$, and let $e(\infty) \in \Frame(T_{\phi(\infty)} \H )$ be an orthonormal frame at $\phi(\infty)$.
An instantaneous map $\phi[t_0]$ is said to be \emph{reconstructed} from an instantaneous dynamic field $\psi_s[t_0]$ with data $(\phi(\infty),e(\infty))$ at infinity if there exists an interval $I$ containing $t_0$ in its interior and extensions $\phi, \psi_s$ of $\phi[t_0]$ of $\psi_s[t_0]$ to $I$ such that the map $\phi$ is reconstructed from the dynamic field $\psi_s$.  In this case we say that $\psi_s[t_0]$ is a \emph{resolution} of $\phi[t_0]$.
\end{definition}

It is not hard to see that the reconstruction $\phi[t_0]$ of $\psi_s[t_0]$ with the data $(\phi(\infty),e(\infty))$ does not depend on the particular choice of extension used for $\psi_s$.

Given an instantaneous dynamic field $\psi_s[t_0]$, we can define the differentiated fields $\psi_{x,t}, A_{x,t}$ on $\R^+ \times \{t_0\} \times \R^2$ 
by \eqref{psi-odd}, \eqref{psi-even}.  We can also define the first time derivative of $\psi_x, A_x$ by the formulae
\begin{align*}
\partial_t \psi_x &= \nabla_x \psi_t + A_t \psi_x - A_x \psi_t\\
\partial_t A_x &= \nabla_x A_t - \psi_t \wedge \psi_x - [A_t, A_x]
\end{align*}
from \eqref{zerotor}, \eqref{constcurv}.

We cannot define the wave-tension field $w$ for instantaneous dynamic fields (it would require two derivatives of $\psi_s$ in time, but we only have one), but we can define the heat-tension field $h$, as well as its first time derivative
\begin{align*}
 \partial_t h &= D_t h - A_t h \\
 &= D_t \psi_s - D_t D_i \psi_i - A_t h \\
 &= \partial_s \psi_t + (\psi_t \wedge \psi_i) \psi_i - D_i D_i \psi_t - A_t h
\end{align*}
As before, we say that an instantaneous dynamic field is a \emph{heat flow} if $h$ and $\partial_t h$ vanishes identically, or equivalently if one has the equations
\begin{align}
\psi_s &= D_i \psi_i \label{heat-eq-instant}\\
\partial_s \psi_t &= D_i D_i \psi_t - (\psi_t \wedge \psi_i) \psi_i\label{psit-eq-instant}
\end{align}
(cf. \eqref{heat-eq}, \eqref{psit-eq}).  Clearly a dynamic field $\psi_s$ on $I$ is a heat flow if and only if $\psi_s[t_0]$ is an (instantaneous) heat flow for each $t_0 \in I$.

There is an analogue of Theorem \ref{exist}:

\begin{theorem}[Existence and uniqueness of instantaneous caloric gauge]\label{exist2}  
Let $\phi[t_0] \in \S$ which equals $\phi(\infty)$ at infinity, and let $e(\infty) \in T_{\phi(\infty)}\H$.  Then there exists a unique resolution $\psi_s[t_0]$ of $\phi[t_0]$ with data $(\phi(\infty),e(\infty))$ at infinity which is a heat flow.
\end{theorem}

\begin{proof} 
To show existence, we extend $\phi$ in $\S$ to a small neighbourhood of $t_0$, apply Theorem \ref{exist}, and then restrict back to $t_0$.  To show uniqueness, suppose one has two different resolutions $\psi_s[t_0], \psi'_s[t_0]$ of $\phi[t_0]$ with the same data.  Applying \cite[Theorem 3.12]{tao:heatwave2} one has $\psi_s(t_0) = \psi'_s(t_0)$, which by \eqref{psi-odd}, \eqref{psi-even} implies that $\psi_x(t_0) = \psi'_x(t_0)$ and $A_x(t_0) = A'_x(t_0)$, and so $e(t_0) = e'(t_0)$ also.  By \eqref{psit-eq-instant} $\psi_t(t_0), \psi'_t(t_0)$ obey the same (linear) covariant heat equation; 
also, since $e(t_0) = e'(t_0)$, $\psi_t(0,t_0) = \psi'_t(0,t_0)$.  The difference $\psi_t(t_0)-\psi'(t_0)$ thus also obeys \eqref{psit-eq-instant} and vanishes at $s=0$, and thus vanishes for all $s$ by \cite[Lemma 4.8]{tao:heatwave2}.
\end{proof}

We recall some useful estimates for instantaneous heat flows:

\begin{proposition}[Parabolic regularity]\label{corbound}  Let $\psi_s[t_0]$ be an instantaneous heat flow at time $t_0$ whose associated instantaneous map $\phi[t_0]$ has energy at most $E$ for some $E>0$.  Then one has
\begin{align}
\int_0^\infty s^{k-1} \| \nabla_x^k \psi_{t,x}(t_0) \|_{L^2_x(\R^2)}^2 ds &\lesssim_{E,k} 1 \label{l2-integ-ord} \\
\sup_{s > 0} s^{(k-1)/2} \| \nabla_x^{k-1} \psi_{t,x}(s,t_0) \|_{L^2_x(\R^2)} &\lesssim_{E,k} 1 
\label{l2-const-ord} \\
\int_0^\infty s^{k-1} \| \nabla_x^{k-1} \psi_{t,x}(t_0) \|_{L^\infty_x(\R^2)}^2\ ds &\lesssim_{E,k} 1\label{linfty-integ-ord} \\
\sup_{s > 0} s^{k/2} \| \nabla_x^{k-1} \psi_{t,x}(s,t_0) \|_{L^\infty_x(\R^2)} &\lesssim_{E,k} 1 
\label{linfty-const-ord}
\end{align}
for all $k \geq 1$.  Similar estimates hold if one replaces $\psi_{t,x}$ with $A_x$, $\nabla_x \psi_{t,x}$ with $\psi_s$, $\nabla_x^2$ with $\partial_s$, and/or $\nabla_x$ with $D_x$.  Finally, one has
\begin{equation}\label{corheat}
\int_0^\infty s^{-2/p} \| \Psi_{t,x}(s,t_0) \|_{L^p_x(\R^2)}^2\ ds \lesssim_{E,p} 1
\end{equation}
for all $2 < p \leq \infty$.
\end{proposition}

\begin{proof} See \cite[Corollary 4.6]{tao:heatwave2}, \cite[Lemma 4.8]{tao:heatwave2}, and \cite[Proposition 4.3]{tao:heatwave2}.  To establish \eqref{corheat} for the $A_{t,x}$ component of $\Psi_{t,x}$, we observe from \eqref{asa}, \eqref{linfty-const-ord} that
$$ \|A_{t,x}(s,t_0) \|_{L^p_x(\R^2)} \lesssim \int_s^\infty \| \psi_{t,x}(s',t_0) \|_{L^p_x(\R^2)}\ ds'$$
and so the $A_{t,x}$ component of \eqref{corheat} follows from the $\psi_{t,x}$ component and Hardy's inequality.
\end{proof}


We can rephrase the above estimates in terms of a frequency envelope:

\begin{proposition}[Parabolic regularity, frequency envelope version]\label{corbound-freq}  Let $\psi_s[t_0]$ be an instantaneous heat flow at time $t_0$ whose associated instantaneous map $\phi[t_0]$ has energy at most $E$ for some $E>0$.  Then there exists a frequency envelope $c$ of energy $O_E(1)$ such that
\begin{align}
\| \nabla_s^j \nabla_x^{k+1} \Psi_{t,x}(s,t_0) \|_{\dot H^0_{k(s)}(\R^2)} &\lesssim_{E,j,k} s^{-(2j+k+1)/2} c(s)\label{al1} \\
\| \nabla_s^j \nabla_x^{k+1} A_{t,x}(s,t_0) \|_{\dot H^0_{k(s)}(\R^2)} &\lesssim_{E,j,k} s^{-(2j+k+1)/2} c^2(s)\label{al2} \\
\| \nabla_s^j \nabla_x^k \psi_s(s,t_0) \|_{\dot H^0_{k(s)}} &\lesssim_{E,j,k} s^{-(2j+k+1)/2} c(s) \label{al3}\\
\| \nabla_s^j \nabla_x^{k} \nabla_{t,x} \Psi_{x}(s,t_0) \|_{\dot H^0_{k(s)}(\R^2)} &\lesssim_{E,j,k} s^{-(2j+k+1)/2} c(s)\label{al4} \\
\| \nabla_s^j \nabla_x^{k} \nabla_{t,x} A_{x}(s,t_0) \|_{\dot H^0_{k(s)}(\R^2)} &\lesssim_{E,j,k} s^{-(2j+k+1)/2} c(s) \label{al5}\\
\| \Psi_{t,x}(s,t_0) \|_{L^p_x(\R^2)} &\lesssim_{E,p} s^{\frac{1}{p}-\frac{1}{2}} c(s)\label{al6}\\
\| \Psi_{t,x}(s,t_0) \|_{L^2_x(\R^2)} &\lesssim_{E} (\int_s^\infty c(s')^2 \frac{ds}{s'})^{1/2}\label{al7}
\end{align}
for all $j,k \geq 0$, $s>0$, and $2+\delta_0<p\leq \infty$.
\end{proposition}

\begin{proof}
We fix $E$ and allow all implied constants to depend on $E$.  Set
$$ c(s) := \sup_{s'} \min( (s/s')^{\delta_0}, (s'/s)^{\delta_0} ) \sum_{j=0}^{10} s^{(j+2)/2} \| \nabla_x^j \nabla_{t,x} \psi_s(s',t_0) \|_{\dot H^0_{k(s')}(\R^2)}.$$
From Proposition \ref{corbound} (expressing $\psi_s = \bigO( \nabla_x \psi_x + \Psi_x \Psi_x )$) we see that $c$ is a frequency envelope of energy $O(1)$.  By construction, we have
$$ \| \nabla_x^j \nabla_{t,x} \psi_s(s,t_0) \|_{\dot H^0_{k(s)}(\R^2)} \lesssim c(s) s^{-(j+2)/2}$$
for all $s>0$ and $0 \leq j \leq 10$.  Applying \cite[Theorem 6.1]{tao:heatwave3} we obtain the estimates \eqref{al1}-\eqref{al5}.

To prove \eqref{al6}, we use \eqref{psi-odd}, \eqref{psi-even} to bound the left-hand side by
$$ \sum_{j=1}^\infty O(1)^j \int_{s<s_1<\ldots<s_j} \| \psi_s(s_1) \ldots \psi_s(s_{j-1}) \nabla_{t,x} \psi_s(s_j) \|_{L^p_x(\R^2)}\ ds_1 \ldots ds_j.$$
Using H\"older's inequality, Bernstein's inequality, and \eqref{al3} one can bound this by
$$ \sum_{j=1}^\infty O(1)^j \int_{s<s_1<\ldots<s_j} \frac{c(s_1)}{s_1} \ldots \frac{c(s_{j-1})}{s_{j-1}} \frac{c(s_j)}{s_j^{3/2 - 1/p}}\ ds_1 \ldots ds_j;$$
using \eqref{sss} we obtain the claim.  

To establish \eqref{al7}, we first observe from \eqref{al3} and Fourier analysis that
$$ \| \int_{s < s_1} \nabla_{t,x} \psi_s(s_1)\ ds_1 \|_{L^2_x(\R^2)} \lesssim (\int_s^\infty c(s')^2 \frac{ds}{s'})^{1/2}$$
so it suffices by \eqref{psi-odd}, \eqref{psi-even} to show that
$$ \int_{s<s_1<\ldots<s_j} \| \psi_s(s_1) \ldots \psi_s(s_{j-1}) \nabla_{t,x} \psi_s(s_j) \|_{L^2_x(\R^2)}\ ds_1 \ldots ds_j \lesssim O(1)^j j^{-j/2} c(s) $$
for all $j>1$.  But if one uses H\"older's inequality, placing $\psi_s(s_1)$ in $L^2$ and the other factors in $L^\infty$, then uses \eqref{al3} and \eqref{sss} as before, one obtains the claim.
\end{proof}

\subsection{The energy space}\label{esd-sec}

We can define the energy $\E(\psi_s[t_0])$ of an instantaneous dynamic field to be $\E(\phi[t_0])$, where $\phi[t_0]$ is any instantaneous map reconstructed from $\psi_s[t_0]$; it is easy to see that this quantity is well defined, and is independent of $t_0$ if $\psi_s$ solves the wave map equation.  

In \cite{tao:heatwave2}, \cite{tao:heatwave3}, the energy metric $\dist_\Energy(\psi_s[t_0], \psi'_s[t_0])$ between two instantaneous dynamic field heat flows was defined by the formula
\begin{equation}\label{energy-dist-def-old}
\dist_\Energy(\psi_s[t_0], \psi'_s[t_0]) := (\| \psi_s(t_0) - \psi'_s(t_0) \|_{L^2_{s,x}(\R^+ \times \R^2)}^2 + \frac{1}{2} \| \psi_t(0,t_0) - \psi'_t(0,t_0) \|_{L^2_x(\R^2)}^2.
\end{equation}
Its properties are studied in \cite{tao:heatwave2}.  The relationship between this metric and the energy for heat flows is given by the identity
$$ \E(\psi_s[t_0]) = \dist_\Energy(\psi_t[t_0],0)^2.$$
We can also define the energy metric $\dist_\Energy(\phi[t_0], \phi'[t_0])$ between two instantaneous maps by
$$ \dist_\Energy(\phi[t_0], \phi'[t_0]) = \inf \dist_\Energy(\psi_s[t_0], \psi'_s[t_0]),$$
where $\psi_s, \psi'_s$ range over heat flow resolutions of $\phi, \phi'$; the rotation invariance of the energy metric on heat flows shows that this is indeed a (semi-)metric; if we quotient the space $\S$ of instantaneous maps by rotations \eqref{rotate-data}, then we recover a genuine metric (see \cite{tao:heatwave2} for details).  We define the energy space $\Energy$ to be the metric completion of $\S$.

In this paper, we need to extend the energy space to other instantaneous dynamic fields that are not necessarily heat flows.  It turns out that the metric as defined in \eqref{energy-dist-def-old} is no longer suitable, because without the parabolic regularising effect of the heat flow, it is insufficient to control certain higher derivatives of the field.  Instead, we will use the following more technical (and artificial) metric, which (implicitly) appeared in \cite{tao:heatwave3}:

\begin{definition}[Modified energy metric]  Let $\psi_s[t_0],\psi'_s[t_0]$ be two instantaneous dynamic fields.  The (modified) energy metric $\tilde \dist_\Energy(\psi_s[t_0], \psi'_s[t_0])$ between these flows is given by the formula
\begin{equation}\label{energy-dist-def}
\tilde \dist_\Energy(\psi_s[t_0], \psi'_s[t_0]) := \sum_{j=0}^{10} \sum_k ( \sup_{2^{-2k-2} \leq s \leq 2^{-2k}} s^{1+j/2} \| \nabla_x^j \nabla_{t,x} (\psi_s - \psi'_s)(s,t_0) \|_{\dot H^0_k(\R^2)}^2 )^{1/2}.
\end{equation}
\end{definition}

It is easy to see that this is a metric on the space of instantaneous dynamic fields.  When restricted to heat flows, the modified energy metric and the ordinary metric are compatible:

\begin{proposition}[Compatibility of metrics]\label{compati}  Let $\psi_s[t_0], \psi'_s[t_0]$ be instantaneous heat flows.
\begin{itemize}
\item[(i)] If $\E(\psi_s[t_0]) \leq E$, then $\tilde \dist_\Energy(0,\psi_s[t_0]) \lesssim_E 1$.
\item[(ii)] Conversely, if $\tilde \dist_\Energy(0,\psi_s[t_0]) \leq M$, then $\E(\psi_s[t_0]) \lesssim_M 1$.  
\item[(iii)] If $\E(\psi_s[t_0]), \E(\psi'_s[t_0]) \leq E$, then $\tilde \dist_\Energy(\psi_s[t_0], \psi'_s[t_0]) \sim_E \dist_\Energy(\psi_s[t_0], \psi'_s[t_0])$.
\end{itemize}
\end{proposition}

\begin{proof}  We begin with (i).  If we allow implied constants to depend on $E$, and omit the $t_0$ parameter, we see from \eqref{energy-dist-def} that it will suffice to show that
\begin{equation}\label{shoa}
 \| \nabla_x^j \nabla_{t,x} \psi_s(s) \|_{\dot H^0_{k(s)}(\R^2)} \lesssim c(s) s^{-(j+2)/2}
\end{equation}
for all $0 \leq j \leq 10$ and $s \geq 0$, and some frequency envelope $c$ of energy $O(1)$.  But this follows from \cite[Lemma 7.6]{tao:heatwave3}.

Now we show (ii).  We allow implied constants to depend on $M$, omit the $t_0$ parameter.  By \eqref{energy-dist-def}, we can find a frequency envelope $c$ of energy $O(1)$ such that \eqref{shoa} holds for all $0 \leq j \leq 10$ and $s \geq 10$.  
We now claim that
$$
\| \Psi_{t,x}(s) \|_{L^p_x(\R^2)} \lesssim c(s) s^{\frac{1}{p}-\frac{1}{2}}$$
for all $s>0$ and $4 \leq p<\infty$ (say).  To see this, we may rescale $s=1$.  By \eqref{psi-odd}, \eqref{psi-even}, it suffices to show that
\begin{equation}\label{ssoj}
\| \int_{1 < s_1 < s_2 < \ldots < s_j} \psi_s(s_1) \ldots \psi_s(s_{j-1}) \nabla_{t,x} \psi_s(s_j)\ ds_1 \ldots ds_j \|_{L^p_x(\R^2)}
\lesssim_p O(1)^j j^{-j/2} c(1)
\end{equation}
for all $j \geq 0$.  But from \eqref{shoa} and the Gagliardo-Nirenberg inequalities one has
$$ \|\psi_s(s)\|_{L^\infty_x(\R^2)} \lesssim c(s) s^{-1}$$
and
$$ \|\nabla_{t,x} \psi_s(s)\|_{L^p_x(\R^2)} \lesssim c(s) s^{\frac{1}{p}-\frac{3}{2}}$$
and so by the H\"older and Minkowski inequalities and \eqref{sss}, one can bound the left-hand side of \eqref{ssoj} by
$$ \lesssim_\eps O(1)^j j^{-j/2} \int_1^\infty c(s) s^{\frac{1}{p}-\frac{3}{2}+\delta_0}\ ds$$
and the claim follows from the envelope property \eqref{sm}.

A similar argument also yields that
$$
\| \nabla_x \Psi_{t,x}(s) \|_{L^2_x(\R^2)} \lesssim c(s) s^{-1/2}$$
for all $s > 0$.  From this and \eqref{heat-eq} we see that
$$
\| \psi_s(s) \|_{L^2_x(\R^2)} \lesssim c(s) s^{-1/2}$$
and thus
$$ \int_0^\infty \|\psi_s(s)\|_{L^2_x(\R^2)}^2\ ds \lesssim 1.$$
To conclude, we need to show that $\| \psi_t(0)\|_{L^2_x(\R^2)} \lesssim 1$.  Using \eqref{psi-odd}, it suffices to show that
$$ \| \int_1^\infty X_j(s)\ ds \|_{L^2_x(\R^2)} \lesssim O(1)^j j^{-j/2} $$
for each $j \geq 1$, where
$$ 
X_j(s_1) := \int_{s_1 < s_2 < \ldots < s_j} \psi_s(s_1) \ldots \psi_s(s_{j-1}) \nabla_{t,x} \psi_s(s_j)\ ds_2 \ldots ds_j.$$
By Fourier analysis it suffices to show that
$$ \| X_j(s_1) \|_{\dot H^0_{k(s_1)}(\R^2)} \lesssim O(1)^j j^{-j/2} c(s_1) s_1^{-1}$$
for all $s_1>0$.  From \eqref{shoa} one has 
$$\|\psi_s(s) \|_{\dot H^1_{k(s)}(\R^2)}, \|\nabla_{t,x} \psi_s(s) \|_{\dot H^0_{k(s)}(\R^2)} \lesssim c(s) s^{-1}$$
so by Minkowski's inequality and Lemma \ref{ek-prod}, one can bound
$$ \| X_j(s_1) \|_{\dot H^0_{k(s_1)}(\R^2)} \lesssim O(1)^j
\int_{s_1 < s_2 < \ldots < s_j} \frac{c(s_1)}{s_1} \ldots \frac{c(s_{j-1})}{s_{j-1}} \frac{c(s_j)}{s_j} (s_j/s_1)^{-\delta_1}\ ds_2 \ldots ds_j;$$
using \eqref{sss} we can bound the right-hand side by
$$ \lesssim O(1)^j j^{-j/2} \int_{s_1}^\infty \frac{c(s)}{s} (s/s_1)^{-\delta_1+\delta_0}\ ds$$
and the claim follows.

Now we show (iii).  The upper bound $\tilde \dist_\Energy(\psi_s[t_0], \psi'_s[t_0]) \lesssim_E \dist_\Energy(\psi_s[t_0], \psi'_s[t_0])$ follows from \cite[Lemma 7.6]{tao:heatwave3}, so we turn to the lower bound.  We again allow implied constants to depend on $E$ and omit the $t_0$ parameter.
By (i) and \eqref{energy-dist-def}, there exists frequency envelopes $c(s)$, $\delta c(s)$ of energy $O(1)$, $O( \dist_\Energy(\psi_s[t_0], \psi'_s[t_0])^2 )$ respectively such that 
$$
 \| \nabla_x^j \nabla_{t,x} \psi^*_s(s) \|_{\dot H^0_{k(s)}(\R^2)} \lesssim c(s) s^{-(j+2)/2}$$
and
$$
 \| \nabla_x^j \nabla_{t,x} \delta \psi_s(s) \|_{\dot H^0_{k(s)}(\R^2)} \lesssim \delta c(s) s^{-(j+2)/2}$$
for all $0 \leq j \leq 10$ and $s > 0$, where $\psi^*_s := (\psi_s,\psi'_s)$ and $\delta \psi_s := \psi_s - \psi'_s$.  It will suffice to show that
\begin{equation}\label{simo}
\|\delta \psi_s(s) \|_{L^2_x(\R^2)} \lesssim \delta c(s) s^{-1/2}
\end{equation}
for $s>0$, and
\begin{equation}\label{simo-2}
 \|\delta \psi_t(0) \|_{L^2_x(\R^2)} \lesssim \dist_\Energy(\psi_s[t_0], \psi'_s[t_0]).
 \end{equation}
But by adapting the arguments used to prove (ii) to differences one sees that
\begin{align*}
\| \Psi^*_{t,x}(s) \|_{L^p_x(\R^2)} &\lesssim c(s) s^{\frac{1}{p}-\frac{1}{2}}\\
\| \delta \Psi_{t,x}(s) \|_{L^p_x(\R^2)} &\lesssim \delta c(s) s^{\frac{1}{p}-\frac{1}{2}}\\
\| \nabla_x \Psi^*_{t,x}(s) \|_{L^2_x(\R^2)} &\lesssim c(s) s^{-1/2}\\
\| \nabla_x \delta \Psi_{t,x}(s) \|_{L^2_x(\R^2)} &\lesssim \delta c(s) s^{-1/2}
\end{align*}
for all $s>0$ and $4 \leq p < \infty$, which gives \eqref{simo}; in a similar vein, the arguments used to prove (ii) yield
$$
\| \delta X_j(s_1) \|_{\dot H^0_{k(s_1)}(\R^2)} \lesssim O(1)^j j^{-j/2} \delta c(s_1) s_1^{-1}$$
for all $s_1 > 0$, which then gives \eqref{simo-2} as before.
\end{proof}

One corollary of this is

\begin{corollary}[Heat flow contracts distances]\label{heato}  Let $\psi_s[t_0], \psi'_s[t_0]$ be instantaneous dynamic fields.  Then
Then for any $s_0 > 0$, one has
$$ \tilde \dist_\Energy(\psi_s(s_0+\cdot)[t], \psi'_s(s_0+\cdot)[t]) \lesssim \tilde \dist_\Energy(\psi_s[t_0], \psi'_s[t_0]).$$
If furthermore $\psi_s[t_0], \psi'_s[t_0]$ are heat flows of energy at most $E$, we have
$$ \dist_\Energy(\psi_s(s_0+\cdot)[t], \psi'_s(s_0+\cdot)[t]) \lesssim_E \dist_\Energy(\psi_s[t_0], \psi'_s[t_0])$$
or equivalently
$$ \dist_\Energy(\phi(s_0)[t], \phi(s_0)[t]) \lesssim_E \dist_\Energy(\phi(0)[t_0], \phi(0)[t_0]).$$
\end{corollary}

\begin{proof} The first claim follows from an inspection of \eqref{energy-dist-def}.  The second claim then follows from the first by Lemma \ref{compati}.
\end{proof}

We can now recall the \emph{energy spectral distribution} from \cite{tao:heatwave4}.  Given any instantaneous dynamic field $\psi_s[t_0]$, we define the \emph{energy spectral distribution} $\ESD(\psi_s[t_0]): \R^+ \to \R^+$ by the formula
\begin{equation}\label{esd-def}
\ESD(\psi_s[t_0])(s) := \| \psi_s(s) \|_{L^2_x(\R^2)}^2 + \| D_x \psi_t(s) \|_{L^2_x(\R^2)}^2 + \frac{1}{2} \| \psi_t \wedge \psi_x(s) \|_{L^2_x(\R^2)}^2.
\end{equation}
Note that this quantity is invariant under rotations of the field $\psi_s[t_0]$.  Thus we may abuse notation and also write $\ESD(\phi[t_0]) = \ESD(\psi_s[t_0])$ for any instantaneous map $\phi[t_0]$ reconstructed from $\psi_s[t_0]$.

The properties of the energy spectral distribution were discussed in \cite{tao:heatwave4}.  For now, we simply record the energy identity
\begin{equation}\label{energy-ident}
\E(\psi_s[t_0]) = \int_0^\infty \ESD(\psi_s[t_0])(s)\ ds;
\end{equation}
see \cite{tao:heatwave4}.

For future reference we record the following convenient lemma.

\begin{lemma}[Energy gap implies energy stability]\label{gapstable}  Let $\phi[t_0] \in \S$ have energy at most $E$, let $\phi(s)[t_0]$ be the heat flow generated by this data, and let $0 < s_1 < s_2$ be such that
\begin{equation}\label{psisep}
 \int_{\eps s_1}^{s_2/\eps} \ESD(\phi[t_0])(s)\ ds \leq \eps
 \end{equation}
for some $E,\eps > 0$.
Then
$$ \dist_\Energy( \phi(s_1)[t_0], \phi(s_2)[t_0] ) \leq o_{\eps \to 0;E}(1).$$
In particular, by Lemma \ref{compati}, we have
$$ \tilde \dist_\Energy( \phi(s_1)[t_0], \phi(s_2)[t_0] ) \leq o_{\eps \to 0;E}(1).$$
\end{lemma}

\begin{proof}  We fix $E$ and allow constants to depend on $E$, and abbreviate $o_{\eps \to 0;E}(1)$ as $o(1)$.  We may normalise $s_2=1$, and omit the explicit mention of $t_0$. We may assume that $\eps$ is small (e.g. $\eps < 1/100$) since the claim follows from the triangle inequality otherwise. By \eqref{energy-dist-def}, we need to show that
\begin{equation}\label{psiss}
\int_0^\infty \| \psi_s(s+s_1) - \psi_s(s+1)\|_{L^2_x(\R^2)}^2\ ds = o(1)
\end{equation}
and
\begin{equation}\label{psits}
\| \psi_t(s_1) - \psi_t(1) \| = o(1).
\end{equation}
We begin with \eqref{psiss}.  From \eqref{psisep}, \eqref{esd-def} we have
\begin{equation}\label{jam}
\int_{\eps s_1}^{1/\eps} \| \psi_s(s) \|_{L^2_x(\R^2)}^2 + \| D_x \psi_t(s) \|_{L^2_x(\R^2)}^2\ ds = O(\eps).
\end{equation}
From this and the triangle inequality we see that the contribution of the interval $0 \leq s \leq \eps^{-1/2}$ to \eqref{psiss} is acceptable, so it suffices to control the region where $s>\eps^{-1/2}$.  But from Proposition \ref{corbound} one has
$$ \| \partial_s \psi_s(s) \|_{L^2_x(\R^2)} \lesssim s^{-3/2}$$
and so by the fundamental theorem of calculus and Minkowski's inequality
$$ \| \psi_s(s+s_1) - \psi_s(s+1)\|_{L^2_x(\R^2)} \lesssim s^{-3/2}$$
and the claim follows.

Now we establish \eqref{psits}.  By \eqref{psit-eq}, the fundamental theorem of calculus, and the triangle inequality it suffices to show that
\begin{equation}\label{ddpsi}
\| \int_{s_1}^1 D_i D_i \psi_t(s)\ ds \|_{L^2_x(\R^2)} = o(1)
\end{equation}
and
\begin{equation}\label{psitxx}
\| \int_{s_1}^1 |\psi_t(s)| |\psi_x(s)|^2\ ds \|_{L^2_x(\R^2)} = o(1).
\end{equation}
We begin with \eqref{ddpsi}.  Squaring and using symmetry, it suffices to show that
$$ \int\int_{s_1 \leq s < s' \leq 1} |\langle D_i(s) D_i(s) \psi_t(s), D_j(s') D_j(s') \psi_t(s') \rangle_{L^2_x(\R^2)}|\ ds ds' = o(1),$$
where $D_i(s) = \partial_i + A_i(s)$.  We integrate by parts and use H\"older's inequality, and reduce to showing that
\begin{equation}\label{ddpsi-1}
 \int\int_{s_1 \leq s < s' \leq 1} \|A_i(s)\|_{L^2_x(\R^2)} \|D_i(s) \psi_t(s)\|_{L^2_x(\R^2)} \|D_j(s') D_j(s') \psi_t(s')\|_{L^\infty_x(\R^2)}\ ds ds' = o(1)
 \end{equation}
and
\begin{equation}\label{ddpsi-2}
 \int\int_{s_1 \leq s < s' \leq 1} \| D_i(s) \psi_t(s)\|_{L^2_x(\R^2)} \| \nabla_x D_j(s') D_j(s') \psi_t(s') \|_{L^2_x(\R^2)}\ ds ds' = o(1).
\end{equation}
We first check \eqref{ddpsi-1}.  From Proposition \ref{corbound}, $\|A_i(s)\|_{L^2_x(\R^2)}=O(1)$, so by \eqref{jam} and Cauchy-Schwarz it suffices to show that
$$ \int_{s_1}^1 (\int_s^1 \|D_j(s') D_j(s') \psi_t(s')\|_{L^\infty_x(\R^2)}\ ds')^2\ ds = O(1).$$
But from Proposition \ref{corbound} one has
$$ \int_0^\infty s^2 \|D_j(s) D_j(s) \psi_t(s)\|_{L^\infty_x(\R^2)}^2\ ds' = O(1)$$
and the claim now follows from Hardy's inequality.

The proof of \eqref{ddpsi-2} proceeds exactly as in \eqref{ddpsi-1}, but with $\|\nabla_x D_j D_j \psi_t\|_{L^2_x(\R^2)}$ plyaing the role of $\|D_j D_j \psi_t \|_{L^\infty_x(\R^2)}$ (and with no $A_i$ term to discard).  Finally we show \eqref{psitxx}.  From Proposition \ref{corbound} one has
$$ 	
\int_0^\infty s^{-1/2} \| \psi_x(s) \|_{L^4_x(\R^2)}^2\ ds = O(1)$$
and $\| \psi_x(s)\|_{L^\infty_x(\R^2)} \lesssim s^{-1/2}$, so by H\"older's inequality it suffices to show that
$$ \int_{s_1}^1 s^{-1/2} \| \psi_t(s) \|_{L^4_x(\R^2)}^2\ ds = o(1).$$
But from Proposition \ref{corbound} we have
$$ \int_{s_1}^1 s^{-2/3} \| \psi_t(s) \|_{L^3_x(\R^2)}^2\ ds = O(1)$$
so from the Gagliardo-Nirenberg inequality
$$ \| f \|_{L^4_x(\R^2)} \lesssim \| f\|_{L^3_x(\R^2)}^{3/4} \| \nabla_x |f| \|_{L^2_x(\R^2)}^{1/4}$$
and H\"older's inequality, it suffices to show that
$$ \int_{s_1}^1 \| \nabla_x |\psi_t|(s) \|_{L^2_x(\R^2)}^2\ ds = o(1).$$
But this follows from \eqref{jam} and the diamagnetic inequality $|D_x \psi_t| \leq \nabla_x |\psi_t|$.
\end{proof}

\subsection{How to exploit delocalisation}\label{delocal-sec}

Three of the main results in this paper, namely Theorems \ref{freqbound}, \ref{spacbound}, \ref{spacdeloc}, are instances of the following general situation.  We have a classical wave map $(\phi,I)$ (which, in practice, will depend on an asymptotic parameter $n$) which we wish to control accurately (in particular, we wish to control the $(A,\mu)$-entropy of this map).  At some initial time $t_0 \in I$ (usually we will take $t_0=0$), we have data $\phi[t_0]$, which we have resolved (via a caloric gauge) into a heat flow $\psi_s[t_0]$; we similarly resolve the rest of $\phi$ into a heat flow $\psi_{s}$.  We suppose that this data $\phi[t_0]$ is somehow \emph{delocalised}, in that it is spread out among two widely separated regions in\footnote{Not coincidentally, each of these scenarios corresponds to one of the non-compact symmetries  of the wave maps equation, namely \eqref{cov-scaling}, \eqref{space-trans}, and \ref{time-trans} respectively.  The remaining non-compact symmetry \eqref{rotate} is completely decoupled from the analysis by the passage to the dynamic field.} frequency (in the case of Theorem \ref{freqbound}), space (in the case of Theorem \ref{spacdeloc}), or time (in the case of Theorem \ref{spacbound}). As the instantaneous map $\phi[t_0]$ is not scalar, one cannot isolate these distinct regions on the level of the map; however, by passing to the instantaneous dynamic field $\psi_s[t_0]$, it now becomes possible to use linear decomposition tools (e.g. cutoff functions and Littlewood-Paley projections) to decompose
$$ \psi_s[t_0] = \tilde \psi_s^1[t_0] + \tilde \psi_s^2[t_0] + \operatorname{error}$$
where $\tilde \psi_s^1[t_0]$, $\tilde \psi_s^2[t_0]$ are two widely separated components of $\psi_s[t_0]$ (e.g. low and high frequency components, or components supported near the spatial origin and far away from that origin), and $\operatorname{error}$ is a small error which we will neglect. While $\psi_s[t_0]$ is a heat flow, the components $\tilde \psi_s^{1,2}[t_0] = (\tilde \psi_s^1[t_0], \tilde \psi_s^2[t_0])$ need not be; however, in practice they will be \emph{approximate} heat flows, and because of this it will be possible to perturb the instantaneous dynamic fields $\tilde \psi_s^{1,2}[t_0]$ into dynamic fields $\psi_s^{1,2}[t_0]$ which are genuine heat flows.

The hypotheses in our theorems will allow us to ensure that the maps reconstructed from these instantaneous dynamic fields $\psi_s^{1,2}[t_0]$ have energy strictly less than the initial energy $E_0$, and so by hypothesis they can be extended to dynamic fields $\psi_s^{1,2}$ on all of $I$ that obey both the heat flow equation and the wave map equation (and will obey good entropy bounds).  Because $\psi_s^1[t_0], \psi_s^2[t_0]$ were already widely separated, we will eventually be able to show that $\psi_s^1$ and $\psi_s^2$ remain essentially separated on the remainder of the time interval $I$.

Now we reconstitute the two components by defining
$$ \tilde \psi_s := \psi_s^1 + \psi_s^2,$$
thus $\tilde \psi_s$ is a dynamic field.  Because of the nonlinear nature of the heat flow and wave map equations, the field $\tilde \psi_s$ need not obey these equations exactly; however, in practice the separation properties of $\psi_s^1$ and $\psi_s^2$ will allow us to show that $\tilde \psi_s$ solves these equations \emph{approximately}.  For similar reasons, the initial data $\tilde \psi_s[t_0]$ of $\tilde \psi_s$ will be close to $\psi_s[t_0]$.  Because of this, we will be able to show that $\tilde \psi_s$ and $\psi_s$ are close to each other, which allows us to control $\phi$ as desired.

The above strategy is depicted schematically by the following diagram:

\newarrow{Corresponds}<--->

\begin{equation}\label{tstruct}
\begin{diagram}
\phi          & \rCorresponds & \psi_s         & \lTo      & \tilde \psi_s                & \lTo & \psi^{1,2}_s          & \rCorresponds & \phi^{1,2}\\
\uCorresponds &               & \uCorresponds  & \ldDotsto &                              &      & \uCorresponds         &               & \uCorresponds\\
\phi[t_0]     & \rCorresponds & \psi_s[t_0]& \rTo      & \tilde \psi^{1,2}_s[t_0] & \rTo & \psi^{1,2}_s[t_0] & \rCorresponds & \phi^{1,2}[t_0]
\end{diagram}
\end{equation}

To summarise, to control the evolution $\phi$ of initial data $\phi[t_0]$ under the wave map equation, one first converts the data $\phi[t_0]$ to an instantaneous dynamic field heat flow $\psi_s[t_0]$ (thus similarly converting $\phi$ to a dynamic field heat flow $\psi_s$, which continues to solve the wave map equation at $s=0$). Then, one linearly decomposes $\psi_s[t_0]$ into well-separated components $\tilde \psi^{1,2}_s[t_0]$.  This breaks the heat flow property, but we repair it to create instantaneous heat flows $\psi^{1,2}_s[t_0]$, which can be used to reconstruct data $\phi^{1,2}[t_0]$ of energy strictly smaller than $E_0$, and which are well-separated from each other.  We then use the induction hypothesis to evolve $\phi^{1,2}[t_0]$ to $\phi^{1,2}$ (and thus $\psi^{1,2}_s[t_0]$ to $\psi^{1,2}_s$) by the wave map equation (hopefully preserving the well-separation property), then linearly superimpose the dynamic fields $\psi^{1,2}_s$ to create $\tilde \psi_s$.  This superposition only obeys the heat flow and wave map equations approximately, but has initial data close to the data $\psi_s[t_0]$, and so is close to $\psi_s$, and can thus be used to control $\phi$ as required.

In order to make this scheme work, we will (broadly speaking) need five key tools (in addition to the more standard theory of linear cutoffs in space and frequency):

\begin{enumerate}
\item An \emph{instantaneous parabolic repair theory}, which takes an instantaneous dynamic field, such as $\tilde \psi^i_s[t_0]$, which \emph{approximately} obeys the heat flow equation, and perturbs it to another instantaneous dynamic field $\psi^i_s[t_0]$ which obeys the heat flow equation \emph{exactly}.
\item A \emph{non-instantaneous parabolic repair theory}, which takes a dynamic field such as $\tilde \psi_s$, which \emph{approximately} obeys the heat flow equation, and perturbs it to another dynamic field $\psi'_s$ which obeys the heat flow equation \emph{exactly}.  
\item A \emph{hyperbolic repair theory}, which takes a dynamic field heat flow $\psi'_s$ which obeys the wave maps equation \emph{approximately}, and perturbs it to obtain another dynamic field heat flow $\psi''_s$ which obeys the wave map equation \emph{exactly}, and which has the same initial data: $\psi'_s[t_0] = \psi''_s[t_0]$.
\item A \emph{stability theory}, which takes a dynamic field heat flow $\psi''_s$ that obeys the wave maps equation \emph{exactly}, and which matches a specified set of initial data $\psi_s[t_0]$ \emph{approximately}, and perturbs it to obtain another dynamic field heat flow $\psi_s$ that obeys the wave maps equation \emph{exactly}, and also matches the initial data $\psi_s[t_0]$ \emph{exactly}.
\item A theory to ensure that well-separated initial data lead to solutions which are also well-separated.
\end{enumerate}

The purpose of the next few sections is to establish these sorts of tools, in preparation for establishing Theorems \ref{freqbound}, \ref{spacbound}, \ref{spacdeloc}.  The first four tools are general and apply equally to all three theorems; the fifth tool has to be adapted to each of these theorems separately.  

In the course of setting up these tools, we will also introduce a number of key function spaces which will be used in these theorems, and also to establish Theorem \ref{symscat}.

\section{Abstract parabolic repair theory}

As discussed in Ingredients 1 and 2 in Section \ref{delocal-sec}, we will need to establish a parabolic repair theory both for instantaneous and non-instantaneous heat flows.  It will be convenient to unify these two theories together as an abstract parabolic repair theory involving an unspecified family of function spaces obeying certain axioms; the two specific parabolic repair theories we will then will be special cases of this abstract theory using specific spaces.

We first recall a useful definition from \cite{tao:heatwave3}.

\begin{definition}[Algebra family]\label{algebra-fam}  Let $I$ be a compact interval.  An \emph{algebra family} is a collection of continuous seminorms $S_k(I \times \R^2)$ on $\Sch(I \times \R^2)$ that obeys the product estimate that obeys the following estimates:
\begin{itemize}
\item The product estimate
\begin{equation}\label{prod1-abstract}
 \| \phi^{(1)} \phi^{(2)} \|_{S_{\max(k_1,k_2)}(I \times \R^2)} \lesssim \| \phi^{(1)} \|_{S_{k_1}(I \times \R^2)} \| \psi^{(2)} \|_{S_{k_2}(I \times \R^2)}
\end{equation}
for all $k_1,k_2 \in \Z$ and $\phi^{(1)}, \phi^{(2)} \in \Sch(I \times \R^2)$;
\item The parabolic regularity estimate
\begin{equation}\label{parreg-abstract}
 \| \nabla_x^2 \phi(s) \|_{S_k(I \times \R^2)} \lesssim s^{-1} \| \phi(s/2) \|_{S_k(I \times \R^2)} + \sup_{s/2 \leq s' \leq s} \| (\partial_s - \Delta) \phi(s') \|_{S_k(I \times \R^2)}
\end{equation}
for all smooth $\phi: \R^+ \to \Sch(I \times \R^2)$ and $s > 0$, as well as the variant
\begin{equation}\label{parreg-abstract-alt}
 \| \nabla_x^2 \phi(s) \|_{S_k(I \times \R^2)} \lesssim \| \nabla_x^2 \phi(s'') \|_{S_k(I \times \R^2)} + \sup_{s'' \leq s' \leq s} \| (\partial_s - \Delta) \phi(s') \|_{S_k(I \times \R^2)}
\end{equation}
for all smooth $\phi: \R^+ \to \Sch(I \times \R^2)$ and $s/2 \leq s'' \leq s$;
\item The additional parabolic regularity estimate
\begin{equation}\label{parreg-abstract2}
 \| \nabla_x^j e^{(s-s')\Delta} \phi \|_{S_{k(s)}(I \times \R^2)} \lesssim_j (s'/s)^{\delta_1/10} (s-s')^{-j/2} \| \phi \|_{S_{k(s')}(I \times \R^2)}
\end{equation}
for all smooth $\phi \in \Sch(I \times \R^2)$, $s > s' \geq 0$, and $j \geq 0$;
\item The comparability estimate
\begin{equation}\label{physical-abstract}  \| \phi\|_{S_{k_1}(I \times \R^2)} \lesssim \chi_{k_1=k_2}^{-\delta_1} \|\phi\|_{S_{k_2}(I \times \R^2)}
\end{equation}
for all $k_1,k_2 \in \Z$ and $\phi \in \Sch(I \times \R^2)$. 
\item The uniform bound
\begin{equation}\label{uniform} \|\phi\|_{L^\infty_{t,x}(I \times \R^2)} \lesssim \| \phi \|_{S_k(I \times \R^2)}
\end{equation}
for all $k\in \Z$ and $\phi \in \Sch(I \times \R^2)$.
\end{itemize}
\end{definition}

\begin{example} For any $t_0 \in I$ the seminorms 
$$ \|\phi \|_{S_k(I \times \R^2)} := \| \nabla_{t,x} \phi(t_0) \|_{\dot H^0_{k}(\R^2)}$$
form an algebra family, where the $\dot H^m_k$ norms were defined in \eqref{ek-def}; see Theorem \ref{parab-thm-abstract} below.  This is the example required for the instantaneous theory.  The non-instantaneous theory will instead need the much more complicated $S$-type spaces from earlier wave maps papers, in particular \cite{tao:wavemap2}.
\end{example}

We can now state the main result of this section.

\begin{theorem}[Abstract parabolic repair]\label{parab-thm-abstract}  Let $I$ be a compact interval, and let $S_k(I \times \R^2)$ be an algebra family. 
Let $\tilde \psi_s$ be a dynamic field on $I$, with the associated differentiated fields $\tilde \Psi_{s,t,x} = (\tilde \psi_{s,t,x}, \tilde A_{s,t,x})$.   

Let $c$ be a frequency envelope of at most $E$ for some $E>0$, and assume that $0 < \kappa < 1$ is sufficiently small depending on $E$.  We assume the bounds
\begin{equation}\label{psisjk-abstract}
\| \nabla_x^j \tilde \psi_s(s) \|_{S_{k(s)}(I \times \R^2)} \lesssim c(s) s^{-(j+2)/2}
\end{equation}
and
\begin{equation}\label{psisjk-star-abstract}
\| \nabla_x^j (\tilde \psi_s - \tilde D_i \tilde \psi_i) (s) \|_{S_{k(s)}(I \times \R^2)} \lesssim \kappa c(s) s^{-(j+2)/2}
\end{equation}
for all $0 \leq j \leq 20$ and $s > 0$.  

Let $\phi$ be a map reconstructed from $\tilde \psi_s$.  Then there exists a dynamic field heat flow $\psi_s$, which also has $\phi$ as a reconstruction, such that
\begin{align}
\| \nabla_x^j \Psi_x(s) \|_{S_{k(s)}(I \times \R^2)} &\lesssim_{E} c(s) s^{-(j+1)/2} \label{parab-1}\\
\| \nabla_x^j A_x(s) \|_{S_{k(s)}(I \times \R^2)} &\lesssim_{E} c(s)^2 s^{-(j+1)/2}\label{parab-2}\\
\| \nabla_x^j \psi_s(s) \|_{S_{k(s)}(I \times \R^2)} &\lesssim_{E} c(s) s^{-(j+2)/2}\label{parab-3}\\
\| \nabla_x^j (\tilde \Psi_x - \Psi_x)(s) \|_{S_{k(s)}(I \times \R^2)} &\lesssim_{E} \kappa c(s) s^{-(j+1)/2}\label{parab-4}\\
\| \nabla_x^j (\tilde A_x - A_x)(s) \|_{S_{k(s)}(I \times \R^2)} &\lesssim_{E} \kappa c(s)^2 s^{-(j+1)/2}\label{parab-5}\\
\| \nabla_x^j (\tilde \psi_s - \psi_s)(s) \|_{S_{k(s)}(I \times \R^2)} &\lesssim_{E} \kappa c(s) s^{-(j+2)/2}\label{parab-6}
\end{align}
for all $0 \leq j \leq 15$ and $s > 0$.  One can also obtain similar bounds with $\nabla_x^2$ replaced by $\partial_s$ as one wishes (e.g. one could replace $\nabla_x^j$ by $\nabla_x^{j-4} \partial_s^2$ if $j \geq 4$).
\end{theorem}

In the remainder of this section we establish Theorem \ref{parab-thm-abstract}.  We allow all implied constants to depend on $E$, and we abbreviate $S_{k(s)}(I \times \R^2)$ as $S_{k(s)}$.  The methods draw heavily from similar results in \cite[Section 6]{tao:heatwave3}.

From \eqref{psisjk-abstract} and (the proof of) \cite[Lemma 6.4]{tao:heatwave3} we have
\begin{equation}\label{bongo-1}
\| \nabla_x^j \tilde \psi_x(s) \|_{S_{k(s)}} \lesssim c(s) s^{-(j+1)/2}
\end{equation}
for all $0 \leq j \leq 19$; by (the proof of) \cite[Lemma 6.7]{tao:heatwave3} we then have
\begin{equation}\label{bongo-2}
\| \nabla_x^j \tilde A_x(s) \|_{S_{k(s)}} \lesssim c(s)^2 s^{-(j+1)/2}
\end{equation}
for all $0 \leq j \leq 19$.

We now introduce yet another time variable $\sigma$, and work on the domain
$$ \Omega := \{ (\sigma,s,t,x): 0 \leq \sigma \leq s; (t,x) \in I \times \R^2 \}.$$
We define the map $\phi: \Omega \to \H$ by imposing the boundary condition
$$ \phi(s,s,t,x) := \tilde \phi(s,t,x)$$
for all $(s,t,x) \in \R^+ \times I \times \R^2$, and then solving the heat flow
$$ \partial_s \phi = (\phi^* \nabla)_i \partial_i \phi$$
in the interior of $\Omega$.  The existence (and uniqueness) of a smooth solution $\phi$ to this system is guaranteed by of \cite[Theorem 3.16]{tao:heatwave2}.  Observe that $\phi(0,s,t,x)$ is nothing more than the evolution of the initial data $\phi(0,0,t,x) = \tilde \phi(0,t,x)$ under the harmonic map heat flow, so the main task will be to place a caloric gauge on $\phi(0,s,t,x)$ so that the associated differentiated fields $\Psi_{s,t,x}$ are close to $\tilde \Psi_{s,t,x}$ in the required sense.

We begin by placing a preliminary orthonormal frame $e \in \Gamma( \Frame(\phi^* T\H ) )$ on $\Omega$, defined by imposing the boundary condition
$$
e(s,s,t,x) = \tilde e(s,t,x)$$
and the secondary caloric gauge condition
$$ (\phi^* \nabla)_\sigma e = 0$$
or equivalently (in terms of the differentiated fields $\Psi_{\sigma,s,t,x} = (\psi_{\sigma,s,t,x}, A_{\sigma,s,t,x})$)
\begin{equation}\label{asig}
A_\sigma = 0.
\end{equation}
(We will need to use this secondary condition, rather than the more natural condition $A_s = 0$, because the latter would require one to solve one equation forwards in $s$ (the heat flow equation) and one equation backwards in $s$ (the caloric gauge condition), and one will not be able to close a bootstrap in the large energy case.)  The existence and smoothness of $e$ can be easily guaranteed by the Picard existence theorem for ODE.

On the diagonal $s=\sigma$, we observe from the boundary condition that 
$$ (\psi_s + \psi_\sigma)(s,s) = \tilde \psi_s(s)$$
while from the heat flow equation and the boundary condition we have
$$ \psi_s(s,s) = \tilde D_x \tilde \psi_x(s)$$
and hence by \eqref{psisjk-star-abstract} we have
\begin{equation}\label{sigma-0}
 \| \nabla_x^j \psi_\sigma(s,s) \|_{S_{k(s)}} \lesssim \kappa c(s) s^{-(j+2)/2}
\end{equation}
for all $0 \leq j \leq 20$ and $s > 0$.

We now extend this bound to the interior of $\Omega$.

\begin{proposition}\label{rego} If $C_0$ is a sufficiently large constant (depending only on $E$), and $C_1$ is sufficiently large depending on $C_0,E$, then one has the bounds
\begin{equation}\label{sigma-ok}
\| \nabla_x^j \psi_\sigma(\sigma,s) \|_{S_{k(s)}} \leq C_1 \kappa c(s) s^{-(j+2)/2}
\end{equation}
for all $0 \leq j \leq 18$ and $0 \leq \sigma \leq s$, and
\begin{equation}\label{splat-ok}
\| \nabla_x^j \psi_x(\sigma,s) \|_{S_{k(s)}} \leq C_0 c(s) s^{-(j+1)/2}
\end{equation}
and
\begin{equation}\label{asplat-ok}
\| \nabla_x^j A_x(\sigma,s) \|_{S_{k(s)}} \leq C_0 c(s)^2 s^{-(j+1)/2}
\end{equation}
for all $0 \leq j \leq 17$ and $0 \leq \sigma \leq s$.
\end{proposition}

\begin{proof}  It suffices to establish these bounds for $S' \leq \sigma \leq s \leq S$ and $0 \leq S' \leq S < \infty$, so long as $C_0$ is independent of $S, S'$.  

Fix $S$.  We shall use a continuity argument. The set of $S'$ for which the desired estimates hold is clearly closed (recall that the $S$ norms are smooth in the Schwartz topology), and contains an open neighbourhood of $S$ in $[0,S]$ by \eqref{sigma-0}, \eqref{bongo-1}, \eqref{bongo-2} (and continuity), so it suffices that it is also open in $(0,S)$.  We will do this by a showing (for $C_0$ large enough) that the hypothesised bounds \eqref{sigma-ok}, \eqref{splat-ok}, \eqref{asplat-ok} can be bootstrapped to the better bounds
\begin{equation}\label{sigma-ok-2}
\| \nabla_x^j \psi_\sigma(\sigma,s) \|_{S_{k(s)}} \leq \frac{1}{2} C_1 \kappa c(s) s^{-(j+2)/2}.
\end{equation}
for all $0 \leq j \leq 18$ and $S' \leq \sigma \leq s \leq S$, and
\begin{equation}\label{splat-ok-2}
\| \nabla_x^j \psi_x(\sigma,s) \|_{S_{k(s)}} \leq \frac{1}{2} C_0 c(s) s^{-(j+1)/2},
\end{equation}
and
\begin{equation}\label{asplat-ok-2}
\| \nabla_x^j A_x(\sigma,s) \|_{S_{k(s)}} \leq \frac{1}{2} C_0 c(s)^2 s^{-(j+1)/2},
\end{equation}
for all $0 \leq j \leq 17$ and $S' \leq \sigma \leq s \leq S$.

To prove this, we first observe from \eqref{constcurv}, \eqref{zerotor} that we have the equations
\begin{equation}\label{pops}
\partial_\sigma \psi_x = \nabla_x \psi_\sigma + \bigO( A_x \psi_\sigma )
\end{equation}
and
\begin{equation}\label{simmer}
 \partial_\sigma A_x = \bigO( \psi_x \psi_\sigma ).
\end{equation}
From \eqref{simmer}, the fundamental theorem of calculus, Minkowski's inequality, and the Leibniz rule, we have
$$
\|\nabla_x^j A_x(\sigma,s) - \nabla_x^j A_x(s,s) \|_{S_{k(s)}} \lesssim \sum_{j=j_1+j_2} \int_\sigma^s \| (\nabla_x^{j_1} \psi_x) (\nabla_x^{j_2} \psi_\sigma)(\sigma',s)\|_{S_{k(s)}}\ d\sigma';
$$
applying \eqref{splat-ok}, \eqref{sigma-ok}, \eqref{prod1-abstract} we conclude
$$ \|\nabla_x^j A_x(\sigma,s) - \nabla_x^j A_x(s,s) \|_{S_{k(s)}} \lesssim_{C_1} \kappa c(s)^2 s^{-(j+1)/2}$$
for all $0 \leq j \leq 17$ and $S' \leq \sigma \leq s \leq S$.  Applying \eqref{bongo-2} and the triangle inequality (and assuming $C_0$ sufficiently large, and $\kappa$ sufficiently small depending on $C_0,C_1$) we conclude \eqref{asplat-ok-2} as required.  A similar argument using \eqref{pops} gives 
$$ \|\nabla_x^j \psi_x(\sigma,s) - \nabla_x^j \psi_x(s,s) \|_{S_{k(s)}} \lesssim_{C_1} \kappa c(s) s^{-(j+1)/2}$$
for all $0 \leq j \leq 17$ and $S' \leq \sigma \leq s \leq S$, which gives \eqref{splat-ok-2}.

To get back to \eqref{sigma-ok-2}, we use the heat equation
$$
\partial_s \psi_\sigma = D_i D_i \psi_\sigma - (\psi_\sigma \wedge \psi_i) \psi_i$$ 
(cf. \eqref{psis-eq}, \eqref{psit-eq}).  We write this schematically as
\begin{equation}\label{joke}
 (\partial_s - \Delta) \psi_\sigma = \bigO( A_x \nabla_x \psi_\sigma ) + \bigO( (\nabla_x A_x) \psi_\sigma ) + \bigO( \Psi_x^2 \psi_\sigma ).
 \end{equation}
Fix $\sigma$, and write
$$ f(s) := \sum_{j=0}^{17} s^{(j+2)/2} \| \nabla_x^j \psi_\sigma(\sigma,s) \|_{S_{k(s)}}.$$
From \eqref{sigma-0} we have $f(\sigma) \lesssim \delta c(\sigma)$.  More generally, using \eqref{joke}, \eqref{splat-ok}, \eqref{asplat-ok} (or \eqref{splat-ok-2}, \eqref{asplat-ok-2}), \eqref{parreg-abstract2} (with $j=1,0$), \eqref{prod1-abstract}, Minkowski's inequality, and the Leibniz rule, we see that
\begin{align*}
f(s) &\lesssim_{C_0} (\sigma/s)^{\delta_1/10} f(\sigma) + \int_\sigma^s (s-s')^{-1/2} (s'/s)^{\delta_1/10} s^{-1/2} c(s')^2 f(s')\ ds\\
&\quad + \int_\sigma^s (s'/s)^{\delta_1/10}  s^{-1} c(s')^2 f(s') ds
\end{align*}
which we can simplify as
$$ f(s) \lesssim_{C_0} \delta c(s) + \int_\sigma^s (s-s')^{-1/2} (s'/s)^{\delta_1/10} c(s')^2 f(s')\ \frac{ds}{s}.$$
Applying \cite[Lemma 2.16]{tao:heatwave3} we conclude that
$$ f(s) \lesssim_{C_0} \delta c(s).$$
Inserting this back into the right-hand side of \eqref{joke} and using \eqref{splat-ok}, \eqref{asplat-ok}, \eqref{prod1-abstract} we see that
$$ \| \nabla_x^j (\partial_s - \Delta) \psi_\sigma(s) \|_{S_{k(s)}} \lesssim_{C_0} \kappa c(s) s^{-(j+4)/2} $$
for $0 \leq j \leq 16$ and $\sigma \leq s \leq S$; applying \eqref{parreg-abstract} we conclude that
$$ \| \nabla_x^j \psi_\sigma(s) \|_{S_{k(s)}} \lesssim_{C_0} \kappa c(s) s^{-(j+2)/2}$$
for $0 \leq j \leq 18$ and $2\sigma \leq s \leq S$, which gives \eqref{sigma-ok-2} when $s \geq 2 \sigma$; the corresponding claim for $s<2\sigma$ follows from using \eqref{parreg-abstract-alt} (and \eqref{sigma-0}) instead of \eqref{parreg-abstract}.  This closes the continuity argument and establishes the claim.
\end{proof}

We now allow all constants to depend on $C_0,C_1$.
We observe from the above proof that we have established the estimates
\begin{equation}\label{ax1}
 \|\nabla_x^j A_x(0,s) - \nabla_x^j \tilde A_x(s) \|_{S_{k(s)}} \lesssim \kappa c(s)^2 s^{-(j+1)/2}
 \end{equation}
and
\begin{equation}\label{ax2}
\|\nabla_x^j \psi_x(0,s) - \nabla_x^j \tilde \psi_x(s) \|_{S_{k(s)}} \lesssim \kappa c(s) s^{-(j+1)/2}
\end{equation}
for all $0 \leq j \leq 17$ and $s > 0$.

Meanwhile, from \eqref{constcurv} and \eqref{asig} we have
$$ \partial_\sigma A_s = \bigO( \psi_\sigma \psi_s );$$
from the fundamental theorem of calculus we conclude that
$$ \| \nabla_x^j A_s(0,s) \|_{S_{k(s)}} \lesssim \int_0^s \| \nabla_x^j (\psi_\sigma \psi_s)(\sigma,s) \|_{S_{k(s)}}\ d\sigma.$$
for any $j \geq 0$ and $s \geq 0$.  Applying \eqref{rego} we conclude the approximate caloric gauge condition
\begin{equation}\label{joyful}
 \| \nabla_x^j A_s(0,s) \|_{S_{k(s)}} \lesssim \kappa c(s)^2 s^{-(j+2)/2}
\end{equation}
for all $0 \leq j \leq 17$ and $s > 0$.  

The next step is to apply a gauge transform 
\begin{equation}\label{gauge}
\begin{split}
\phi \mapsto \phi; &\quad e \mapsto e U; \quad \psi_i \mapsto U^{-1} \psi_i; \\
D_i \mapsto U^{-1} D_i U; \quad &A_i \mapsto U^{-1} \partial_i U + U^{-1} A_i U
\end{split}
\end{equation}
for some $U: \{0\} \times \R^+ \times I \times \R^2 \to SO(m)$ that will transform $A_s(0,s)$ to zero, or in other words $U$ needs to solve the equation
\begin{equation}\label{use}
\partial_s U = - A_s U.
\end{equation}
We also need to let $U$ decay to the identity matrix in the limit $s \to \infty$.  To attain this boundary condition at infinity, we will do the usual trick of solving instead at a finite value of $s$, and then taking limits at the end.  More precisely, for any $0 < S < \infty$, let $U_S$ be the solution to \eqref{use} with $U_S(0,S) = I$.  From the bounds \eqref{joyful} and the product bound \eqref{prod1-abstract} we see that $U_S$ is defined in the completion of the Schwartz space under any of the $S_k$ norms, and obeys the bounds
\begin{equation}\label{usos}
 \| U_S(0,s) - I\|_{S_{k(s)}} \lesssim \kappa c(s)^2
\end{equation}
whenever $S/2 \leq s \leq 2S$, and we also have the semigroup law
\begin{equation}\label{semi}
U_S(0,s) = U_{s'}(0,s) U_S(0,s')
\end{equation}
for any $s, s', S > 0$.  Differentiating \eqref{use} in space and using Duhamel's formula, \eqref{joyful}, and Gronwall's inequality, we also see that
\begin{equation}\label{splatter}
 \| \nabla_x^j(U_S(0,s) - I)\|_{S_{k(s)}} \lesssim \kappa c(s)^2 s^{-j/2}
\end{equation}
for all $0 \leq j \leq 17$ and $S/2 \leq s \leq 2S$.

We can write \eqref{usos} schematically as
\begin{equation}\label{usos-2}
 U_S(0,s) = I + O_{S_k}( \kappa c(2^{-2k})^2 )
\end{equation}
whenever $S, s \sim 2^{-2k}$, where $O_{S_k}(X)$ denotes a quantity of $S_k$ norm $O(X)$.  If instead $S$ and $s$ are not comparable, but are widely separated with $s < S$, then we can use the semigroup law \eqref{semi} and \eqref{usos-2} to obtain
$$ U_S(0,s) = \prod_{s \lesssim 2^{-2k} \lesssim S} ( I + O_{S_k}( \kappa c(2^{-2k})^2 ) );$$
multiplying through using \eqref{prod1-abstract} and the envelope properties of $c$, we see that
$$ U_S(0,s) = I + \sum_{s \lesssim 2^{-2k} \lesssim S} O_{S_k}( \kappa c(2^{-2k})^2 ).$$
Differentiating using \eqref{splatter}, we see that
$$ \nabla_x^j U_S(0,s) = \sum_{s \lesssim 2^{-2k} \lesssim S} 2^{kj} O_{S_k}( \kappa c(2^{-2k})^2 )$$
for $1 \leq j \leq 17$; applying \eqref{physical-abstract}, we conclude that \eqref{splatter} in fact holds for all $1 \leq j \leq 17$ and $S \geq s$.

From \eqref{uniform}, \eqref{usos}, \eqref{semi}, we see that $U_S$ converges uniformly as $S \to \infty$ to a limit $U$, and we can write
\begin{equation}\label{us-1}
 U(0,s) = I + \sum_{2^{-2k} \gtrsim s} O_{S_k}( \kappa c(2^{-2k})^2 )
 \end{equation}
where the sum is uniformly convergent.  In particular $U(0,s)$ converges uniformly to the identity as $s \to \infty$.  From the derivative bounds we also have
\begin{equation}\label{us-2}
  \| \nabla_x^j U(0,s) \|_{S_{k(s)}} \lesssim \kappa c(s)^2 s^{-j/2}
\end{equation}
for all $1 \leq j \leq 17$ and $s > 0$.

We now apply \eqref{gauge} for $\sigma=0$ to obtain fields in the caloric gauge, which we shall denote temporarily by $\phi', e', \Psi'_{s,t,x} = (\psi'_{s,t,x}, A'_{s,t,x})$ to distinguish them from the fields on $\Omega$.  
From \eqref{gauge}, \eqref{us-1}, \eqref{us-2}, \eqref{prod1-abstract}, and Proposition \ref{rego} we see that
$$
\| \nabla_x^j (\Psi'_x(s) - \Psi_x(0,s)) \|_{S_{k(s)}(I \times \R^2)} \lesssim \kappa c(s) s^{-(j+1)/2}$$
and
$$
\| \nabla_x^j (A'_x(s) - A_x(0,s)) \|_{S_{k(s)}(I \times \R^2)} \lesssim \kappa c(s)^2 s^{-(j+1)/2}
$$
for $0 \leq j \leq 17$ and $s>0$; from \eqref{ax1}, \eqref{ax2} and the triangle inequality, we thus have
$$
\| \nabla_x^j (\Psi'_x(s) - \tilde \Psi_x(s)) \|_{S_{k(s)}(I \times \R^2)} \lesssim \kappa c(s) s^{-(j+1)/2}$$
and
$$
\| \nabla_x^j (A'_x(s) - \tilde A_x(s)) \|_{S_{k(s)}(I \times \R^2)} \lesssim \kappa c(s)^2 s^{-(j+1)/2}
$$
for the same range of $j$ and $s$.  In particular, by \eqref{bongo-1}, \eqref{bongo-2} one has
$$
\| \nabla_x^j \Psi'_x(s) \|_{S_{k(s)}(I \times \R^2)} \lesssim c(s) s^{-(j+1)/2}$$
and
$$
\| \nabla_x^j A'_x(s) \|_{S_{k(s)}(I \times \R^2)} \lesssim c(s)^2 s^{-(j+1)/2}
$$
for $0 \leq j \leq 17$ and $s>0$. As $\psi'_s = D'_i \psi'_i$, we then see from \eqref{prod1-abstract} that
$$
\| \nabla_x^j \psi'_s(s) \|_{S_{k(s)}(I \times \R^2)} \lesssim c(s) s^{-(j+2)/2}$$
for $0 \leq j\leq 16$ and $s>0$; a similar argument using \eqref{psisjk-star-abstract} gives
$$
\| \nabla_x^j (\psi'_s - \tilde \psi_s)(s) \|_{S_{k(s)}(I \times \R^2)} \lesssim \kappa c(s) s^{-(j+2)/2}.$$
We now obtain all the required claims of Theorem \ref{parab-thm-abstract} by dropping the primes from $\phi', e', \Psi'_{s,t,x}$.

\subsection{An instantaneous version}

We now specialise the above abstract theorem to an instantaneous version.

\begin{theorem}[Instantaneous parabolic repair]\label{parab-thm-instant}  Let $t_0$ be  atime, let $\tilde \psi_s[t_0]$ be an instantaneous dynamic field, with the associated fields $\tilde \Psi_{s,t,x} = (\tilde \psi_{s,t,x}, \tilde A_{s,t,x})$ be the associated differentiated fields.   

Let $c$ be a frequency envelope of at most $E$ for some $E>0$, and assume that $0 < \kappa < 1$ is sufficiently small depending on $E$.  We assume the bounds
\begin{equation}\label{psisjk-abstract-instant}
\| \nabla_{t,x} \nabla_x^j \tilde \psi_s(s,t_0) \|_{\dot H^0_{k(s)}(\R^2)} \lesssim c(s) s^{-(j+2)/2}
\end{equation}
and
\begin{equation}\label{psisjk-star-abstract-instant}
\| \nabla_{t,x} \nabla_x^j (\tilde \psi_s - \tilde D_i \tilde \psi_i) (s,t_0) \|_{\dot H^0_{k(s)}(\R^2)} \lesssim \kappa c(s) s^{-(j+2)/2}
\end{equation}
for all $0 \leq j \leq 20$ and $s > 0$.  

Let $\phi[t_0]$ be an instantaneous map reconstructed from $\tilde \psi_s[t_0]$.  Then there exists a instantaneous dynamic field heat flow $\psi_s[t_0]$, which also has $\phi[t_0]$ as a reconstruction, such that
\begin{align}
\| \nabla_{t,x} \nabla_x^j \Psi_x(s,t_0) \|_{\dot H^0_{k(s)}(\R^2)} &\lesssim_{E} c(s) s^{-(j+1)/2} \label{parab-1-instant}\\
\| \nabla_{t,x} \nabla_x^j A_x(s,t_0) \|_{\dot H^0_{k(s)}(\R^2)} &\lesssim_{E} c(s)^2 s^{-(j+1)/2}\label{parab-2-instant}\\
\| \nabla_{t,x} \nabla_x^j \psi_s(s,t_0) \|_{\dot H^0_{k(s)}(\R^2)} &\lesssim_{E} c(s) s^{-(j+2)/2}\label{parab-3-instant}\\
\| \nabla_{t,x} \nabla_x^j (\tilde \Psi_x - \Psi_x)(s,t_0) \|_{\dot H^0_{k(s)}(\R^2)} &\lesssim_{E} \kappa c(s) s^{-(j+1)/2}\label{parab-4-instant}\\
\| \nabla_{t,x} \nabla_x^j (\tilde A_x - A_x)(s,t_0) \|_{\dot H^0_{k(s)}(\R^2)} &\lesssim_{E} \kappa c(s)^2 s^{-(j+1)/2}\label{parab-5-instant}\\
\| \nabla_{t,x} \nabla_x^j (\tilde \psi_s - \psi_s)(s,t_0) \|_{\dot H^0_{k(s)}(\R^2)} &\lesssim_{E} \kappa c(s) s^{-(j+2)/2}\label{parab-6-instant}
\end{align}
for all $0 \leq j \leq 15$ and $s > 0$.  One can also obtain similar bounds with $\nabla_x^2$ replaced by $\partial_s$ as one wishes.
\end{theorem}

\begin{proof} By Theorem \ref{parab-thm-abstract} (and extending $\tilde \psi_s$ in time to a small neighbourhood of $t_0$), it suffices to verify that the seminorms
$$ \| \phi \|_{S_{k}} := \| \nabla_{t,x} \phi(t_0) \|_{\dot H^0_{k}(\R^2)}$$
form an algebra family.

The comparability estimate \eqref{physical-abstract} follows from \eqref{ek-def}, while the product estimate \eqref{prod1-abstract} follows from Lemma \ref{ek-prod}.  The uniform bound \eqref{uniform} follows from Bernstein's inequality (it also follows from \eqref{prod1-abstract} by the spectral theory of Banach algebras).  The parabolic regularity estimate \eqref{parreg-abstract2} is easily verified by Fourier analysis.

It remains to establish the estimates \eqref{parreg-abstract}, \eqref{parreg-abstract-alt}.  From \eqref{parreg-abstract2} and Duhamel's formula, it suffices to establish the estimate
$$ \| \nabla_x^2 \int_\infty^0 e^{s\Delta} F(s)\ ds\|_{\dot H^0_k(\R^2)} \lesssim \sup_{s>0} \|F(s)\|_{\dot H^0_k(\R^2)}$$
for any $k$ and any Schwartz function $F: \R^+ \times \R^2 \to \R$.  Taking Littlewood-Paley projections, it suffices to show that
$$ \| \nabla_x^2 \int_\infty^0 e^{s\Delta} F(s)\ ds\|_{L^2(\R^2)} \lesssim \sup_{s>0} \|F(s)\|_{L^2(\R^2)}$$
for all such $F$.  But this follows from Plancherel's theorem.
\end{proof}

\section{Function spaces and estimates}

In order to proceed further we need to recall some function spaces from earlier papers \cite{tao:wavemap2}, \cite{tao:heatwave3}, and then introduce some refinements of these spaces.

\subsection{Function spaces}

In \cite{tao:wavemap2} two function spaces $S_k(\R \times \R^2), N_k(\R \times \R^2)$ for $k \in \Z$ were constructed, in terms of some variants $S[k], N[k]$ of these spaces.  The precise definition of these spaces is extremely lengthy and will not be reviewed here.  These spaces were then restricted to slabs $I \times \R^2$, where $I$ is a compact interval, in the usual manner, thus
\begin{equation}\label{usk}
\|u\|_{S_k(I \times \R^2)} := \inf \{ \| \tilde u\|_{S_k(\R \times \R^2)}: \tilde u|_{I \times \R^2} = u \}
\end{equation}
where $\tilde u$ ranges over all extensions of $u$ to $S_k(\R \times \R^2)$.  The $S_k( I \times \R^2)$ norm was also very slightly modified in \cite{tao:heatwave3} to an equivalent norm $\overline{S_k(I \times \R^2)}$ which is continuous with respect to dilations (the $S_k$ norm is merely quasicontinuous).  For technical reasons that will be important later, it is also necessary to replace the $N_k$ norm with a stronger norm $N_k^\strong$; in Appendix \ref{nk-refine} we define this stronger norm, and show that all the properties established for the $N_k$ in \cite{tao:wavemap2} (and then inherited in \cite{tao:heatwave3}) continue to hold for this stronger norm\footnote{In the recent preprint \cite{sterbenz} it was shown, using the machinery of $U^p$ and $V^p$ spaces, that the key Strichartz and Wolff estimates that we need hold even without replacing $N_k$ with $N_k^\strong$; this provides an alternate approach to the selection of function spaces and verification of estimates.} $N_k^\strong$.

These $S_k$ and $N_k$ spaces were adequate for dealing with the small energy theory.  However, when dealing with large energy wave maps $\phi$, the $S_k(I \times \R^2)$ norms of various quantities related to $\phi$ also became quite large, even after localising $I$ to be small.  Thus, to obtain a suitable perturbation theory for large energy, one must modify the spaces slightly.

To deal with this problem, a variant $\| \|_{S_{\mu,k}(I \times \R^2)}$ of the $S_k(I \times \R^2)$ norm was introduced in \cite[Theorem 8.1]{tao:heatwave3}, where $\mu > 0$ was a small parameter.  The $S_{\mu,k}$ norm is larger than the $S_k$ (or $\overline{S_k}$) norm, but for sufficiently small time intervals $I$, the two norms are comparable.  Furthermore, in some key bilinear, and trilinear estimates enjoyed by the $S_k$ norms, a gain of a power of $\mu$ appears when $S_k$ is replaced by $S_{\mu,k}$.  Because of these facts, it was possible in \cite{tao:heatwave3} to transport the small energy theory\footnote{The theory in \cite{tao:wavemap2} was for spherical targets rather than hyperbolic ones, and used a microlocal gauge rather than the caloric gauge, but we will ignore these technical distinctions for this informal discussion.} of \cite{tao:wavemap2} to the large energy setting.

However, the norms $S_{\mu,k}(I \times \R^2)$ used in \cite{tao:heatwave3} are not suitable for this paper.  The reason is that in order to decompose a large time interval into smaller ones on which the $S_{\mu,k}$ norms of (say) an energy-class solution to the free wave equation are adequately controlled, the number of intervals required can become unbounded\footnote{This is ultimately because the $S_{\mu,k}$ norms defined in \cite{tao:heatwave3} involve the $L^1_t L^\infty_x$ norm, which is under control for frequency-localised solutions on extremely short time intervals, but becomes unbounded on large time intervals such as $\R$.}.  We shall refer to the property of being able to subdivide a large interval into a \emph{bounded} number of smaller intervals on which a norm is controlled as \emph{divisibility}.  As we will rely heavily on divisibility in this paper, it is thus necessary to replace the $S_{\mu,k}$ norm defined in \cite{tao:heatwave3} with a more divisible norm.  The downside to this is that it becomes more difficult to establish the required bilinear and trilinear estimates that gain a factor of the parameter $\mu$; indeed, we will need to rely on an inverse theorem for a basic bilinear estimate established recently by the author in \cite{tao:inverse}.

Nevertheless, the spaces $S_{\mu,k}$ can at least be \emph{defined} with relatively little effort.  Given a direction $\omega \in S^1$ and an initial position $x_0 \in \R^2$, as well as a frequency parameter $k$, define the infinite tube $T_{x_0,\omega,k} \subset \R \times \R^2$ to be the set
$$ T_{x_0,\omega,k} := \{ (t,x) \in \R \times \R^2: |x-x_0-\omega t| \leq 2^{-k} \}.$$
For any interval $I$, we say that a function $\phi: I \times \R^2 \to \R$ is \emph{Schwartz} if it is smooth, and $\partial_t^j \nabla_x^k \phi$ is rapidly decreasing in space for all $j,k \geq 0$, uniformly in time.  We let $\Sch(I)$ denote the space of all Schwartz functions, with the usual Frechet space topology.  

\begin{definition}[$S_{\mu,k}$ norm]\label{smuk-def}  Let $0 < \mu < 1$.  We let $C_0$ be a large absolute constant (e.g. $C_0 := 10^{10}$ will do; the exact value depends to some extent on the implied constants in \cite{tao:inverse}).  For any $k$ and $I$, we define the enhanced $S_k(I \times \R^2)^+$ space to be the subspace of $S_k(I \times \R^2)$ consisting of functions $\phi$ whose norm
\begin{equation}\label{phis-plus}
 \| \phi \|_{S_k^\strong(I \times \R^2)} := \|\phi\|_{S_k(I \times \R^2)} + \sup_{k'} \chi_{k=k'}^{-\delta_1} \|P_{k'} \Box \phi\|_{N_{k'}^\strong(I \times \R^2)}
 \end{equation}
is finite.  We then define $\tilde S_{\mu,k}(I \times \R^2)$ to be the atomic Banach space whose atoms $\phi \in S_k^\strong(I \times \R^2)$ are one of the following two types:
\begin{itemize}
\item (Small atoms) $\| \phi \|_{S_k^\strong(I \times \R^2)} \leq \mu$;
\item (Null-dispersed atoms) $\| \phi \|_{S_k^\strong(I \times \R^2)} \leq 1$ and $\| \phi \|_{S'_{\mu,k}(I \times \R^2)} \leq \mu^{C_0}$ and 
\begin{equation}\label{mondo}
\| P_{k'} \nabla_{t,x} \phi\|_{L^5_t L^\infty_x(I \times \R^2)} \lesssim \chi_{k=k'}^{-\delta_1} 2^{4k'/5} \mu
\end{equation}
for all $k'$, where
\begin{equation}\label{spunk}
 \| \phi \|_{S'_{\mu,k}(I \times \R^2)} := \sup_{k',\kappa,x_0,\omega} 2^{|k-k'|/C_0} 
2^{-k'/2} \| P_{k',\kappa} \nabla_{t,x} \phi \|_{L^2_t L^\infty_x(T_{x_0,\omega,k'} \cap (I \times \R^2))} 
\end{equation}
where $k'$ ranges over $\Z$, $\kappa$ ranges over all caps $\kappa \in K_l$ of length $2^{-l}$ at least $\mu^{C_0}$, $x_0$ ranges over $\R^2$, and $\omega$ ranges over all directions in $S^1$ with $\dist(\omega, \kappa), \dist(\omega, -\kappa) \geq \mu^{C_0}$.  (For a definition of $K_l$ and $P_{k,\kappa}$, see \cite[p. 481]{tao:wavemap2}.)
\end{itemize}
We then define the dilation-smoothed variant $S_{\mu,k}(I \times \R^2)$ of these norms by the formula
\begin{align*}
\| \phi \|_{S_{\mu,k}(I \times \R^2)} &:= \sup_{j\in \R} 2^{-100|j|} \| \phi^{(j)} \|_{\tilde S_{\mu,k}(2^{-j} I \times \R^2)}.
\end{align*}
\end{definition}

By repeating the arguments in \cite{tao:heatwave3} we see that the two norms $S_{\mu,k}$ and $\tilde S_{\mu,k}$ are comparable:
\begin{equation}\label{smuk}
\| \phi \|_{S_{\mu,k}(I \times \R^2)} \sim \| \phi \|_{\tilde S_{\mu,k}(I \times \R^2)}.
\end{equation}

\subsection{Multilinear estimates}

In this section we show the following variant of \cite[Theorem 8.1]{tao:heatwave3}:

\begin{theorem}[Function space norms]\label{func}  For every interval $I$, every integer $k$, and every $\mu > 0$, we have the following properties for all integers $k,k_1,k_2,k_3$ and $\phi, \phi^{(1)}, \phi^{(2)}, \phi^{(3)}, F \in \Sch(I \times \R^2)$:
\begin{itemize}
\item (Continuity and monotonicity) If $I = [t_-,t_+]$, then $\|\phi\|_{S_{\mu,k}([a,b])}$ is a continuous function of $a,b$ for $t_- \leq a < b \leq t_+$, and is decreasing in $a$ and increasing in $b$.
\item ($S^\strong_k$ and $S_{\mu,k}$ are comparable) We have
\begin{equation}\label{sksk-star}
 \| \phi \|_{S^\strong_{k}(I \times \R^2)} \lesssim \| \phi \|_{S_{\mu,k}(I \times \R^2)} \lesssim \mu^{-1} \| \phi \|_{S^\strong_k(I \times \R^2)}.
\end{equation}
\item (Vanishing)  If $\phi \in \Sch(I \times \R^2)$ and $t_0 \in I$, then there exists an interval $J \subset I$ containing $t_0$ such that 
\begin{equation}\label{shrinko}
\| \phi \|_{S_{\mu,k}(J \times \R^2)} \lesssim \sum_{k'} \chi_{k=k'}^{-\delta_1} \| \nabla_{t,x} P_{k'} \phi(t_0) \|_{L^2_x(\R^2)}.
\end{equation}
\item (First product estimate) We have
\begin{equation}\label{prod1}
 \| \phi^{(1)} \phi^{(2)} \|_{S^\strong_{\max(k_1,k_2)}(I \times \R^2)} \lesssim \| \phi^{(1)} \|_{S^\strong_{k_1}(I \times \R^2)} \| \phi^{(2)} \|_{S^\strong_{k_2}(I \times \R^2)}.
 \end{equation} 
\item ($L^\infty$ estimates) We have
\begin{equation}\label{algebra}
\| \phi \|_{L^\infty_t L^\infty_x(I \times \R^2)} \lesssim \| \phi \|_{S^\strong_k(I \times \R^2)}
\end{equation}
as well as the refinement
\begin{equation}\label{algebra-2}
\| \nabla_{t,x} \phi \|_{L^\infty_t L^\infty_x(I \times \R^2)} \lesssim \| \nabla_x \phi \|_{S^\strong_k(I \times \R^2)}.
\end{equation}
\item ($N$ contains $L^1_t L^2_x$)  If $F \in \Sch(I \times \R^2)$ has Fourier support in the region $\{ \xi: |\xi| \sim 2^k \}$, then
\begin{equation}\label{fl1l2}
 \| F \|_{N_k(I \times \R^2)^\strong} \lesssim \|F\|_{L^1_t L^2_x(I \times \R^2)}.
 \end{equation}
\item (Adjacent $N_k$ or $S_k$ are equivalent) If $\phi \in \Sch(I \times \R^2)$, then
\begin{equation}\label{physical}
 \| \phi\|_{S^\strong_{k_1}(I \times \R^2)} \lesssim \chi_{k_1=k_2}^{-\delta_1} \|\phi\|_{S^\strong_{k_2}(I \times \R^2)}
\end{equation}
and
\begin{equation}\label{physical-star}
 \| \phi\|_{S_{\mu,k_1}(I \times \R^2)} \lesssim \chi_{k_1=k_2}^{-\delta_1} \|\phi\|_{S_{\mu,k_2}(I \times \R^2)}.
\end{equation}
Similarly, if $F \in \Sch(I \times \R^2)$ and $k_1=k_2+O(1)$, then
\begin{equation}\label{physicaln-eq} \| F\|_{N_{k_1}(I \times \R^2)^\strong} \sim \|F\|_{N_{k_2}(I \times \R^2)^\strong}.
\end{equation}
\item (Energy estimate) If $\phi \in \Sch(I \times \R^2)$ has Fourier transform supported in the region $\{ |\xi| \sim 2^k\}$, and $t_0 \in I$, then
\begin{equation}\label{energy-est}
\| \phi \|_{S^\strong_k(I \times \R^2)} \lesssim 
\| \phi[t_0] \|_{\dot H^1(\R^2) \times L^2(\R^2)} + \| \Box \phi \|_{N_k(I \times \R^2)^\strong}.
\end{equation}
\item (Parabolic regularity estimate)  If $\phi: \R^+ \to \Sch(I \times \R^2)$ is smooth and $s > 0$, then
\begin{equation}\label{parreg}
 \| \nabla_x^2 \phi(s) \|_{S^\strong_k(I \times \R^2)} \lesssim s^{-1} \| \phi(s/2) \|_{S^\strong_k(I \times \R^2)} + \sup_{s/2 \leq s' \leq s} \| (\partial_s - \Delta) \phi(s') \|_{S^\strong_k(I \times \R^2)}
\end{equation}
and similarly for $S_{\mu,k}(I \times \R^2)$ or $N_k(I \times \R^2)$.
\item (Second product estimate) We have
\begin{equation}\label{second-prod} \| P_k( \phi F ) \|_{N_k(I \times \R^2)^\strong} \lesssim \chi_{k \geq k_2}^{\delta_2} 
\chi_{k=\max(k_1,k_2)}^{\delta_2} \| \phi \|_{S^\strong_{k_1}(I \times \R^2)} \|F\|_{N_{k_2}(I \times \R^2)^\strong}.
\end{equation}
\item (Improved trilinear estimates) We have
\begin{equation}\label{trilinear-improv}
\begin{split}
 \| P_k( \phi^{(1)} \partial_\alpha \phi^{(2)} \partial^\alpha \phi^{(3)}) \|_{N_k(I \times \R^2)} &\lesssim \mu^{1-\eps} \chi_{k=\max(k_1,k_2,k_3)}^{\eps \delta_1} \chi_{k_1 \leq \min(k_2,k_3)}^{\eps \delta_1} \\
&\quad \times  \|\phi^{(1)} \|_{S_{\mu_1, k_1}(I \times \R^2)}
 \|\phi^{(2)}\|_{S_{\mu_2,k_2}(I \times \R^2)}
 \|\phi^{(3)} \|_{S_{\mu_3,k_3}(I \times \R^2)}
\end{split}
\end{equation}
and
\begin{equation}\label{trilinear-improv2}
\begin{split}
 \| P_k( \phi^{(1)} \partial_\alpha \phi^{(2)} \partial^\alpha \phi^{(3)}) \|_{N_k(I \times \R^2)} &\lesssim \mu^{2-2\eps} \chi_{k=\max(k_1,k_2,k_3)}^{\eps \delta_1} \chi_{k_1 \leq \min(k_2,k_3)}^{\eps \delta_1} \\
&\quad  \|\phi^{(1)} \|_{S_{\mu,k_1}(I \times \R^2)}
 \|\phi^{(2)}\|_{S_{\mu,k_2}(I \times \R^2)}
 \|\phi^{(3)} \|_{S_{\mu,k_3}(I \times \R^2)}
\end{split}
\end{equation}
and
\begin{equation}\label{trilinear-improv3}
\begin{split}
 \| P_k( \phi^{(1)} \phi^{(2)} \Box \phi^{(3)}) \|_{N_k(I \times \R^2)} &\lesssim \chi_{k=\max(k_1,k_2,k_3)}^{\delta_1} 
\chi_{\min(k_1,k_2) \leq k_3}^{\delta_1} \\
&\quad  \|\phi^{(1)} \|_{S_{k_1}^\strong(I \times \R^2)}
 \|\phi^{(2)}\|_{S_{k_2}^\strong(I \times \R^2)}
 \|\phi^{(3)} \|_{S_{k_3}^\strong(I \times \R^2)}
\end{split}
\end{equation}
for every $0 \leq \eps \leq 1$ and $k_1 \geq \min(k_2,k_3)-O(1)$, whenever two of the $\mu_1,\mu_2,\mu_3$ are equal to $\mu$ and the third is equal to $1$ (with the convention that $S_{1,k} = S_k^\strong$).
\item (Strichartz estimates)  If $\phi$ has Fourier support in the region $\{ \xi: |\xi| \lesssim 2^k\}$, then we have
\begin{equation}\label{outgo}
 \sup_{t \in I} \| \phi[t] \|_{H^1(\R^2) \times L^2_x(\R^2)} \lesssim \|\phi\|_{S_k^\strong(I \times \R^2)}
 \end{equation}
and
\begin{equation}\label{lstrich-4}
\| \nabla_{t,x} \phi\|_{L^5_t L^\infty_x(I \times \R^2)} \lesssim 2^{4k/5} \mu \|\phi\|_{S_{\mu,k}(I \times \R^2)}
\end{equation}
and
\begin{equation}\label{lstrich}
\| \nabla_{t,x} \phi\|_{L^q_t L^\infty_x(I \times \R^2)} \lesssim 2^{(1-1/q) k} \|\phi\|_{S_{k}^\strong(I \times \R^2)}.
\end{equation}
for $5 \leq q \leq \infty$.
\item (Hyperbolic estimate)  We have
\begin{equation}\label{hypo}
\| P_{k_1} \Box \phi \|_{N_{k_1}^\strong(I \times \R^2)} \lesssim \chi_{k=k_1}^{\delta_1} \| \phi \|_{S_{k}^\strong(I \times \R^2)}
\end{equation}
\end{itemize}
\end{theorem}

\begin{proof}  We first dispose of some estimates which are either trivial or easily deduced from the existing literature.
The continuity and monotonicity of $S_{\mu,k}$ is established just as in \cite{tao:heatwave3}.  The comparability \eqref{sksk-star} follows from Definition \ref{smuk-def} and Lemma \ref{boxf}.  The comparability properties \eqref{physical}, \eqref{physical-star}, \eqref{physicaln-eq} are easily verified from the definition, while the hyperbolic estimate \eqref{hypo} follows from \eqref{phis-plus}.  The estimate \eqref{fl1l2} is immediate from Definition \ref{atom-def2}. The energy estimate \eqref{energy-est} follows from \eqref{phis-plus} and \cite[Equation (27)]{tao:wavemap2}, while the second product estimate \eqref{second-prod} follows from \cite[Equation (29)]{tao:wavemap2}.  The estimate \eqref{outgo} follows from \cite[Equation (32)]{tao:wavemap2}, the estimate \eqref{lstrich} follows from \cite[Lemma 3.1]{krieger:2d} (or \cite[Lemma 6.7]{krieger:2d}; see also \cite{sterbenz}), and the estimate \eqref{hypo} follows from \eqref{phis-plus}.  The parabolic regularity estimate \eqref{parreg} is proven exactly as \cite[Equation (122)]{tao:heatwave3} or \cite[Equation (61)]{tao:heatwave3} is proven.  The estimate \eqref{algebra} follows from \cite[Equation (82)]{tao:wavemap2}, while the refinement \eqref{algebra-2} comes from Bernstein's inequality and the fact that the frequency building block $S[k]$ of the $S_k$ norm controls the energy norm $\| \nabla_{t,x} \phi \|_{L^\infty_t L^2_x}$.

Now we turn to \eqref{prod1}.  We normalise 
$$ \| \phi^{(1)} \|_{S^\strong_{k_1}(I \times \R^2)} = \| \phi^{(2)} \|_{S^\strong_{k_2}(I \times \R^2)} = 1$$
and $k_1 \leq k_2 = 0$.
If the $S_k^\strong$ norm on the left-hand side was replaced with $S_k$, then the claim follows from \cite[Equation (20)]{tao:wavemap2}, so by \eqref{phis-plus} it suffices to show that
$$ \|P_{k'} \Box (\phi^{(1)} \phi^{(2)})\|_{N_{k'}^\strong(I \times \R^2)} \lesssim
\chi_{k' = 0}^{\delta_1}$$
for all $k'$.

Fix $k'$.  By the Leibniz rule, we may split $\Box (\phi^{(1)} \phi^{(2)})$ as a combination of $(\Box \phi^{(1)}) \phi^{(2)}$, $\Box \phi^{(1)} (\Box \phi^{(2)})$, and $\partial^\alpha \phi^{(1)} \partial_\alpha \phi^{(2)}$.  The third term is acceptable from the null form estimate in \cite[Equation (30)]{tao:wavemap2} (and Theorem \ref{enw}).  The other two terms are acceptable from \eqref{phis-plus} and \eqref{second-prod} (and Littlewood-Paley decomposition).

Now we look at the vanishing estimate \eqref{shrinko}.  We may set $t_0=0$.  From Definition \ref{smuk-def} we can replace $S_{\mu,k}$ by $\tilde S_{\mu,k}$, and it will then suffice to show that
\begin{equation}\label{phipl}
  \limsup_{\eps \to 0} \| \phi \|_{S_k^\strong([-\eps,\eps] \times \R^2)} \lesssim \sum_{k'} \chi_{k=k'}^{-\delta_1} \| \nabla_{t,x} P_{k'} \phi(0) \|_{L^2_x(\R^2)}
\end{equation}
and
$$  \limsup_{\eps \to 0} \| \phi \|_{S'_{\mu,k}([-\eps,\eps] \times \R^2)} = 0.$$
The second claim is easily established from the Schwartz nature of $\phi$, so we turn to \eqref{phipl}.  From the energy estimates in \cite[Section 11]{tao:wavemap2} (and the definition of $S_k$ in \cite[Section 10]{tao:wavemap2}) we have
$$ \| \phi \|_{S_k^\strong([-\eps,\eps] \times \R^2)} \lesssim \sum_{k'} \chi_{k=k'}^{-\delta_1} \| \nabla_{t,x} P_{k'} \phi(0) \|_{L^2_x(\R^2)}
+ \sum_{k'} \chi_{k=k'}^{-\delta_1} \| P_{k'} \Box \phi \|_{L^1_t L^2_x([-\eps,\eps] \times \R^2)},$$
and the claim follows from the monotone convergence theorem and the Schwartz nature of $\phi$.

Next, we consider the Strichartz estimate \eqref{lstrich-4}.  It suffices to verify this for atoms.  For small atoms, the claim follows from \eqref{lstrich}, while for dispersed atoms the claim follows from \eqref{mondo}.

Now, we look at the trilinear estimates \eqref{trilinear-improv}, \eqref{trilinear-improv2}.  The arguments here will rely heavily on those in \cite{tao:wavemap2}, as well as the inverse theory in Section \ref{inverse-sec}.  The basic strategy is to first use the inverse theory show that the only counterexamples to these estimates could be generated along light rays, which in turn can only arise from parallel interactions; but parallel interactions in the $S$ and $N$ type spaces are known to be negligible, especially when null structure is present, as is for instance demonstrated at various points in \cite{tao:wavemap2}.

The $\eps=1$ cases of these inequalities follow from \cite[Equation (87)]{tao:wavemap2}, so it suffices to verify the $\eps=0$ case; a similar argument allows us to assume $\mu$ to be small.  By allowing $k,k_1,k_2,k_3$ to be real rather than integer, we may replace the $S_{\mu_i,k}$ norms with $\tilde S_{\mu_i,k}$ for $i=1,2,3$, and then reduce to atoms.  Normalising the norms on the right-hand side to be one, our task is now to show that
\begin{equation}\label{trilinear-improv-1}
 \| P_k( \phi^{(1)} \partial_\alpha \phi^{(2)} \partial^\alpha \phi^{(3)}) \|_{N_k(I \times \R^2)^\strong} \lesssim \mu 
\end{equation}
when $\phi^{(i)}$ is a $S_{\mu_i,k_i}(I \times \R^2)$ atom for $i=1,2,3$, and
\begin{equation}\label{trilinear-improv2-1}
 \| P_k( \phi^{(1)} \partial_\alpha \phi^{(2)} \partial^\alpha \phi^{(3)}) \|_{N_k(I \times \R^2)^\strong} \lesssim \mu^{2} 
\end{equation}
when $\phi^{(i)}$ is a $S_{\mu,k_i}(I \times \R^2)$ atom for $i=1,2,3$.

We first prove \eqref{trilinear-improv-1}.  When $\phi^{(i)}$ is a small atom and $\mu_i=\mu$, the claim again follows from \cite[Equation (87)]{tao:wavemap2}, so we may assume that $\phi^{(i)}$ is a null-dispersed atom whenever $\mu_i=\mu$.  By symmetry we may also take $k_2 \leq k_3$.
  
We assume that \eqref{trilinear-improv2-1} fails (so that the left-hand side is at least $\mu$) and derive a contradiction (for $C_0$ large enough).
We recall that implied constants are allowed to depend on exponents such as $\delta_1$.

By \cite[Equation (87)]{tao:wavemap2}, this forces $\chi_{k=\max(k_1,k_2,k_3)}^{\delta_1} \chi_{k_1 \leq \min(k_2,k_3)}^{\delta_1} \gtrsim \mu$, and thus $k = k_3 + O(\log \frac{1}{\mu})$ and $k_1 = k_2 + O(\log \frac{1}{\mu})$ by the hypotheses on $k_1,k_2$. By rescaling, we may thus normalise
$$  k_1, k_2 = O(\log \frac{1}{\mu}); \quad k = k_3 +O(\log \frac{1}{\mu}); k_3 \geq O(\log \frac{1}{\mu}).$$
From \cite[Equations (29), (30)]{tao:wavemap2}, we see that
$$
 \| P_k( \phi^{(1)} P_{k'}(\partial_\alpha \phi^{(2)} \partial^\alpha \phi^{(3)}) ) \|_{N_k(I \times \R^2)^\strong}$$
decays exponentially away from $k' = k_3$, so by the pigeonhole principle we may thus find
$$ k' = k_3+O(\log \frac{1}{\mu})$$
such that
$$
 \| P_k( \phi^{(1)} P_{k'}(\partial_\alpha \phi^{(2)} \partial^\alpha \phi^{(3)}) ) \|_{N_k(I \times \R^2)^\strong} \gtrsim \mu^{O(1)}.$$
Fix this $k'$.  We extend the $\phi^{(i)}$ to all of $\R^{1+2}$ so that they remain bounded in $S_{k_i}(\R \times \R^2)$.  From \cite[Equation (30)]{tao:wavemap2}, one has
$$
 \| P_{k'}(\partial_\alpha \phi^{(2)} \partial^\alpha \phi^{(3)}) \|_{\dot X_{k'}^{0,-1/2,\infty}(I \times \R^2)} \lesssim \mu^{O(1)},$$
so by Lemma \ref{core-improve}, the contribution of $Q_{<j'} P_{k'}(\partial_\alpha \phi^{(2)} \partial^\alpha \phi^{(3)})$ is negligible for some $j' = O(\log \frac{1}{\mu})$, and thus
$$
 \| P_k( \phi^{(1)} P_{k'} Q_{\geq j'} (\partial_\alpha \phi^{(2)} \partial^\alpha \phi^{(3)}) ) \|_{N_k(I \times \R^2)^\strong} \gtrsim \mu^{O(1)}$$
 for this $j'$.  
 
Let us now consider the case when $\mu_2=\mu$, so $\phi^{(2)}$ is a null-dispersed atom.  Applying \cite[Equation (29)]{tao:wavemap2}, we conclude that
$$ \| P_{k'} Q_{\geq j'} (\partial_\alpha \phi^{(2)} \partial^\alpha \phi^{(3)}) ) \|_{N_k(I \times \R^2)^\strong} \gtrsim \mu^{O(1)};$$
controlling the $N_k$ norm by the $\dot X^{0,-1/2,1}$ norm and discarding the $P_{k'} Q_{\geq j'}$ projections, we conclude
$$ \| \partial_\alpha \phi^{(2)} \partial^\alpha \phi^{(3)} \|_{L^2_{t,x}(I \times \R^2)} \gtrsim \mu^{O(1)}.$$
Applying (a slightly rescaled version of) Theorem \ref{tao-thm3} and the triangle inequality, we can find a collection $(T_\beta)_{\beta \in B}$ of tubes of frequency $1$ with cardinality $O( \mu^{-O(1)} )$ such that
$$ \| \partial_\alpha \phi^{(2)} \partial^\alpha \phi^{(3)} \|_{L^2_{t,x}((I \times \R^2) \cap \bigcup_{\beta \in B} T_\beta)} \gtrsim \mu^{O(1)},$$
and thus by the triangle inequality and pigeonhole principle we can find one of these tubes $T = T_{x_0,\omega,0}$ for which
$$ \| \partial_\alpha \phi^{(2)} \partial^\alpha \phi^{(3)} \|_{L^2_{t,x}((I \times \R^2) \cap T)} \gtrsim \mu^{O(1)}.$$
From \eqref{energy-est}, $\nabla_{t,x} \phi^{(3)}$ has an $L^\infty_t L^2_x(I \times \R^2)$ norm of $O(1)$.  Using the null-dispersed nature of $\phi^{(2)}$ and H\"older's inequality to control all components of $\phi^{(2)}$ not oriented near the direction $\omega$, we may thus conclude that
$$ \| \partial_\alpha P \phi^{(2)} \partial^\alpha \phi^{(3)} \|_{L^2_{t,x}((I \times \R^2) \cap T)} \gtrsim \mu^{O(1)},$$
where $P$ is a smooth Fourier projection to the region $\{ |\xi| \sim 1; \angle \xi, \omega \lesssim \mu^{cC_0} \}$ for some small absolute constant $c>0$, assuming $C_0$ is large enough.

Now let $P'$ be a smooth Fourier projection to the region $\{ |\xi| \sim 1; \angle \xi, \omega \lesssim \mu^{c'C_0} \}$ for some small absolute constant $c'$ (which will be much smaller than $c$).  We consider first the contribution of $(1-P') \phi^{(3)}$.  The transversality to $T$ implies that $\nabla_{t,x} (1-P') \phi^{(3)}$ has an $L^2_{t,x}((I \times \R^2) \cap T)$ norm of $O( \mu^{-O(c' C_0) + O(1)} )$ (here we use the plane wave component of the $S$ norms and dyadic decomposition, see \cite[Proposition 2]{tao:wavemap2}).  On the other hand, from Bernstein's inequality we see that $\nabla_{t,x} P \phi^{(2)}$ has an $L^\infty_{t,x}((I \times \R^2) \cap T)$ norm of $O(\mu^{cC_0/2 - O(1)})$.  For $C_0$ large enough (and $c'$ small enough) and H\"older's inequality\footnote{Again, one has to deal with the $P_{k'} Q_{\geq j'}$ projection, but the localisation of the associated kernel allows one to do this easily.}, this contribution is negligible, thus we have
$$ \| \partial_\alpha P \phi^{(2)} \partial^\alpha P' \phi^{(3)} \|_{L^2_{t,x}((I \times \R^2) \cap T)} \gtrsim \mu^{O(1)}.$$
Consider the contribution of $Q_{>j_2} \phi^{(2)}$ for some $j_2$.  The quantity $\nabla_{t,x} P Q_{>j_2} \phi^{(2)}$ has an $L^2_t L^\infty_x$ norm of $O( \mu^{\tilde c C_0-O(1)} 2^{-j_2/2} )$ for some $\tilde c>0$ (by $\dot X^{0,1/2,\infty}$ estimates and Bernstein's inequality), while $\nabla_{t,x} P' \phi^{(3)}$ has an $L^\infty_t L^2_x$ norm of $O(1)$, so this contribution is negligible for $j_2 := -c'' C_0 \log \frac{1}{\mu}$ and some absolute constant $c'' > 0$.  Similarly, to control the contribution of $Q_{>j_3} \phi^{(3)}$ for some $j_3$, one observes that $\nabla_{t,x} P \phi^{(2)}$ has an $L^\infty_{t,x}$ norm of $O( \mu^{-O(1)})$ and $\nabla_{t,x} P' Q_{>j_3} \phi^{(3)}$ has a $L^2_t L^\infty_x$ norm of $O( \mu^{\tilde c C_0-O(1)} 2^{-j_2/2} )$ (by $\dot X^{0,1/2,\infty}$ estimates and Bernstein's inequality), again giving a negligible contribution for $j_3 = - c'' C_0 \log \frac{1}{\mu}$.  So we have
$$ \| \partial_\alpha P Q_{<j_2} \phi^{(2)} \partial^\alpha P' Q_{<j_3} \phi^{(3)} \|_{L^2_{t,x}((I \times \R^2) \cap T)} \gtrsim \mu^{O(1)}.$$
for these choices of $j_2, j_3$.

There are two components to $P Q_{<j_2} \phi^{(2)}$, one which is transverse to $T$ and one which is essentially parallel to $T$.  The transverse component can be dealt with by earlier arguments, so we may restrict attention to the component which is parallel to $T$; similarly for $P' Q_{<j_3} \phi^{(3)}$.  But now the expression $\partial_\alpha P Q_{<j_2} \phi^{(2)} \partial^\alpha P' Q_{<j_3} \phi^{(3)}$ lies at a distance $\mu^{c''' C_0}$ from the light cone, for some $c''' > 0$, and the arguments in \cite[Step 17.2(b).3]{tao:wavemap2} then give an upper bound 
$$ \| \partial_\alpha P Q_{<j_2} \phi^{(2)} \partial^\alpha P' Q_{<j_3} \phi^{(3)} \|_{L^2_{t,x}(I \times \R^2)} \lesssim \mu^{c'''' C_0 - O(1)}$$
for some $c'''' > 0$, giving a contradiction.

Now we turn to the case when $\mu_2=1$, thus $\mu_1=\mu_3=\mu$ and $\phi^{(1)}, \phi^{(3)}$ are both null-dispersed.   If $k_3<k_2+10$ then by swapping the roles of $\phi^{(2)}$ and $\phi^{(3)}$ we can use the previous arguments, so suppose instead that $k_3 \geq k_2+10$, which forces $k = k_3+O(1)$.

Using Theorem \ref{tao-thm3} as before, one can find a collection $(T_\beta)_{\beta \in B}$ of tubes of frequency $1$ with cardinality $O( \mu^{-O(1)})$ such that
$$
 \| P_k( \phi^{(1)} P_{k'} Q_{\geq j'} ( 1_{\bigcup_{\beta \in B} T_\beta} \partial_\alpha \phi^{(2)} \partial^\alpha \phi^{(3)}) ) \|_{N_k(I \times \R^2)^\strong} \gtrsim \mu^{O(1)}.$$
On the other hand, from the preceding arguments one has
$$ \| 1_{\bigcup_{\beta \in B} T_\beta} \partial_\alpha \phi^{(2)} \partial^\alpha \phi^{(3)}) \|_{L^2_{t,x}(I \times \R^2)} \lesssim \mu^{-O(1)}.$$
Thus, there exists a tube $T = T_{x_0,\omega,0}$ and a function $F_T$ supported on that tube with $L^2_{t,x}$ norm $O(1)$ such that
\begin{equation}\label{ppk}
 \| P_k( \phi^{(1)} P_{k'} Q_{\geq j'} F_T ) \|_{N_k(I \times \R^2)^\strong} \gtrsim \mu^{O(1)}.
\end{equation}
Consider a portion of $\phi^{(1)}$ whose spatial Fourier transform supported on a cap at an angle of at least $\mu^{cC_0}$ to $\omega_0$ for some small absolute constant $c>0$.  Then applying the null-dispersion hypothesis, we see that this portion has an $L^2_t L^\infty_x$ norm of $O(\mu^{c' C_0})$ for some other constant $c'>0$ on $T$ and all of its translates; this easily leads to a $L^1_t L^2_x$ norm of $O(\mu^{c' C_0-O(1)})$ for this contribution to \eqref{ppk}, which is negligible.  Because of this, we may replace $\phi^{(1)}$ with $P \phi^{(1)}$ where $P$ is as before.  Now improved Bernstein estimates become available, and one sees that the contribution of $Q_{>j} P \phi^{(1)}$ for some $j = -c' C_0 \log \frac{1}{\mu}$ becomes negligible because it is small in $L^2_t L^\infty_x$.  Thus we only need consider the contribution of $Q_{\leq j} P \phi^{(1)}$.  But by construction of the $S$ norm, one can estimate this term in the space $S[O(1),\kappa]$ for some cap $\kappa$ of size $O( \mu^{c'''C_0 - O(1)})$, with a norm of $O(\mu^{-O(1)})$.
Multiplying this against the $L^2_{t,x}$-bounded quantity $P_{k'} Q_{\geq j'} F_T$ (using \cite[Equation (72)]{tao:wavemap2}) we see that this contribution is also negligible for $C_0$ large enough, leading to the desired contribution.

Next, we prove \eqref{trilinear-improv2-1}.  When any of the $\phi^{(i)}$ are small atoms, then the claim follows from \eqref{trilinear-improv} (setting $\mu_i=1$ and the other two $\mu_j$ equal to $\mu$), so we may assume that all of the $\phi^{(i)}$ are null-dispersed atoms.  The claim now follows by repeating the previous arguments (losing a few more factors of $\mu$).

Next, we establish \eqref{trilinear-improv3}.  We may normalise
$\|\phi^{(i)} \|_{S_{k_i}^\strong(I \times \R^2)}=1$ for $i=1,2,3$, and also $k_1 \geq k_2$ and $k_3=0$.  Write $F := \Box \phi^{(3)}$, thus $\|F\|_{N_{0}(I \times \R^3)} \lesssim 1$ by \eqref{hypo}.  Our task is to show that
$$
 \| P_k( \phi^{(1)} \phi^{(2)} F) \|_{N_k(I \times \R^2)} \lesssim \chi_{k=\max(k_1,0)}^{\delta_1} 
\chi_{k_2 \leq 0}^{\delta_1}.$$
From \eqref{second-prod} we already obtain the bound of
$$ 
\chi_{k \geq 0}^{\delta_2} 
\chi_{k=\max(k_1,0)}^{\delta_2}.$$
So we are done unless $k_2 \geq 10$, at which point it will suffice to show that
$$
 \| P_k( \phi^{(1)} \phi^{(2)} F) \|_{N_k(I \times \R^2)} \lesssim \chi_{k=k_1}^{-C} 
\chi_{k_2 \leq 0}^{\delta_1}$$
for some large absolute constant $C$. For simplicity of notation we will suppose that $k=k_1$, but the general case is similar.

From \cite[Lemma 12]{tao:wavemap2} we know that $\phi^{(2)} F$ (and hence $\phi^{(1)} \phi^{(2)} F$ decays if $F$ has modulation much less than $1$, so we may assume that $F$ has modulation at least $2^{-\delta_2^2 k_2}$ (say).  Dually, by \cite[Lemma 13]{tao:wavemap2}, we see that $\phi^{(1)} \phi^{(2)}$ decays when the modulation is much less than $2^{k_2}$, so we may assume that the modulation here is at least $2^{-\delta_2^2 k_2} 2^{k_2}$ (say).

By the arguments in \cite[Section 15.2(ac)]{tao:wavemap2} we may assume that $F$ is a $\dot X^{0,-1/2,1}$ atom with some modulation $2^j$.  
Let $j'$ denote the modulation of $\phi^{(1)} \phi^{(2)}$ (one can formalise this by Littlewood-Paley projections), then $\phi^{(1)} \phi^{(2)}$ has an $L^2_{t,x}$ norm of $O( 2^{-k} 2^{-j'/2} )$, while $F$ has an $L^2_t L^\infty_x$ norm of $2^{j'/2}$.  This gives an acceptable contribution unless $2^j \gg 2^{1.9 k} 2^{j'}$ (say), in which case $\phi^{(1)} \phi^{(2)}$ has a modulation of $\sim 2^j$.  Since $\phi^{(1)} \phi^{(2)}$ has an $L^\infty_t L^2_x$ norm of $O( 2^{-k})$, $\phi^{(1)} \phi^{(2)} F$ has an $L^2_{t,x}$ norm of $O( 2^{-k} 2^{j'/2} )$, which is acceptable.


\end{proof}

\begin{proposition}\label{algae}  The $S_k^\strong$ norms form an algebra family in the sense of Definition \ref{algebra-fam}.
\end{proposition}

\begin{proof}  Comparing Theorem \ref{func} with Definition \ref{algebra-fam}, we see that we already have all of the properties required except for
\eqref{parreg-abstract-alt} and \eqref{parreg-abstract2}.

Now we prove \eqref{parreg-abstract2}.  From \eqref{physical} it suffices to show that the operator $\nabla_x^j e^{(s-s')\Delta}$ has an operator norm of $O_j((s-s')^{-j/2})$ on $S^\strong_{k(s)}(I\times \R^2)$.  However, from the fundamental solution of the heat equation we see that this operator is a convolution operator whose kernel has an $L^1$ norm of $O_j((s-s')^{-j/2})$, and the claim follows from Minkowski's inequality (and the translation-invariance of $S^\strong_{k(s)}(I\times \R^2)$).

Finally, we prove \eqref{parreg-abstract-alt}.  We apply Duhamel's formula to write
$$ \phi(s) = e^{(s-s'')\Delta} \phi(s'') + \int_{s''}^s e^{(s-s')\Delta} F(s')\ ds'$$
where $F := (\partial_s - \Delta) \phi$.
The contribution of the first term is acceptable by \eqref{parreg-abstract2}, so it suffices to show that
$$ \| \nabla_x^2 \int_{s''}^s e^{(s-s')\Delta} F(s') \ ds' \|_{S_k^\strong(I \times \R^2)} \lesssim \sup_{s'' \leq s' \leq s} \| F(s') \|_{S_k^\strong(I \times \R^2)}.$$
We may normalise $k=0$ and
\begin{equation}\label{soups}
\sup_{s'' \leq s' \leq s} \| F(s') \|_{S_k^\strong(I \times \R^2)} = 1.
\end{equation}
By \eqref{phis-plus} and the definition of the $S_k$ norms in \cite[Section 10]{tao:wavemap2}, it suffices to show that
\begin{equation}\label{coconut}
 \| \int_{s''}^s \nabla_x^2 P_{k'} e^{(s-s')\Delta} F(s') \ ds' \|_{S[k']^\strong(I \times \R^2)} \lesssim \chi_{k'=0}^{\delta_1}
 \end{equation}
where
$$ \| f \|_{S[k']^\strong(I \times \R^2)} := \| f\|_{S[k'](I \times \R^2)} + \|\Box f \|_{N[k'](I \times \R^2)}$$
and the $S[k']$ norm was defined in \cite[Section 10]{tao:wavemap2}.  From \eqref{soups} one has
$$ \| P_{k'} F(s') \|_{S[k']^\strong(I \times \R^2)} \lesssim \chi_{k'=0}^{\delta_1}.$$
Meanwhile, one can write 
$$\nabla_x^2 P_{k'} e^{(s-s')\Delta} F(s') = \nabla_x^2 \tilde P_{k'} e^{(s-s')\Delta} P_{k'} F(s')$$
where $\tilde P_{k'}$ is a slight enlargement of $P_{k'}$.  From the fundamental solution of the heat equation we see that $\nabla_x^2 P_{k'} e^{(s-s')\Delta}$ is a convolution operator whose kernel has an $L^1$ norm of $O( 2^{2k'} (1 + 2^{2k'} (s-s'))^{-10} )$, and so by Minkowski's inequality (and the translation-invariance of $S[k']^\strong(I \times \R^2)$) one can bound the left-hand side of \eqref{coconut} by
$$ \int_{s''}^s O( 2^{2k'} (1 + 2^{2k'} (s-s'))^{-10} \chi_{k'=0}^{\delta_1} )\ ds'$$
and the claim follows.
\end{proof}

\subsection{Norms}

For each compact interval $I$ and each $\mu>0$, we define a metric $S^1_\mu(I)$ on the space of dynamic fields $\psi_s$ by the formula
\begin{equation}\label{ds1-eq}
\dist_{S^1_\mu(I)}(\psi_s,\psi'_s) := = \| \sum_{j=0}^{10} \sup_{2^{-2k-2} \leq s \leq 2^{-2k}} s^{1+\frac{j}{2}} \| \nabla_x^j (\psi_s - \psi'_s)(s) \|_{S_{\mu,k}(I \times \R^2)} \|_{\ell^2_k(\Z)} 
\end{equation}
and then define $\| \psi_s \|_{S^1_\mu(I)}:= \dist_{S^1_\mu(I)}(0,\psi_s)$, thus
\begin{equation}\label{ds2-eq}
\| \psi_s \|_{S^1_\mu(I)}= \| \sum_{j=0}^{10} \sup_{2^{-2k-2} \leq s \leq 2^{-2k}} s^{1+\frac{j}{2}} \| \nabla_x^j \psi_s(s) \|_{S_{\mu,k}(I \times \R^2)} \|_{\ell^2_k(\Z)} 
\end{equation}

We now have

\begin{theorem}[A priori estimates in the caloric gauge]\label{apriori-thm2} For any interval $I$, and $0 < \mu < 1$, we have the following properties:
\begin{itemize}
\item[(i)] (Monotonicity)  If $I \subset J$, $0 < \mu \leq 1$, and $\psi_s, \psi'_s$ are dynamic fields on $J$ then
$\| \psi_s\downharpoonright_I \|_{S^1_\mu(I)} \leq \|\psi_s\|_{S^1_\mu(J)}$.
\item[(ii)] (Continuity)  If $\psi_s$ is a dynamic field on $[t_-,t_+]$ and $0 < \mu \leq 1$, then $\|\psi_s\downharpoonright_{[a,b]} \|_{S^1_\mu([a,b])}$ is a continuous function of $a,b$ in the region $t_- \leq a < b \leq t_+$.
\item[(iii)] (Vanishing) If $I_n$ is a decreasing sequence of compact intervals with $\bigcap_n I_n = \{t_0\}$, $0 < \mu \leq 1$, and $\psi_s, \psi'_s$ are dynamic fields on $I_1$, then $\lim_{n \to \infty} \| (\psi_s-\psi'_s)\downharpoonright_{I_n} \|_{S^1_\mu(I_n)} \lesssim_E \tilde d_{\Energy}(\psi_s[t_0], \psi'_s[t_0])$, where the modified energy metric $\dist_\Energy$ is defined in \eqref{energy-dist-def}, and $\E(\psi_s[t_0]), \E(\psi'_s[t_0]) \leq E$.
\item[(iv)] ($S^1$ controls energy)  For any $I$, any $0 < \mu \leq 1$ and $M>0$, any dynamic fields $\psi_s, \psi'_s$
 with $\|\psi_s\|_{S^1_\mu(I)}, \|\psi'_s\|_{S^1_\mu(I)} \leq M$, and any $t_0 \in I$ and $s_0 \geq 0$, we have $\tilde d_{\Energy}(\psi_s(s_0)[t_0],\psi'_s(s_0)[t_0]) \lesssim_{M} \| \psi_s - \psi'_s \|_{S^1_\mu(I)}$.  In particular, $\E(\psi_s[t_0]) = O_M(1)$.
\item[(v)] (Stability estimate) If $M > 0$, $0 < \mu \leq 1$ is sufficiently small depending on $M$, $I$ is an interval, $t_0 \in I$, and $\psi_s, \psi'_s$ are dynamic fields on $I$ obeying the wave map and heat flow equations, with $\|\psi_s\|_{S^1_\mu(I)}, \|\psi'_s\|_{S^1_\mu(I)} \leq M$, then
\begin{equation}\label{apriori2a}
\dist_{S^1_\mu(I)}(\psi_s, \psi'_s) \lesssim_{M,\mu} d_{\Energy}(\psi_s[t_0], \psi'_s[t_0]).
\end{equation}
\end{itemize}
\end{theorem}

\begin{proof} Claims (i), (ii), (iii), (v) are deduced from Theorem \ref{func} in exactly the same way that \cite[Theorem 5.1]{tao:heatwave3} is deduced from \cite[Theorem 8.1]{tao:heatwave3} (i.e. by repeating the arguments in \cite[Section 9]{tao:heatwave3}).  (An inspection of the arguments in \cite{tao:heatwave3} reveals that it is the modified energy metric which is really being used in these arguments, rather than the original energy metric, though from Lemma \ref{compati} we know that these metrics are equivalent in the heat flow case.)  Finally, (iv) follows from \eqref{ds1-eq}, \eqref{energy-dist-def}, and \eqref{outgo}.
\end{proof}

Theorem \ref{apriori-thm2}(v) is the stability estimate we will need as Ingredient 4 in the plan outlined in Section \ref{delocal-sec}.

Specialising Theorem \ref{apriori-thm2} to the case $\psi'_s=0$, we see that if $\psi_s$ is a dynamic field on $I$ obeying the wave map and heat flow equations with $\|_{S^1_\mu(I)} \leq M$, with $\mu$ sufficiently small depending on $M$, then
$$
\| \psi_s\|_{S^1_\mu(I)} \lesssim_{M,\mu} \E(\psi_s[t_0])^{1/2}.$$
In fact one has a more precise statement, in which the frequency envelope of $\psi_s$ on $I$ is controlled by the frequency envelope of $\psi_s[t_0]$.  Indeed, an inspection of the arguments in \cite[Section 9.5]{tao:heatwave3} reveals the following list of useful estimates:

\begin{lemma}[Local stability of frequency envelope]\label{freqstable}  Let $I$ be an interval, $E, M>0$, and suppose $0 < \mu < 1$ is sufficiently small depending on $E, M$.  Let $\psi_s$ be a dynamic field on $I$ obeying the wave map and heat flow equations with $\|\psi_s\|_{S^1_\mu(I)} \leq M$, and let $t_0 \in I$.  Let $c_0$ be a frequency envelope of energy at most $E$, such that
$$ \| \nabla_{t,x} \nabla_x^j \psi_s(s,t_0) \|_{\dot H^0_{k(s)}(\R^2)} \lesssim s^{-(j+2)/2} c_0(s)$$
for all $0 \leq j \leq 10$ and $s > 0$.  Then one has
$$ \| \nabla_x^j \psi_s(s) \|_{S_{k(s)}(I \times \R^2)} \lesssim_{M,E,j} s^{-(j+2)/2} c_0(s)$$
and
$$ \| \nabla_x^j \psi_x(s) \|_{S_{k(s)}(I \times \R^2)} \lesssim_{M,E,j} s^{-(j+1)/2} c_0(s)$$
for all $j \geq 0$ and $s>0$.  In particular, one has
$$ \| \nabla_{t,x} \nabla_x^j \psi_s(s,t) \|_{\dot H^0_{k(s)}(\R^2)} \lesssim_{M,E,j} s^{-(j+2)/2} c_0(s)$$
for all $t \in I$, $j \geq 0$, and $s>0$.  In addition, one has the fixed estimates
\begin{equation}\label{l2s-fixed-special-psitn-freq}
\begin{split}
\| \nabla_x^{j+1} \Psi_{t,x}(s,t) \|_{L^2_x(\R^2)} 
+ \| \nabla_x^{j} \Psi_{t,x}(s,t) \|_{L^\infty_x(\R^2)} \quad &\\
+ \| \nabla_x^j \nabla_{t,x} \Psi_x(s,t) \|_{L^2_x(\R^2)} +
\| \nabla_x^j \psi_s(s,t_n) \|_{L^2_x(\R^2)} \quad &\\
+ \| \nabla_x^{j-1} \nabla_{t,x} \psi_s(s,t) \|_{L^2_x(\R^2)}
&\lesssim_{M,E,j} c(s) s^{-(j+1)/2} 
\end{split}
\end{equation}
and
\begin{equation}\label{l2s-fixed-special-psitn-freq2}
\begin{split}
\| \nabla_x^{j+1} A_{t,x}(s,t) \|_{L^2_x(\R^2)} 
+ \| \nabla_x^{j} A_{t,x}(s,t) \|_{L^\infty_x(\R^2)} \quad &\\
+ \| \nabla_x^j \nabla_{t,x} A_x(s,t) \|_{L^2_x(\R^2)} &\lesssim_{M,E,j} c(s)^2 s^{-(j+1)/2} 
\end{split}
\end{equation}
for all $s > 0$ and $j \geq 0$, and $t \in I$, with the convention that we drop all terms involving $\nabla_x^{-1}$; we also have the estimates
\begin{equation}\label{wave-ten}
\| \partial_s^i \partial_x^j P_k w(s)\|_{N_{k(s)}^\strong(I \times \R^2)} \lesssim_{M,E,j,\eps} \mu^{2-\eps} c(s) \chi_{k(s)=k}^{\delta_1/10} s^{-(j+2i)/2}
\end{equation}
and
\begin{equation}\label{wave-ten2}
\| \partial_x^j P_k \Box \psi_s(s)\|_{N_{k(s)}^\strong(I \times \R^2)} \lesssim_{M,E,j,\eps} \mu^{2-\eps} c(s) \chi_{k(s)=k}^{\delta_1/10} s^{-(j+2)/2}
\end{equation}
for all $i=0,1$, $0 \leq j \leq 10$, and $s>0$.
\end{lemma}

\section{$(A,\eps)$-wave maps}\label{aesec}

In this section we define a $(A,\eps)$-wave map, and then establish Theorem \ref{symscat}.

\begin{definition}[$(A,\eps)$-wave map]\label{aes}  Let $A$ and $0 < \mu \leq 1$.  An \emph{$(A,\mu)$-wave map} is a pair $(\phi,I)$ with $I$ a compact time interval and  $\phi$ is a classical wave map on $I$ of energy, which has a resolution $\psi_s$ which is a heat flow with $\| \psi_s \|_{S^1_\mu(I)} \leq A$.
\end{definition}

With this definitions, the monotonicity and symmetry properties (i), (ii) of Theorem \ref{symscat} are routinely verified.   Now we turn to the other properties:

\subsection{Verification of the continuity property}

We now show the continuity property.  Let $A,\mu,I,I_n,\phi^{(n)}$ be as in Theorem \ref{symscat}(iv).  
By time reversal symmetry, it suffices to show continuity at the upper endpoint $t_+ := \sup(I)$ of $I$.  Write $I_n = [t_-^n, t_+ - \eps_n]$ where $\eps_n \to 0$.  We can shift each $\phi^{(n)}$ forward in time by $2\eps_n$ to create a new $(A,\mu)$-wave map $(\tilde \phi^{(n)}, I_n+2\eps_n)$.  These wave maps still converge in $\Energy$ on any compact interval in the interior of $I$ (by uniform continuity and uniform convergence on such intervals), to the same limit as the original $\phi^{(n)}$.  From Theorem \ref{apriori-thm2}(iv),(v), we see that $\tilde \phi^{(n)}$ is a uniformly Cauchy sequence in $\Energy$ on $[t_0,t_+]$ for any $t_0$ in the interior of $I$, and the continuity of the limit at $t_+$ as required.

\subsection{Verification of the divisibility property}\label{symscat-sec}

Let us now show the divisibility property (iii).  Fix $A, \mu, \mu'$, and let $(\phi,I)$ be an $(A,\mu)$ wave map.  We allow all implied constants to depend on $A,\mu$.

Let $t_0 \in I$ be fixed.  From \cite[Lemma 7.7]{tao:heatwave3} we have a frequency envelope $c$ of energy $O(1)$ such that
$$ \| P_k \nabla_{t,x} \psi_s(s,t_0) \|_{L^2_x(\R^2)} \lesssim_j c(s) s^{-1} \chi_{k \leq k(s)}^j \chi_{k=k(s)}^{0.1}$$
for all $k \in \Z$, $j \geq 0$, and $s > 0$.  Meanwhile, from (the proof of) \cite[Lemma 9.10]{tao:heatwave3}, and by enlarging the frequency envelope $c$ if necessary, we have
$$ \| P_k \Box \psi_s(s) \|_{N_k^\strong(I \times \R^2)} \lesssim_{A,\eps,j} \mu^{2-\eps} c(s) s^{-1} \chi_{k \geq k(s)}^{\delta_2/10} \chi_{k \leq k(s)}^j$$
for all $\eps > 0$, $j \geq 0$. Combining these estimates with \eqref{energy-est}, we conclude that
$$ \| P_k \psi_s(s) \|_{S^\strong_k(I \times \R^2)}
 \lesssim_j c(s) s^{-1} \chi_{k \leq k(s)}^j \chi_{k \geq k(s)}^{\delta_2/10}$$
for all $k \in \Z$ and $s>0$.  

We now invoke the following harmonic analysis lemma:

\begin{lemma}[Effective divisbility]\label{boxf}  Let $I$ be an interval, let $k \in \Z$, and let $\phi^\strong \in S_k(I \times \R^2)$.  Then we have
$$ \| \phi \|_{S'_{\mu,k}(I \times \R^2)} \lesssim \mu^{-O(1)}  
\|\phi\|_{S^\strong_k(I \times \R^2)}.$$
Furthermore, given any $\eps > 0$, one can partition $I$ into at most $O_{\eps,\mu,C_0}(1)$ intervals $J$ such that
$$ \| \phi \|_{S'_{\mu,k}(J \times \R^2)} 
\lesssim \varepsilon  
\|\phi\|_{S^\strong_k(I \times \R^2)}$$
for each $J$.
\end{lemma}

\begin{proof} See Appendix \ref{fungi-sec}.
\end{proof}

Applying this lemma, we see that for each $k \in \Z$ and $s>0$, one can find a partition ${\mathcal I}_{k,j,s}$ of $I$ into at most $O_{\mu',C_0}(1)$ intervals $J$ such that
\begin{equation}\label{smukk} \| P_k \psi_s(s)  \|_{S'_{\mu',k}(J \times \R^2)} \lesssim (\mu')^{100 C_0} c(s) s^{-1} \chi_{k \leq k(s)}^{100} \chi_{k \geq k(s)}^{\delta_2/10}.
\end{equation}
Also, using \eqref{lstrich} and partitioning the intervals $J$ further, we may also assume that
\begin{equation}\label{box3} \| P_k \nabla_{t,x} \psi_s(s)\|_{L^5_t L^\infty_x(J \times \R^2)} \lesssim 2^{4k/5} (\mu')^{100 C_0} c(s) s^{-1} \chi_{k \leq k(s)}^{100} \chi_{k \geq k(s)}^{\delta_2/10}.
\end{equation}
We return to \eqref{smukk}, converting $S'_{\mu',k}$ to $S'_{\mu',k(s)}$, we conclude that
$$ \| P_k \psi_s(s)  \|_{S'_{\mu',k(s)}(J \times \R^2)} \lesssim (\mu')^{100 C_0} c(s) s^{-1} \chi_{k \leq k(s)}^{99} \chi_{k \geq k(s)}^{\delta_2/20}$$
which implies
$$ \| \nabla_x^j P_k \psi_s(s)  \|_{S'_{\mu',k(s)}(J \times \R^2)} \lesssim (\mu')^{100 C_0} c(s) s^{-(j+2)/2} \chi_{k \leq k(s)}^{50} \chi_{k \geq k(s)}^{\delta_2/20}$$
for all $0 \leq j \leq 20$ (say).

Applying Corollary \ref{bcl-2}, we can find a partition ${\mathcal I}_k$ of $I$ into $O_{\mu',C_0}(1)$ intervals $J$ for each $k \in \Z$ such that
$$
s^{\frac{j}{2}} \int_{2^{-2k-2}}^{2^{-2k}} \| \nabla_x^j \psi_s(s) \|_{S'_{\mu',k}(J \times \R^2)}\ ds \lesssim (\mu')^{100 C_0} c(2^{-2k}) 
$$
for all $0 \leq j \leq 20$; by another application of Corollary \ref{bcl-2}, we can thus find a partition ${\mathcal I}$ into $O_{\mu',C_0}(1)$ intervals $J$ such that
\begin{equation}\label{box1}
(\sum_k [s^{\frac{j}{2}} \int_{2^{-2k-2}}^{2^{-2k}} \| \nabla_x^j \psi_s(s) \|_{S'_{\mu',k}(J \times \R^2)}\ ds]^2)^{1/2} \lesssim (\mu')^{100 C_0}. \end{equation}
On the other hand, from \cite[Theorem 6.1]{tao:heatwave3} we have
\begin{equation}\label{box2}
(\sum_k [s^{\frac{j}{2}+1} \sup_{2^{-2k-2}}^{2^{-2k}} \| \nabla_x^j \psi_s(s) \|_{S_k(J \times \R^2)}\ ds]^2)^{1/2} \lesssim 1
\end{equation}
and hence by Lemma \ref{boxf} we have
$$
(\sum_k [s^{\frac{j}{2}+3} \sup_{2^{-2k-2} \leq s \leq 2^{-2k}} \| \partial^2_s \nabla_x^j \psi_s(s) \|_{S'_{\mu',k}(J \times \R^2)}]^2)^{1/2} \lesssim (\mu')^{-O(1)}$$
and thus by \eqref{box1} and the Gagliardo-Nirenberg inequality
$$
(\sum_k [s^{\frac{j}{2}+1} \sup_{2^{-2k-2}}^{2^{-2k}} \| \nabla_x^j \psi_s(s) \|_{S'_{\mu',k}(J \times \R^2)}\ ds]^2)^{1/2} \lesssim (\mu')^{C_0}.$$
Combining this with \eqref{box2}, \eqref{box3} we conclude that
$$
(\sum_k [s^{\frac{j}{2}+1} \sup_{2^{-2k-2}}^{2^{-2k}} \| \nabla_x^j \psi_s(s) \|_{S_{\mu,k}(J \times \R^2)}\ ds]^2)^{1/2} \lesssim 1.
$$
Meanwhile, from Definition \ref{aes} and \eqref{ds2-eq} one has
$$ \sup_{t \in J} \sup_{s > 0} \E( \psi_s(s)[t] ) \lesssim 1$$
and thus by \eqref{ds2-eq} again
$$\| \psi_s \|_{S^1_{\mu'}(J)} \lesssim 1$$
and the claim follows.

\subsection{Proof of small data scattering}

We now show (v). Let $\mu$ be fixed, and suppose $E$ is sufficiently small depending on $\mu$.  Let $I$ be a compact time interval, let $\phi$ be a classical wave map on $I$ with energy at most $E$, and let $\psi_s$ be a heat flow resolution of $\phi$.  It will suffice to show that $\| \psi_s \|_{S^1_\mu(I)} \leq \mu$.

Suppose for contradiction that $\| \psi_s \|_{S^1_\mu(I)} > \mu$.
From Theorem \ref{apriori-thm2}(ii), (iii), we can then find a subinterval $J$ of $I$ such that $\| \psi_s \|_{S^1_\mu(J)} \sim \mu$.  Applying Theorem \ref{apriori-thm2}(v) (with $A, E$ chosen to be $O(1)$, and $\psi'_s$ chosen to be trivial) we conclude that $\| \psi_s \|_{S^1_\mu(J)} \lesssim E^{1/2}$, giving the desired a contradiction if $E$ is small enough.  The proof of Theorem \ref{symscat} is now complete.

\subsection{Proof of local well-posedness}

Now we show (vi).  This will be a simplified version of the arguments in \cite[Section 4]{tao:heatwave3}.

Let $E, \mu, \Phi, t_0$ be as in Theorem \ref{symscat}.  We may assume $\mu$ is sufficiently small depending on $E$.  Let $\eps > 0$ be chosen later, and let $\Phi^{(\eps)}$ be classical data within $\eps$ of $\Phi$ in the energy metric, with energy at most $E$.  By classical well-posedness, we can find a compact interval $I$ with $t_0$ in the interior and a classical wave map $\phi^{(\eps)}$ with initial data $\Phi^{(\eps)}$.  By Theorem \ref{apriori-thm2}(iii), we may shrink $I$ and assume that $\| \psi^{(\eps)}_s \|_{S^1_\mu(I)} \lesssim_E 1$ for some resolution heat flow $\psi^{(\eps)}_{s,t,x}$ of $\phi^{(\eps)}$.  Since all resolution heat flows of the same map differ from each other by a rotation, the same bound is true for all resolution heat flows.

Now if $(\phi_0,\phi_1)$ is classical data which is sufficiently close to $\Phi$ in the energy metric, then it will be within $O(\eps)$ of $\Phi^{(\eps)}$ in that metric, and have energy at most $E$.  Applying Theorem \ref{apriori-thm2}(ii), (iii), (v) and a continuity argument, we see (if $\eps$ is small enough) that there exists a classical wave map $\phi$ on $I$ with initial data $(\phi_0,\phi_1)$ and $\| \Psi_{s,t,x} \|_{S^1_\mu(I)} \lesssim_E 1$ for all resolution heat flows $\psi_s$ of $\phi$, and the claim follows.

The proof of Theorem \ref{symscat} is now complete.

\subsection{Frequency envelope stability and entropy}

By iterating Lemma \ref{freqstable} we obtain a variant for wave maps of bounded entropy:

\begin{lemma}[Local stability of frequency envelope, II]\label{freqstable2}  Let $(\phi,I)$ be an wave map with $(A,\mu)$-entropy $K$, thus let $I = I_1 \cup \ldots \cup I_K$ be such that each $(\phi,I_i)$ is an $(A,\mu)$-wave map, let $t_0 \in I$ and $E>0$, and suppose that $\mu$ is sufficiently small depending on $A,K,E$.  Let $\psi_s$ be a heat flow resolution of $\phi$, and suppose that $c_0$ be a frequency envelope of energy at most $E$, such that
$$ \| \nabla_{t,x} \nabla_x^j \psi_s(s,t_0) \|_{\dot H^0_{k(s)}(\R^2)} \lesssim s^{-(j+2)/2} c_0(s)$$
for all $0 \leq j \leq 10$ and $s > 0$.  Then one has
$$ \| \nabla_x^j \psi_s(s) \|_{S_{k(s)}(I_i \times \R^2)} \lesssim_{A,K,E,j} s^{-(j+2)/2} c_0(s)$$
and
$$ \| \nabla_x^j \psi_x(s) \|_{S_{k(s)}(I_i \times \R^2)} \lesssim_{A,K,E,j} s^{-(j+1)/2} c_0(s)$$
for all $j \geq 0$, $1 \leq i \leq k$, and $s>0$.  In particular, one has
$$ \| \nabla_{t,x} \nabla_x^j \psi_s(s,t) \|_{\dot H^0_{k(s)}(\R^2)} \lesssim_{A,K,E,j} s^{-(j+2)/2} c_0(s)$$
for all $t \in I$, $j \geq 0$, and $s>0$.
\end{lemma}

\section{Hyperbolic repair}

We now establish our main hyperbolic repair result, which incorporates Ingredients 2-4 from Section \ref{delocal-sec}.  More precisely, the purpose of this section is to establish prove

\begin{theorem}[Hyperbolic repair]\label{hyp-repair}
Let $I$ be an interval, and let $\tilde \psi_s$ be a dynamic field on $I$.
Let $c$ be a frequency envelope of at most $E$, and assume that $0 < \mu,\kappa < 1$ are sufficiently small depending on $E$.  We assume the bounds
\begin{equation}\label{psisjk-hyp}
\| \nabla_x^j \tilde \psi_s(s) \|_{S_{\mu,k(s)}(I \times \R^2)} \lesssim_j c(s) s^{-(j+2)/2}
\end{equation}
and
\begin{equation}\label{psisjk-star-hyp}
\| \nabla_x^j \tilde h(s) \|_{S_{\mu,k(s)}(I \times \R^2)} \lesssim_j \mu^2 \kappa c(s) s^{-(j+2)/2}
\end{equation}
for all $j \geq 0$ and $s > 0$ (where $\tilde h$ is the heat-tension field \eqref{heat-tension} of $\tilde \psi_s$), as well as the bound 
\begin{equation}\label{wave-off}
\| P_k \tilde w(0) \|_{N_{k}(I \times \R^2)} \lesssim \mu^2 \kappa c(2^{-2k})
\end{equation}
for every $k \in \R$ (where $\tilde w$ is the wave-tension \eqref{wave-tension} field of $\tilde \psi_s$).  Let $t_0 \in I$ and $\psi_s$ be a dynamic field on $I$ such that
\begin{equation}\label{energy-close}
\tilde \dist_\Energy( \tilde \psi_s[t_0], \psi_s[t_0] ) \lesssim \mu^2 \kappa.
\end{equation}
Then, one has
\begin{equation}\label{darcy}
\dist_{S^1_\mu(I)}(\psi_s, \tilde \psi_s) \lesssim_{E,\mu} \kappa.
\end{equation}
Furthermore, one has
\begin{equation}\label{toast}
\tilde \dist_\Energy( \tilde \psi_s(s)[t], \psi_s(s)[t] ) \lesssim_E \kappa
\end{equation}
for every $s \geq 0$ and $t \in I$, and 
\begin{align}
\| \psi_s(s) - \tilde \psi_s(s) \|_{A^{5}(s)} &\lesssim_E \kappa c(s) s^{-1} \label{fslice-1}\\
\| \Psi_{t,x}(s) - \tilde \Psi_{t,x}(s)  \|_{A^{5}(s)} &\lesssim_E \kappa c(s) s^{-1/2} \label{fslice-2}
\end{align}
for all $s>0$, where the $A^m$ spaces are defined in Definition \ref{sobspace}.
\end{theorem}

The rest of this section is devoted to the proof of Theorem \ref{hyp-repair}.

Let $I, \tilde \psi_s, \psi_s, t_0, \mu, \kappa, c, E$ be as in that theorem.  We allow implied constants to depend on $E$.  We may set $t_0=0$.

\subsection{Reduction to heat flows}

The first step in the argument is to reduce to the case where $\tilde \psi_s$ obeys the heat flow equation; this corresponds to Ingredient 2 from Section \ref{delocal-sec}.

To do this, we perform parabolic repair (Theorem \ref{parab-thm-abstract}) to obtain an intermediate set of dynamic field heat flows $\psi'_s$.  Applying that theorem (and Proposition \ref{algae}) to $\tilde \psi_s$, we can find a dynamic field heat flow $\psi'_s$ which can reconstruct the same map as $\tilde \psi_s$, and obeys the bounds
\begin{align}
\| \nabla_x^j \Psi^*_x(s) \|_{S_{k(s)}^\strong(I \times \R^2)} &\lesssim c(s) s^{-(j+1)/2} \label{parab-1-w}\\
\| \nabla_x^j A^*_x(s) \|_{S_{k(s)}^\strong(I \times \R^2)} &\lesssim c(s)^2 s^{-(j+1)/2}\label{parab-2-w}\\
\| \nabla_x^j \psi^*_s(s) \|_{S_{k(s)}^\strong(I \times \R^2)} &\lesssim c(s) s^{-(j+2)/2}\label{parab-3-w}\\
\| \nabla_x^j \delta \Psi_x(s) \|_{S_{k(s)}^\strong(I \times \R^2)} &\lesssim \mu^2 \kappa c(s) s^{-(j+1)/2}\label{parab-4-w}\\
\| \nabla_x^j \delta A_x(s) \|_{S_{k(s)}^\strong(I \times \R^2)} &\lesssim \mu^2 \kappa c(s)^2 s^{-(j+1)/2}\label{parab-5-w}\\
\| \nabla_x^j \delta \psi_s(s) \|_{S_{k(s)}^\strong(I \times \R^2)} &\lesssim \mu^2 \kappa c(s) s^{-(j+2)/2}\label{parab-6-w}
\end{align}
for all $0 \leq j \leq 15$ and $s > 0$, and similarly with $\nabla_x^2$ replaced by $\partial_s$, where we write $f^*$ for $(\tilde f, f')$ and $\delta f$ for $\tilde f - f'$.

From these estimates and \eqref{algebra}, \eqref{algebra-2}, \eqref{outgo} we conclude

\begin{lemma}[Fixed-time estimates]\label{bread}  We have
\begin{align}
\| \psi^*_{x,t}(s,t) \|_{A^{10}(s)} \lesssim c(s) s^{-1/2}\label{time-1}\\
\| A^*_{x,t}(s,t) \|_{A^{10}(s)} \lesssim c(s)^2 s^{-1/2}\label{time-2}\\
\| \delta \psi_{x,t}(s,t) \|_{A^{10}(s)} \lesssim \mu^2 \kappa c(s) s^{-1/2}\label{time-3}\\
\| \delta A_{x,t}(s,t) \|_{A^{10}(s)} \lesssim \mu^2 \kappa c(s)^2 s^{-1/2}\label{time-4}
\end{align}
for all $t \in I$ and $s>0$.  (The $A^k(s)$ norms were defined in Definition \ref{sobspace}.)  In a similar spirit, one has
\begin{align}
\| \psi^*_s(s,t) \|_{\dot H^1_{k(s)}} \lesssim c(s) s^{-1}\label{stime-1}\\
\| \nabla_{t,x} \psi^*_s(s,t) \|_{\dot H^0_{k(s)}} \lesssim c(s) s^{-1}\label{stime-2}\\
\| \delta \psi^*_s(s,t) \|_{\dot H^1_{k(s)}} \lesssim \mu^2 \kappa c(s) s^{-1}\label{stime-3}\\
\| \delta \nabla_{t,x} \psi^*_s(s,t) \|_{\dot H^0_{k(s)}} \lesssim \mu^2 \kappa c(s) s^{-1}\label{stime-4}
\end{align}
where the $\dot H^m_k$ norms were defined in \eqref{ek-def}.
\end{lemma}

As a consequence of these estimates and \eqref{ds2-eq}, \eqref{energy-dist-def}, one has

\begin{lemma}\label{jamm} We have
$$\dist_{S^1_\mu(I)}(\psi'_s, \tilde \psi_s) \lesssim \mu^2 \kappa$$
and
$$ \tilde \dist_\Energy( \tilde \phi(s)[t], \phi'(s)[t] ) \lesssim \mu^2 \kappa$$
for all $t \in I$ and $s \geq 0$.
\end{lemma}

Next, we show that $\psi'_s$ inherits the property \eqref{wave-off}:

\begin{lemma}[$\psi'_s$ is an approximate wave map]  Let $w'$ be the wave-tension field \eqref{wave-tension} of $\psi'_s$.  Then
\begin{equation}\label{wave-off-2}
\| P_k w'(0) \|_{N_{k}^\strong(I \times \R^2)} \lesssim \mu^2 \kappa c(2^{-2k})
\end{equation}
for all $k \in \Z$.
\end{lemma}

\begin{proof} By scale invariance we may normalise $k=0$.   By \eqref{wave-off} and the triangle inequality, it suffices to show that
\begin{equation}\label{pod}
 \| P_0 \delta w(0) \|_{N_{0}^\strong(I \times \R^2)} \lesssim \mu^2 \kappa c(1).
\end{equation}
This estimate is similar to those in \cite[Lemma 9.9]{tao:heatwave3}, but in that setting, the analogues of $\psi'_{s}, \tilde \psi_{s}$ were assumed to obey both the heat flow and the wave map equation.  Here, $\Psi'_{s,t,x}$ obeys the heat flow equation exactly, and is not yet known to obeys the wave map equation approximately (that is the point of this lemma!), whilst $\tilde \psi_{s}$ obeys the heat flow and wave map equation approximately.  So we cannot repeat the proof of \cite[Lemma 9.9]{tao:heatwave3} directly (as it assumed the wave map equations $w'(0)=\tilde w(0)=0$), and so shall proceed slightly differently.

We will first show the easier estimate
\begin{equation}\label{pop}
\| P_0 w'(0) \|_{N_{0}^\strong(I \times \R^2)} \lesssim c(1)
\end{equation}
and then indicate what changes need to be made to the argument to yield \eqref{pod}.

For brevity we will omit the primes in the analysis that follows, thus writing $\psi_\alpha$ for $\psi'_\alpha$, etc; we also abbreviate $N_0^\strong(I \times \R^2)$ as $N$.  We expand
$$ w(0) = \bigO( \partial^\alpha \psi_\alpha(0) ) + \bigO( A^\alpha(0) \psi_\alpha(0) ).$$
Applying \eqref{psi-odd}, \eqref{psi-even} we see that it suffices to show that\footnote{To justify weak convergence of the infinite series in the $N$ topology one can use the fact that the closed unit ball in $N$ is closed with respect to the uniform topology, which can be shown from the definition of $N$ using such tools as Fatou's lemma in a tedious but straightforward manner; we omit the details.}
\begin{equation}\label{cs1}
\| P_0(\partial^\alpha X_{j,\alpha}) \|_N \lesssim_C 2^{-Cj} c(1)
\end{equation}
and
\begin{equation}\label{cs2}
 \| P_0(X_{j,\alpha} X_{k}^\alpha) \|_N \lesssim_C 2^{-C(j+k)} c(1)
\end{equation}
for all odd $j \geq 1$ and even $k \geq 2$, and all $C>0$.

We begin with \eqref{cs1}.  Applying the product rule and \eqref{xjdef}, we can distribute the derivative $\partial^\alpha$ amongst the various factors in \eqref{xjdef}.  If $j \geq 3$ and the derivative falls on the $\psi_s(s_1)$ factor (which is the highest in frequency), one can use \eqref{parab-1-w}, \eqref{prod1}, \eqref{trilinear-improv}, and Minkowski's inequality to bound this contribution by
$$ \leq O(1)^j \int_{0 < s_1 < \ldots < s_j}  \chi_{0=k(s_1)}^{\delta_1} \chi_{k(s_2) \geq k(s_j)}^{\delta_1}
\frac{c(s_1)}{s_1} \ldots \frac{c(s_j)}{s_j}\ ds_1 \ldots ds_j $$
which by \eqref{sss} is bounded by
$$ \leq O(1)^j j^{-j/2} \int_{0 < s_1 < s_2 \leq s_j}  \chi_{0=k(s_1)}^{\eps \delta_1} \chi_{k(s_2) \geq k(s_j)}^{\delta_1-\delta_0} \frac{c(s_1)}{s_1} \frac{c(s_2)}{s_2} \frac{c(s_j)}{s_j}\ ds_1 ds_2 ds_j $$
which is acceptable (thanks to the envelope properties of $c$).

Similarly if the derivative falls on other factors except for the final factor $\psi_\alpha(s_j)$.  For this, we use \eqref{parab-1-w}, \eqref{prod1}, \eqref{second-prod}, and Minkowski's inequality to bound this contribution by
$$ \leq O(1)^j \int_{0 < s_1 < \ldots < s_j}  \chi_{0 \leq k(s_j)}^{\delta_2} \chi_{0 = k(s_1)}^{\delta_2}
\frac{c(s_1)}{s_1} \ldots \frac{c(s_j)}{s_j}\ ds_1 \ldots ds_j $$
which is acceptable by a similar argument to the preceding.  Note also that this argument also works in the $j=1$ case.

Now we turn to \eqref{cs2}.  By Minkowski's inequality, the left-hand side is bounded by
\begin{align*}
& \int_{0 < s_1 < \ldots < s_j} \int_{0 < s'_1 < \ldots < s'_k} 
&\quad \| P_0(\psi_s(s_1) \ldots \psi_s(s_{j-1}) \partial_\alpha \psi_s(s_j) \psi_s(s'_1) \ldots \psi_(s'_{k-1}) \partial^\alpha \psi_s(s'_k)) \|_N \\
&\quad ds_1 \ldots ds_j ds'_1 \ldots ds'_k.
\end{align*}
Applying \eqref{parab-1-w}, \eqref{prod1}, \eqref{trilinear-improv}, we can bound the integrand by
$$ O(1)^{j+k} \chi_{0=\max(k(s_1),k(s'_1))}^{\eps \delta_1} \chi_{\max(k(s_1),k(s'_1)) \leq \min(k(s_j),k(s'_k))}^{\eps \delta_1}
 \frac{c(s_1)}{s_1} \ldots \frac{c(s_j)}{s_j} \frac{c(s'_1)}{s'_1} \ldots \frac{c(s'_k)}{s'_k}.$$
 We can bound
$$ \chi_{0=\max(k(s_1),k(s'_1))}^{\eps \delta_1} \chi_{\max(k(s_1),k(s'_1)) \leq \min(k(s_j),k(s'_k))}^{\eps \delta_1}
\leq \chi_{0=\max(k(s_1)}^{\eps \delta_1/2} \chi_{k(s_1) = k(s_j)}^{\eps \delta_1/2}
\chi_{0=\max(k(s'_1)}^{\eps \delta_1/2} \chi_{k(s'_1) = k(s'_k)}^{\eps \delta_1/2}$$
and then the integral splits into two integrals of the type previously considered, and the claim follows.

This establishes \eqref{pop}.  The claim \eqref{pod} is similar but uses the Leibniz rule \eqref{disc-leib-eq}
repeatedly to distribute the $\delta$ operator.  One then obtains terms similar to those previously considered, but with one of the $\psi_s$ factors endowed with a $\delta$ rather than an asterisk.  Using \eqref{parab-4-w} for that term instead of \eqref{parab-1-w} one gains the required factor of $\mu^2 \kappa$.  (One also loses a factor of $O(j)$ or $O(j+k)$ from the Leibnitz rule, but this can of course be absorbed in to the exponentially decaying factors in those parameters.)
\end{proof}

From this lemma, we see that $\psi'_s$ obeys essentially the same hypotheses as $\tilde \Psi_{s,t,x}$, but also has the additional advantage of being a heat flow.  If one assumes that Theorem \ref{hyp-repair} holds for the heat flow $\Psi'_{s,t,x}$, then the claim for $\tilde \Psi_{s,t,x}$ then follows from Lemmas \ref{bread}, \ref{jamm}, and the triangle inequality.

Thus we may assume henceforth that $\tilde \Psi_{s,t,x}$ obeys the heat flow equation.

\subsection{Conclusion of the argument}\label{repairconc}

Having concluded the ``parabolic'' portion of the proof of Theorem \ref{hyp-repair}, in which the heat flow equation of $\tilde \Psi_{s,t,x}$ is repaired, we now turn to the ``hyperbolic'' portion, in which we repair the wave maps equation.  This portion of the proof follows the arguments used to establish Theorem \ref{apriori-thm2}(vi) in \cite[Section 9]{tao:heatwave3}, with the added complication that $\tilde \Psi_{s,t,x}$ is now an approximate wave map rather than a true wave map.

We turn to the details.  We first remark that the estimate \eqref{toast} now follows from \eqref{darcy} and Theorem \ref{apriori-thm2}(iv), so we can drop that estimate from the conclusions.  Similarly, \eqref{fslice-1}, \eqref{fslice-2} follows from \eqref{darcy} by repeating the proof of \cite[Lemma 9.7]{tao:heatwave3}.  Thus the only conclusion we need to verify is \eqref{darcy}.

Secondly, we observe that it suffices to prove the claim under the additional hypothesis
\begin{equation}\label{tpsipsi}
 \tilde \Psi_{s,t,x}(t_0) = \Psi_{s,t,x}(t_0).
\end{equation}
Indeed, once one establishes Theorem \ref{hyp-repair} in this special case, we conclude in the general case by a continuity argument (using Theorem \ref{apriori-thm2}(ii)) that there exists a wave map $\Psi'_{s,t,x}$ with $\Psi'_{s,t,x}(t_0) = \tilde \Psi_{s,t,x}(t_0)$ and
$$
\dist_{S^1_\mu(I)}(\Psi'_{s,t,x}, \tilde \Psi_{s,t,x}) \lesssim \kappa
$$
Next, by Theorem \ref{apriori-thm2}(vi) and \eqref{energy-close} we conclude that
$$
\dist_{S^1_\mu(I)}( \Psi'_{s,t,x}, \Psi_{s,t,x} ) \lesssim \kappa$$
and the claim follows from the triangle inequality.

Henceforth we assume \eqref{tpsipsi}.  For this section, we write $\delta f := f - \tilde f$ and $f^* := (f,\tilde f)$ (note that this is a slightly different convention from the preceding section).  Let $C_0 > 1$ be a sufficiently large number.  It will suffice to show that
\begin{equation}\label{darcy-2}
\dist_{S^1(I)}(\Psi'_{s,t,x}, \tilde \Psi_{s,t,x}) \leq C_0 \kappa,
\end{equation}
where the $S^1(I)$ metric is defined just as $S^1_\mu(I)$, but with the $S_{\mu,k}(I \times \R^2)$ norms replaced by the $S_k^\strong(I \times \R^2)$ norms (which are comparable up to factors of $\mu$, by \eqref{sksk-star}).

By a standard continuity argument (using the analogue of Theorem \ref{apriori-thm2}(ii) for the $S^1(I)$ metric), it suffices to establish the claim \eqref{darcy-2} under the additional assumption
\begin{equation}\label{darcy-3}
\dist_{S^1(I)}(\Psi'_{s,t,x}, \tilde \Psi_{s,t,x}) \leq 2C_0 \kappa
\end{equation}
Henceforth $C_0$ is fixed, and the implied constants are allowed to depend on $C_0$.

By \eqref{psisjk-hyp}, \eqref{ds2-eq} we have
$$ \| \tilde \psi_s \|_{S^1_\mu(I)} \lesssim 1$$
and thus by the triangle inequality and \eqref{sksk-star}
$$ \| \psi_s \|_{S^1_\mu(I)} \lesssim 1$$
also (if $\kappa$ is small enough depending on $\mu$).  From \cite[Proposition 7.1]{tao:heatwave3} (which did not use the wave map equation) we also have the following qualitative estimates, which we need to obtain sufficient decay at infinity:

\begin{proposition}[Qualitative estimates]\label{qual-prop} We have
\begin{equation}\label{l1-eq}
 \| \nabla_x^j \psi^*_{t,x}(s,t) \|_{L^1_x(\R^2)} \lesssim_{\Psi^*_{s,t,x},j} \langle s \rangle^{-j/2}
\end{equation}
and
\begin{equation}\label{qual-eq}
 \| \nabla_x^j \nabla_{t,x} \psi^*_s(s,t) \|_{L^1_x(\R^2)} \lesssim_{\Psi^*_{s,t,x},j} \langle s \rangle^{-(j+2)/2}
\end{equation}
for all $t \in I$, $j \geq 0$, and $s \geq 0$.
\end{proposition}

Applying \cite[Corollary 9.6]{tao:heatwave3} and \cite[Lemma 9.7]{tao:heatwave3} (neither of which required $\tilde \Psi_{s,t,x}$ to be a wave map), one now has the estimates
$$ \| \nabla_x^j \psi_s^*(s) \|_{S_{\mu,k(s)}(I \times \R^2)} \lesssim_{j,\eps} \mu^{-\eps} c(s) s^{-(j+2)/2}$$
for all $j \geq 0$, $\eps > 0$, and $s>0$, as well as the Strichartz type estimates
\begin{align}
\| \nabla_x^j \nabla_{t,x} \Psi^*_x(s) \|_{L^q_t L^\infty_x(I \times \R^2)} &\lesssim_{j,\eps} \mu^{(5-\eps)/q} c(s) s^{-(j+2)/2+1/(2q)}  \label{strich1}\\
\| \nabla_x^j \psi^*_s(s) \|_{L^q_t L^\infty_x(I \times \R^2)} &\lesssim_{j,\eps} \mu^{(5-\eps)/q} c(s) s^{-(j+2)/2+1/(2q)} \label{strich2}\\
\| \nabla_x^j \nabla_{t,x} \psi^*_s(s) \|_{L^q_t L^\infty_x(I \times \R^2)} &\lesssim_{j,\eps} \mu^{(5-\eps)/q} c(s) s^{-(j+4)/2+1/(2q)}  \label{strich3}\\
\| \nabla_x^j \Psi^*_{t,x}(s) \|_{L^q_t L^\infty_x(I \times \R^2)} &\lesssim_{j,\eps} \mu^{(5-\eps)/q} c(s) s^{-(j+1)/2+1/(2q)}  \label{strich4}\\
\| \nabla_x^j A^*_{t,x}(s) \|_{L^r_t L^\infty_x(I \times \R^2)} &\lesssim_{j,\eps} \mu^{(5-\eps)/2r} c(s) s^{-(j+1)/2+1/(2r)}  \label{strich5}\\
\| \nabla_x^j \nabla_{t,x} \Psi^*_x(s) \|_{L^\infty_t L^2_x(I \times \R^2)} &\lesssim_{j} c(s) s^{-(j+1)/2}\label{strich6} \\
\| \nabla_x^j \psi^*_s(s) \|_{L^\infty_t L^2_x(I \times \R^2)} &\lesssim_{j} c(s) s^{-(j+1)/2}\label{strich7} \\
\| \nabla_x^j \nabla_{t,x} \psi^*_s(s) \|_{L^\infty_t L^2_x(I \times \R^2)} &\lesssim_{j} c(s) s^{-(j+2)/2}\label{strich8} 
\end{align}
for all $j \geq 0$, $s > 0$, $5 \leq q \leq \infty$, $5/2 \leq r \leq \infty$, and $\eps > 0$.  One also has analogues of all the above estimates in which $\Psi^*$ is replaced by $\delta \Psi$, $c$ is replaced by $\kappa c$, etc., and all powers of $\mu$ are discarded.  

Set $F := (\psi_\alpha \wedge \psi_i) D_i \psi^\alpha$, and define $\tilde F$, $F^*$, $\delta F$ accordingly.  Repeating the proof of \cite[Lemma 9.8]{tao:heatwave3} (which, again, did not require $\tilde \Psi_{s,t,x}$ to be a wave map) one concludes that
\begin{equation}\label{force-eq} \| P_k F^*(s) \|_{N_k(I \times \R^2)} \lesssim_{j,\eps} \mu^{2-\eps} c(s) s^{-1} \chi_{k \geq k(s)}^{\delta_2} \chi_{k \leq k(s)}^j
\end{equation}
and
\begin{equation}\label{force-diff-eq} \| P_k \delta F(s) \|_{N_k(I \times \R^2)} \lesssim_{j,\eps} \mu^{1-\eps} \kappa c(s) s^{-1} \chi_{k \geq k(s)}^{\delta_2} \chi_{k \leq k(s)}^j
\end{equation}
for all $j \geq 0$, $s > 0$, and integers $k$, and $\eps > 0$.

Next, we introduce the wave tension field $w = D^\alpha \psi_\alpha$, and similarly define $\tilde w$, $w^*$, $\delta w$.  We would like to estimate these expressions using an analogue of \cite[Lemma 9.9]{tao:heatwave3}.  Unfortunately, and in contrast to the previous material from \cite[Section 9]{tao:heatwave3} that we have been using, the proof of \cite[Lemma 9.9]{tao:heatwave3} uses the wave map equation $w^*(0) = 0$ and so cannot be directly repeated here.  Instead, from \eqref{wave-off} (and the wave maps equation for $\Psi_{s,t,x}$) we have that
\begin{equation}\label{poke}
\| P_k w^*(0) \|_{N_{k}(I \times \R^2)} \lesssim \mu^2 \kappa c(2^{-2k})
\end{equation}
for all $k \in \Z$.  This is enough to obtain a slightly weaker version of \cite[Lemma 9.9]{tao:heatwave3}:

\begin{lemma}[Control of wave-tension field] We have
\begin{equation}\label{pkw-1}
\| P_k w^* \|_{N_k(I \times \R^2)} \lesssim_{j,\eps} \mu^{2-\eps} c(2^{-2k}) \chi_{k \leq k(s)}^j
\end{equation}
and
\begin{equation}\label{pkw-2}
 \| P_k \partial_s w^* \|_{N_k(I \times \R^2)} \lesssim_{j,\eps} \mu^{2-\eps} c(s) s^{-1} \chi_{k \geq k(s)}^{\delta_2/10} \chi_{k \leq k(s)}^j
 \end{equation}
for all $j \geq 0$, $s > 0$, $\eps > 0$, and integers $k$.  Similarly we have
\begin{equation}\label{pkw-1-diff}
\| P_k \delta w \|_{N_k(I \times \R^2)} \lesssim_{j,\eps} \mu^{1-\eps} \kappa c(2^{-2k}) \chi_{k \leq k(s)}^j
\end{equation}
and
\begin{equation}\label{pkw-2-diff}
 \| P_k \partial_s \delta w \|_{N_k(I \times \R^2)} \lesssim_{j,\eps} \mu^{1-\eps} \kappa c(s) s^{-1} \chi_{k \geq k(s)}^{\delta_2/10} \chi_{k \leq k(s)}^j
 \end{equation}
for the same range of $j,s,\eps,k$.
\end{lemma}

\begin{proof} We repeat the proof of \cite[Lemma 9.9]{tao:heatwave3} with some technical modifications, basically having to do with the slightly worse control on $w^*$ and $\delta w$ at low frequencies.

Fix $j$.  Without loss of generality we may take $j \geq 10$ (say).

From \eqref{w-eq} we have the schematic heat equation
\begin{equation}\label{wsax}
\partial_s w^* = \Delta w^* + \bigO( A^*_x \nabla_x w^* ) + \bigO( (\nabla_x A^*_x) w^* ) + \bigO( (\Psi^*_x)^2 w^* ) + \bigO( F^* ).
\end{equation}
From Duhamel's formula and \eqref{cov-wave} we conclude that
\begin{equation}\label{w-duh}
w^*(s) = e^{s\Delta} w^*(0) +
\int_0^s e^{(s-s')\Delta} (\bigO( A^*_x \nabla_x w^* ) + \bigO( (\nabla_x A^*_x) w^* ) + \bigO( (\Psi^*_x)^2 w^* ) + \bigO( F^* ))(s')\ ds'.
\end{equation}
Let
$$f_j(s) := \sup_k c(2^{-2k})^{-1} \chi_{k \leq k(s)}^{-j} \| P_k w^* \|_{N_k(I \times \R^2)},$$
thus
\begin{equation}\label{pkw}
\| P_k w^* \|_{N_k(I \times \R^2)} \lesssim f_j(s) c(2^{-2k}) \chi_{k \leq k(s)}^{j}
\end{equation}
for all $k$.  The claim \eqref{pkw-1} is then equivalent to showing that $f_j(s) \lesssim_{j,\eps} \mu^{2-\eps}$ for all $s>0$, $j \geq 0$, and $\eps > 0$.

Observe that the convolution kernel of $P_k e^{(s-s')\Delta}$ has total mass $O_j( \chi_{k \leq k(s-s')}^{100j} )$ for any $j$. From \eqref{w-duh}, \eqref{poke}, and Minkowski's inequality, we can thus bound $f_j(s)$ by
\begin{align*}
 \lesssim_j \mu^2 \kappa +
 &\sup_k c(2^{-2k})^{-1} \chi_{k \leq k(s)}^{-j} 
\int_0^s \chi_{k \leq k(s-s')}^{100j} \times \\
&\quad \| P_k (\bigO( A^*_x \nabla_x w^* ) + \bigO( (\nabla_x A^*_x) w^* ) + \bigO( (\Psi^*_x)^2 w^* ) + \bigO( F^* ))(s') \|_{N_k(I \times \R^2)}\ ds'.
\end{align*}
For any $s' > 0$, we see from \eqref{parab-1}, \eqref{parab-2}, \eqref{prod1} that
$$ \| \nabla_x A^*_x(s') \|_{S_{k(s')}(I \times \R^2)}, \| (\Psi^*_x)^2(s') \|_{S_{k(s')}(I \times \R^2)} \lesssim c(s')^2 (s')^{-1}$$
and thus by \eqref{pkw}, \eqref{second-prod} and dyadic decomposition we have
$$ \| P_k (\bigO( (\nabla_x \Psi^*_x) w^* ) + \bigO( (\Psi^*_x)^2 w^* ))(s') \|_{N_k(I \times \R^2)}
\lesssim_{j} (s'/s)^{\delta_2/2} c(s')^2 c(2^{-2k}) f_j(s') (s')^{-1} \chi_{k \geq k(s')}^{\delta_2/10} \chi_{k \leq k(s')}^{j}.$$
A similar argument also gives
$$\| P_k (\bigO( \Psi^*_x \nabla_x w^* ) \|_{N_k(I \times \R^2)}
\lesssim_{j} c(s')^2 c(2^{-2k}) (s'/s)^{\delta_2/2} f_j(s') (s')^{-1} \chi_{k \geq k(s')}^{\delta_2/10} \chi_{k \leq k(s')}^{j-1}.$$ 
Combining these estimates with \eqref{force-eq}, we conclude that
\begin{align*}
f_j(s) &\lesssim_{j,\eps} \mu^2 \kappa + \sup_k c(2^{-2k})^{-1} \chi_{k \leq k(s)}^{-j} \\
&\quad \int_0^s \chi_{k \leq k(s-s')}^{100j} [(s'/s)^{\delta_2/2} c(s')^2 c(2^{-2k}) f_j(s') \chi_{k \geq k(s')}^{\delta_2/10} \chi_{k \leq k(s')}^{j-1} \\
&\quad\quad + \mu^{2-\eps} c(s') \chi_{k \geq k(s')}^{\delta_2} \chi_{k \leq k(s')}^{100j} ]\ \frac{ds'}{s'}
\end{align*}
for all $\eps > 0$.

From \eqref{sm} we have
$$ c(s') \chi_{k \geq k(s')}^{\delta_2} \chi_{k \leq k(s')}^{100j} \lesssim_j c(2^{-2k}) \chi_{k \geq k(s')}^{\delta_2/2} \chi_{k \leq k(s')}^{99j}$$
and
$$ \chi_{k \leq k(s)}^{-j} \lesssim_j \chi_{k \leq k(s-s')}^{-j} \chi_{k \leq k(s')}^{-j}$$
while a direct computation shows that
$$ \int_0^s \chi_{k \leq k(s-s')}^{99j} \chi_{k \geq k(s')}^{\delta_2/2} \chi_{k \leq k(s')}^{98j}\ \frac{ds'}{s'}
\lesssim_j 1$$
so we can simplify slightly to
\begin{align*}
f_j(s) &\lesssim_{j,\eps} \mu^{2-\eps} + \sup_k \\
&\quad \int_0^s \chi_{k \leq k(s-s')}^{99j} (s'/s)^{\delta_2/2} c(s')^2 f_j(s') \chi_{k \geq k(s')}^{\delta_2/10} \chi_{k \leq k(s')}^{-1}\ \frac{ds'}{s'}.
\end{align*}
Bounding 
$$ \chi_{k \leq k(s-s')} \chi_{k \leq k(s')}^{-1} \lesssim ((s-s')/s)^{-1/2}$$
we conclude the integral inequality
$$ f_j(s) \lesssim_{j,\eps} \mu^{2-\eps} + \int_0^s ((s-s')/s)^{-1/2} (s'/s)^{\delta_2/2} c(s')^2 f_j(s') \frac{ds'}{s'}$$
for all $\eps > 0$. Also, from \eqref{fl1l2}, Proposition \ref{qual-prop}, Bernstein's inequality, and H\"older's inequality we have $f_j(s) \to 0$ as $s \to 0$.  Applying Lemma \ref{gron-lem-2} we have $f_j(s) \lesssim_{j,\eps} \mu^{2-\eps}$, and the claim \eqref{pkw-1} follows.

The claim \eqref{pkw-2} then follows by using \eqref{wsax} to write $\partial_s w^*$ in terms of $\nabla_x^2 w^*$ (which can be controlled by \eqref{pkw-1}), together with several additional terms which were already estimated in the required manner in the first part of the proof.

The proof of \eqref{pkw-1-diff} follows from \eqref{pkw-1} as in previous propositions.  A little more specifically, we define
$$\delta f_j(s) := \sup_k \kappa^{-1} c(s)^{-1} \chi_{k \leq k(s)}^{-j} \| P_k \delta w \|_{N_k(I \times \R^2)}$$
and repeat the above arguments (and using the bound \eqref{pkw-1} just established) to eventually obtain the integral inequality
$$ \delta f_j(s) \lesssim_{j,\eps} \mu^{1-\eps} + \int_0^s ((s-s')/s)^{-1/2} (s'/s)^{\delta_2/2} c(s')^2 \delta f_j(s') \frac{ds'}{s'}$$
Using the continuity method as before we obtain $\delta f_j(s) \lesssim_{j,\eps} \mu^{1-\eps}$ for all $s>0$, $j \geq 0$, and $\eps > 0$, giving \eqref{pkw-1-diff}.  The proof of \eqref{pkw-2-diff} then follows by the differenced version of \eqref{wsax}.  We leave the details to the reader.
\end{proof}

The bounds on $\partial_s w^*$ and $\partial_s \delta w$ in the above lemma are identical to those in \cite[Lemma 9.9]{tao:heatwave3}.  We may then repeat the proof of \cite[Lemma 9.10]{tao:heatwave3} to conclude that
\begin{equation}\label{psis-iter-eq} \| P_k \psi^*_s(s) \|_{S_k(I \times \R^2)} \lesssim_{j,\eps}  \| \nabla_{t,x} P_k \psi^*_s(s,t_0) \|_{L^2_x(\R^2)} + \mu^{2-\eps} c(s) s^{-1} \chi_{k \geq k(s)}^{\delta_2/10} \chi_{k \leq k(s)}^j 
\end{equation}
and
\begin{equation}\label{psis-iter-diff} \| P_k \delta \psi_s(s) \|_{S_{k}(I \times \R^2)} \lesssim_{j,\eps} \| \nabla_{t,x} P_k \delta \psi_s(s,t_0) \|_{L^2_x(\R^2)} + \mu^{1-\eps} \kappa c(s) s^{-1} \chi_{k \geq k(s)}^{\delta_2/10} \chi_{k \leq k(s)}^j 
\end{equation}
for all $j \geq 0$, $\eps > 0$, $s > 0$, and integers $k$.  In particular, by \eqref{tpsipsi}, one has
$$ \| P_k \delta \psi_s(s) \|_{S_{k}(I \times \R^2)} \lesssim_{\eps} \mu^{1-\eps} \kappa c(s) s^{-1} \chi_{k \geq k(s)}^{\delta_2/10} \chi_{k \leq k(s)}^{100},$$
which by the analogue of \eqref{ds1-eq} for $S_k$ gives \eqref{darcy} as required.

\section{Frequency delocalisation implies spacetime bound}\label{freqbound-sec}

We now begin the proof of Theorem \ref{freqbound}.  We omit the superscript $(n)$ from the solution $\phi$, and allow all implied constants to depend on $E_0$ and $\eps$.  

\subsection{Construction of the low frequency component}

The low frequency component can be constructed globally without difficulty.  We let $\phi^{lo}$ be the evolution of $\phi$ under the heat flow by $K_n^{1/2}$, thus
$$ \phi^{lo} = \phi(K_n^{1/2}),$$
and the associated heat flow resolution $\psi^{lo}_s$ is given by
$$
\psi^{lo}_s(s) := \psi_s(s+K_n^{1/2})
$$
for all $s \geq 0$, where $\psi_s$ is a heat flow resolution of $\phi$.  In particular, the energy spectral distribution of $\phi^{lo}[0]$ is simply the translation of that of $\phi[0]$ by $K_n^{1/2}$:
$$ \ESD(\phi^{lo}[0])(s) = \ESD(\phi[0])(s + K_n^{1/2}).$$
From \eqref{enbound}, \eqref{energy-ident} one has
\begin{equation}\label{eo-1}
 \int_0^\infty \ESD(\phi[0])(s)\ ds = E_0 + o_{n \to \infty}(1)
\end{equation}
and thus by \eqref{fdl-1}, \eqref{fdl-2}
$$ \eps/2 - o_{n \to \infty}(1) \leq \int_0^\infty \ESD(\phi^{lo}[0])(s)\ ds \leq E_0 - \eps/2 + o_{n \to \infty}(1).$$
Thus, by \eqref{energy-ident}
\begin{equation}\label{energy-lo}
\eps/4 \leq \E( \phi^{lo}[0] ) \leq E_0 - \eps/4
\end{equation}
if $n$ is large enough.  But by hypothesis, $E_0 - \eps/4$ is good.  Then by Lemma \ref{quant}, we can extend $\phi^{lo}$ to a global classical wave map, which will then have a $(O(1), 1)$-entropy of $O(1)$.  

Let $\mu > 0$ be a small quantity (depending only on $\eps$) to be chosen later, and let $\mu' > 0$ be an even smaller quantity (depending on $\mu, \eps$) to be chosen later.  We allow the decay rates in the $o_{n \to \infty}()$ notation to depend on $\mu,\mu'$. 
By Proposition \ref{symscat}(iii), we may find a partition ${\mathcal I}_n$ of $I_n$ into $O_{\mu'}(1)$ intervals $J$, such that $\phi^{lo}$ is a $(O(1),\mu')$-wave map on each such $J$.

By Lemma \ref{gapstable}, we have
$$ \dist_\Energy( \phi(K_n^{1/4})[0], \phi^{lo}[0] ) = o_{n \to \infty}(1).$$

To finish the proof of Theorem \ref{freqbound}, it now suffices to show

\begin{proposition}[$\phi$ stays close to $\phi^{lo}$]\label{sweden} Let $n$ be a sequence going to infinity, let $L_n \to \infty$ be such that $L_n \leq K_n^{1/4}$, let $J \in {\mathcal I}_n$, and let $t_n \in J$ be such that
\begin{equation}\label{phocus}
 \int_{1/L_n \leq s \leq L_n} \ESD(\phi[t_n])(s)\ ds = o_{n \to \infty}(1)
\end{equation}
and
\begin{equation}\label{phocus-2}
\dist_\Energy( \phi^{lo}[t_n], \phi(L_n^{1/4})[t_n] ) = o_{n \to \infty}(1)
\end{equation}
Then there exists $L'_n \leq L_n$ with $L'_n \to \infty$ such that
\begin{equation}\label{summer}
\int_{1/L'_n \leq s \leq L'_n} \ESD(\phi[t])(s)\ ds = o_{n \to \infty}(1)
\end{equation}
and
\begin{equation}\label{simmer2}
\dist_\Energy( \phi^{lo}[t], \phi((L'_n)^{1/4})[t] ) = o_{n \to \infty}(1)
\end{equation}
for all $t \in J$.  Furthermore, $(\phi\downharpoonright_J,J)$ has a $(O(1),1)$-entropy of $O_{\mu}(1)$.
\end{proposition}

Indeed, by iterating this proposition once for each interval $J$ in ${\mathcal I_n}$ (starting with $L_n = K_n^{1/4}$ and decreasing $L_n$ each time), we conclude that $(\phi,I_n)$ has a $(O(1),1)$-entropy of $O(1)$, as required.

\subsection{Control of initial data}

We begin the proof of Proposition \ref{sweden}.  The first step is to control the differentiated fields $\Psi_{s,t,x}$ at time $t_n$:

\begin{lemma}\label{psitn-lem}  There exists a frequency envelope $c = c^{(n)}$ of energy $O(1)$ (depending on $n$ and $t_n$) such that
\begin{equation}\label{css}
\int_{L_n^{-0.9} \leq s \leq L_n^{0.9}} c(s)^2 \frac{ds}{s} = o_{n \to \infty}(1)
\end{equation}
and such that 
\begin{equation}\label{l2s-fixed-special-psitn}
\begin{split}
\| \nabla_x^{j+1} \Psi_{t,x}(s,t_n) \|_{L^2_x(\R^2)} 
+ \| \nabla_x^{j} \Psi_{t,x}(s,t_n) \|_{L^\infty_x(\R^2)} \quad &\\
+ \| \nabla_x^j \nabla_{t,x} \Psi_x(s,t_n) \|_{L^2_x(\R^2)} +
\| \nabla_x^j \psi_s(s,t_n) \|_{L^2_x(\R^2)} \quad &\\
+ \| \nabla_x^{j-1} \nabla_{t,x} \psi_s(s,t_n) \|_{L^2_x(\R^2)}
&\lesssim_{j} c(s) s^{-(j+1)/2} 
\end{split}
\end{equation}
for all $s > 0$ and $j \geq 0$, with the convention that we drop all terms involving $\nabla_x^{-1}$. In particular, from Definition \ref{sobspace} we have
\begin{equation}\label{spak}
\| s^{1/2} \Psi_{t,x}(s,t_n) \|_{A^k(s)} 
+  \| s \Psi_s(s,t_n) \|_{A^k(s)} 
\lesssim_k c(s)
\end{equation}
for all $s>0$ and $k \geq 0$.  Also, we have
\begin{equation}\label{jojoe}
\| \partial_s^i \nabla_x^j \nabla_{t,x} \psi_s(s,t_n) \|_{\dot H^0_{k(s)}} \lesssim_{i,j} c(s) s^{-(2i+j+2)/2}
\end{equation}
for all $s>0$ and $j,i \geq 0$.
\end{lemma}

\begin{proof} We omit the time index $t_n$ for brevity.  Applying \cite[Theorem 7.2]{tao:heatwave3} and\footnote{Lemma 7.6 of \cite{tao:heatwave3} only yields \eqref{jojoe} for $j \leq 10$ and $i=0$, but the $j>10$ cases and $i=0$ can then be deduced from the $j \leq 10$ case and the estimates in \eqref{l2s-fixed-special-psitn}, and then the $i>0$ case then follows by using the equations of motion to replace $\partial_s$ with spatial derivatives and using \eqref{l2s-fixed-special-psitn} again.} \cite[Lemma 7.6]{tao:heatwave3}, one can find an envelope $c$ obeying the bounds \eqref{jojoe}; indeed, from the construction in the proof of that theorem, one can take
\begin{equation}\label{cs-mash}
 c(s) := \sum_{j=0}^{10} \sup_{s' > 0} \min( (s'/s)^{\delta_0}, (s/s')^{\delta_0} )
(s')^{(j+2)/2} \| \nabla_{t,x} \nabla_x^j \psi_s(s') \|_{\dot H^0_{k(s')}}.
\end{equation}
It remains to establish \eqref{css}.  We first observe from \eqref{phocus} and \eqref{esd-def}, and the Bochner-Weitzenb\"ock inequality
\begin{align*}
\int_{\R^2} D_j \psi_i(s,x)|^2 + \frac{1}{2} |\psi_i \wedge \psi_j(s,x)|^2\ dx &\leq \int_{\R^2} |D_i \psi_i(s,x)|^2\ dx \\
&= \|\psi_s(s)\|_{L^2_x(\R^2)}
\end{align*}
we have
\begin{equation}\label{joo}
\int_{1/L_n}^{L_n} \| D_x \psi_{t,x}(s)\|_{L^2_x(\R^2)}^2\ ds = o_{n \to \infty}(1).
\end{equation}
From \eqref{joo} and the diamagnetic inequality one has
$$ \int_{1/L_n}^{L_n} \| \nabla_x |\psi_{t,x}(s)| \|_{L^2_x(\R^2)}^2\ ds = o_{n \to \infty}(1);$$
meanwhile, from \eqref{corheat} one has
$$ \int_{1/L_n}^{L_n} s^{-2/p} \| \psi_{t,x}(s) \|_{L^p_x(\R^2)}^2\ ds \lesssim_p 1$$
for any $2 < p < \infty$.  Interpolating these facts using the Gagliardo-Nirenberg and Cauchy-Schwarz inequalities, we conclude that
$$ \int_{1/L_n}^{L_n} s^{-2/p} \| \psi_{t,x}(s) \|_{L^p_x(\R^2)}^2\ ds = o_{n \to \infty;p}(1)$$
for any $2 < p < \infty$.  Also, from \eqref{l2s-fixed-special-psitn} one has
$$ \int_0^\infty s^j \| \nabla_x^j D_x \psi_{t,x}(s)\|_{L^2_x(\R^2)}^2\ ds \lesssim_j 1$$
for any $j \geq 0$,
so by interpolation with \eqref{joo}
$$ \int_0^\infty s^j \| \nabla_x^j D_x \psi_{t,x}(s)\|_{L^2_x(\R^2)}^2\ ds = o_{n \to \infty;j}(1)$$
for any $j \geq 0$.  Meanwhile, from
From \eqref{psit-eq} one has
\begin{align*}
\nabla_{t,x} \psi_s &= D_{t,x} \psi_s - A_{t,x} \psi_s \\
&= \partial_s \psi_{t,x} + \bigO( A_{t,x} \psi_s ) \\
&= \bigO( \nabla_x D_x \psi_{t,x} ) + \bigO( A_x D_x \psi_{t,x} ) + \bigO( \psi_x \psi_x \psi_{t,x} ) + \bigO( A_{t,x} \psi_s )
\end{align*}
and from Proposition \ref{corbound} one has $\|A_{t,x}(s)\|_{L^\infty_x(\R^2)} = O(s^{-1/2})$, $\| \psi_x(s)\|_{L^p(\R^2)} = O(s^{1/p-1/2})$.  Putting all this together we see that
$$ \int_{1/L_n}^{L_n} s\| \partial_t \psi_s(s)\|_{L^2_x(\R^2)}^2\ ds = o_{n \to \infty}(1)$$
while from \eqref{l2s-fixed-special-psitn} one has
$$ \int_{1/L_n}^{L_n} s^{j+1} \| \nabla_{t,x} \nabla_x^j \psi_s(s)\|_{L^2_x(\R^2)}^2\ ds \lesssim_j 1.$$
Interpolating this, one obtains
$$ \int_{1/L_n}^{L_n} \| \nabla_t \nabla_x^j \psi_s(s)\|_{\dot H^1_{k(s)}(\R^2)}^2\ ds = o_{n \to \infty;j}(1)$$
for each $j$.  Inserting this into \eqref{cs-mash} one obtains \eqref{css} as required.
\end{proof}

This has the following consequence:

\begin{lemma}[Some energy at low frequencies]\label{some-energy}
$$ \frac{1}{2} \int_{\R^2} |\psi_x(L_n^{0.8},t_n,x)|^2 + |\psi_t(L_n^{0.8},t_n,x)|^2 \ dx \geq \eps/8 - o_{n \to \infty}(1).$$
\end{lemma}

\begin{proof} 
In view of \eqref{phocus-2}, \eqref{energy-lo}, energy conservation of $\phi^{lo}$, and the triangle inequality, it suffices to show that
$$
\dist_\Energy( \phi(L_n)[t_n], \phi(L_n^{0.8})[t_n] ) = o_{n \to \infty}(1).$$
But this follows from Lemma \ref{psitn-lem} and Lemma \ref{gapstable}.  
\end{proof}

\subsection{Control of the low frequency component}

We continue the proof of Proposition \ref{sweden}.   The plan is to run the scheme in Section \ref{delocal-sec}, thus decomposing $\psi_s[t_n]$ into ``low frequency'' and ``high frequency'' components $\tilde \psi^{lo}_s[t_n], \tilde \psi^{hi}_s[t_n]$, which one then repairs to instantaneous heat flows $\psi^{lo}_s[0], \psi^{hi}_s[t_n]$, then to dynamic field heat flows $\psi^{lo}_s$, $\psi^{hi}_s$, leading to a superposition $\tilde \psi_s$, which we then compare to $\psi_s$.

The low frequency component has already been constructed\footnote{The reason why we can define $\phi^{lo}$ globally (as opposed to $\phi^{hi}$, which has to be constructed anew for each interval $J$ in ${\mathcal I}_n$, is that the nonlinear effects of the wave map equation (even in the caloric gauge, and even with large frequency separation) have a non-trivial ``$\operatorname{low} \times \operatorname{high} \to \operatorname{high}$'' component that will continually force us to update the high frequency component, whereas the dual ``$\operatorname{high} \times \operatorname{high} \to \operatorname{low}$'' component is much smaller and thus has a negligible impact on the low frequency component.  This feature of the wave map dynamics is not present in other equations to which the minimal blowup solution strategy has been applied; for instance, for the energy-critical NLS, both types of frequency dynamics are negligible.}: we can take $\tilde \psi^{lo}_s[t_n] = \psi^{lo}_s[t_n]$ be the instantaneous data of $\psi^{lo}_s$ at time $t_n$.  

We can now control the low frequency component:

\begin{lemma}[$\phi^{lo}$ is low frequency]  Let the hypotheses be as in Proposition \ref{sweden}.  Then there exists a frequency envelope $c^{lo}$ of energy $O(1)$ such that
\begin{equation}\label{slice-dice-lo}
\| \nabla_x^j \psi^{lo}_s(s) \|_{S_{\mu',k(s)}(J \times \R^2)} \lesssim_j c^{lo}(s) s^{-(j+2)/2}
\end{equation}
and
\begin{equation}\label{slice-dice-lo-2}
\| \nabla_x^j \psi^{lo}_x(s) \|_{S_{\mu',k(s)}(J \times \R^2)} \lesssim_j c^{lo}(s) s^{-(j+1)/2}
\end{equation}
for all $j \geq 0$.  Furthermore, one has
\begin{equation}\label{clo}
\int_{s \geq L_n^{-0.7}} c^{lo}(s)^2 \frac{ds}{s} = o_{n \to \infty}(1).
\end{equation}
\end{lemma}

\begin{proof}  From the analogue of Lemma \ref{psitn-lem} for time $t=0$, one can find a frequency envelope $c_0$ of energy $O(1)$ such that
$$
\| \nabla_x^j \nabla_{t,x} \psi_s(s,0) \|_{\dot H^0_{k(s)}} \lesssim_{j} c_0(s) s^{-(j+2)/2}$$
for all $j \geq 0$ and $s>0$, and such that
$$
\int_{K_n^{-0.9} \leq s \leq K_n^{0.9}} c_0(s)^2 \frac{ds}{s} = o_{n \to \infty}(1).$$
Since $\psi_s^{lo}(s,0) = \psi_s(s+K_n^{1/4},0)$ by construction, we have
$$
\| \nabla_x^j \nabla_{t,x} \psi^{lo}_s(s,0) \|_{\dot H^0_{k(s)}} \lesssim_{j} c'_0(s) s^{-(j+2)/2}$$
for all $j \geq 0$ and $s>0$, where $c'_0(s) := (\frac{s}{s+K_n^{1/4}})^{\delta_0} c_0(s+K_n^{1/4})$.  Note that $c'_0$ is essentially (up to multiplicative constants) a frequency envelope of energy $O(1)$.

Applying Lemma \ref{freqstable2}, one thus has
$$
\| \nabla_x^j \psi^{lo}_s(s) \|_{S_{\mu',k(s)}(J \times \R^2)} \lesssim_{\mu',j} c'_0(s) s^{-(j+2)/2}
$$
and
$$
\| \nabla_x^j \psi^{lo}_x(s) \|_{S_{\mu',k(s)}(J \times \R^2)} \lesssim_{\mu',j} c'_0(s) s^{-(j+1)/2}.
$$
Meanwhile, since $\psi^{lo}_s$ is a $(O(1),\mu')$-wave map, we see (from \eqref{ds2-eq} and Lemma \ref{freqstable}) that there exists another frequency envelope $\tilde c$ of energy $O(1)$ such that
$$
\| \nabla_x^j \psi^{lo}_s(s) \|_{S_{\mu',k(s)}(J \times \R^2)} \lesssim_{j} \tilde c(s) s^{-(j+2)/2}
$$
and
$$
\| \nabla_x^j \psi^{lo}_x(s) \|_{S_{\mu',k(s)}(J \times \R^2)} \lesssim_{j} \tilde c(s) s^{-(j+1)/2}.
$$
Taking $c^{lo}$ to be the infimum of $\tilde c$ and $O_{\mu'}(c')$, we obtain the claim.
\end{proof}

\subsection{Construction of the high frequency component}

Now we turn to the construction of the high-frequency component.  We first define the preliminary dynamic field $\tilde \psi^{hi}_s[t_n]$ by setting
\begin{equation}\label{psihi}
\tilde \psi^{hi}_s(s,t_n,x) := \eta(s/L_n^{0.7}) \psi_s(s,t_n,x),
\end{equation}
and
\begin{equation}\label{psit-hi}
\partial_t \tilde \psi^{hi}_s(t_n,x) := \eta(s/L_n^{0.7}) \partial_t \psi_s(s,t_n,x)
\end{equation}
where $\eta: \R^+ \to \R$ is a smooth cutoff that equals $1$ on $[0,1]$ and is supported on $[0,2]$; thus $\tilde \psi^{hi}_s$ is a smooth truncation of $\psi_s$ to small values of $s$ (which correspond to high frequency scales). 

We now place some estimates on $\tilde \psi^{hi}$ at time $t_n$:

\begin{lemma}[Estimates on $\tilde \psi^{hi}$] \label{Dels-lem}  Write $\delta \Psi_{t,x}(s,t_n) := \tilde \Psi^{hi}_{t,x}(s,t_n)  - \Psi_{t,x}(s,t_n) + \Psi_{t,x}(L_n^{0.8},t_n)$; then we have
\begin{equation}\label{dels}
 \| \delta \Psi_{t,x}(s,t_n) \|_{A^j_{k(s)}(\R^2)} 
\lesssim_j \frac{c(s)}{s^{1/2}} [ 1_{s \geq L_n^{0.6}} + o_{n \to \infty}(1) ] 
\end{equation}
for all $j \geq 0$ and $0 \leq s \leq L_n^{0.8}$.  In particular we have
\begin{equation}\label{flow}
 \| \delta \Psi_{t,x}(s,t_n) \|_{L^2_x(\R^2)} \lesssim o_{n \to \infty}(1)
 \end{equation}
for all $0 \leq s \leq L_n^{0.8}$.
\end{lemma}

\begin{proof}  For brevity, we omit the $t_n$ time coordinate here.  Write $S := L_n^{0.8}$.  By construction, $\delta \Psi_{t,x}(S) = 0$.  The strategy is then to propagate $\delta \Psi_{t,x}$ backwards in the heat-temporal variable $s$ from $S$ to $0$.

We first need an evolution equation for $\delta \Psi_{t,x}$.  We split $\delta \Psi_{t,x}$ into its two components $\delta \psi_{t,x}$, $\delta A_{t,x}$ in the obvious manner.

From the zero-torsion property \eqref{zerotor} (and the caloric gauge) we have
$$ \partial_s \psi_{t,x} = \nabla_{t,x} \psi_s + A_{t,x} \psi_s$$
and similarly
$$ \partial_s \tilde \psi^{hi}_{t,x} = \nabla_{t,x} \tilde \psi^{hi}_s + \tilde A^{hi}_{t,x} \tilde \psi^{hi}_s;$$
subtracting the two statements, we obtain
$$ \partial_s \delta \psi_{t,x} = \nabla_{t,x} \delta \psi_s + \bigO( \delta A_{t,x} \psi^*_s ) + \bigO( A_x \delta \psi_s ) + \bigO( A_{t,x}(S) \psi^*_s )$$
where $\psi^*_s := (\psi_s, \tilde \psi^{hi}_s)$, $A^*_x := (A_x, \tilde A^{hi}_x)$, and
\begin{equation}\label{deltap}
\delta \psi_s := \psi_s - \tilde \psi^{hi}_s = (1-\eta(s/L_n^{0.7})) \psi_s
\end{equation}
(thanks to \eqref{psihi}).
A similar argument using the constant curvature property \eqref{constcurv} in place of \eqref{zerotor} gives
$$ \partial_s \delta A_{t,x} = \bigO( \delta \psi_{t,x} \psi^*_s ) + \bigO( \psi_{t,x} \delta \psi_s ) + \bigO( \psi_{t,x}(S) \psi^*_s )$$
where $\psi^*_s := (\psi_s, \tilde \psi^{hi}_s)$.  Writing $\delta \Psi_{t,x} := (\delta \psi_{t,x}, \delta A_{t,x})$, etc., we can unify these two equations as
\begin{equation}\label{spider}
 \partial_s \delta \Psi_{t,x} = \bigO( \nabla_{t,x} \delta \psi_s ) + \bigO( \delta \Psi_{t,x} \psi^*_s ) + \bigO( \Psi_{t,x} \delta \psi_s ) + \bigO( \Psi_{t,x}(S) \psi^*_s ).
\end{equation}

Fix $j \geq 0$, and allow implied constants to depend on $j$.  To prove \eqref{dels}, we introduce the quantity
$$ h(s) := \| \delta \Psi_{t,x}(s) \|_{A^j_{k(s)}(\R^2)},$$

Since $h(S)=0$, we see from \eqref{spider}, the fundamental theorem of calculus, Minkowski's inequality, the Leibniz rule, and Lemma \ref{ek-prod} that
$$ h(s) \lesssim \int_s^{S} f(s') + h(s') g(s')\ ds'$$
for all $0 \leq s \leq S$, where
$$ f(s) := \| \nabla_{t,x} \delta \psi_s(s) \|_{A^j_{k(s)}(\R^2)} + \| \Psi_{t,x}(s)\|_{A^j_{k(s)}(\R^2)} \| \delta \psi_s(s) \|_{A^j_{k(s)}(\R^2)} + \| \Psi_{t,x}(S) \|_{A^j_{k(S)}(\R^2)} \| \psi^*_s(s) \|_{A^j_{k(s)}(\R^2)}$$
and
$$ g(s) := \| \psi^*_s(s) \|_{A^j_{k(s)}(\R^2)}.$$
By Gronwall's inequality, we conclude
$$ h(s) \lesssim \int_s^S f(s') \exp( \int_{s}^{s'} g(s'')\ ds'' )\ ds'.$$
From \eqref{spak} and Cauchy-Schwarz we have
$$ \int_{s}^{s'} g(s'')\ ds'' \lesssim \log^{1/2}(s'/s)$$
and thus
\begin{equation}\label{hass}
 h(s) \lesssim \int_s^{L_n} f(s') \exp( O( \log^{1/2}(s'/s)) )\ ds'.
\end{equation}
From Lemma \ref{psitn-lem} we have
$$ f(s) \lesssim \frac{c(s)}{s^{3/2}} [ 1_{s \geq L_n^{0.7}} + c(S) ].$$
From \eqref{css} and the envelope property, $c(S) = o_{n \to \infty}(1)$.  Applying \eqref{hass} and the envelope property we conclude that
$$ h(s) \lesssim \frac{c(s)}{s^{1/2}} [ 1_{s \geq L_n^{0.6}} + o_{n \to \infty}(1) ],
$$
which is \eqref{dels}.
\end{proof}

Let $\phi^{hi}[t_n] = \tilde \phi^{hi}[t_n]$ be a reconstruction of $\tilde \psi_s[t_n]$.  We now make the key observation that $\phi^{hi}$ has energy strictly less than $E_0$:

\begin{lemma}[High frequency component has less energy]\label{highless} We have
$$ \E( \phi^{hi}[t_n] ) \leq E_0 - \eps/8 + o_{n \to \infty}(1).$$
\end{lemma}

\begin{proof}  We omit time variable $t_n$.  From Lemma \ref{some-energy} it suffices to show that
$$ \int_{\R^2} |\tilde \psi^{hi}_\alpha(0,x)|^2 + |\psi_\alpha(L_n^{0.8},x)|^2\ dx = \int_{\R^2} |\psi_\alpha(0,x)|^2\ dx + o_{n \to \infty}(1)$$
for $\alpha=0,1,2$.  From Lemma \ref{Dels-lem} and \eqref{psit-hi} we have
$$
\| \psi_\alpha(0) - \psi_\alpha(L_n^{0.8}) - \tilde \psi^{hi}_\alpha(0) \|_{L^2_x(\R^2)} = o_{n \to \infty}(1)
$$
so by the triangle inequality and the cosine rule, it suffices to show the asymptotic orthogonality property
$$
\int_{\R^2} (\psi_\alpha(L_n^{0.8},x)) (\psi_\alpha(0,x) - \psi_\alpha(L_n^{0.8},x))\ dx = o_{n \to \infty}(1)
$$
for $\alpha=0,1,2$.

By the fundamental theorem of calculus, it suffices to show that
$$ \int_0^{L_n^{0.8}} \int_{\R^2} \psi_\alpha(L_n^{0.8},x) \partial_s \psi_\alpha(s,x)\ dx ds = o_{n \to \infty}(1).$$
From \eqref{psit-eq} and an integration by parts, we can bound the left-hand side by
\begin{equation}\label{poky}
 \lesssim \int_0^{L_n^{0.8}} \int_{\R^2} |D_x \psi_{t,x}(L_n^{0.8},x)| |D_x \psi_{t,x}(s,x)| + |\psi_{t,x}(L_n^{0.8},x)| |\psi_{t,x}(s,x)|^3\ dx ds
\end{equation}
where it is understood that the covariant derivative $D_x$ always uses the connection $A_x(s)$ rather than $A_x(L_n^{0.8})$.  By H\"older's inequality and Lemma \ref{psitn-lem} (placing $\psi_{t,x}(L_n^{0.8})$ in $L^\infty_x$ rather than $L^2_x$ whenever possible), we have
\begin{align*}
\| D_x \psi_{t,x}(L_n^{0.8}) \|_{L^2_x(\R^2)} &\lesssim \frac{c(L_n^{0.8})}{(L_n^{0.8})^{1/2}} \\
\| D_x \psi_{t,x}(s) \|_{L^2_x(\R^2)} &\lesssim \frac{c(s)}{s^{1/2}} \\
\| |\psi_{t,x}(L_n^{0.8})| |\psi_{t,x}(s)|^3 \|_{L^1_x(\R^2)} &\lesssim \frac{c(s)^3 c(L_n^{0.8})}{s^{1/2}} (L_n^{0.8})^{1/2} 
\end{align*}
so we can bound \eqref{poky} by
$$ \lesssim c(L_n^{0.8}) \int_0^{L_n^{0.8}} \frac{c(s)}{s^{1/2} (L_n^{0.8})^{1/2}}\ ds.$$
By Cauchy-Schwarz and the fact that $c$ has energy $O(1)$, this is $O( c(L_n^{0.8}) )$, and the claim now follows from \eqref{css}.
\end{proof}

The field $\tilde \psi^{hi}_s[t_n]$ is not actually a heat flow. However, it is close to one:

\begin{lemma}[${\tilde \psi_s^{hi}[t_n]}$ is an approximate heat flow]  We have
\begin{equation}\label{stosh}
\| \nabla_{t,x} \nabla_x^j \tilde \psi^{hi}_s(s,t_n) \|_{\dot H^0_{k(s)}(t_n)} \lesssim c^{hi}(s) 
s^{-(j+2)/2}
\end{equation}
and
\begin{equation}\label{stosh-2}
\| \nabla_{t,x} \nabla_x^j (\tilde \psi^{hi}_s - \tilde D^{hi}_i \tilde \psi^{hi}_i) (s,t_n) \|_{\dot H^0_{k(s)}(t_n)} \lesssim c^{hi}(s) s^{-(j+2)/2} (1_{L^{0.6} \leq s \leq L^{0.8}} + o_{n \to \infty}(1))
\end{equation}
for all $0 \leq j \leq 20$ and $s \geq 0$, where the $\dot H^m_k$ norms are defined in \eqref{ek-def}, and $c^{hi}$ is the frequency envelope $c^{hi}(s) := \min( c(s), \chi_{k(s) \geq k(L_n^{0.7})}^{\delta_0} )$.
\end{lemma}

\begin{proof}  We again omit the $t_n$ time coordinate for brevity.  We may restrict attention to $0 \leq s \leq L_n^{0.8}$ since the expression inside the norms vanishes for $s > L_n^{0.8}$, in which case $c^{hi} \sim c$.  

The claim \eqref{stosh} follows from \eqref{psihi}, \eqref{jojoe}, so it suffices to prove \eqref{stosh-2}.

We first suppose that $s \geq L^{0.6}$.  From \eqref{stosh} and the triangle inequality, it suffices to show in this case that
$$
\| \nabla_{t,x} \nabla_x^j (\tilde D^{hi}_i \tilde \psi^{hi}_i) (s,t_n) \|_{\dot H^0_{k(s)}(t_n)} \lesssim c(s) s^{-(j+2)/2} $$
for all $0 \leq j \leq 20$.  For this, we express
\begin{align*}
\nabla_{t,x} \tilde D^{hi}_i \tilde \psi^{hi}_i &= \tilde D^{hi}_{t,x} \tilde D^{hi}_i \tilde \psi^{hi}_i + \bigO( \tilde A^{hi}_{t,x} \tilde D^{hi}_x \tilde \psi^{hi}_x ) \\
&= \tilde D^{hi}_i \tilde D^{hi}_i \tilde \psi^{hi}_{t,x} + \bigO( \tilde \psi^{hi}_{t,x} \tilde \psi^{hi}_x \tilde \psi^{hi}_x ) + \bigO( \tilde A^{hi}_{t,x} \tilde D^{hi}_x \tilde \psi^{hi}_x ).
\end{align*}
Using \eqref{dels}, Lemma \ref{ek-prod}, the Leibniz rule, and the triangle inequality, we obtain the claim.

Now suppose that $s < L^{0.6}$.  Then $\tilde \psi^{hi}_s = \psi_s$; writing $\psi_s = D_i \psi_i$, it thus suffices to show that
$$ \|\nabla_x^j \nabla_{t,x} (D_i \psi_i - \tilde D^{hi}_i \tilde \psi^{hi}_i)(s) \|_{\dot H^0_{k(s)}(\R^2)} \lesssim o_{n \to \infty}(c(s) s^{-(j+2)/2})$$
for $0 \leq j \leq 20$.

Fix $j$.  We can decompose $D_i \psi_i - \tilde D^{hi}_i \tilde \psi^{hi}_i$ as
\begin{align*}
&\bigO( \nabla_x \delta \psi_x(s) ) + \bigO( \nabla_x \psi_x(L_n^{0.8}) ) \\
&\quad + \bigO( A^*_x(s) \delta \psi_x(s) ) + \bigO( A^*_x(s) \psi_x(L_n^{0.8}) ) \\
&\quad + \bigO( \delta A_x \psi^*_x(s) ) + \bigO( A_x(L_n^{0.8}) \psi^*_x(s) ) 
\end{align*}
where $A^*_x(s) = (A_x(s), \tilde A_x(s)) = \bigO( A_x(s), \delta A_x(s), A_x(L_n^{0.8}) )$, and similarly for $\psi^*_x$.  Applying the Leibniz rule followed by H\"older's inequality (placing $\psi_x(L_n^{0.8})$ and $A_x(L_n^{0.8})$ in $L^\infty$ whenever possible, in order to gain at least one factor of $(s/L_n^{0.8})^{1/2} = o_{n \to \infty}(1)$) and applying Lemma \ref{psitn-lem}, Lemma \ref{Dels-lem}, and Lemma \ref{ek-prod} we obtain the claim.
\end{proof}

From \eqref{css} and the definition of $c^{hi}$, we note that
\begin{equation}\label{css-hi}
 \int_{s \geq L_n^{-0.9}} c^{hi}(s)^2 \frac{ds}{s} = o_{n \to \infty}(1).
\end{equation}

We can now invoke Theorem \ref{parab-thm-instant} (enlarging $c$ on the interval $L^{0.5} \leq s \leq L^{0.9}$ somewhat so that \eqref{psisjk-star-abstract} holds for some $\kappa = o_{n \to \infty}(1)$), to obtain an instantaneous dynamic field heat flow $\psi^{hi}_s[t_n]$ with the same resolution as $\tilde \psi^{hi}_s[t_n]$ obeying the estimates
\begin{align}
\| \nabla_{t,x} \nabla_x^j \Psi^{hi}_x(s) \|_{\dot H^0_{k(s)}(t_n)} &\lesssim \tilde c^{hi}(s) s^{-(j+1)/2} \label{parab-1h}\\
\| \nabla_{t,x} \nabla_x^j A^{hi}_x(s) \|_{\dot H^0_{k(s)}(t_n)} &\lesssim \tilde c^{hi}(s)^2 s^{-(j+1)/2}\label{parab-2h}\\
\| \nabla_{t,x} \nabla_x^j \psi^{hi}_s(s) \|_{\dot H^0_{k(s)}(t_n)} &\lesssim \tilde c^{hi}(s) s^{-(j+2)/2}\label{parab-3h}\\
\| \nabla_{t,x} \nabla_x^j (\tilde \Psi^{hi}_x - \Psi^{hi}_x)(s) \|_{\dot H^0_{k(s)}(t_n)} &\lesssim \tilde c^{hi}(s) s^{-(j+1)/2}
(1_{L^{0.5} \leq s \leq L^{0.9}} + o_{n \to \infty}(1))
\label{parab-4h}\\
\| \nabla_{t,x} \nabla_x^j (\tilde A^{hi}_x - A^{hi}_x)(s) \|_{\dot H^0_{k(s)}(t_n)} &\lesssim o_{n \to \infty}( \tilde c^{hi}(s)^2 s^{-(j+1)/2} )
 \label{parab-5h}\\
\| \nabla_{t,x} \nabla_x^j (\tilde \psi^{hi}_s - \psi^{hi}_s)(s) \|_{\dot H^0_{k(s)}(t_n)} &\lesssim o_{n \to \infty}( \tilde c^{hi}(s) s^{-(j+2)/2} ) \label{parab-6h}
\end{align}
for all $0 \leq j \leq 15$ and $s > 0$, where $\tilde c^{hi}$ is a frequency envelope slightly larger than $c^{hi}$, but still of energy $O(1)$ and still obeying the bound \eqref{css-hi}, and $\Psi^{hi}_{s,t,x}(t_n)$ are the instantaneous differentiated fields associated to $\psi_s[t_n]$.

By Lemma \ref{highless} and the induction hypothesis, the energy of $\phi^{hi}[t_n]$ is good.  From this and Lemma \ref{quant}, we may extend $\phi^{hi}$ to a global classical wave map, and in particular to the time interval $J$, and thus also extend the differentiated fields $\Psi^{hi}_{s,t,x}$ to a wave map heat flow on $J$.  Furthermore, this wave map has an $(O(1),1)$-entropy of $O(1)$, and in particular by Proposition \ref{symscat}(iii), has a $(O(1),\mu)$-entropy of $O_\mu(1)$.   In other words, one can partition $J$ into a collection ${\mathcal K}$ of $O_\mu(1)$ intervals $K$ such that $\phi^{hi}$ is a $(O(1),\mu)$-wave map on each such interval $K$.

\subsection{Control of the high frequency component}

\begin{lemma}[$\phi^{hi}$ is high frequency]  Let the hypotheses be as in Proposition \ref{sweden}.  For each $K \in {\mathcal K}$, one has a frequency envelope $c^{hi}_K$ of energy $O(1)$ such that
\begin{equation}\label{slice-dice-hi}
\| \nabla_x^j \psi^{hi}_s(s) \|_{S_{\mu,k(s)}(K \times \R^2)} \lesssim_j c^{hi}_K(s) s^{-(j+2)/2}
\end{equation}
and
\begin{equation}\label{slice-dice-hi-2}
\| \nabla_x^j \psi^{hi}_x(s) \|_{S_{\mu,k(s)}(K \times \R^2)} \lesssim_j c^{hi}_K(s) s^{-(j+1)/2}
\end{equation}
for all $j \geq 0$ and $s>0$.  Furthermore, one has
\begin{equation}\label{chi}
\int_{s \leq L_n^{0.5}} c^{hi}(s)^2 \frac{ds}{s} = o_{n \to \infty}(1).
\end{equation}
\end{lemma} 

\begin{proof}
From \eqref{parab-3h} and Lemma \ref{freqstable2} one has
$$
\| \nabla_x^j \psi^{hi}_s(s) \|_{S_{\mu,k(s)}(K \times \R^2)} \lesssim_{j,\mu} c^{hi}(s) s^{-(j+2)/2}
$$
and
$$
\| \nabla_x^j \psi^{hi}_x(s) \|_{S_{\mu,k(s)}(K \times \R^2)} \lesssim_{j,\mu} c^{hi}(s) s^{-(j+1)/2}
$$
for all $j \geq 0$ and $s>0$.  Furthermore, since $\psi^{hi}_s$ is a $(O(1),\mu)$-map on $K$, we see from \eqref{ds2-eq} and Lemma \ref{freqstable} that
$$
\| \nabla_x^j \psi^{hi}_s(s) \|_{S_{\mu,k(s)}(K \times \R^2)} \lesssim_{j} \tilde c_K^{hi}(s) s^{-(j+2)/2}
$$
and
$$
\| \nabla_x^j \psi^{hi}_x(s) \|_{S_{\mu,k(s)}(K \times \R^2)} \lesssim_{j} \tilde c_K^{hi}(s) s^{-(j+1)/2}
$$
for all $j \geq 0$ and $s>0$, and some frequency envelope $\tilde c_K^{hi}$ of energy $O(1)$.  Setting $c_K^{hi}$ to be the infimum of $\tilde c_K^{hi}$ and $O_{\mu}(c^{hi})$, the claim now follows from \eqref{css-hi}.
\end{proof}

\subsection{Superimposing the low and high frequency components}

Having constructed the low-frequency fields $\Psi^{lo}_{s,t,x}$ and high-frequency fields $\Psi^{hi}_{s,t,x}$ on all of $I_n \cap J_n$, we now superimpose them to create an approximant $\tilde \Psi_{s,t,x}$ to the fields $\Psi_{s,t,x}$ of the original wave map $\phi$.  Following Section \ref{delocal-sec}, we define the dynamic field
\begin{equation}\label{tpsis-def}
 \tilde \psi_s := \psi^{lo}_s + \psi^{hi}_s.
\end{equation}
We let $\tilde \Psi_{s,t,x}$ be the associated differentiated fields, and $\tilde \phi$ be an associated reconstructed map.

Now we establish various estimates on this approximate field $\tilde \Psi_{s,t,x}$.

\begin{proposition}[Envelope and approximate heat flow bounds]\label{slice-dice}  For each $K \in {\mathcal K}$, one has a frequency envelope $c_K$ of energy $O(1)$ such that
\begin{equation}\label{slice-dice-1}
\| \nabla_x^j \tilde \psi_s(s) \|_{S_{\mu,k(s)}(K \times \R^2)} \lesssim_j c_K(s) s^{-(j+2)/2}
\end{equation}
and
\begin{equation}\label{slice-dice-2}
\| \nabla_x^j \tilde h(s) \|_{S_{\mu,k(s)}(K \times \R^2)} \lesssim_j o_{n \to \infty}( c_K(s) s^{-(j+2)/2} )
\end{equation}
for all $j \geq 0$ and $s > 0$,
where $\tilde h := \tilde \psi_s - \tilde D_i \tilde \psi_i$ is the heat-tension field of $\tilde \psi$.
In particular (by \eqref{ds2-eq}) we have
$$ \| \tilde \psi_s \|_{S^1_\mu(K)} \lesssim 1.$$
\end{proposition}

\begin{proof}  From \eqref{clo}, \eqref{chi} one verifies that
$$ \int_0^\infty c^{lo}(s)^2 c^{hi}_K(s)^2\ ds = o_{n \to \infty}(1).$$
Thus, if we set $\tilde c_K := c^{lo} + c^{hi}_K + w c^{lo} c^{hi}_K$ for some $w = w_n$ growing to infinity sufficiently slowly, and then set $c_K(s) := \sup_{s'} \min((s/s')^{\delta_0}, (s'/s)^{\delta_0}) \tilde c_K(s)$, then $c_K$ is a frequency envelope of energy $O(1)$, and
one has the pointwise estimates
$$ c^{lo}(s), c^{hi}_K(s) \lesssim c_K(s)$$
and
\begin{equation}\label{jungo}
 c^{lo}(s) c^{hi}_K(s) \lesssim o_{n \to \infty}(c_K(s)).
 \end{equation}

The bound \eqref{slice-dice-1} follows from \eqref{tpsis-def}, \eqref{slice-dice-lo}, \eqref{slice-dice-hi}, and the triangle inequality.  

The remaining estimate \eqref{slice-dice-2} in this proposition is not quite so simple, because $\tilde \Psi_{t,x}$ is not a linear superposition of $\Psi_{t,x}^{lo}$ and $\Psi_{t,x}^{hi}$.  However, we can show that it is an \emph{approximate} linear superposition by arguing as follows.  From \eqref{PSA} one has
\begin{align*}
\partial_s \Psi_{t,x}^{lo} &= \bigO( \nabla_{t,x} \psi_s^{lo} ) + \bigO( \psi_s^{lo} \Psi_{t,x}^{lo} ) \\
\partial_s \Psi_{t,x}^{hi} &= \bigO( \nabla_{t,x} \psi_s^{hi} ) + \bigO( \psi_s^{hi} \Psi_{t,x}^{hi} ) \\
\partial_s \tilde \Psi_{t,x} &= \bigO( \nabla_{t,x} \tilde \psi_s ) + \bigO( \tilde \psi_s \tilde \Psi_{t,x} ).
\end{align*}
Subtracting the first two equations from the second and using \eqref{tpsis-def}, we conclude that
\begin{equation}\label{hypos}
 \partial_s \delta \Psi_{t,x} = F_{t,x} + \bigO( \tilde \psi_s \delta \Psi_{t,x} )
\end{equation}
where $\delta \Psi_{t,x} := \tilde \Psi_{t,x} - \Psi^{lo}_{t,x} - \Psi^{hi}_{t,x}$, and the forcing term $F_{t,x}$ takes the form
$$ F_{t,x} = \bigO( \psi_s^{lo} \Psi_{t,x}^{hi} ) + \bigO( \psi_s^{hi} \Psi_{t,x}^{lo} ).$$

Since $\psi^{lo}_s, \psi^{hi}_s$ are heat flows, one has  
\begin{align*}
\tilde h &= \tilde \psi_s - \tilde D_i \tilde \psi_i \\
0 &= \psi^{lo}_s - D^{lo}_i \psi^{lo}_i \\
0 &= \psi^{hi}_s - D^{hi}_i \psi^{hi}_i;
\end{align*}
subtracting the latter two equations from the former, one obtains the formula
\begin{equation}\label{htl}
 \tilde h = \bigO( \partial_x \delta \Psi_x + \Psi^{lo}_x \Psi^{hi}_x + \Psi^*_x \delta \Psi_x + \delta \Psi_x \delta \Psi_x)
\end{equation}
where $\Psi^*_x := (\Psi^{lo}_x, \Psi^{hi}_x)$.

From \eqref{slice-dice-lo-2}, \eqref{slice-dice-hi-2}, \eqref{prod1} one has
$$ \| \bigO( \Psi^{lo}_x \Psi^{hi}_x(s) ) \|_{S_{\mu,k}(K \times \R^2)} \lesssim s^{-1} c^{lo}(s) c^{hi}_K(s);$$
the right-hand side is $o_{n \to \infty}(c_K(s) / s )$ by construction of $c_K$, and so the contribution of this term to \eqref{slice-dice-2} is acceptable.  Since we have
$$ \| \Psi^*_x(s) \|_{S_{\mu,k}(K \times \R^2)} \lesssim c_K(s) s^{-1/2},$$
we see that to control the remaining terms in \eqref{htl} it suffices by \eqref{prod1} to show that
\begin{equation}\label{fief}
\| \nabla_x^m \delta \Psi_x(s) \|_{S_{\mu,k}(K \times \R^2)} \lesssim_m o_{n \to \infty}(c_K(s) s^{-(m+1)/2}),
\end{equation}
for all $m \geq 0$.  

We can rescale $s=1$.  Iterating (the spatial version of) \eqref{hypos} we have
\begin{equation}\label{shoofly}
\delta \Psi_{x}(1) = \sum_{j=1}^\infty \int_1^\infty X_j(s_1)\ ds_1
\end{equation}
where
\begin{align*}
 X_j(s_1) &:= \bigO_j( \int_{s_1 < s_2 < \ldots < s_j} \tilde \psi_s(s_1) \ldots \tilde \psi_s(s_{j-1}) F_{x}(s_j)\ ds_2 \ldots ds_j),\\
F_{x} &:= \bigO( \psi_s^{lo} \Psi_{x}^{hi} ) + \bigO( \psi_s^{hi} \Psi_{x}^{lo} ),
\end{align*}
and the implied constants are $O(1)^j$, so it suffices by the triangle inequality to show that
\begin{equation}\label{nabab}
 \| \nabla_x^m \int_{1 \leq s_1 < s_2 < \ldots < s_j} \tilde \psi_s(s_1) \ldots \tilde \psi_s(s_{j-1}) F_{x}(s_j)\ ds_1 \ldots ds_j \|_{S_{\mu,0}(K \times \R^2)} \lesssim_m o_{n \to \infty}( O(1)^j j^{-j/2} c_K(1) )
 \end{equation}
for all $j \geq 1$.  From \eqref{slice-dice-lo}, \eqref{slice-dice-hi} we have
$$ \| \nabla_x^m \tilde \psi_s(s) \|_{S_{\mu,k(s)}(K \times \R^2)} \lesssim s^{-(m+2)/2} c_K(s),$$
while from \eqref{slice-dice-lo}, \eqref{slice-dice-lo-2}, \eqref{slice-dice-hi}, \eqref{slice-dice-hi-2}, \eqref{prod1}, and the Leibniz rule one has
$$ \| \nabla_x^m F_x(s) \|_{S_{\mu,k(s)}(K \times \R^2)} \lesssim_m s^{-(m+3)/2} c^{lo}_K(s) c^{hi}(s) \lesssim o_{n \to \infty}( s^{-(m+2)/2} c_K(s) ).$$
From \eqref{prod1}, the Leibniz rule, Minkowski's inequality, and \eqref{physical-star}, one can thus bound the left-hand side of \eqref{nabab} by
$$ 
\lesssim_m O(1)^j o_{n \to \infty}( \int_{1 \leq s_1 < s_2 < \ldots < s_j} 
\frac{c_K(s_1)}{s_1} \ldots \frac{c_K(s_{j-1})}{s_{j-1}} \frac{c_K(s_j)}{s_j^{3/2-\delta_0}}\ ds_1 \ldots ds_j ).$$
 Applying \eqref{sss}, we can bound this by
$$ 
\lesssim_m O(1)^j j^{-j/2} o_{n \to \infty}( \int_{1 < s_j} \frac{c_K(s_j)}{s_j^{3/2-2\delta_0}}\ ds_j )$$
 which is acceptable by the envelope property.
\end{proof}

\begin{proposition}[Good approximation at initial time]\label{tinit} We have
$$ \tilde \dist_\Energy( \psi[t_n], \tilde \psi[t_n] ) = o_{n \to \infty}(1).$$
\end{proposition}

\begin{proof}
We suppress the $t_n$ parameter.  By \eqref{energy-dist-def} it suffices to show that
$$
\sum_k ( \sup_{2^{-2k-2} \leq s \leq 2^{-2k}} s^{1+j/2} 
\| \nabla_x^j \nabla_{t,x} (\psi_s - \tilde \psi_s)(s) \|_{\dot H^0_k(\R^2)}^2 )^{1/2} \lesssim o_{n \to \infty}(1)
$$
for all $0 \leq j \leq 10$.  

Fix $j$.  From \eqref{tpsis-def}, \eqref{psihi} one has
$$\psi_s(s) - \tilde \psi_s(s) = (1-\eta(s/L_n^{0.7}) \psi_s(s) - \psi_s^{lo}(s).$$
From \eqref{phocus-2}, \eqref{energy-dist-def} one has
$$
\sum_k ( \sup_{2^{-2k-2} \leq s \leq 2^{-2k}} s^{1+j/2} 
\| \nabla_x^j \nabla_{t,x} (\psi^{lo}_s(s) - \psi_s(s+L_n^{1/4})) \|_{\dot H^0_k(\R^2)}^2 )^{1/2} \lesssim o_{n \to \infty}(1)
$$

so by the triangle inequality, it suffices to show that
\begin{equation}\label{sordid}
\sum_k ( \sup_{2^{-2k-2} \leq s \leq 2^{-2k}} s^{1+j/2} 
\| \nabla_x^j \nabla_{t,x} ((1-\eta(s/L_n^{0.7})) \psi_s(s) - \psi_s(s+L_n^{1/4})) \|_{\dot H^0_k(\R^2)}^2 )^{1/2} \lesssim o_{n \to \infty}(1).
\end{equation}
Suppose first that $L_n^{-0.8} \leq s \leq L_n^{0.8}$.  Then by \eqref{jojoe} one has
$$
s^{1+j/2} \| \nabla_x^j \nabla_{t,x} ((1-\eta(s/L_n^{0.7})) \psi_s(s) - \psi_s(s+L_n^{1/4})) \|_{\dot H^0_{k(s)}(\R^2)}
\lesssim c(s) $$
and so this contribution is acceptable by \eqref{css}.  Similarly, if $s \leq L_n^{-0.8}$, then the $1-\eta(s/L_n^{0.7})$ term vanishes, and by \eqref{jojoe} one has
$$
s^{1+j/2} \| \nabla_x^j \nabla_{t,x} ((1-\eta(s/L_n^{0.7})) \psi_s(s) - \psi_s(s+L_n^{1/4})) \|_{\dot H^0_{k(s)}(\R^2)}
\lesssim (s/L_n^{0.7})^{1+j/2}$$
which is also acceptable.  Finally, if $s_n \geq L_n^{0.8}$, then from \eqref{jojoe} one has
$$
\| \nabla_x^j \partial_s \nabla_{t,x} \psi_s(s) \|_{\dot H^0_{k(s)}(\R^2)} \lesssim s^{-(j+4)/2}$$
and hence by the fundamental theorem of calculus and Minkowski's inequality
$$
s^{1+j/2} \| \nabla_x^j \nabla_{t,x} ((1-\eta(s/L_n^{0.7})) \psi_s(s) - \psi_s(s+L_n^{1/4})) \|_{\dot H^0_{k(s)}(\R^2)}
\lesssim L_n^{1/4} s^{-(j+4)/2}$$
which is also acceptable.
\end{proof}

\begin{proposition}[Approximate wave map property]\label{twavemap} We have 
$$ \| P_k \tilde D^\alpha \tilde \psi_\alpha(0) \|_{N_{k}(K \times \R^2)} \lesssim o_{\mu' \to 0}( c(2^{-2k}) )$$
for all $k \in \Z$ and some frequency envelope $c$ of energy $O(1)$, assuming $n$ is large enough depending on $\mu'$.
\end{proposition}

\begin{remark} Note the bound here is only $o_{\mu' \to 0}( c(2^{-2k}) )$ instead of $o_{n \to \infty}( c(2^{-2k}) )$; this is because the interactions between low and high frequencies is not entirely negligible.  This will cause some technical difficulties in the rest of the proof of Proposition \ref{sweden}.  Fortunately, the high-high interactions remain negligible, which will mean that we will be able to control the low frequency component of the solution with an accuracy of $o_{n \to \infty}( c(2^{-2k}) $ rather than $o_{\mu' \to 0}( c(2^{-2k}) )$.
\end{remark}

\begin{proof}
We may rescale $k=0$, and abbreviate $N_0(K \times \R^2)$ as $N_0$ and $o_{n \to \infty}()$ as $o()$.  As $\psi^{lo}_s$, $\psi^{hi}_s$ obey the wave map equation, one has
\begin{align*}
(D^{lo})^\alpha \psi^{lo}_\alpha(0) &= 0\\
(D^{hi})^\alpha \psi^{hi}_\alpha(0) &= 0.
\end{align*}
Subtracting these equations, we obtain the representation
\begin{align}
\tilde D^\alpha \tilde \psi_\alpha(0) &= \bigO( \partial^\alpha \delta \Psi_\alpha(0) ) \label{wave1}\\
&\quad + \bigO( (\Psi^{hi})^\alpha(0) \Psi^{lo}_\alpha(0) ) \label{wave2} \\
&\quad + \bigO( (\delta \Psi^\alpha(0)) \tilde \Psi_\alpha(0) )\label{wave3}
\end{align}
where $\delta \Psi_\alpha := \tilde \Psi_\alpha - \Psi^{lo}_\alpha - \Psi^{hi}_\alpha$.

To estimate the contribution of each component \eqref{wave1}, \eqref{wave2}, \eqref{wave3}, we need to expand out the $\Psi_\alpha$ terms to bring the null structure into focus.  We begin with \eqref{wave1}.  By repeating the derivation of \eqref{shoofly}, one has
$$
\delta \Psi_\alpha(0) = \sum_{j=1}^\infty Y_{j,\alpha}^{lo,hi} + Y_{j,\alpha}^{hi,lo}
$$
where
$$
 Y_{j,\alpha}^{A,B} := 
 \bigO_j( \int_{s_1 < s_2 < \ldots < s_j} \tilde \psi_s(s_1) \ldots \tilde \psi_s(s_{j-1}) \psi_s^A(s_j) \Psi_\alpha^B(s_j) \ ds_1 ds_2 \ldots ds_j )
$$
and the implied constants are $O(1)^j$ (and, crucially, do not depend on $\alpha$).  
Using \eqref{psi-odd}, \eqref{psi-even}, one can expand this further as
\begin{equation}\label{dorito}
\delta \tilde \Psi_\alpha(0) = \sum_{j=1}^\infty \sum_{l=1}^\infty Z_{j,l,\alpha}^{lo,hi} + Z_{j,l,\alpha}^{hi,lo}
\end{equation}
where
$$
 Z_{j,l,\alpha}^{A,B} := 
 \bigO_{j,l}( \int_{s_1 < s_2 < \ldots < s_{j+l}} \tilde \psi_s(s_1) \ldots \tilde \psi_s(s_{j-1}) \psi_s^A(s_j) 
 \psi_s^B(s_{j+1}) \ldots \psi_s^B(s_{j+l-1}) \partial_\alpha \psi_s^B(s_{j+l}) \ ds_1 ds_2 \ldots ds_{j+l} )$$
and the implied constants are $O(1)^{j+l}$ and are independent of $k$.

A closer inspection of the derivation of \eqref{dorito} (cf.\footnote{Alternatively, one can take advantage of the reflection symmetry $\psi_{s,t,x} \mapsto -\psi_{s,t,x}$, $A_{s,t,x} \mapsto A_{s,t,x}$, which preserves the heat flow and wave map equations, to see this.} \eqref{psi-odd}, \eqref{psi-even}) shows that the $\delta \psi_\alpha$ component of $\delta \Psi_\alpha$ consists purely of terms with an odd number of terms (i.e. with $j+l$ odd), while the $\delta A_\alpha$ component consists of terms with an even number of terms (i.e. with $j+l$ even).  This observation will be important to us in order to eliminate some quadratic null form terms for which the trilinear estimate \eqref{trilinear-improv} would not be applicable.

To control the contribution of \eqref{wave1}, it thus suffices by the triangle inequality to show that
\begin{equation}\label{hili}
\| P_0( \partial^\alpha Z_{j,l,\alpha}^{A,B} ) \|_{N_0} \lesssim O(1)^{j+l} j^{-j/2} l^{-l/2} o_{\mu' \to 0}( c(1) )
\end{equation}
for all $j,l \geq 1$ with $j+l$ odd (in particular, $j+l \geq 3$), with $(A,B) = (hi,lo)$ or $(A,B)=(lo,hi)$.

We distribute the $\partial^\alpha$ derivative using the Leibniz rule.  We first consider the terms in which the $\partial^\alpha$ derivative does not fall on the $\partial_\alpha \psi_s^B(s_{j+l})$ term.  Applying \eqref{trilinear-improv}, \eqref{prod1}, \eqref{slice-dice-1}, \eqref{slice-dice-lo}, \eqref{slice-dice-hi}, and Minkowski's inequality, we can estimate this contribution to \eqref{hili} by
\begin{align*}
&\lesssim O(1)^{j+l} \int_{s_1 < s_2 < \ldots < s_{j+l}} \frac{c_K(s_1)}{s_1} \ldots \frac{c_K(s_{j-1})}{s_{j-1}} \frac{c_K^A(s_j)}{s_j}
\frac{c_K^B(s_{j+1})}{s_{j+1}} \ldots \frac{c_K^B(s_{j+l})}{s_{j+l}} \\
&\quad \chi_{k(s_1)=0}^{\delta_1/2} \chi_{k(s_2) \leq k(s_{j+l})}^{\delta_1/2}\ ds_1 \ldots ds_{j+l}
\end{align*}
where we write $c_K^{lo}$ for $c^{lo}$.  If $B=lo$ and $l \geq 2$, we observe from \eqref{trilinear-improv} that we can also insert an additional factor of $(\mu')^{1/2}$ in this bound.

Suppose first that $j > 2$.  Evaluating the $s_3,\ldots,s_{j-1}$ and $s_{j+1},\ldots,s_{j+l}$ integrals using \eqref{sss}, one is left with
\begin{equation}\label{bowl}
\begin{split}
&\lesssim O(1)^{j+l} j^{-j/2} l^{-l/2} \int_{s_1 < s_2 < s_j < s_{j+l}} \frac{c_K(s_1)}{s_1} \frac{c_K(s_2)}{s_2} \frac{c_K^A(s_j)}{s_j} \frac{c_K^B(s_{j+1})}{s_{j+1}}\\
&\quad \chi_{k(s_1)=0}^{\delta_1} \chi_{k(s_2) \leq k(s_{j+l})}^{\delta_1+O(\delta_0)}\ ds_1 ds_2 ds_j ds_{j+l}.
\end{split}
\end{equation}
Using \eqref{jungo} and the envelope property to evaluate the $s_j, s_{j+l}$ integrals, we obtain the claim (in fact we obtain the superior bound of $o_{n \to \infty}(c(1))$ in this case).  A similar argument works for $j=2$ (with slight notational changes).  Now suppose instead that $j=1$, so that $k \geq 2$.  Evaluating the $s_3,\ldots,s_l$ integrals using \eqref{sss}, one is left with
$$ \lesssim O(1)^l l^{-l/2} \int_{s_1 < s_2 < s_{l+1}} \frac{c_K^A(s_1)}{s_1} \frac{c_K^B(s_2)}{s_2} \frac{c_K^B(s_{j+1})}{s_{j+1}}
\chi_{k(s_1)=0}^{\delta_1} \chi_{k(s_2) \leq k(s_{l+l})}^{\delta_1+O(\delta_0)}\ ds_1 ds_2 ds_{l+1}.$$
If $(A,B) = (lo,hi)$, then from \eqref{clo}, \eqref{chi} and the envelope property we see that this contribution is $o_{n \to \infty}(c(1))$.  If instead $(A,B) = (hi,lo)$, then we only obtain a bound of $O(c(1))$ this way; however, in this case the additional factor of $(\mu')^{1/2}$ is available, thus obtaining the required bound $o_{\mu' \to \infty}(c(1))$.

Finally, we consider the contribution to \eqref{hili} when the $\partial^\alpha$ derivative falls on the $\partial_\alpha \psi_s^B(s_{j+l})$ term.  We can bound this contribution by
\begin{align*}
&O(1)^{j+l} \int_{s_1 < s_2 < \ldots < s_{j+l}} \| P_0( \tilde \psi_s(s_1) \ldots \tilde \psi_s(s_{j-1}) \psi_s^A(s_j) 
 \psi_s^B(s_{j+1}) \ldots \\
 & \quad \psi_s^B(s_{j+l-1}) \Box \psi_s^B(s_{j+l})\|_{N_0}\ ds_1 ds_2 \ldots ds_{j+l}.
 \end{align*}
Applying \eqref{trilinear-improv3}, \eqref{prod1} we can bound the integrand 
$$
\lesssim O(1)^{j+l} \frac{c_K(s_1)}{s_1} \ldots \frac{c_K(s_{j-1})}{s_{j-1}} \frac{c_K^A(s_j)}{s_j}
\frac{c_K^B(s_{j+1})}{s_{j+1}} \ldots \frac{c_K^B(s_{j+l})}{s_{j+l}} \chi_{k(s_1)=0}^{\delta_1} \chi_{k(s_2) = k(s_{j+l})}^{\delta_1}
$$
Performing all integrals except for $s_1, s_2, s_j, s_{j+l}$ using \eqref{sss}, we gain a factor of $O(1)^{j+l} j^{-j/2} l^{-l/2}$ while degrading the $\delta_1$ exponents by $O(\delta_0)$, leave one with essentially the same integral as by \eqref{bowl} for $j>2$ (with minor adjustments for $j=1,2$).  The claim now follows by similar calculations to the previous arguments.  
\end{proof}

\subsection{Repairing the superposition}

We now establish the entropy component of Theorem \ref{sweden}:

\begin{proposition}[Entropy bound]\label{kslice}  Let $K$ be one of the intervals in Proposition \ref{slice-dice}.  Then for all $t \in K$, one has
\begin{equation}\label{kslice-dist}
\tilde \dist_\Energy( \phi(s)[t], \tilde \phi(s)[t] ) = o_{\mu' \to \infty;\mu}(1)
\end{equation}
for all $s \geq 0$ and $t \in K$, and that
\begin{equation}\label{kslice-0}
\dist_{S^1_\mu(K)}( \Psi_{s,t,x}, \tilde \Psi_{s,t,x}) = o_{\mu' \to \infty;\mu}(1).
\end{equation}
\end{proposition}

\begin{proof} It suffices to verify the proposition under the assumption that
$$ \tilde \dist_\Energy( \phi(s)[t_0], \tilde \phi(s)[t_0] ) = o_{n \to \infty}(1)$$
for some $t_0 \in K$, as the claim then follows by an induction on $K$ (using the interval $K$ containing $t_n$ and Proposition \ref{tinit} as the base case).  But this follows from Theorem \ref{hyp-repair}, Proposition \ref{slice-dice}, and Proposition \ref{twavemap}.
\end{proof}

From Proposition \ref{kslice}, from Proposition \ref{slice-dice} and the triangle inequality, $\phi$ is a $(O(1),\mu)$-wave map on each $K$, and so $(\phi,J)$ has a $(O(1),\mu)$-entropy of $O_{\mu}(1)$.

To finish the proof of Theorem \ref{sweden}, we need to show the bounds \eqref{summer}, \eqref{simmer2} for some $L'_n \leq L_n$ with $L'_n \to \infty$.  We first show some frequency delocalisation:

\begin{proposition}[Frequency delocalisation]\label{slice-dice-freq}  For each $K \in {\mathcal K}$, one has a frequency envelope $c'_K$ of energy $O(1)$ such that
\begin{equation}\label{slice-dice-soo}
\| \nabla_x^j \psi_s(s) \|_{S_{\mu,k(s)}(K \times \R^2)} \lesssim_j c'_K(s) s^{-(j+2)/2}
\end{equation}
and
\begin{equation}\label{slice-dice-soo2}
\| \nabla_x^j \psi_x(s) \|_{S_{\mu,k(s)}(K \times \R^2)} \lesssim_j c'_K(s) s^{-(j+1)/2}
\end{equation}
and
\begin{equation}\label{odeo}
\int_{L_n^{-0.5} \leq s \leq L_n^{0.5}} c'_K(s)^2 \frac{ds}{s} = o_{n \to \infty}(1).
\end{equation}
\end{proposition}

\begin{proof}  From Lemma \ref{psitn-lem} and Lemma \ref{freqstable2} one has
$$
\| \nabla_x^j \psi_s(s) \|_{S_{\mu,k(s)}(K \times \R^2)} \lesssim_{\mu,j} c_0(s) s^{-(j+2)/2}
$$
and
$$
\| \nabla_x^j \psi_x(s) \|_{S_{\mu,k(s)}(K \times \R^2)} \lesssim_{\mu,j} c_0(s) s^{-(j+1)/2}
$$
for all $j \geq 0$ and $s > 0$; meanwhile, as $(\phi,K)$ is a $(O(1),\mu)$ map one also has (from Lemma \ref{freqstable})
$$
\| \nabla_x^j \psi_s(s) \|_{S_{\mu,k(s)}(K \times \R^2)} \lesssim_{j} \tilde c_K(s) s^{-(j+2)/2}
$$
and
$$
\| \nabla_x^j \psi_x(s) \|_{S_{\mu,k(s)}(K \times \R^2)} \lesssim_{j} \tilde c_K(s) s^{-(j+1)/2}
$$
for all $j \geq 0$ and $s>0$, and some frequency envelope $\tilde c_K$ of energy $O(1)$.  Setting $c'_K$ to be the infimum of $\tilde c_K$ and $O_\mu(c_0)$ one obtains the claim.
\end{proof}

Set $L'_n := L_n^{0.1}$ (say). We can now establish \eqref{summer}:

\begin{proposition}  One has
$$
\int_{L_n^{-0.4} \leq s \leq L_n^{0.4}} \ESD(\phi[t])(s)\ ds = o_{n \to \infty}(1)
$$
for all $t \in K$.
\end{proposition}

\begin{proof}  We fix $t$, and omit mention of this parameter.
By \eqref{esd-def} and writing $\psi_s = \bigO( \nabla_x \Psi_x + \Psi_x \Psi_x )$, it suffices to show that
$$
\int_{L_n^{-0.4} \leq s \leq L_n^{0.4}} \| \nabla_x \Psi_{t,x}(s)\|_{L^2_x(\R^2)}^2 + \| \Psi_{t,x}(s) \|_{L^4_x(\R^2)}^4\ ds = o_{n \to \infty}(1).
$$
The contribution of $\| \nabla_x \Psi_{t,x}(s)\|_{L^2_x(\R^2)}$ can be handled by the analogue of Lemma \ref{psitn-lem} for this time $t$.  To control $\| \Psi_{t,x}(s) \|_{L^4_x(\R^2)}^4$, one first uses the bound $\| \Psi_{t,x}(s) \|_{L^4_x(\R^2)} \lesssim s^{-1/4}$ and the Gagliardo-Nirenberg inequality to bound
$$ \| \Psi_{t,x}(s) \|_{L^4_x(\R^2)}^4 \lesssim s^{-1/2} \| \nabla_x \Psi_{t,x}(s)\|_{L^2_x(\R^2)}^{1/2} \| \Psi_{t,x}(s)\|_{L^3_x(\R^2)}^{3/2}$$
so by H\"older's inequality it suffices to show that
$$
\int_{L_n^{-0.4} \leq s \leq L_n^{0.4}} s^{-2/3} \| \Psi_{t,x}(s) \|_{L^3_x(\R^2)}^2\ ds \lesssim 1.$$
But this follows from Proposition \ref{corbound}.
\end{proof}

Finally, we have to establish \eqref{simmer2}.   From \eqref{phocus-2}, Lemma \ref{gapstable}, and the triangle inequality one has
$$
\dist_\Energy( \phi^{lo}[t_n], \phi((L'_n)^{1/4})[t_n] ) = o_{n \to \infty}(1).$$
Meanwhile, $\psi^{lo}$ and $\psi(\cdot + (L'_n)^{1/4})$ both have energy $O(1)$, are both heat flows, and both have an $S_\mu(K)$ norm of $O(1)$ for each $K \in {\mathcal K}$.  To establish \eqref{simmer2}, it thus suffices by Proposition \ref{hyp-repair} (and Proposition \ref{compati}, and an induction on $K$) to show that
$$
\sum_k \| P_k w((L'_n)^{1/4}) \|_{N_{k}(K \times \R^2)}^2 \lesssim o_{n \to \infty}(1)
$$
for each $K \in {\mathcal K}$.  But by Lemma \ref{freqstable} (and Proposition \ref{slice-dice-freq}), one has
$$ \| P_k w((L'_n)^{1/4}) \|_{N_{k}(K \times \R^2)} \lesssim c'_K( (L'_n)^{1/4} ) \chi_{k = k((L'_n)^{1/4})}^{\delta_1/10}$$
and the claim follows from \eqref{odeo}.The proof of Proposition \ref{sweden} is (finally!) complete.

\section{Spatial dispersion implies spacetime bound}\label{spacbound-sec}

We now begin the proof of Theorem \ref{spacbound}.  Let $E_0, \phi^{(n)}[0]$ be as in the theorem; we allow implied constants to depend on $E_0$, and on the localisation bounds in \eqref{spacloc}. We now omit the $(n)$ subscripts, and assume throughout that $n$ is sufficiently large.

Let $A$ be a sufficiently large quantity depending only on $E_0$ to be chosen later, let $\mu > 0$ be a sufficiently small quantity depending on $A, E_0$ to be chosen later.  Let $I$ be a time interval containing $0$.  We consider the set $\Omega = \Omega_n \subset (0,+\infty)$ of all times $T > 0$ such that $\phi[0]$ can be extended to a $(A,\mu)$-wave map on $[-T,0]$.    From Theorem \ref{symscat}(vi), $\Omega$ contains a neighbourhood of the origin if $A$ is large enough; from Theorem \ref{apriori-thm2}(ii) and Definition \ref{aes} (and the fact that the space of resolutions $\psi_s$ of a map $\phi$ is compact), $\Omega$ is closed in $(0,+\infty)$.  If $\Omega$ is all of $(0,+\infty)$, then we can extend $\phi$ to $I_-$ and the claim follows; thus we may assume that the supremum $T_* = T_{*,n}$ of $\Omega$ is such that $0 < T_* < \infty$.  From Theorem \ref{symscat}(ii) and Definition \ref{aes}, we conclude that there exists an extension of $\phi[0]$ to a wave map $\phi$ on $[-T_*,0]$ with
$$\| \psi_s \|_{S^1_\mu([-T_*,0])} = A.$$

Thus $\psi_s$ has a substantial presence to the past of $0$.  We now claim the same is true for the linear solution:

\begin{lemma}\label{sodo}  Let $\psi_{s,lin}$ be the solution to the free wave equation $\Box \psi_{s,lin}=0$ with initial data $\psi_{s,lin}[0] = \psi_s[0]$.  Then
$$ \| \psi_{s,lin} \|_{S^1_\mu([-T_*,0])} \geq A/2.$$
\end{lemma}

\begin{proof}
From Lemma \ref{freqstable} we see that
$$\| \nabla_x^j P_k \Box \psi_s(s) \|_{N_k^\strong([-T_*,0] \times \R^2)} \lesssim_A \mu^{1.9} s^{-(j+2)/2} c(s) \chi_{k=k(s)}^{\delta_1/10}.$$
(say) for all $k \in \R$, $s>0$, and $0 \leq j \leq 10$, and thus by \eqref{energy-est} one has
$$ \| \nabla_x^j P_k(\psi_s- \psi_{s,lin})(s) \|_{S_{\mu,k}([-T_*,0] \times \R^2)} \lesssim_A \mu^{0.9} c(s) s^{-(j+2)/2} \chi_{k=k(s)}^{\delta_1/10}.$$
By \eqref{ds1-eq} we conclude that
$$ d_{S^1_\mu([-T_*,0])}( \psi_s, \psi_{s,lin} ) \lesssim_A \mu^{0.9}$$
and the claim then follows from the triangle inequality, if $\mu$ is small enough.
\end{proof}

Henceforth we allow implied constants to depend on $A$.

The above lemma leads to concentration on a light ray\footnote{Actually, since the only significant frequency interactions here are between components of comparable frequency, one could obtain concentration at a point rather than along a light ray with more effort, but we will not need to do so here.} in the past:

\begin{lemma}\label{snork}  There exists $k,k'=O_\mu(1)$, $s_0 \sim_\mu 1$, a cap $\kappa \in K_l$ of length $2^{-l} \gtrsim_\mu 1$, and a tube $T_{x_0,\omega,k'}$ with $\dist(\omega, \kappa), \dist(\omega, -\kappa) \gtrsim_\mu 1$ such that
$$
\| P_{k',\kappa} \nabla_{t,x} \psi_{s,lin}(s_0) \|_{L^2_t L^\infty_x(T_{x_0,\omega,k'} \cap ([-T_*,0] \times \R^2))} \gtrsim_\mu 1.$$
\end{lemma}

\begin{proof} 
By repeating the proof of Lemma \ref{psitn-lem}, using \eqref{spacloc} instead of \eqref{phocus} we can find a frequency envelope $c = c^{(n)}$ of energy $O(1)$ (depending on $n$ and $t_n$) such that the sequence of measures $c(s)^2\frac{ds}{s}$ are \emph{tight} in $n$, in the sense that for every $\eps > 0$ there exists $C > 0$ such that
\begin{equation}\label{css-tight}
\int_{s < 1/C} c(s)^2 \frac{ds}{s},
\int_{s > C} c(s)^2 \frac{ds}{s} \leq \eps + o_{n \to \infty; \eps}(1)
\end{equation}
and such that 
\begin{equation}\label{l2s-fixed-special-psitn-aa}
\begin{split}
\| \nabla_x^{j+1} \Psi_{t,x}(s,0) \|_{L^2_x(\R^2)} 
+ \| \nabla_x^{j} \Psi_{t,x}(s,0) \|_{L^\infty_x(\R^2)} \quad &\\
+ \| \nabla_x^j \nabla_{t,x} \Psi_x(s,0) \|_{L^2_x(\R^2)} +
\| \nabla_x^j \psi_s(s,0) \|_{L^2_x(\R^2)} \quad &\\
+ \| \nabla_x^{j-1} \nabla_{t,x} \psi_s(s,0) \|_{L^2_x(\R^2)}
&\lesssim_{j} c(s) s^{-(j+1)/2} 
\end{split}
\end{equation}
for all $s > 0$ and $j \geq 0$, with the convention that we drop all terms involving $\nabla_x^{-1}$. Applying \eqref{energy-est}, we conclude that
$$ \| \nabla_x^j P_k(\psi_{s,lin})(s) \|_{S_{\mu,k}([-T_*,0] \times \R^2)} \lesssim c(s) s^{-(j+2)/2} \chi_{k=k(s)}^{\delta_1/10}$$
for all $0 \leq j \leq 10$, $s>0$, and $k \in \R$.  From this and the tightness, we see that there exists a constant $C > 0$ such that the contribution of those $s$ larger than $C$, or less than $1/C$, or those $k$ with $2^k > C$ or $2^k < C$, to \eqref{sodo} is at most $A/100$ (say).  From the triangle inequality and the pigeonhole principle, we thus conclude that there exists $s_0 \sim 1$, $k = O(1)$, and $0 \leq j \leq 10$ such that
$$ \| \nabla_x^j P_k(\psi_{s,lin})(s) \|_{S_{\mu,k}([-T_*,0] \times \R^2)} \sim 1.$$
By Littlewood-Paley theory we may assume that $j=0$. From Definition \ref{smuk-def} (and an argument by contradiction), we conclude that
$$\| \phi \|_{S'_{\mu,k}(I \times \R^2)} \gtrsim_{\mu} 1;$$
by \eqref{spunk} we conclude that the claim.
\end{proof}

Let $k,k',s_0,\kappa,\omega$ be as above.  Now we use the dispersion hypothesis \eqref{disp} to push the concentration further into the past:

\begin{lemma}[Concentration on a distant light ray]\label{concdist}  There exists $T_n \to \infty$ such that
$$
\| P_{k',\kappa} \nabla_{t,x} \psi_{s,lin}(s_0) \|_{L^2_t L^\infty_x(T_{x_0,\omega,k'} \cap ([-T_*,-T_n] \times \R^2))} \gtrsim_\mu 1.$$
\end{lemma}

\begin{proof}  Let $R \geq 1$ be a quantity to be chosen later (it will grow slowly with $n$).  From \eqref{disp} we that
$$ \int_{|x-x_0| \leq R} |\psi_{t,x}(0,0,x)|^2\ dx = o_{n \to \infty;R}(1).$$
Now we propagate this in the $s$ direction.  Let $\eta((x-x_0)/R)$ be a bump function supported on the ball $\{ |x-x_0| \leq R \}$ which equals one on the ball $\{ |x-x_0| \leq R/2\}$.  From the heat equation \eqref{psit-eq}, one has the Bochner-Weitzenb\"ock identity
$$ \partial_s |\psi_{t,x}|^2 = \Delta |\psi_{t,x}|^2 - 2 |D_x \psi_{t,x}|^2 - |\psi_x \wedge \psi_{t,x}|^2$$
and thus by integrating by parts
$$ \partial_s \int_{|x-x_0| \leq R} |\psi_{t,x}(s,0,x)|^2 \eta((x-x_0)/R) \ dx \leq R^{-2} \int_{|x-x_0| \leq R} |\psi_{t,x}(s,0,x)|^2 (\Delta \eta)((x-x_0)/R)\ dx;$$
since $\|\psi_{t,x}(s,0)\|_{L^2_x(\R^2)} = O(1)$ by Proposition \ref{corbound}, we conclude that
$$ \int_{|x-x_0| \leq R} |\psi_{t,x}(s_0,0,x)|^2 \lesssim o_{n \to \infty;R}(1) + O( s_0 / R^2 ).$$
In particular, by choosing $R = R_n$ growing to infinity sufficiently slowly, one has
$$ \int_{|x-x_0| \leq R} |\psi_{t,x}(s_0,0,x)|^2 \lesssim o_{n \to \infty;s_0}(1).$$
By the fundamental solution of the free wave equation, one then easily computes that
$$
\| P_{k',\kappa} \nabla_{t,x} \psi_{s,lin}(s_0) \|_{L^2_t L^\infty_x(T_{x_0,\omega,k'} \cap ([-T_*,-T_n] \times \R^2))} = o_{n \to \infty;\mu,k,k',s_0}(1)$$
and the claim follows from Lemma \ref{snork} with the stated properties.
\end{proof}

This leads to a decoupling of the initial data $\psi_s[0]$:

\begin{lemma}[Initial data splitting]\label{ids}  There exists a splitting
$$ \psi_s[0] = \tilde \psi'_s[0] + \tilde \psi''_s[0]$$
with the following properties:
\begin{itemize}
\item (Energy decrease of $\psi'$) For $n$ sufficiently large, we have
\begin{equation}\label{eo-decline}
\E( \psi'_s[0] ) \leq E_0 - c
\end{equation}
 for some $c \gtrsim_\mu 1$.
\item (Future dispersal of $\psi''$) If $\psi''_{s,lin}$ is the linear evolution of $\psi''_s[0]$, then
\begin{equation}\label{sourgum}
 \| \psi''_{s,lin} \|_{S^1_{\mu'}(I_+)} \lesssim 1
\end{equation}
for some $\mu' = o_{n \to \infty}(1)$.
\item ($\tilde \psi'_s$, $\tilde \psi''_s$ are approximate heat flows) There exists a frequency envelope $c$ of energy $O(1)$ such that
\begin{equation}\label{haha}
 \| \nabla_x^j \nabla_{t,x} \tilde h'(s) \|_{\dot H^0_{k(s)}(I \times \R^2)} \lesssim o_{n \to \infty}(c(s) s^{-(j+2)/2})
\end{equation}
and
\begin{equation}\label{haha-2}
 \| \nabla_x^j \nabla_{t,x} \tilde h''(s) \|_{\dot H^0_{k(s)}(I \times \R^2)} \lesssim o_{n \to \infty}(c(s) s^{-(j+2)/2})
 \end{equation}
for all $0 \leq j \leq 10$ and $s>0$, where $\tilde h' := \psi'_s - D'_i \psi'_i$, $\tilde h'' := \psi''_{s,lin} - D''_i \psi''_i$ are the heat-tension fields of $\psi'_s$, $\psi''_{s,lin}$.

\item We have
\begin{equation}\label{jollyo}
\| (\nabla_x^j \psi''_{s,lin}(s)) \phi \|_{S_{\max(k,k(s)}^\strong(I_+)} \lesssim o_{n \to \infty; j,\eps}( \chi_{k=k(s)}^{-\eps} (1+s)^{-100} \| \phi \|_{S_k^\strong(I_+)} ) 
\end{equation}
for all $s,\eps>0$, $j \geq 0$, $k \in \R$ and all $\phi$.
\end{itemize}
\end{lemma}

\begin{proof}  We allow implied constants to depend on $\mu$, thus $k,k',l = O(1)$.

By Lemma \ref{concdist}, and the fact that $\nabla_x F$ can be expressed as a convolution of $F$ with a Schwartz function when $F$ is localised in frequency, we see (possibly after shifting $x_0$ slightly) that
$$
\| P_{k',\kappa} \partial_{t}^i \psi_{s,lin}(s_0) \|_{L^2_t L^\infty_x(T_{x_0,\omega,k'} \cap ([-T_*,-T_n] \times \R^2))} \gtrsim 1$$
for some $i=0,1$.  By  duality (taking advantage of frequency localisation), there exists a function $F$ supported on $T_{x_0,\omega,k'} \cap ([-T_*,-T_n] \times \R^2)$ with
$$ \|F\|_{L^2_t L^1_x(T_{x_0,\omega,k'} \cap ([-T_*,-T_n] \times \R^2))} = 1$$
such that
\begin{equation}\label{genie}
 |\langle \partial_t^i \psi_{s}(s_0,0), f \rangle| \gtrsim 1,
 \end{equation}
where
$$ f := \Box^{-1} P_{k',\kappa} F(0).$$
From null energy estimates we see that 
$$ \|f\|_{L^2_x(\R^2)} \lesssim 1;$$
also, from the frequency-localised (to $P_{k',\kappa}$) fundamental solution to the wave equation, which propagates transversely to $T_{x_0,\omega,k'}$, and the fact that $T_n \to \infty$, we obtain the dispersion estimate
\begin{equation}\label{fo}
\|f\|_{L^\infty_x(\R^2)} \lesssim o_{n \to \infty}(1).
\end{equation}
Also, $f$ is localised to frequencies $\{ \xi: |\xi| \sim 1 \}$.

The dispersion property \eqref{fo} will ensure that the harmonic map heat flow behaves like the linear equation as far as components related to $f$ are concerned (thus we expect $\psi_{t,x}(s) \approx e^{s\Delta} \nabla_{t,x} \Phi$ and $\psi_s(s) \approx \Delta e^{s\Delta} \Phi$ for some function $\Phi$).  This heuristic will underlie the unusual seeming definitions which now follow.

Let $0 < \eps < 1$ be a small parameter to be chosen later. We now define $\tilde \psi''_s[0]$ as follows.  If $i=0$, we set 
\begin{equation}\label{soy}
\tilde \psi''_s(s,0) := \eps \Delta^2 e^{(s+s_0)\Delta} f; \quad \partial_t \tilde \psi''_s(s,0) := 0
\end{equation}
while if $i=1$, we set
\begin{equation}\label{sauce}
\tilde \psi''_s(s,0) := 0; \quad \partial_t \psi''_s(s,0) := \eps \Delta e^{(s+s_0)\Delta} f.
\end{equation}
We then set $\tilde\psi'_s[0] := \psi_s[0] - \tilde \psi''_s[0]$.

We first establish \eqref{sourgum}.  Since $\psi''_{s,lin}(s)$ is supported on frequencies $\sim 1$, it suffices by \eqref{ds2-eq} to show that
\begin{equation}\label{gem}
\| \psi''_{s,lin}(s) \|_{S_{\mu',1}(I_+ \times \R^2)} \lesssim (1+s)^{-100}
\end{equation}
for all $s>0$.

We prove the claim just for $s = O(1)$, as the case $s \gg 1$ is similar but exploits the rapid decrease of the fundamental solution of $e^{(s+s_0)\Delta}$ in this regime.

Fix $s$.  From \eqref{energy-est} one already has
$$
\| \psi''_{s,lin}(s) \|_{S^\strong_1(I_+ \times \R^2)} \lesssim 1,
$$ 
so by Definition \ref{smuk-def} it suffices to show that
\begin{equation}\label{figaro}
\| \nabla_{t,x} \psi''_{s,lin}(s) \|_{L^5_t L^\infty_x(I_+ \times \R^2)} = o_{n \to \infty}(1)
\end{equation} 
and
\begin{equation}\label{sordid2}
\| P_{\tilde k, \tilde \kappa} \nabla_{t,x} \psi''_{s,lin}(s) \|_{L^2_t L^\infty_x(T_{\tilde x_0,\tilde \omega,\tilde k} \cap (I \times \R^2))}  = o_{n \to \infty}(1)
\end{equation}
for all $\tilde k = O(1)$, all caps $\tilde \kappa \in K_{\tilde l}$ with $\tilde l = O(1)$, all $\tilde x_0 \in \R^2$, and all $\tilde \omega$ making an angle of $O(1)$ with $\tilde \kappa$.

From Strichartz estimates, one already has
$$
\| \nabla_{t,x} \psi''_{s,lin}(s) \|_{L^4_t L^\infty_x(I_+ \times \R^2)} \lesssim 1
$$ 
while from the fundamental solution of the wave equation localised to $P_{k,\kappa}$, and the support of $F$, one has
$$
\| \nabla_{t,x} \psi''_{s,lin}(s) \|_{L^\infty_t L^\infty_x(I_+ \times \R^2)} = o_{n \to \infty}(1)
$$ 
which gives \eqref{figaro} by interpolation.  To prove \eqref{sordid2}, we observe from linearisation that it suffices to show that the operator $P_{\tilde k,\tilde\kappa} \nabla_{t,x} \Box^{-1} P_{k',\kappa}$ is bounded from $L^2_t$ of any curve $\{ (t,x(t)) \}$ in $T_{x_0,\omega,k'}$ to 
$L^2_t$ of any curve $\{(t,\tilde x(t)) \}$ in $T_{\tilde x_0,\tilde \omega,\tilde k}$.  But from stationary phase one checks that the kernel of this operator obeys the requirements for Schur's test, and the claim follows.

Now we establish \eqref{haha}, \eqref{haha-2}.  To unify the two cases $i=0$ and $i=1$, we write in both cases
$$ \psi''_s(s,0) = \Delta e^{s\Delta} F_0; \quad \partial_t \psi''_s(s,0) = \Delta e^{s\Delta} F_1$$
where $F_0,F_1$ have frequency support in the region $\{ \xi: |\xi| \sim 1 \}$, have $L^2_x(\R^2)$ norm $O(1)$, and $L^\infty_x(\R^2)$ norm $o_{n \to \infty}(1)$.

From \eqref{psi-odd} we can write
$$ \psi''_x(s) = - \int_s^\infty \nabla_x \psi''_s(s')\ ds' + \sum_{j=3}^\infty X_j(s)$$
where
$$ X_j(s) := \int_{s < s_1 < \ldots < s_j} \bigO_j( \psi''_s(s_1) \ldots \psi''_s(s_{j-1}) \nabla_x \psi''_s(s_j) )\ ds_1 \ldots ds_j,$$
where the implied coefficients are $O(1)^j$.  

We can evaluate $- \int_s^\infty \nabla_x \psi''_s(s')\ ds'$ as $\nabla_x e^{s\Delta} F_0$.  Meanwhile, by construction, we have $\| \nabla_x^m \psi''_s(s) \|_{L^2_x(\R^2)} \lesssim_m (1+s)^{-100}$ and $\|\nabla_x^m \psi''_s(s)\|_{L^\infty_x(\R^2)} \lesssim_m o_{n \to \infty}((1+s)^{-100})$.  Using this, one easily shows that 
$$ \|\nabla_x^m X_j(s)\|_{L^p_x(\R^2)} \lesssim \frac{O(1)^j}{j!} o_{n \to \infty}((1+s)^{-10})$$
(say) for $1 \leq p \leq \infty$, $j \geq 3$ and $0 \leq m \leq 20$.  Thus we have
$$ \| \nabla_x^m (\psi''_x(s) - \nabla_x e^{s\Delta} F_0) \|_{L^p_x(\R^2)} \lesssim o_{n \to \infty}((1+s)^{-10})$$
for $1 \leq p \leq \infty$ and $0 \leq m \leq 20$.  A similar argument also gives
$$ \| \nabla_x^m A''_x(s) \|_{L^p_x(\R^2)} \lesssim o_{n \to \infty}((1+s)^{-10}).$$
This already yields the spatial component of \eqref{haha} (in which $\nabla_{t,x}$ is replaced by $\nabla_x$).  By taking time derivatives of the above arguments, we also see that
$$ \| \nabla_x^m (\partial_t \psi''_x(s) - \nabla_x e^{s\Delta} F_1) \|_{L^p_x(\R^2)} \lesssim o_{n \to \infty}((1+s)^{-10})$$
and
$$ \| \nabla_x^m \partial_t A''_x(s) \|_{L^p_x(\R^2)} \lesssim o_{n \to \infty}((1+s)^{-10}).$$
for $1 \leq p \leq \infty$ and $0 \leq m \leq 19$, and this gives the time component of \eqref{haha} also.

To show \eqref{haha-2}, it suffices by \eqref{haha} and \eqref{heat-eq} to establish the claim with $\tilde h''$ replaced by the quantity
$$ (\psi_s - D_i \psi_i) - (\tilde \psi'_s - \tilde D'_i \tilde \psi'_i) - (\tilde \psi''_s - \tilde D''_i \tilde \psi''_i)$$
which can be re-expressed as
$$ \bigO( \nabla_x \delta \Psi_x ) + \bigO( \delta \Psi_x \Psi^*_x ) + \bigO( \tilde \Psi'_x \tilde \Psi''_x )$$
where $\delta \Psi_x := \Psi_x - \tilde \Psi'_x - \tilde \Psi''_x$ and $\Psi^*_x := ( \Psi_x, \tilde \Psi'_x, \tilde \Psi''_x )$.

From \eqref{psi-odd}, \eqref{psi-even} one has
$$ \delta \Psi_x(s) = \sum_{j=2}^\infty Y_j(s)$$
where $Y_j(s)$ is defined as with $X_j(s)$, except that some (but not all) of the $\psi''_s$ terms may be replaced by $\psi'_s$ instead.  The quantity $\nabla_x^m \psi'_s(s)$ enjoys less favourable estimates than $\psi'_s(s)$; by Proposition \ref{corbound-freq}, it has an $L^2_x(\R^2)$ norm of $O(c(s) s^{-(m+1)/2})$ and an $L^\infty_x(\R^2)$ norm of $O(c(s) s^{-(m+2)/2})$.  A computation using \eqref{sss} and the H\"older and Minkowski inequalities (placing at least one $\psi''_s$ factor in an $L^p$ space higher than $2$ to gain the $o()$ factor) then shows that
$$ \|\nabla_x^m Y_j(s)\|_{L^p_x(\R^2)} \lesssim_{p,m} O(1)^j j^{-j/2} o_{n \to \infty}(s^{1/p - m/2} (1+s)^{-10})$$
for $1 < p \leq \infty$ and $0 \leq m \leq 20$; repeating the arguments used to prove \eqref{haha} one then establishes \eqref{haha-2}.

The same arguments used to estimate the $X_j$ then show that
$$ \| \nabla_x^m \partial_t^k \delta \Psi_x(s) \|_{L^p_x(\R^2)} \lesssim o_{n \to \infty}((1+s)^{-10}).$$
for $1 \leq p \leq \infty$, $k=0,1$, and $0 \leq m \leq 19$; ; meanwhile we have
$$ \| \nabla_x^m \partial_t^k \tilde \Psi'_x(s) \|_{L^2_x(\R^2)} \lesssim (1+s)^{-10}$$
and
$$ \| \nabla_x^m \partial_t^k \tilde \Psi'_x(s) \|_{L^\infty_x(\R^2)} \lesssim o_{n \to \infty}((1+s)^{-10}).$$
From these bounds, Proposition \ref{corbound}, and the triangle inequality, we conclude that
$$ \| \nabla_x^m \partial_t^k \tilde \Psi''_x(s) \|_{L^2_x(\R^2)} \lesssim s^{-(m+2k)/2}$$
for the same range of $m,k$. From these estimates we easily obtain \eqref{haha-2}.

Now, we establish \eqref{eo-decline}.  We first consider the case when $i=1$.  Then $\psi_s(0) = \tilde \psi'_s(0)$, and thus $\psi_{s,x}(0) = \tilde \psi_{s,x}(0)$.  To get the energy decrease, it thus suffices (by taking $\eps$ small enough) to show that
$$ \int_\R^2 |\psi_t(0,x)|^2 - |\tilde \psi'_t(0,x)|^2\ dx \gtrsim \eps - O(\eps^2).$$
Writing $\delta \psi_t := \psi_t - \tilde \psi'_t$, it thus suffices to show that
\begin{equation}\label{jim}
 \langle \psi_t(0,0), \delta \psi_t(0,0) \rangle \gtrsim \eps - O(\eps^2)
\end{equation}
and
\begin{equation}\label{deltat} \| \delta \psi_t(0,0) \|_{L^2_x(\R^2)} = O(\eps).
\end{equation}
We first show \eqref{deltat}.  In fact, we will show the more precise estimate
\begin{equation}\label{deltat-2} \| \delta \psi_t(0,0) - \eps \Delta e^{s_0 \Delta} f \|_{L^2_x(\R^2)} = o_{n \to \infty}(1).
\end{equation}
From \eqref{psi-odd}, we can expand
$$ \delta \psi_t(0) = -\int_0^\infty \partial_t \delta \psi_s(s)\ ds + \sum_{j=3}^\infty X_j$$
where
$$ X_j = \bigO_j( \int_{s_1 < \ldots < s_j} \psi_s(s_1) \ldots \psi_s(s_{j-1}) \partial_t \delta \psi_s(s_j)\ ds_1 \ldots ds_j )$$
with the implied constants being of size $O(1)^j$.  From \eqref{sauce} one has
$$ \partial_t \delta \psi_s(s) = \tilde \psi''_s(s,0) := \eps \Delta^2 e^{(s+s_0)\Delta} f$$
so on integrating this in $s$, it will suffice to show that
$$
\| \int_{s_1 < \ldots < s_j} \psi_s(s_1) \ldots \psi_s(s_{j-1}) \Delta^2 e^{(s_j+s_0)\Delta} f\ ds_1 \ldots ds_j \|_{L^2_x(\R^2)} = o_{n \to\infty}(1).
$$
Estimating $\psi_s(s_1)$ in $L^2_x$ and the other terms in $L^\infty_x$ using Proposition \ref{corbound-freq} and \eqref{fo}, we can bound the left-hand side by
$$
\int_{s_1 < \ldots < s_j} s_1^{-1/2} c(s_1) \ldots c(s_{j-1}) o_{n \to \infty}( (1+s_j)^{-100} )\ ds_1 \ldots ds_j$$
for some frequency envelope $c$ of energy $O(1)$; using \eqref{sss} we obtain the claim.

It remains to show \eqref{jim}.  By Proposition \ref{corbound}, $\|\psi_t\|_{L^2_x(\R^2)} = O(1)$.  By \eqref{deltat-2} and Cauchy-Schwarz, it suffices to show that
$$ \langle \psi_t(0,0), \Delta e^{s_0 \Delta} f \rangle \gtrsim 1 - O(\eps).$$
By the fundamental theorem of calculus, it suffices to show that
\begin{equation}\label{emmanuel}
\langle \psi_t(s_0,0), \Delta f \rangle \gtrsim 1-O(\eps)
\end{equation}
and
\begin{equation}\label{samuel}
 \int_0^{s_0} |\partial_s \langle \psi_t(s,0), \Delta e^{(s_0-s)\Delta} f \rangle|\ ds = o_{n \to \infty}(1).
\end{equation}
to show \eqref{samuel}, we use \eqref{psit-eq} and integration by parts to write
$$
\partial_s \langle \psi_t(s,0), \Delta e^{(s_0-s)\Delta} f \rangle = 
\langle \bigO( \Psi_x \partial_x \psi_t + (\partial_x \Psi_x) \psi_t + \Psi_x \Psi_x \psi_t )(s,0), \Delta e^{(s_0-s)\Delta} f \rangle.$$
Using Proposition \ref{corbound} and \eqref{fo}, one can bound this inner product by $o_{n \to \infty}(s^{-1/2})$, and the claim follows.

To show \eqref{emmanuel}, we use \eqref{zerotor} followed by \eqref{psit-eq}, \eqref{heat-eq} to write
\begin{align*}
\partial_t \psi_s &= \partial_s \psi_t + \bigO( A_t \psi_s ) \\
&= \Delta \psi_t + \bigO( \Psi_{t,x} \nabla_x \Psi_{t,x} ) + \bigO( \Psi_{t,x} \Psi_{t,x} \Psi_{t,x} )
\end{align*}
so by \eqref{genie}, integration by parts, and the triangle ienquality, it suffices to show that
$$
\langle \bigO( \Psi_{t,x} \nabla_x \Psi_{t,x} )(s_0), f \rangle = o_{n \to \infty}(1)$$
and
$$
\langle \bigO( \Psi_{t,x} \Psi_{t,x} \Psi_{t,x} )(s_0), f \rangle = o_{n \to \infty}(1).$$
But this follows from Proposition \ref{corbound} and \eqref{fo}.

Now we establish \eqref{eo-decline} when $i=0$.  It will suffice to show the following variants of \eqref{jim}, \eqref{deltat}:
\begin{equation}\label{jim-2}
 \langle \psi_x(0,0), \delta \psi_x(0,0) \rangle \gtrsim \eps - O(\eps^2),
\end{equation}
\begin{equation}\label{deltax-2} \| \delta \psi_x(0,0) \|_{L^2_x(\R^2)} = O(\eps)
\end{equation}
and
\begin{equation}\label{deltax-3} \| \delta \psi_t(0,0) \|_{L^2_x(\R^2)} = o_{n \to \infty}(1).
\end{equation}
Repeating the argument used to establish \eqref{deltat-2} gives \eqref{deltax-3} and
$$
\| \delta \psi_x(0,0) - \eps e^{s_0 \Delta} \nabla_x f \|_{L^2_x(\R^2)} = o_{n \to \infty}(1)
$$
which gives \eqref{deltax-2}.  As before, to show \eqref{jim-2}, it will suffice to show that
$$ \langle \psi_x(0,0), e^{s_0 \Delta} \nabla_x f \rangle \gtrsim 1 - O(\eps).$$
By the fundamental theorem of calculus, it suffices to show that
\begin{equation}\label{emmanuel-2}
\langle \psi_x(s_0,0), \nabla_x f \rangle \gtrsim 1-O(\eps)
\end{equation}
and
\begin{equation}\label{samuel-2}
 \int_0^{s_0} |\partial_s \langle \psi_x(s,0), e^{(s_0-s)\Delta} \nabla_x f \rangle|\ ds = o_{n \to \infty}(1).
\end{equation}
The bound \eqref{samuel-2} is proven in exactly the same way as \eqref{samuel}.  To show \eqref{emmanuel-2}, we repeat the arguments used to prove \eqref{emmanuel} to replace $\Delta \psi_x$ by $\nabla_x \psi_s$, leaving one with
$$
\langle \Delta^{-1} \nabla_x \psi_s(s_0,0), \nabla_x f \rangle 
$$
and the claim follows from integration by parts and \eqref{genie}.

Finally, we show \eqref{jollyo}.  We give only a sketch of the proof here.  We may normalise $\| \phi \|_{S_k^\strong(I_+)}=1$.  As $\psi''_{s,lin}$ has frequency $\sim 1$ it is essentially enough to take $j=0$.
From \eqref{prod1} and \eqref{gem} one already has an upper bound of
$O( (1+s)^{-100})$, so by interpolation it suffices to establish an upper bound of $o_{n \to \infty}( \chi_{k=k(s)}^{-O(1)} (1+s)^{-100} )$.  By \eqref{physical} we may then take $k=k(s)$.  Finally, we shallcontent ourselves with the $s = O(1)$ case, as the $s \gg 1$ case is similar but is aided by the exponential decay of the heat kernel; our task is now to show that
$$
\| \psi''_{s,lin}(s) \phi \|_{S_0^\strong(I_+)} \lesssim o_{n \to \infty}( 1 ).$$
From \eqref{fo} it is not difficult to show that
$$ \| \psi''_{s,lin}(s) \phi[0] \|_{\dot H^1_0(I_+)} \lesssim o_{n \to \infty}( 1 )$$
so by  \eqref{energy-est} it suffices to show that
$$
\| P_k( \Box( \psi''_{s,lin}(s) \phi ) ) \|_{N_k^\strong(I_+)} \lesssim o_{n \to \infty}( 1 )
$$
for $k=O(1)$ (the bounds for $k \neq O(1)$ are similar but are omitted here).

As $\psi''_{s,lin}$ solves the free wave equation, one can split
$$ \Box( \psi''_{s,lin}(s) \phi ) = \psi''_{s,lin}(s) \Box \phi + 2 \partial^\alpha \psi''_{s,lin}(s) \partial_\alpha \phi.$$
The contribution of the null form is acceptable by using Theorem \ref{tao-thm3} as in the proof of \eqref{trilinear-improv2}, so it suffices to show that
$$
\| P_k( \psi''_{s,lin}(s) \Box \phi ) \|_{N_k^\strong(I_+)} \lesssim o_{n \to \infty}( 1 )
$$
In view of \cite[Lemma 12]{tao:wavemap2}, the only components of $\phi$ which cause difficulty are those with modulation $\gg 1$.  It is then not hard to show the main contribution needs to come from the case when $\psi''_{s,lin} \Box \phi$ has Fourier support at distance within $O(1)$ of the light cone, and at an angle of $\sim 1$ from the Fourier support of $\psi''_{s,lin}$.  

Decomposing $F$ into unit balls, one can split $\psi''_{s,lin}$ into an $l^2$ superposition of waves that are concentrated on $O(1)$-neighbourhoods of light cones.  Splitting these light cones into $1 \times 2^m \times 2^{2m}$ plates for $m \gg 1$ (we do not need to consider the contribution of bounded $m$, because we are working in $I_+$ and the source term $F$ is supported far in the past), one can compute that the contribution of each such plate in terms of null frame atoms (with $l = O(1)$) to show that the $N_k^\strong$ norm of these contributions is acceptable.  We omit the details. 
\end{proof}

Applying Theorem \ref{parab-thm-abstract}, we now obtain an instantaneous dynamic field heat flow $\psi'_s[0]$ with the same resolution as $\tilde \psi'_s[0]$ obeying the estimates
\begin{align}
\| \nabla_{t,x} \nabla_x^j \Psi'_x(s,0) \|_{\dot H^0_{k(s)}(\R^2)} &\lesssim c(s) s^{-(j+1)/2} \label{parab-1-instant-disp}\\
\| \nabla_{t,x} \nabla_x^j A'_x(s,0) \|_{\dot H^0_{k(s)}(\R^2)} &\lesssim c(s)^2 s^{-(j+1)/2}\label{parab-2-instant-disp}\\
\| \nabla_{t,x} \nabla_x^j \psi'_s(s,0) \|_{\dot H^0_{k(s)}(\R^2)} &\lesssim c(s) s^{-(j+2)/2}\label{parab-3-instant-disp}\\
\| \nabla_{t,x} \nabla_x^j (\tilde \Psi'_x - \Psi'_x)(s,0) \|_{\dot H^0_{k(s)}(\R^2)} &\lesssim o_{n \to \infty}( c(s) s^{-(j+1)/2} )\label{parab-4-instant-disp}\\
\| \nabla_{t,x} \nabla_x^j (\tilde A'_x - A'_x)(s,0) \|_{\dot H^0_{k(s)}(\R^2)} &\lesssim o_{n \to \infty}( c(s)^2 s^{-(j+1)/2} )\label{parab-5-instant-disp}\\
\| \nabla_{t,x} \nabla_x^j (\tilde \psi'_s - \psi'_s)(s,t_0) \|_{\dot H^0_{k(s)}(\R^2)} &\lesssim o_{n \to \infty}( c(s) s^{-(j+2)/2} )\label{parab-6-instant-disp}
\end{align}
for all $0 \leq j \leq 15$ and $s > 0$, and some frequency envelope $c$ of energy $O(1)$.  One sees from \eqref{spacloc} that $c$ is tight in the sense of \eqref{css}.

One could also repair $\tilde \psi''_s[0]$ to become a heat flow in a similar manner, but we will not need to do so here, as the linear evolution $\tilde \psi''_{s,lin}$ will already suffice for an extension of this component to $I_+$.

As $E_0 - c$ is good, we may apply Lemma \ref{quant} and extend $\psi'_s$ to a dynamic field heat flow on $I_+$ that obeys the wave map equation and has $(O(1),\mu)$-entropy $O_\mu(1)$.  Thus we may obtain a partition ${\mathcal J}$ of $I_+$ into $O_\mu(1)$ intervals $J$, such that $\psi'_s$ determines a $(O(1),\mu)$-wave map on each such interval $J$.  Applying Lemma \ref{freqstable2}, we see that
\begin{equation}\label{nabsa}
 \| \nabla_x^j \psi'_s(s) \|_{S_{k(s)}(J \times \R^2)} \lesssim_{\mu,j} s^{-(j+2)/2} c(s)
\end{equation}
and
\begin{equation}\label{nabsax}
 \| \nabla_x^j \psi'_x(s) \|_{S_{k(s)}(J \times \R^2)} \lesssim_{\mu,j} s^{-(j+1)/2} c(s)
 \end{equation}
for all $j \geq 0$, all $J$, and all $s>0$. 

We then recombine $\psi'_s$ and $\tilde \psi''_{s,lin}$ by defining the dynamic field
$$ \tilde \psi_s := \psi'_s + \tilde \psi''_{s,lin}$$
on $I_+$.

We also have that $\tilde \psi_s$ is an approximate wave map and approximate heat flow, which approximates $\psi_s$ at time zero:

\begin{lemma}\label{apaprox}  Let $J$ be an interval in ${\mathcal J}$.  Then there exists a frequency envelope $c_J$ of energy $O(1)$ such that
\begin{equation}\label{psisjk-hyp-sap}
\| \nabla_x^j \tilde \psi_s(s) \|_{S_{\mu,k(s)}(J \times \R^2)} \lesssim_j c_J(s) s^{-(j+2)/2}
\end{equation}
and
\begin{equation}\label{psisjk-star-hyp-sap}
\| \nabla_x^j \tilde h (s) \|_{S_{\mu,k(s)}(J \times \R^2)} = o_{n \to \infty; j,\mu}(c_J(s) s^{-(j+2)/2})
\end{equation}
for all $j \geq 0$ and $s > 0$ (where $\tilde h$ is the heat-tension field \eqref{heat-tension} of $\tilde \psi_s$), and also
\begin{equation}\label{wave-off-sap}
\| P_k \tilde w(0) \|_{N_{k}(I \times \R^2)} = o_{n \to \infty;j,\mu}(c_J(2^{-2k}))
\end{equation}
for every $k \in \R$ (where $\tilde w$ is the wave-tension \eqref{wave-tension} field of $\tilde \psi_s$).  Furthermore, one has
\begin{equation}\label{energy-close-sap}
\dist_\Energy( \tilde \psi_s[0], \psi_s[0] ) = o_{n \to \infty;j,\mu}(1)
\end{equation}
\end{lemma}

\begin{proof}  From Lemma \ref{freqstable} and Proposition \ref{corbound-freq} we see that the analogue of \eqref{psisjk-hyp-sap} for $\tilde \psi'_s$ holds for some frequency envelope $c'_J$ of energy $O(1)$.  From \eqref{nabsa} and the tightness of $c$, one has tightness of $c'_J$, in the sense that for every $\eps > 0$ there exists $C > 0$ depending on $\eps$ and $\mu$ such that 
$$ \int_{s>C, s<1/C} c'_J(s)^2/s\ ds \leq \eps$$
for sufficiently large $n$.  From \eqref{sourgum} we see that a similar claim holds for $\tilde \psi''_{s,lin}$ (because of the frequency localisation of $\psi''_{s,lin}$, one can take the relevant frequency envelope $c''_J$ here to also be tight). Adding the two estimates gives \eqref{psisjk-hyp-sap} (for a suitable envelope $c_J$).

Now we turn to \eqref{psisjk-star-hyp-sap}.  From \eqref{haha-2} it suffices (after enlarging $c_J$ slightly) to show that
$$
\| \nabla_x^j (\tilde h - h' - \tilde h'') (s) \|_{S_{\mu,k(s)}(J \times \R^2)} = o_{n \to \infty; j,\mu}(c_J(s) s^{-(j+2)/2})
$$
for all $j \geq 0$ and $s > 0$, where $\tilde h''$ is the heat-tension field associated to $\psi''_{s,lin}$.  But we can expand
\begin{equation}\label{hahaha}
 \tilde h - h' - \tilde h'' = \bigO( \partial_x \delta \psi_x ) + \bigO( \Psi^*_x \delta \psi_x ) + \bigO( \Psi'_x \tilde \Psi''_x )
 \end{equation}
where $\delta \psi_x := \tilde \psi_x - \psi'_x - \tilde \psi''_x$ and $\Psi^*_x := (\tilde \psi_x,\psi'_x, \tilde\psi''_x)$.
From the proof of Lemma \ref{ids}, we already have
$$ \| \nabla_x^j \tilde \Psi''_{s,lin}(s) \|_{S_{\mu',k(s)}(J \times \R^2)} \lesssim_j (1+s)^{-100}$$
for any $j \geq 0$.  From this and \eqref{psi-odd}, \eqref{psi-even}, \eqref{prod1}, one easily obtains that
\begin{equation}\label{sanitary}
 \| \nabla_x^j \tilde \Psi''_x(s) \|_{S_{\mu',k(s)}(J \times \R^2)} \lesssim_j (1+s)^{-100}.
 \end{equation}
Meanwhile, observe from \eqref{nabsax} and Lemma \ref{freqstable} that
\begin{equation}\label{sanitary-2}
\| \nabla_x^j \tilde \psi_x(s) \|_{S_{\mu,k(s)}(J \times \R^2)} \lesssim_j c_J(s) s^{-(j+1)/2}
\end{equation}
for all $j \geq 0$ and $s > 0$.  

From \eqref{sanitary}, \eqref{sanitary-2}, \eqref{jollyo} one has
\begin{equation}\label{psior}
 \| \nabla_x^j (\Psi'_x(s)\tilde \Psi''_x(s)) \|_{S_{k(s)}^\strong(J \times \R^2)} \lesssim_j o_{n \to \infty}((1+s)^{-100})
\end{equation}
and thus
$$ \| \nabla_x^j (\Psi'_x(s)\tilde \Psi''_x(s)) \|_{S_{\mu,k(s)}(J \times \R^2)} \lesssim_j o_{n \to \infty;\mu}((1+s)^{-100}).$$
This shows that the contribution of the last term of \eqref{hahaha} is acceptable.  To handle the other two terms, we see from \eqref{sanitary}, \eqref{sanitary-2}, \eqref{prod1} that it suffices to show that
$$ \| \nabla_x^j \delta \psi_x(s) \|_{S_{\mu,k(s)}(J \times \R^2)} \lesssim_j o_{n \to \infty;\mu}(c_J(s) s^{-(j+1)/2})$$
for all $j \geq 0$ and $s > 0$.

From \eqref{psi-odd} one has
$$ \delta \psi_x(s) = \sum_{j \geq 3} X_j(s)$$
where $X_j$ is the sum of $O(1)^j$ terms of the form
$$  \int_{s < s_1 < \ldots < s_j} \bigO( \psi_s^{(1)}(s_1) \ldots \psi_s^{(j-1)}(s_{j-1}) \partial_x \psi_s^{(j)}(s_j) )\ ds_1 \ldots ds_j$$
with one of the $\psi_s^{(i)}$ being equal to $\tilde \psi''_{s,lin}$, and the other terms being $\psi^*_s$.  Using \eqref{jollyo}, one has
\begin{align*}
\| \nabla_x^m X_j(s) \|_{S_{\mu,k(s)}(J \times \R^2)} &\lesssim_m o_{n \to \infty}( 
O(1)^j s^{-m/2} \int_{s < s_1 < \ldots < s_j} (s_j/s_1)^{0.1} \\
\frac{c_J(s_1)}{s} \ldots &\frac{c_J(s_{j-1})}{s_{j-1}} \frac{c_J(s_{j})}{s_{j}^{3/2}}\ ds_1 \ldots ds_j )
\end{align*}
(with one of the $c_J$ possibly replaced by $(1+s)^{-100}$), which by \eqref{sss} simplifies to
$$ \| \nabla_x^m X_j(s) \|_{S_{\mu,k(s)}(J \times \R^2)} \lesssim_m o_{n \to \infty}( O(1)^j j^{-j/2} s^{-(m+1)/2} c(s) )$$
and the claim follows.

Next, we show \eqref{wave-off-sap}.  By arguing as in the proof of Proposition \ref{twavemap}, it suffices to show  \eqref{hili} for all $j,l \geq 1$ with $j+l$ odd (in particular, $j+l \geq 3$), with  $(A,B)=(1,2)$ and $(A,B)=(2,1)$, where $\psi^1_s := \psi'_s$ and $\psi^2_s := \tilde \psi''_{s,lin}$.

Again, we first consider the terms in which the $\partial^\alpha$ derivative does not fall on the $\partial_\alpha \psi_s^B(s_{j+l})$ term.  Using \eqref{trilinear-improv}, the integrand can be bounded by
$$
\lesssim O(1)^{j+l} \frac{c_J(s_1)}{s_1} \ldots \frac{c_J(s_{j-1})}{s_{j-1}} \frac{c_J(s_j)}{s_j}
\frac{c_J(s_{j+1})}{s_{j+1}} \ldots \frac{c_J(s_{j+l})}{s_{j+l}} \chi_{k(s_1)=0}^{\delta_1/2} \chi_{k(s_2) \leq k(s_{j+l})}^{\delta_1/2};$$
using instead the dispersion bound \eqref{sourgum}, one obtains the bound
$$
\lesssim o_{n \to \infty}( O(1)^{j+l} \frac{c_J(s_1)}{s_1} \ldots \frac{c_J(s_{j-1})}{s_{j-1}} \frac{c_J(s_j)}{s_j}
\frac{c_J(s_{j+1})}{s_{j+1}} \ldots \frac{c_J(s_{j+l})}{s_{j+l}}).$$
Taking the geometric mean of these bounds and using the arguments from Proposition \ref{twavemap}, we see that the contribution of this term is acceptable.  The contribution of the term where the derivative falls on the $\partial_\alpha \psi_s^B(s_{j+l})$ term is similar.

Finally, we show \eqref{energy-close-sap}.  Write $\delta \psi_s(0) := \tilde \psi_s(0) - \psi_s(0) = \psi'_s(0) - \tilde \psi'_s(0)$.  From \eqref{parab-6-instant-disp} one has
$$
\| \nabla_{t,x} \nabla_x^j \delta \psi_s(s,0) \|_{\dot H^0_{k(s)}(\R^2)} \lesssim o_{n \to \infty}(c(s) s^{-(j+2)/2})$$
for all $0 \leq j \leq 10$ and $s>0$; meanwhile one has
$$
\| \nabla_{t,x} \nabla_x^j \psi^*_s(s,0) \|_{\dot H^0_{k(s)}(\R^2)} \lesssim o_{n \to \infty}(c(s) s^{-(j+2)/2})$$
where $\psi^*_s = (\tilde \psi_s, \psi_s)$ by Proposition \ref{corbound-freq}.  On the other hand, from \eqref{psi-odd}, \eqref{psi-even}, one can write
$$ \delta \Psi_{t,x}(s) = \sum_{j=1}^\infty \int_{s < s_1 < \ldots < s_j} \delta \bigO( \psi_s(s_1) \ldots \psi_s(s_{j-1}) \nabla_{t,x} \psi_s(s_j) )\ ds_1 \ldots ds_j,$$
where the implied constants are $O(1)^j$; using Lemma \ref{ek-prod} and \eqref{sss} as in previous arguments, we conclude that
$$
\| \nabla_x^j \delta \Psi_{t,x}(s,0) \|_{\dot H^0_{k(s)}(\R^2)} \lesssim o_{n \to \infty}(c(s) s^{-j/2})$$
for any $1 \leq j \leq 9$, which in turn implies (from \eqref{heat-eq}) that
$$ \| \delta \psi_s(s,0)\|_{L^2_x(\R^2)} \lesssim o_{n \to \infty}(c(s) s^{-1/2});$$
the same arguments also give
$$ \| \delta \psi_{t,x}(0)\|_{L^2_x(\R^2)} \lesssim o_{n \to \infty}(1)$$
and the claim now follows from \eqref{energy-dist-def}.
\end{proof}

Applying Theorem \ref{hyp-repair} iteratively for each interval $J$ in turn as in the proof of Proposition \ref{kslice}, we conclude for each such $K$ that
$$
\dist_\Energy( \phi(s)[t], \tilde \phi(s)[t] ) = o_{n \to \infty;j,\mu}(1)
$$
for all $s \geq 0$ and $t \in J$, and that
$$
\dist_{S^1_\mu(K)}( \Psi_{s,t,x}, \tilde \Psi_{s,t,x}) = o_{n \to \infty;j,\mu}(1).
$$
In particular, from the triangle inequality one has
$$ \| \Psi_{s,t,x} \|_{S^1_\mu(J)} \lesssim 1$$
for all $J \in {\mathcal J}$, and thus $(\phi,I_+)$ has $(O(1),\mu)$-entropy of $O_\mu(1)$, as required.  The proof of Theorem \ref{spacbound} is complete.

\section{Spatial delocalisation implies spacetime bound}\label{spacdeloc-sec}

We now begin the proof of Theorem \ref{spacdeloc}.  Let $E_0, \phi^{(n)}[0], \eps, R_n, I$ be as in the theorem; we allow all implied constants to depend on $E_0, \eps$, and on the localisation bounds in \eqref{spacloc}.  We now drop the $(n)$ superscript, and assume $n$ sufficiently large.  Our task is to show that $\phi$ can be extended to $I$ with an $(O(1),1)$-entropy of $O(1)$.  By time reversal symmetry, it suffices to do this for $I_+ := I \cap [0,+\infty)$.

Let $F: \R^+ \to \R^+$ be a (rapidly growing) function with $F(x) \geq x$ for all $x$ to be chosen later.  For each $k$, let $F^{(k)} = F \circ \ldots \circ F$ be the $k$-fold iterate of $F$.  Since $R_n \to \infty$, one can find $R_{(0),n} \to \infty$ and $K_n \to \infty$ such that $F^{(K_n)}( R_{(0),n} ) \leq R_n$.  Applying the pigeonhole principle, one can thus find (for sufficiently large $n$) a $k=k_n$ between $10$ and $K_n$ such that $R_{(k)} \leq R_n$ and
$$ \int_{F^{(k-10)}(R_{(0),n}) \leq |x| \leq F^{(k)}(R_{(0),n})} \T_{00}(\phi)(0,x)\ dx \leq o_{n \to \infty}(1).$$
If we set $r = r_n := F^{(k-10)}(R_{(0),n})$, we thus have $r \to \infty$ as $n \to \infty$, and
\begin{align}
\int_{|x| < r} \T_{00}(\phi)(0,x)\ dx &\leq E_0 - \eps\label{hot}\\ 
\int_{r \leq |x| \leq F^{(10)}(r)} \T_{00}(\phi)(0,x)\ dx \leq o_{n \to \infty}(1)\label{hot-2}\\
\int_{|x| > F^{(10)}(r)} \T_{00}(\phi)(0,x)\ dx \leq E_0 - \eps\label{hot-3}
\end{align} 
by the hypotheses on $\phi[0]$.

Let $\psi_s[0]$ be a heat flow resolution of $\phi[0]$.  We let $\eta \in C^\infty_0(\R^2)$ be a bump function supported on the ball $\{ x: |x| \leq 2 \}$ which equals one on the ball $\{ x: |x| \leq 1\}$, and define
\begin{align*}
\tilde \psi'_s[0](s,x) &:= \eta( \frac{x}{F(r)} ) \eta( s / r ) \psi_s[0](s,x) \\
\tilde \psi''_s[0](s,x) &:= (1 - \eta(s/F^{(8)}(r)) \eta(\frac{x}{F^{(9)}(r)})) \psi_s[0](s,x).
\end{align*}
for the ``local'' and ``global'' components of this resolution.

As in the previous section, we can find a tight frequency envelope $c = c^{(n)}$ of energy $O(1)$ such that
\begin{equation}\label{ss1}
\| \nabla_{t,x} \nabla_x^j \psi_s(s,0) \|_{\dot H^0_{k(s)}(\R^2)} \lesssim_j c(s) s^{-(j+2)/2}
\end{equation}
for all $j \geq 0$ and $s>0$, and similarly
\begin{equation}\label{ss2}
\| \nabla_x^j \Psi_{t,x}(s,0) \|_{\dot H^0_{k(s)}(\R^2)} \lesssim_j c(s) s^{-j/2}.
\end{equation}
It is easy to see that these bounds are inherited by $\tilde \psi'_s, \tilde \psi''_s$ and hence (by \eqref{psi-odd}, \eqref{psi-even}) by $\tilde \psi'_x, \tilde \psi''_x$.

\begin{lemma}[${\tilde \psi'_s[0]}$, ${\tilde \psi''_s[0]}$ are approximate heat flows]  Modifying $c$ if necessary, we have
\begin{equation}\label{claim1}
\| \nabla_{t,x} \nabla_x^j (\tilde \psi'_s - \tilde D'_i \tilde \psi'_i) (s,0) \|_{\dot H^0_{k(s)}(\R^2)} \lesssim_j o_{n \to \infty}(c(s) s^{-(j+2)/2})
\end{equation}
and
\begin{equation}\label{claim2}
\| \nabla_{t,x} \nabla_x^j (\tilde \psi''_s - \tilde D''_i \tilde \psi''_i) (s,0) \|_{\dot H^0_{k(s)}(\R^2)} \lesssim_j o_{n \to \infty}(c(s) s^{-(j+2)/2})
\end{equation}
for all $0 \leq j \leq 20$ and $s>0$.
\end{lemma}

\begin{proof}  We will just prove the second claim \eqref{claim2}, as the first one is similar (and slightly simpler).

Set $\tilde h'' := \tilde \psi''_s - \tilde D''_i \tilde \psi''_i$ 
From \eqref{ss1}, \eqref{ss2} for $\tilde \psi''_s, \tilde \Psi''_x$, one already establishes the bound 
$$
\| \nabla_{t,x} \nabla_x^j \tilde h'' (s,0) \|_{\dot H^0_{k(s)}(\R^2)} \lesssim_j c(s) s^{-(j+2)/2}$$
for all $j$.  Given the tightness of $c$, it will then suffice to establish the claim \eqref{claim1} for each fixed dyadic range $2^m \leq s \leq ^{m+1}$ (with the decay rate in $o_{n \to \infty}$ depending on $m$), as the general case can then be established by increasing the envelope $c$ slightly when $s$ is very large or very small.

Fix $m$; we now allow all implied constants to depend on $m$, thus $s \sim 1$, and our task is now to show that
$$ \| \nabla_{t,x} \nabla_x^j \tilde h' (s,0) \|_{\dot H^0_{k(s)}(\R^2)} \lesssim_j o_{n \to \infty}(1)$$
for all such $s$ and $0 \leq j \leq 20$.  By \eqref{heat-eq} we may replace $\tilde h'$ by $\tilde h' - h \tilde \eta$, where $\tilde \eta(s,x) := 1 - \eta(x/F(r)) \eta(s/r)$.  But one can expand $\tilde h'-h \tilde \eta$ as
$$\tilde h'-h \tilde \eta = \bigO( \partial_x \delta\psi_x ) + \bigO( \tilde \eta \Psi_x \delta \Psi_x ) + \bigO( \delta \Psi_x \delta \Psi_x )$$
where $\delta \Psi_x := \tilde \Psi'_x - \Psi_x \tilde \eta$.  So it would suffice to establish that
\begin{equation}\label{sood}
\| \nabla_{t,x} \nabla_x^j \delta \Psi_x (s,0) \|_{\dot H^0_{k(s)}(\R^2)} \lesssim_j o_{n \to \infty}(1)
\end{equation}
for $s \sim 1$ and $0 \leq j \leq 30$.

Using \eqref{PSA}, one sees that
\begin{equation}\label{january}
 \partial_s \delta \Psi_x = F + \bigO( \tilde \eta \psi_s \delta \Psi_x ) 
\end{equation}
where
$$ F  = \bigO( (\partial_x \tilde \eta) \psi_s ) + \bigO( (\partial_s \tilde \eta) \Psi_x ) + \bigO( \tilde \eta (1-\tilde \eta) \psi_s \Psi_x ).$$
Meanwhile, from \eqref{hot-2} and the energy estimates used in the proof of Lemma \ref{concdist}, one sees that
$$ 
\int_{2r \leq |x| \leq F^{(10)}(r)/2} |\psi_{t,x}(s,0,x)|^2\ dx = o_{n \to \infty}(1)$$
for all $s \sim 1$ (say).  From this and interpolation with Proposition \ref{corbound} we see that
$$ \| \nabla_{t,x} \nabla_x^j F (s,0) \|_{\dot H^0_{k(s)}(\R^2)} \lesssim_j o_{n \to \infty}(c(s) s^{-(j+3)/2})$$
for $s \leq 2r$, while $F$ vanishes for $s > 2r$.    The claim then follows from integrating \eqref{january} as in previous arguments.
\end{proof}

Applying Theorem \ref{parab-thm-instant}, we now obtain instantaneous dynamic field heat flows $\psi'_s[0], \psi''_s[0]$ with the same resolutions as $\tilde \psi'_s[0], \psi''_s[0]$ respectively obeying the estimates \eqref{parab-1-instant}-\eqref{parab-6-instant}.

\begin{lemma}[Compactness]\label{cpt}  $\psi'_s[0]$ lies in a compact subset $K_r$ of the energy space $\Energy$ which depends only on $r$ (and on the localisation bounds in \eqref{spacloc}).
\end{lemma}

\begin{proof}  By \cite[Proposition 3.10]{tao:heatwave4}, it suffices to show that the measures $\ESD(\psi'_s[0])$ and $\T_{00}(\psi'_s[0])$ are tight (viewed as a sequence in $n$).  The first claim follows from \eqref{ss1}, \eqref{ss2} for $\psi'_s$, $\Psi'_x$ and the tightness of $c$.  The second claim follows from the fact that $\psi'_s[0]$ is supported in the region $s = O(r), x = O(F(r))$ and energy estimates as in the proof of Lemma \ref{concdist}.
\end{proof}

\begin{lemma}[Energy decrease]\label{energy-dec} One has $\E( \psi'_s[0] ), \E( \psi''_s[0] ) \leq E_0-\eps/2$.
\end{lemma}

\begin{proof}  From Proposition \ref{corbound-freq} and the tightness of $c$, we have
$$ \int_{\R^2} |\psi_{t,x}(S,x)|^2\ dx = o_{n \to \infty}(1)$$
whenever $S = S_n$ goes to infinity as $n \to \infty$. Meanwhile, from \eqref{hot-2}-\eqref{hot-3}, one has
$$ \int_{|x| \leq F^{(10)}(r)} |\psi_{t,x}(0,x)|^2\ dx \leq E_0 - \eps + o_{n \to \infty}(1),$$
so from energy estimates we see that
$$ \int_0^S \int_{|x| \leq F^{(10)}(r)/2} \frac{1}{2} |\psi_{t,x} \wedge \psi_x|^2 + \frac{1}{2} |D_x \psi_x|^2\ dx ds \leq E_0 - \eps + o_{n \to \infty}(1)$$
if $S$ grows to infinity slowly enough.  From the tightness of $\ESD(\psi_s[0])$ we conclude that
\begin{equation}\label{iss}
 \int_{1/S}^S \int_{|x| \leq F^{(10)}(r)/2} \frac{1}{2} |\psi_{t,x} \wedge \psi_x|^2 + \frac{1}{2} |D_x \psi_x|^2\ dx ds \leq E_0 - \eps + o_{n \to \infty}(1).
\end{equation}
By repeating the proof of \eqref{sood}, one has
$$
\| \nabla_x^{j+1} [\Psi'_{t,x}(s,0,x) - \eta(x/F(r)) \eta(s/r) \psi_{t,x}(s,0,x)] \|_{\dot H^0_{k(s)}(\R^2)} \lesssim_j o_{n \to \infty}(1)
$$
for $j \geq 0$ and $s \sim 1$; from this (choosing $S$ sufficiently slowly growing) we conclude that we may replace $\psi_x$ by $\tilde \psi'_x$, etc.  From the tightness of $\ESD(\tilde \psi'_s[0])$ and energy estimates, we then conclude that
$$
 \int_0^\infty \int_{\R^2} \frac{1}{2} |\tilde \psi'_{t,x} \wedge \tilde \psi'_x|^2 + \frac{1}{2} |\tilde D'_x \tilde \psi'_x|^2\ dx ds \leq E_0 - \eps + o_{n \to \infty}(1).$$
 But by energy estimates and \eqref{claim1}, the left-hand side is just $\E(\tilde \psi'_s[0]) + o_{n \to \infty}(1)$, giving the first claim.  The proof of the second claim is similar.
\end{proof}

Let $\mu > 0$ be a small quantity depending on $E_0, \eps$ to be chosen later.
Since $E_0-\eps/2$ is good, we may now apply Lemma \ref{quant} and extend $\psi'_s, \psi''_s$ to dynamic field heat flows on $I_+$ obeying the wave map equation, such that $\psi'_s$, $\psi''_s$ have a $(O(1),\mu)$-entropy of $O_\mu(1)$.

Now we show that $\psi'_s, \psi''_s$ are separated in spacetime.  We decompose $I_+$ into the near future $I' := I_+ \cap [0,F^{(7)}(r)]$ and the far future $I'' := I_+ \backslash [0,F^{(6)}(r)]$.  In the near future we can separate $\psi'_s$ and $\psi''_s$:

\begin{proposition}[Spatial separation in the near future]\label{sapporo}  For any $s_0 > 0$ and $j \geq 0$, one has
\begin{equation}\label{samo}
\| \nabla_x^j(\eta(x / F^{(7)}(r)) \psi'_s(s_0)) \|_{S_{\mu,k(s_0)}(I' \times \R^2)} = o_{n \to \infty; \mu,s_0,j}(1)
\end{equation}
and
\begin{equation}\label{slimo}
\| \nabla_x^j (1- \eta(x / 10 F^{(7)}(r))) \psi''_s(s_0) \|_{S_{\mu,k(s_0)}(I' \times \R^2)} = o_{n \to \infty; \mu,s_0,j}(1)
\end{equation}
\end{proposition}

\begin{proof}  Fix $s,j$; we allow implied constants to depend on these quantities and on $\mu$. 

We begin with the proof of \eqref{samo}.
Observe that $\tilde \psi'_s[0]$ is supported on the region $|x| \leq 2 F(r)$.  By \eqref{psi-odd}, \eqref{psi-even}, the same is true for $\tilde \psi'_{t,x}(0)[0]$, and thus so is $\psi'_{t,x}(0)[0]$.  By finite speed of propagation for the wave map equation, we conclude that $\psi'_{t,x}(0)[t]$ is supported on the region $|x| \leq 2 F^{(6)}(r)$ (say).

Now we propagate this information to other values of $s$.  We introduce the second energy density
$$ e_2(s,x) := |D'_x \psi'_{t,x}(s,t,x)|^2$$
and observe (from \eqref{heat-eq}) the energy identity
$$ \partial_s e_2 = \Delta e_2 - 2 |D'_x D'_x \psi'_{t,x}|^2 + \bigO( |\psi'_{t,x}|^2 e_2 ).$$
By Proposition \ref{corbound-freq} we have $|\psi'_{t,x}|^2 = O( c(s)^2/s )$.  Thus, for any test function $\varphi$, if we define $E_\varphi(s) := \int_{\R^2} e_2 \varphi^4$, then
$$ \partial_s E_\varphi(s) \leq \int_{\R^2} e_2 \Delta(\varphi^4) + O( c(s)^2 E_\varphi(s) / s ).$$
Using Cauchy-Schwarz and the bounds $\int_{\R^2} e_2 = O( c(s)^2 / s )$ and $\Delta(\varphi^4) = O( A \varphi^2 )$ where $A := \|\nabla \varphi\|_{L^\infty_x}^2 + \| \nabla^2 \varphi\|_{L^\infty_x} \| \varphi\|_{L^\infty_x} )$, one has
$$ |\int_{\R^2} e_2 \Delta(\varphi^4)| \lesssim A E_\varphi(s)^{1/2} c(s) s^{-1/2},$$
and thus
$$ \partial_s E_\varphi(s)^{1/2} \lesssim A c(s) s^{-1/2} + c(s)^2 E_\varphi(s)^{1/2} / s.$$
If $\varphi$ is supported outside of the region $|x| \leq 2F(r)+|t|$, then $E_\varphi(0)=0$, and thus by Lemma \ref{gron-lem-2} one has
$$ E_\varphi(s)^{1/2} \lesssim A c(s) s^{1/2}.$$
Choosing $\varphi$ appropriately, we conclude that
\begin{equation}\label{shiba}
 \| D'_x \psi'_{t,x}(s,t) \|_{L^2_x(\{ |x| \geq 3 F^{(6)}(r) \})} \lesssim o_{n \to \infty}(c(s) s^{1/2})
\end{equation}
for all $s > 0$.  In particular
\begin{equation}\label{paas}
 \| \psi'_s(s,t) \|_{L^2_x(\{ |x| \geq 3F^{(6)}(r) \})} \lesssim o_{n \to \infty}(c(s) s^{1/2}).
\end{equation}

Next, from the pointwise identity
$$ |\psi'_{t,x}(s,t)| \leq e^{s\Delta} |\psi'_{t,x}(0,t)|$$
(see \cite[Lemma 4.8]{tao:heatwave2}), the support hypothesis, and the decay of the fundamental solution, we obtain the bounds
\begin{equation}\label{shiba-2}
 \| \psi'_{t,x}(s,t) \|_{L^\infty_x(\{ |x| \geq 3F^{(6)}(r) \})} \lesssim_m F^{(6)}(r)^{-m} s^m
\end{equation}
for any $m>0$.

Interpolating these bounds at $t=0$ and $s=s_0$ with Proposition \ref{corbound} (writing $\psi'_s := \bigO( D_x \psi'_x )$ and $\nabla_{t,x} \psi'_s := \bigO( A_{t,x} \psi'_s + \psi_{t,x} \psi_x \psi_x + D_x D_x \psi_{t,x} )$), one easily obtains that
$$
\| \nabla_x^j \nabla_{t,x} (\eta(x / F^{(7)}(r)) \psi'_s(s_0)) \|_{\dot H^0_{k(s_0)}} = o_{n \to \infty}(1)
$$
and so by \eqref{energy-est}, \eqref{fl1l2}, and Bernstein's inequality it suffices to show that
$$
\| \nabla_x^j \Box (\eta(x / F^{(7)}(r)) \psi'_s(s_0)) \|_{L^1_t L^p_x(I' \times \R^2)} = o_{n \to \infty;p}(1)
$$
for all $p$ close to $2$ (e.g. $1.9 \leq p \leq 2.1$).

Now we study the wave-tension field $w'$. Let $t\in I'$ be arbitrary.  From \eqref{w-eq} and the diamagnetic inequality (cf. \cite[Lemma 4.8]{tao:heatwave2}) one has
$$ \partial_s |w'| \leq \Delta |w'| + O( |\psi'_x| |\psi'_{t,x}| |D'_x \psi'_{t,x}| ).$$
On the one hand, from Proposition \ref{corbound} one has
$$ \||\psi'_x| |\psi'_{t,x}| |D'_x \psi'_{t,x}| (s,t) \|_{L^1_x(\R^2)} \lesssim 1 / s$$
(cf. \cite[Lemma 7.4]{tao:heatwave2}).  On the other hand, from \eqref{shiba}, \eqref{shiba-2}, and Proposition \ref{corbound-freq} one has
$$ \||\psi'_x| |\psi'_{t,x}| |D'_x \psi'_{t,x}| (s,t) \|_{L^1_x(|x| \geq 4F^{(6)}(r))} \lesssim_m F^{(6)}(r)^{-m} s^m$$
for any $m>0$.
From this and the decay of the fundamental solution of the heat equation, we see that
$$ \| w'(s_0,t) \|_{L^p_x( |x| \geq 5F^{(6)} r )} \lesssim_m F^{(6)}(r)^{-m}.$$
for all $p$ close to $2$.  Interpolating this with \cite[Lemma 7.5]{tao:heatwave2} one concludes that
$$ \| \nabla_x^m w'(s_0,t) \|_{L^p_x( |x| \geq 6 F^{(6)}( r )} \lesssim_{m,m'} F^{(6)}(r)^{-m'}.$$
for all $m,m' \geq 0$ also.  A final application of \eqref{w-eq} (and Proposition \ref{corbound}) then gives
$$ \| \nabla_x^m \partial_s w'(s_0,t) \|_{L^p_x( |x| \geq 7 F^{(6)} r )} \lesssim_{m,m'} F^{(6)}(r)^{-m'}.$$
Applying \eqref{psis-box}, \eqref{paas}, \eqref{shiba-2}, Proposition \ref{corbound}, and interpolation one then concludes that
$$ \| \nabla_x^m \Box (\eta(x / F^{(7)}(r)) \psi'_s(s_0,t)) \|_{L^p_x( \R^2 )} \lesssim_m F^{(6)}(r)^{-m}.$$
Integrating in time we obtain the claim.

The proof of \eqref{slimo} is similar.  The one new difficulty is that whereas $\tilde \psi'_s[0]$ vanished entirely for $|x| \geq 3F^{(6)}(r)$, $\tilde \psi''_s[0]$ does not quite vanish in the region $|x| \leq 100 F^{(7)}(r)$ (say), because of the case when $s > F^{(8)}(r)$.  However, the contribution of that case goes to zero as $F^{(8)}(r) \to \infty$, and so ends up being negligible compared to all other terms here if $F$ is chosen sufficiently rapidly growing.  We omit the details.
\end{proof}

Now we assert that $\psi'_s$ decays for large time $t$. 

\begin{proposition}[$\psi'_s$ decays for large time $t$]\label{psiton}  $\psi'_s$ is a $(O(1),1/r)$-wave map on $I''$.
We also have the analogue of \eqref{jollyo},
\begin{equation}\label{jollyo-2}
\| (\nabla_x^j \psi'_{s}(s)) \phi \|_{S_{\max(k,k(s)}^\strong(I'')} \lesssim o_{n \to \infty; j,\eps}( \chi_{k=k(s)}^{-\eps} s^{-(j+2)/2} c(s) \| \phi \|_{S_k^\strong(I'')} ) 
\end{equation}
for all $s,\eps>0$, $j \geq 0$, $k \in \R$ and all $\phi$.
\end{proposition}

\begin{proof}  By the induction hypothesis, $\psi'_s$ has a $(O(1),1/2r)$-entropy of $O_r(1)$ on every interval $[0,T]$, thus one can partition $[-T,T]$ into intervals $[t_{i,T},t_{i+1,T}]$ for $1 \leq i \leq m = O_r(1)$ and $0 = t_{0,T} < \ldots < t_{m,T} = T$.  Sending $T \to \infty$ and passing to subsequences as necessary, we may then find an $1 \leq i < m$ such that $t_{i,T} \to t_i < \infty$ and $t_{i+1,T} \to \infty$ for some sequence of $T \to \infty$.  In particular, we see that $\psi'_s$ is a $(O(1),1/2r)$-wave map on $[2t_i,T]$ (say) for all $T$.

To conclude the proposition, we need to show that $t_i$ is bounded by a quantity depending only on $r$.  Here we can appeal to compactness.  From Theorem \ref{apriori-thm2} we see that there is a small neigbourhood of $\psi'_s[0]$ in the energy space such that all data in this neighbourhood extend to a wave map on $[0,T]$ which is a $(O(1),1/r)$-wave map on $[2t_i,T]$ for any $T$.  By Lemma \ref{cpt}, we may make this $t_i$ depend only on $r$, and the claim follows.

Now we prove \eqref{jollyo-2}; we give only a sketch here.  From the tightness of $c(s)$, it suffices to establish the claim for $s \sim 1$, increasing $c$ slightly away from $s \sim 1$ if necessary.  As in the proof of \eqref{jollyo}, we may set $k=k(s)$ and $\|\phi\|_{S_k^\strong([2t_i,T])}=1$.  By \cite[Lemma 9.10]{tao:heatwave3}, one can approximate $\psi'_s$ in $S_k^\strong(I'')$ to error $o_{n \to \infty} ( c(1) )$ by a free solution, which one can approximate to have Schwartz initial data.  The claim now follows by arguments similar to those used to prove \eqref{jollyo} by making $t_i$ large enough, and then using compactness to bound $t_i$ by some quantity depending only on $r$.
\end{proof}

Now let $\tilde \psi_s$ be the dynamic field on $I_+$ defined by the formula
$$ \tilde \psi_s := \psi'_s + \psi''_{s}$$
This is an approximate wave map and approximate heat flow that is close to $\psi_s$ at time zero (cf. Lemma \ref{apaprox}):

\begin{lemma}  There exists a partition ${\mathcal J}$ of $I_+$ into $O_\mu(1)$ intervals $J$, such that on each such interval $J$, there exists a frequency envelope $c_J$ of energy $O(1)$ such that
\begin{equation}\label{psisjk-hyp-sap-2}
\| \nabla_x^j \tilde \psi_s(s) \|_{S_{\mu,k(s)}(J \times \R^2)} \lesssim_j c_J(s) s^{-(j+2)/2}
\end{equation}
and
\begin{equation}\label{psisjk-star-hyp-sap-2}
\| \nabla_x^j \tilde h (s) \|_{S_{\mu,k(s)}(J \times \R^2)} = o_{n \to \infty; j,\mu}(c_J(s) s^{-(j+2)/2})
\end{equation}
for all $j \geq 0$ and $s > 0$ (where $\tilde h$ is the heat-tension field \eqref{heat-tension} of $\tilde \psi_s$), and also
\begin{equation}\label{wave-off-sap-2}
\| P_k \tilde w(0) \|_{N_{k}(I \times \R^2)} = o_{n \to \infty;j,\mu}(c_J(2^{-2k}))
\end{equation}
for every $k \in \R$ (where $\tilde w$ is the wave-tension \eqref{wave-tension} field of $\tilde \psi_s$).  Furthermore, one has
\begin{equation}\label{energy-close-sap-2}
\dist_\Energy( \tilde \Psi_{s,t,x}(0), \Psi_{s,t,x}(0) ) = o_{n \to \infty;\mu}(1)
\end{equation}
\end{lemma}

\begin{proof}  We begin with \eqref{energy-close-sap-2}.  We write
\begin{align*}
\delta \psi_s(0) &:= \tilde \psi_s(0) - \psi_s(0) \\
&= (\psi'_s(0) - \tilde \psi'_s(0)) \\
&\quad + (\psi''_s(0) - \tilde \psi''_s(0))\\
&\quad - (\eta(s/F^{(8)}(r)) \eta(\frac{x}{F^{(9)}(r)}) - \eta( \frac{x}{F(r)} ) \eta( s / r ) ) \psi_s(0).
\end{align*}
By repeating the arguments used to prove \eqref{energy-close-sap}, it suffices to show that
$$
\| \nabla_{t,x} \nabla_x^j \delta \psi_s(s,0) \|_{\dot H^0_{k(s)}(\R^2)} \lesssim o_{n \to \infty}(c_J(s) s^{-(j+2)/2})$$
for all $0 \leq j \leq 10$.  The contributions of the first two terms in the above expansion of $\delta \psi_s$ to this estimate are acceptable by 
\eqref{parab-6-instant}.  As for the final term, one can divide into two cases, $s<r$ and $s\geq r$.  When $s \geq r$ one already has a bound of $O( c_J(s) s^{-(j+2)/2} )$ from Proposition \ref{corbound-freq}, and one can upgrade this to $o_{n \to \infty}(c_J(s) s^{-(j+2)/2})$ by increasing $c_J$ slightly for large $s$, using the tightness of $c_J$.  For $s<r$, one can instead use energy estimates as in the proof of Lemma \ref{energy-dec} to obtain an acceptable bound.

Now we prove the remaining claims.  It suffices to do this for $I'$ and $I''$ separately.  We begin with the contribution of the distant future $I''$, which is somewhat easier.  As $\psi'_s$ has a $(O(1),\mu)$-entropy of $O_\mu(1)$, we can partition $I''$ into $O_\mu(1)$ intervals $J$ such that $\psi'_s$ is a $(O(1),\mu)$ wave map on $J$, while $\psi''_s$ is a $(O(1),o_{n \to \infty}(1))$-wave map on $J$ by Proposition \ref{psiton}.  The claims now follow by repeating the arguments used to prove Lemma \ref{apaprox}.

Now we turn to the claim for $I'$.  One can repeat the arguments used to prove Lemma \ref{apaprox}; the main new difficulty is to establish the required analogues of \eqref{psior} and \eqref{hili}.  But these can be established by decomposing $\psi'_s = \eta(x/F^{(7)}(r)) \psi'_s + (1-\eta(x/F^{(7)}(r))) \psi'_s$ and $\psi''_s = \eta(x/10 F^{(7)}(r)) \psi''_s + (1-\eta(x/10 F^{(7)}(r)) \psi''_s$.  The interactions between these components is either small by Lemma \ref{sapporo}, or is negligible by the disjoint support of the components (note that the contributions of very large or very small $s$ are acceptable by increasing the tight envelope $c_J$ at very low or very high frequencies, so we may focus on values of $s$ comparable to $1$, and allow constants to depend on $s$).  We omit the details.
\end{proof}

Applying Theorem \ref{hyp-repair} iteratively for each interval $J$ in turn as in the proof of Proposition \ref{kslice}, we conclude for each such $K$ that
$$
\dist_\Energy( \phi(s)[t], \tilde \phi(s)[t] ) = o_{n \to \infty;j,\mu}(1)
$$
for all $s \geq 0$ and $t \in J$, and that
$$
\dist_{S^1_\mu(K)}( \Psi_{s,t,x}, \tilde \Psi_{s,t,x}) = o_{n \to \infty;j,\mu}(1).
$$
In particular, from the triangle inequality one has
$$ \| \Psi_{s,t,x} \|_{S^1_\mu(J)} \lesssim 1$$
for all $J \in {\mathcal J}$, and thus $(\phi,I_+)$ has $(O(1),\mu)$-entropy of $O_\mu(1)$, as required.  The proof of Theorem \ref{spacdeloc} is complete.

\section{Scattering}\label{scattering-sec}

We now use the above results to establish some results of scattering type for classical wave maps.

The arguments in \cite{tao:heatwave}-\cite{tao:heatwave4} and the current paper establish that every energy $E$ is \emph{good} in the sense of \cite{tao:heatwave4}.  In particular, given any energy $E > 0$ and $\mu>0$, and any classical initial data $\phi[0]$ with energy at most $E$, there exists a global classical wave map $\phi$ with this data, which has a $(O_E(1),\mu)$-entropy of $O_{E,\mu}(1)$.  One can view this assertion as a sort of ``uniform spacetime bound'' on $\phi$, though technically the bounds apply to the dynamic map heat flow resolution $\psi_s$ of $\phi$ rather than to $\phi$ itself.  (In \cite{sterbenz}, \cite{sterbenz2}, a non-uniform variant of this statement was obtained in the case where the target manifold is compact rather than hyperbolic; roughly speaking, this variant would assert in our language that the entropy of every global solution was finite, but without a uniform bound on that entropy in terms of the energy.)

This implies some scattering results in the caloric gauge.  We give some representative results of this type here:

\begin{proposition}[Scattering in the caloric gauge]  Let $\phi$ be a global classical wave map, and let $\psi_s$ be a dynamic field heat flow resolution of $\phi$.
\begin{itemize}
\item (Scattering of $\psi_s$) For each $s > 0$, there exists a solution $\psi_s^+(s): \R \times \R^2 \to \R^m$ to the free wave equation such that $\| \psi_s(s) - \psi^+_s(s) \|_{\dot S_{k(s)}( [T,+\infty) \times \R^2 )} \to 0$ as $T \to +\infty$.  In particular, $\| \nabla_{t,x} (\psi_s(s,t) - \psi^+(s,t) ) \|_{L^2_x(\R^2)} \to 0$ as $t \to +\infty$.
\item (Uniform Sobolev control) One has the regularity bound $\| \phi[t] \|_{{\mathcal H}^{1+\delta_0/2}_\loc} \lesssim_{E, \| \phi[0] \|_{{\mathcal H}^{100}_\loc}} 1$ for all $t \in \R$, where the Sobolev-type norms ${\mathcal H}^s_\loc$ are defined in \cite[Section 3]{tao:heatwave3}.
\item (Decay of $\phi$)  We have $\sup_x \dist_{\H}(\phi(t,x),\phi(\infty)) \to 0$ as $t \to \pm \infty$.
\end{itemize}
\end{proposition}

\begin{proof}  We begin with the scattering of $\psi_s$.  We allow implied constants to depend on $E$.  For any $\mu$, we see from the previous discussion that we can find $T_\mu > 0$ such that $\phi$ is a $(O(1),\mu)$-wave map on $[T_\mu,\infty)$.  

Let $\psi_s^{+,\mu}(s)$ be the solution to the free wave equation with initial data $\psi_s(s)[T_\mu]$.  By \cite[Lemma 9.10]{tao:heatwave3} we see that
$$ \| P_k (\psi_s(s) - \psi_s^{+,\mu}(s) ) \|_{S_k([T_\mu,+\infty) \times \R^2)} \lesssim_{j,\eps} \mu^{2-\eps} c(s) s^{-1} \chi_{k \geq k(s)}^{\delta_2/10} \chi_{k \leq k(s)}^j$$
for all $j,\eps > 0$ and some frequency envelope $c$ of energy $O(1)$.  This implies that
$$\| \psi_s(s) - \psi^{+,\mu}_s(s) \|_{\dot S_{k(s)}( [T_\mu,+\infty) \times \R^2 )} \lesssim \mu$$
(say).  This implies in particular that $\psi^{+,\mu}_s(s)[0]$ is a Cauchy sequence in $\dot H^1_{k(s)}$ as $\mu \to 0$, and thus converges to a limit $\psi^+_s(s)[0]$.  Letting $\psi^+_s(s)$ be the associated solution to the free wave equation, we obtain the claim.

Now we obtain the uniform Sobolev control.  We allow implied constants to depend on $E$ and $\|\phi[0]\|_{{\mathcal H}^{100}_\loc}$.  From \cite[Lemma 9.15]{tao:heatwave3} we see that $\phi[0]$ is controlled (in the sense of Proposition \ref{corbound-freq}) by a frequency envelope $c$ of energy $O(1)$ and obeying the decay estimate $c(s) = O(s^{-\delta_0/2})$ for all $s \geq 1$.  Applying Lemma \ref{freqstable}, we conclude that this control persists for all time $t$ (up to constant factors of $O(1)$).  The claim now follows by the arguments in \cite[Section 9.13]{tao:heatwave3}.

Finally, to prove the decay estimate, we observe from the scattering estimate (and approximating $\psi^+_s(s)$ by Schwartz class solutions to the free wave equation, which clearly disperse to zero at infinity) that $\| \psi_s(s,t)\|_{L^\infty_x(\R^2)} \to 0$ as $t \to +\infty$.  On the other hand, from the arguments in \cite[Section 9.13]{tao:heatwave} one also has the bounds of $\| \psi_s(s,t)\|_{L^\infty_x(\R^2)} \lesssim \min(1, s^{-\delta_0/2})$.  From the (non-abelian) triangle inequality, one has $\sup_x \dist_{\H}(\phi(t,x),\phi(\infty)) \leq \int_0^\infty \| \psi_s(s,t)\|_{L^\infty_x(\R^2)}$, and the claim now follows from the dominated convergence theorem.
\end{proof}

\begin{remark} One can also embed $\H$ in $\R^{1+m}$ and show that $\phi$ approximates a solution to the free wave equation by pushing the arguments further; we omit the details (see also \cite{tataru:wave3}, \cite{sterbenz} for related arguments).
\end{remark}

\appendix

\section{A covering lemma}

\begin{lemma}[Besicovitch-type covering lemma]\label{bcl}  Let $(Y,\nu)$ be a measure space with $\nu(Y) \leq A$ for some $A>0$.  Let $I$ be an interval, let $M \geq 1$ be an integer, and for each $y \in Y$, let ${\mathcal I}_y$ be a collection of at most $M$ intervals that partition $I$, in a manner which is measurable\footnote{In other words, one can write ${\mathcal I}_y=\{ I_1,\ldots,I_{M(y)}\}$, where $M(y) \in \{1,\ldots,M\}$ is measurable, and $I_i$ are intervals such that the set $\{ (y,t): y \in Y, t \in I_i \}$ is a measurable subset of $Y \times I$.} with respect to $y$.  Then there exists a partition ${\mathcal I}$ of $I$ into at most $M$ intervals, such that for each $J$ in ${\mathcal I}$, we have
$$ \int_Y \# \{ K \in {\mathcal I}_y: J \cap K \neq \emptyset \}\ d\nu(y) \leq 3A.$$
\end{lemma}

\begin{proof}  For sake of notation we shall write $I$ as a closed interval $I=[a,b]$, though the argument works for any other interval.  Let $f: I \to \R^+$ be the counting function
$$ f(t) := \int_Y \# \{ K \in {\mathcal I}_y: K \subset [a,t) \}\ d\nu(y).$$
Then $f$ is a non-decreasing function taking values in the interval $[0, AM)$.  If we then let ${\mathcal I}$ be the the set of all (non-empty) intervals of the form $f^{-1}( [iA, (i+1)A) )$ for $i=0,\ldots,M-1$, then we see that ${\mathcal I}$ is a collection of at most $M$ intervals partitioning ${\mathcal I}$.  Furthermore, if $J$ is an interval in ${\mathcal I}$, then $f$ increases at most $A$ within $J$, so we have
$$ \int_Y \# \{ K \in {\mathcal I}_y: K \subset J \} \leq A.$$
Also, if $x$ is one of the two endpoints of $J$, we have
$$ \int_Y \# \{ K \in {\mathcal I}_y: x \in J_y \}\ d\nu(y) \leq \int_Y \ d\nu \leq A.$$
Summing, we obtain the claim.
\end{proof}

This has the following corollary:

\begin{corollary}\label{bcl-2}  Let $I$ be an interval.  For each sub-interval $J \subset I$, let $X(J)$ be some Banach space norm on functions $\phi: J \times \R^2 \to \C$ obeying the monotonicity property
$$ \| \phi|_{J' \times \R^2} \|_{X(J')} \leq \|\phi \|_{X(J)}$$
for all $\phi: J \times \R^2 \to \C$ and $J' \subset J$, as well as the sub-additivity property
$$ \|\phi \|_{X(J)} \leq \| \phi|_{J_1 \times \R^2} \|_{X(J_1)} + \| \phi|_{J_2 \times \R^2} \|_{X(J_2)}.$$
Let $(Y,\nu)$ be a finite measure space, and let $(\phi_y)_{y \in Y}$ be a collection of functions indexed by $Y$ such that $\| \phi_y|_{J \times \R^2} \|_{X(J)}$ is jointly measurable in $y$ and $X(J)$.  Let $M \geq 1$, and suppose for each $y \in Y$ one can find a partition ${\mathcal I}_y$ of $I$ into at most $M$ intervals such that
$$ \| \phi_y|_{J \times \R^2} \|_{X(J)} \leq F(y)$$
for all $j \in {\mathcal I}_y$, where $F: Y \to \R^+$ is an absolutely integrable function.  Then one can find a partition ${\mathcal I}$ of $I$ into at most $M$ intervals such that
$$ \| \int_Y \phi_y\ d \nu(y)|_{J \times \R^2} \|_{X(J)} \leq 3 \int_Y F(y)\ d\nu(y).$$
\end{corollary}

\begin{proof} By multiplying $\nu$ with $F$ if necessary we may assume $F \equiv 1$.  From monotonicity in $J$ one can easily ensure that the partition ${\mathcal I}_y$ is measurable in the sense of Lemma \ref{bcl}.  The claim then follows from that lemma, Minkowski's integral inequality, and subadditivity.
\end{proof}

\begin{remark} The arguments here are closely related to the theory of variational norms for spacetime functions, see e.g. \cite{hadac} for a discussion.
\end{remark}

\section{Refinement of the $N_k$ space}\label{nk-refine}

This appendix is best read with \cite{tao:wavemap2} at hand, as it will refer very frequently to this paper. We will of course specialise to the case of two dimensions (thus, in the notation of \cite{tao:wavemap2}, we have $n=2$).

In the paper \cite{tao:wavemap2}, a number of spaces $S(c), S_k, N_k$ were constructed, obeying a large number of estimates and other properties; see \cite[Theorem 3]{tao:wavemap2} for a precise statement.  The spaces $N_k$ were constructed as follows:

\begin{definition}[Nonlinearity space]\label{atom-def}\cite[Definitions 7,8]{tao:wavemap2} Let $k$ be an integer, and let $F$ be a Schwartz function with Fourier support in the region $2^{k-4} \leq D_0 \leq 2^{k+4}$.  We say that $F$ is an \emph{$L^1_t L^2_x$-atom at frequency $2^k$} if
$$ \| F \|_{L^1_t L^2_x} \leq 1.$$
If $j \in \Z$, we say that $F$ is a \emph{$\dot X^{0,-1/2,1}$-atom with frequency $2^k$ and modulation $2^j$} if $F$ has Fourier support in the region $2^{k-4} \leq D_0 \leq 2^{k+4}$, $2^{j-5} \leq D_- \leq 2^{j+5}$ and
$$ \| F \|_{L^2_t L^2_x} \leq 2^{j/2}.$$
Finally, if $l > 10$ is a real number and $\pm$ is a sign, we say that $F$ is a \emph{$\pm$-null frame atom with frequency $2^k$ and angle $2^{-l}$} if there exists a decomposition $F = \sum_{\kappa \in K_l} F_\kappa$ such that each $F_\kappa$ has Fourier support in the region
\begin{equation}\label{suka}
\{ (\tau, \xi): \pm \tau > 0; D_- \leq 2^{k-2l-50}; 2^{k-4} \leq D_0 \leq 2^{k+4}; \Theta \in \frac{1}{2} \kappa \}
\end{equation}
and
\begin{equation}\label{sukf}
 (\sum_{\kappa \in K_l} \| F_\kappa \|_{NFA[\kappa]}^2)^{1/2} \leq 1.
\end{equation}
If $F$ is an atom of one of the above three types, then we say that $F$ is an $N[k]$ atom.  We let $N[k]$ be the atomic Banach space generated by the $N[k]$ atoms.

We say that $F$ is an \emph{$N_k$ atom} if there exists a $k' \in \Z$ such that $2^{200|k-k'|} F$ is a $N[k']$ atom.  We define $N_k(\R^{1+2})$ to be the atomic Banach space generated by the $N_k$ atoms.  Finally, for any time interval $I$, we define $N_k := N_k(I \times \R^2)$ to be the restriction of $N_k(\R^{1+2})$ to the slab $I \times \R^2$.
\end{definition}

(For the definitions of terms such as $D_0, D_-, \Theta, K_l, NFA[\kappa]$, see \cite[Section 7]{tao:wavemap2}.)

For technical reasons, this space is a little bit too large for our applications, specifically because the $\ell^2$ summability in \eqref{sukf} for null frame atoms is only barely sufficient to close the estimates in \cite{tao:wavemap2}, with no room left over to extract a good inverse theorem for these estimates.  However, it turns out that there is enough ``slack'' in other estimates to tighten this $\ell^2$ summability to $\ell^{2-}$ summability, which will assist in our inverse theorem.  More precisely, we make the following definitions:

\begin{definition}[Strong nonlinearity space]\label{atom-def2} Let $k$ be an integer, and let $F$ be a Schwartz function with Fourier support in the region $2^{k-4} \leq D_0 \leq 2^{k+4}$.  If $l > 10$ is a real number and $\pm$ is a sign, we say that $F$ is a \emph{strong $\pm$-null frame atom with frequency $2^k$ and angle $2^{-l}$} if there exists a decomposition $F = \sum_{\kappa \in K_l} F_\kappa$ such that each $F_\kappa$ has Fourier support in the region \eqref{suka} and
\begin{equation}\label{sukf-2}
 (\sum_{\kappa \in K_l} \| F_\kappa \|_{NFA[\kappa]}^{2-\delta_{-1}})^{1/(2-\delta_{-1})} \leq 1.
\end{equation}
(Thus, a strong null frame atom is the same concept as a null frame atom, but with the $\ell^2$ summability in \eqref{sukf} upgraded slightly.)
If $F$ is a $L^1_t L^2_x$-atom at frequency $2^k$, a $\dot X^{0,-1/2,1}$-atom with frequency $2^k$ and modulation $2^j$ for some $j \in \Z$, or a strong $\pm$-null frame atom with frequency $2^k$ and angle $2^{-l}$ for some $l>10$, we say that $F$ is a \emph{strong $N[k]$ atom}.  We let $N[k]^\strong$ be the atomic Banach space generated by the strong $N[k]$ atoms.  We then define $N_k(\R^{1+2})^\strong$ and $N_k(I \times \R^2)^\strong$ from $N[k]^\strong$ in exactly the same fashion as in Definition \ref{atom-def2}.
\end{definition}

Clearly $N_k^\strong$ is a stronger space than $N_k$, thus $\|\phi\|_{N_k(I \times \R^2)} \leq \|\phi\|_{N_k(I \times \R^2)^\strong}$.

The main claim of this appendix is then

\begin{theorem}\label{enw}  All the results of \cite[Theorem 3]{tao:wavemap2} continue to hold if $N_k$ is replaced by $N_k^\strong$ throughout.
\end{theorem}

\begin{proof}  This is easy for many components of \cite[Theorem 3]{tao:wavemap2}:
\begin{itemize}
\item All the properties that do not involve $N_k$ at all (the quasicontinuity property, scaling properties of $S(c), S_k$, and the estimates (18)-(24) and (32) from \cite{tao:wavemap2}) of course continue to hold.
\item The scale invariance of $N_k$, the compatibility condition (26), and the containment (25) of $L^1_t L^2_x$  from \cite{tao:wavemap2} were already easy to establish for the original $N_k$, and the proofs carry over without difficulty to $N_k^\strong$.
\item The energy estimate (27) from \cite{tao:wavemap2} for strong $N_k$ is clearly implied by that estimate for the original $N_k$.
\end{itemize}
The only remaining components are the product estimates (28), (29), the null form estimate (30), and the trilinear estimate (31) from \cite{tao:wavemap2}.

These estimates were all proven using a core product estimate, which we reproduce here:

\begin{lemma}\label{core}  \cite[Lemma 12]{tao:wavemap2}
Let $j, k, k_1, k_2$ be integers such that $j \leq \min(k_1,k_2)+O(1)$.  Then we have
\begin{equation}\label{core-est}
\| P_k (F \psi) \|_{N[k]} \lesssim \chi^{\delta}_{k=\max(k_1,k_2)} \chi^{\delta}_{j=\min(k_1,k_2)}
\| F \|_{\dot X^{0,-1/2,\infty}_{k_1}} \| \psi \|_{S[k_2]}
\end{equation}
for all Schwartz functions $F$ on $\R^{1+2}$ with Fourier support in $2^{k_1-5} \leq D_0 \leq 2^{k_1+5}$, $D_- \sim 2^j$ and Schwartz functions $\psi \in S[k_2]$, and some absolute constant $\delta > 0$.
\end{lemma}

We now verify 

\begin{lemma}\label{core-improve} In Lemma \ref{core}, one can strengthen the $N[k]$ norm to an $N[k]^\strong$ norm on the left-hand side (after replacing $\delta$ with, say, $\delta/2$).
\end{lemma}

\begin{proof} The first thing to do is search through the proof of Lemma \ref{core} in \cite{tao:wavemap2} for all occurrences of null frame atoms (or more specifically, all applications of \cite[Equation (93)]{tao:wavemap2}), and works out how to strengthen them to strong null frame atoms.  One easily finds that Case 1(c), Case 2(c), and Case 3(d).3 (basically, all the cases where $F$ has the dominant modulation) are the only places in that proof where such atoms appear.

In Case 1(c), we have (after rescaling) that $k_1 =  0$, $k_2=O(1)$, $k \leq O(1)$, $j \leq k+O(1)$, and the $l$ parameter is taken to be $(k-j)/2+O(1)$.  We can then borrow a small power from the $\chi^{\delta}_{k=\max(k_1,k_2)} \chi^{\delta}_{j=\min(k_1,k_2)}$ factors to gain a small power of $2^{-l}$, which one can use to upgrade \eqref{sukf} to \eqref{sukf-2}.

Similarly, in Case 3(d).3, we have (after rescaling) that $k=0$, $k_1=O(1)$, $k_2 \leq O(1)$, $j \leq k_2+O(1)$, and the $l$ parameter is taken to be $(k_2-j)/2+O(1)$.  Again, we can borrow a small power from the $\chi^{\delta}_{k=\max(k_1,k_2)} \chi^{\delta}_{j=\min(k_1,k_2)}$ factors to gain a small power of $2^{-l}$, which one can use to upgrade \eqref{sukf} to \eqref{sukf-2}.

Finally, we turn to Case 2(c), which is more interesting, because $l$ is now significantly larger.  After rescaling, we have that $k=0$, $k_2 = O(1)$, $k_1 \leq O(1)$, and $j \leq k_1+O(1)$, and the $l$ parameter is taken to be $-(k_1+j)/2 + O(1)$.  The arguments on \cite[p. 506]{tao:wavemap2} estimate the contribution here by
$$
\lesssim (\sum_{\kappa, \kappa' \in K_l: \dist(\kappa,\kappa') \sim 2^{-l+C}} 
\| P_0 P_{0,\kappa} Q^+_{<k_1+j-2C} (F P_{k_2,\kappa'} Q^+_{<k_1+j-C} \psi) \|_{NFA[\kappa]}^2)^{1/2}$$
for some large constant $C$.

Now we make a key new observation.  The frequency localisation operators allow us to replace $F$ by the spacetime Fourier projection to the $2^{k_1} \times 2^{(k_1+j)/2} \times 2^{k_1+j}$ slab
$$
S_{\kappa'} := \{ (\tau,\xi): |\xi| \sim 2^{k_1}; \tau = \xi \cdot \omega_{\kappa'} + O( 2^{k_1+j} ); |\pi_{\omega_{\kappa'}^\perp}(\xi)| \sim 2^{-(k_1+j)/2} \}
$$
oriented in the plane Minkowski-orthogonal to the null direction $(\omega_{\kappa'}, 1)$, where $\omega_{\kappa'}$ is the centre of the cap $\kappa'$, and $\pi_{\omega_{\kappa'}^\perp}(\xi)$ is the projection of $\xi$ to the orthogonal complement of $\omega_{\kappa'}$.  Let us denote this projection by $F_{\kappa'}$.  If one makes this replacement and then continues the argument in \cite[p. 506]{tao:wavemap2} (but using strong null frame atoms instead of null frame atoms), one ends up with
$$
\lesssim 2^{l/2}
(\sum_{\kappa' \in K_l}
\| F_{\kappa'} \|_{L^2_t L^2_x}^{2-\delta_{-1}}
\| P_{k_2,\kappa'} Q^+_{<k_1 + j-C} \psi \|_{S[k_2,\kappa']}^{2-\delta_{-1}})^{1/(2-\delta_{-1})}.$$
Now one makes the crucial geometric observation that the $S_{\kappa'}$ have a bounded overlap as $\kappa'$ varies, and so by Plancherel's theorem we have
$$ (\sum_{\kappa' \in K_l} \| F_{\kappa'} \|_{L^2_t L^2_x}^2)^{1/2} \lesssim \|F\|_{L^2_t L^2_x}$$
and in particular
$$ (\sum_{\kappa' \in K_l} \| F_{\kappa'} \|_{L^2_t L^2_x}^q)^{1/q} \lesssim \|F\|_{L^2_t L^2_x}$$
for all $2 \leq q \leq \infty$.  Applying H\"older's inequality with a large but finite value of $q$, we can then close the argument as in \cite[p. 506]{tao:wavemap2}.
\end{proof}

Now one can obtain the analogue of the product estimates in \cite[Equations (28), (29)]{tao:wavemap2} for the strong $N_k$ spaces.  These estimates are proven in \cite[Section 15]{tao:wavemap2} via a unified estimate \cite[Equations (119)]{tao:wavemap2}, which contains $N_k$ norms on both the left-hand side and right-hand side.  Strengthening the $N_k$ norms on the right-hand side of course makes the task of proving these estimates easier; the difficulty comes from strengthening the $N_k$ norms on the left-hand side.  Accordingly, one needs to search the proof for when the left-hand side is estimated using null frame atoms (other than via invocation of Lemma \ref{core}, as we have already checked the strong version of this estimate).  The only place where this occurs is in Case 3(c).2 of \cite[Section 15]{tao:wavemap2}, where null frame atoms appear on both the left and the right-hand side.  But one easily verifies that changing the exponent of summation from $2$ to $2-\delta_{-1}$ on both sides does not impact the argument, and the proof can be modified without difficulty.

An inspection of the proof of the null form estimate \cite[Equation (30)]{tao:wavemap2} in \cite[Section 17]{tao:wavemap2} reveals that null frame atoms are not mentioned explicitly in the argument (although they implicitly appear through the use of Lemma \ref{core}), so one can replace $N_k$ by strong $N_k$ here without difficulty.

Finally, we turn to the strong version of the (difficult) trilinear estimate \cite[Equation (31)]{tao:wavemap2}, proven in \cite[Section 18]{tao:wavemap2}, i.e. that
\begin{equation}\label{o-lemma}
\| P_k L(\phi^{(1)}, \phi_{,\alpha}^{(2)}, \phi^{(3),\alpha}) \|_{N_k^\strong}
\lesssim  
\chi^{-\delta}_{k = \max(k_1,k_2,k_3)}
\chi^{\delta}_{k_1 \leq \min(k_2,k_3)}
\prod_{i=1}^3
\| \phi^{(i)} \|_{S_{k_i}} 
\end{equation}
for some $\delta > 0$, whenever $\phi^{(i)} \in S_{k_i}$ for $i=1,2,3$.

As before, we search the proof of this estimate for all explicit mention of null frame atoms; these occur at Case 4(b) and Case 4(e).3(c) in \cite[Section 18]{tao:wavemap2}.  In both of these cases we have normalised $k_3=0$, $k=k_1+O(1)$, and $-\delta' k_1 \leq k_2 \leq 0$.

In Case 4(b), one sees that one can borrow a small power from $\chi^{\delta}_{k_1 \leq \min(k_2,k_3)}$ to gain a small power of $2^{-l}$, where $l$ is defined as $l = k_1/2 + O(1)$, and this allows us to upgrade the null frame atom to a strong null frame atom, as before.  The case 4(c) is similar, except $l$ is now defnied as $k_1 + 3 \delta' k_1$.  The claim follows.
\end{proof}

\section{Proof of Lemma \ref{boxf}}\label{fungi-sec}

We now prove Lemma \ref{boxf}.  We shall just prove the second claim; it will be clear from the proof that the same argument also gives the first claim.

By shrinking $\eps$ appropriately, we may allow all implied constants to depend on $\mu, C_0$.  By a limiting argument we may take $\phi$ to be Schwartz. 

In view of Definition \ref{smuk-def}, Lemma \ref{bcl}, and the triangle inequality, it suffices to show that for any $I$, any Schwartz $\phi$ and any $k,k'$, that one can partition $I$ into $O_\eps(1)$ intervals $J$ such that
$$ 2^{|k-k'|/2C_0} 2^{-k'/2} \| P_{k',\kappa} \nabla_{t,x} \phi \|_{L^2_t L^\infty_x(T_{x_0,\omega,k} \cap (J \times \R^2))} 
\lesssim \varepsilon \|\phi\|_{S^\strong_k(I \times \R^2)}$$
for each $J$, and all $\kappa \in K_l$, $x_0$, $\omega$ as in Definition \ref{smuk-def}.
From the definition of $S_k$ (in \cite[Section 10]{tao:wavemap2}) and $N_k^\strong$ (in Definition \ref{atom-def2}) it thus suffices to show that for all $I$ and all Schwartz $\phi: I \times \R^2 \to \C$ with Fourier support in the region $\{ D_0 \sim 2^k \}$, one can partition $I$ into $O_\eps(1)$ intervals $J$ such that
$$ 2^{-k/2} \| P_{k,\kappa} \nabla_{t,x} \phi \|_{L^2_t L^\infty_x(T_{x_0,\omega,k} \cap (I \times \R^2))} 
\lesssim \varepsilon ( \|\phi\|_{S[k](I \times \R^2)} + \| \Box \phi \|_{N[k]^\strong(I \times \R^2)} )$$
for all $\kappa, x_0, \omega$.

We can rescale $k=0$.  The number of caps $\kappa$ is $O(1)$, so we can also fix $\kappa$.  Our task is now to partition $I$ into $O_\eps(1)$ intervals $J$ such that
$$ \sup_{x_0,\omega} \| P_{0,\kappa} \nabla_{t,x} \phi \|_{L^2_t L^\infty_x(T_{x_0,\omega,0} \cap (J \times \R^2))} 
\lesssim \varepsilon ( \|\phi\|_{S[0](I \times \R^2)} + \| \Box \phi \|_{N[0]^\strong(I \times \R^2)} )$$
where it is always understood that one has the transversality condition $\dist(\omega,\kappa), \dist(\omega,-\kappa) \sim 1$.

Let us first establish this claim for free solutions, when $\Box \phi = 0$.

\begin{lemma}[Divisibility for free solutions]\label{lemfree}  Let $\phi$ be a Schwartz solution to the free wave equation $\Box \phi = 0$ supported on the frequency annulus $\{ D_0 \sim 1 \}$ and energy $\E(\phi)$.  Then one can partition $\R$ into $O_\eps(1)$ intervals $J$ such that
$$ \sup_{x_0,\omega} \| P_{0,\kappa} \nabla_{t,x} \phi \|_{L^2_t L^\infty_x(T_{x_0,\omega,0} \cap (J \times \R^2))} 
\lesssim \varepsilon \E(\phi)^{1/2}$$
for each $J$.
\end{lemma}

\begin{proof} We can normalise $\E(\phi)=1$.   Let $M$ be a large integer depending on $\eps$ to be chosen later.  Suppose the claim failed; then by the greedy algorithm, one could find disjoint time intervals $J_1,\ldots,J_M$ (which we may order from left to right), and tubes $T_{x_i,\omega_i,0}$ for $i=1,\ldots,M$, such that $\dist(\omega_i,\kappa), \dist(\omega_i,-\kappa) \sim 1$ and
$$ \| P_{0,\kappa} \nabla_{t,x} \phi \|_{L^2_t L^\infty_x(T_{x_i,\omega_i,0} \cap (J_i \times \R^2))} \gtrsim \eps$$
for all $1 \leq i \leq M$.  From Bernstein's inequality we see that $|J_i| \gtrsim \eps^2$ for all $i$.

Now we use the $TT^*$ method.  Linearising the above inequality, we may then find curves $x_i: J_i \to \R^2$ with $x_i(t) = x_i + \omega_i t + O(1)$ for each $1 \leq i \leq M$, and a function $f_i: J_i \to \R^3$ with $\int_{J_i} |f_i(t)|^2\ dt = 1$, such that
$$ \int_{J_i} P_{0,\kappa} \nabla_{t,x} \phi(t, x_i(t)) \cdot f_i(t)\ dt \gtrsim \eps.$$
We sum in $i$ and dualise to conclude that
$$ \| \sum_{i=1}^M \int_{J_i} U(-t) P_{0,\kappa} \delta_{x_i(t)} f_i(t)\ dt \|_{L^2_x(\R^2)} \gtrsim \eps M,$$
where $U(-t)$ is the wave propagator, whose (matrix-valued) symbol is a linear combination of $e^{2\pi it|\xi|}$ and $e^{-2\pi it|\xi|}$.  Squaring this, we see that
\begin{equation}\label{jij}
 \sum_{i=1}^M \sum_{j=1}^M \int_{J_i} \int_{J_j} |\langle U(t'-t) P_{0,\kappa}^2 \delta_{x_i(t)}, \delta_{x_j(t')} \rangle| f_i(t) f_j(t')\ dt dt' \gtrsim_\eps M^2.
\end{equation}
A routine stationary phase calculation shows that
$$ |\langle U(t'-t) P_{0,\kappa}^2 \delta_{x}, \delta_{x'} \rangle| \lesssim (1+|t-t'|)^{-1/2} (1+\dist(((t-t'),(x-x')), \Sigma))^{-100}$$
where $\Sigma$ consists of the portion of the light cone $\{ (t,x): |t| = |x| \}$ where $x/|x|$ lies in $\kappa$ or $-\kappa$.  Because the $\omega_i$ are transverse to $\kappa$, this implies in particular that
$$ |\langle U(t'-t) P_{0,\kappa}^2 \delta_{x_i(t)}, \delta_{x_i(t')} \rangle| \lesssim (1+|t-t'|)^{-100}$$
for any $1 \leq i \leq M$ and $t,t' \in J_i$.  Another application of transversality reveals that
$$ \int_{t'} |\langle U(t'-t) P_{0,\kappa}^2 \delta_{x_i(t)}, \delta_{x_j(t')} \rangle|\ dt' \lesssim (1+\dist(J_i,J_j))^{-1/2}$$
and
$$ \int_{t} |\langle U(t'-t) P_{0,\kappa}^2 \delta_{x_i(t)}, \delta_{x_j(t')} \rangle|\ dt \lesssim (1+\dist(J_i,J_j))^{-1/2}$$
for all distinct $1 \leq i,j \leq M$ and $t \in J_i, t' \in J_j$.  Applying Schur's test we conclude that
$$ \int_{J_i} \int_{J_j} |\langle U(t'-t) P_{0,\kappa}^2 \delta_{x_i(t)}, \delta_{x_j(t')} \rangle| f_i(t) f_j(t')\ dt dt' \lesssim (1+\dist(J_i,J_j))^{-1/2}$$
for all $i,j$.  Since $|J_i| \gtrsim \eps^2$, we have $\dist(J_i,J_j) \gtrsim \eps^2 |j-i|$; summing, we conclude that
$$ \sum_{i=1}^M \sum_{j=1}^M \int_{J_i} \int_{J_j} |\langle U(t'-t) P_{0,\kappa}^2 \delta_{x_i(t)}, \delta_{x_j(t')} \rangle| f_i(t) f_j(t')\ dt dt' \lesssim_\eps M^{3/2}$$
which contradicts \eqref{jij} if $M$ is large enough.
\end{proof}

Now we return to the inhomogeneous solutions $\Box \phi = F$.  We may assume that $I = [-T,T]$ for some $I$. Using Duhamel's formula and Lemma \ref{lemfree} we may assume that $\phi[0]=0$.
Applying Corollary \ref{bcl-2}, it suffices to show that for every strong $N[0]$ atom $F$, that we can partition $I$ into $O_\eps(1)$ intervals $J$ such that
$$ \sup_{x_0,\omega} \| P_{0,\kappa} \nabla_{t,x} \phi \|_{L^2_t L^\infty_x(T_{x_0,\omega,0} \cap (J \times \R^2))} 
\lesssim \varepsilon.$$

There are three cases.  Suppose first that $F$ is a $\dot X^{0,-1/2,1}$ atom.  Then $\nabla_{t,x} \tilde \phi$ is bounded in $\dot X^{1,1/2,1}_0$, plus a free solution of frequency $1$ that is bounded in energy (cf. \cite[Section 11]{tao:wavemap2}), and can be written as a modulated average of free solutions in the usual manner (cf. \cite[Lemma 2.9]{tao:cbms}), in which case the claim follows from Lemma \ref{lemfree} and Corollary \ref{bcl-2}.

In a similar fashion, if $F$ is an $L^1_t L^2_x$ atom, then by Duhamel's formula $\phi$ can be expressed as an average of free solutions of frequency $1$ and bounded energy, localised to half-spaces $\{ (t,x): t \geq t_0 \}$ in spacetime, and the claim again follows from Lemma \ref{lemfree} and Corollary \ref{bcl-2} (note that localisation in spacetime cannot increase the $L^2_t L^\infty_x$ norm).

Now we turn to the most difficult case, when $F$ is a strong $\pm$-null frame atom with frequency $1$ and angle $2^{-l}$ for some $l>10$.  It is here that we will crucially rely on the fact that we have a \emph{strong} null frame atom.

Let us first deal with the case when the angle $l$ is very large depending on $\eps$.  We express $F = \sum_{\kappa \in K_l} F_\kappa$ as in Definition \ref{atom-def2}, which induces a corresponding decomposition $\phi = \sum_{\kappa \in K_l} \phi_\kappa$.

By (null) energy estimates and Duhamel's formula, we know that $\phi_\kappa$ is equal to $\|F_\kappa\|_{NFA[\kappa]})$ times an average of free solutions of frequency $1$ and bounded energy, and also localised in frequency angle to a neighbourhood of $\pm \kappa$, localised to null half-spaces.  A standard application of the $TT^*$ method\footnote{Intuitively, the frequency localisation to $\kappa$ will cause a spreading in physical space at scale $2^l$ in the direction orthogonal to $\kappa$, which is the cause of the $2^{-l/2}$ gain.  Alternatively, one can foliate the sector of the light cone associated to $\kappa$ into a superposition of light rays (with the net measure of this superposition being $O(2^{-l})$), obtain an $L^2_t L^\infty_x$ estimate for the contribution of each light ray, and then average using Cauchy-Schwarz (cf. \cite{tataru:wave2} for a similar argument).} (similar to that used to prove Lemma \ref{lemfree}) then shows that such free solutions are bounded in $L^2_t L^\infty_x(T_{x_0,\omega,0})$ with a norm of $O(2^{-l/2})$ uniformly over all choices of $x_0,\omega$ with $\dist(\omega, \pm \kappa) \sim 1$.  By Minkowski's inequality, we thus have
$$ \sup_{x_0,\omega} \| P_{0,\kappa} \nabla_{t,x} \phi_\kappa \|_{L^2_t L^\infty_x(T_{x_0,\omega,0} \cap (I \times \R^2))} 
\lesssim 2^{-l/2} \|F_\kappa\|_{NFA[\kappa]}.$$
Taking $\ell^2$ norms of both sides and using H\"older's inequality and \eqref{sukf-2}, we conclude that
$$ \sup_{x_0,\omega} \| P_{0,\kappa} \nabla_{t,x} \phi \|_{L^2_t L^\infty_x(T_{x_0,\omega,0} \cap (I \times \R^2))} 
\lesssim 2^{-cl}$$
for some absolute constant $c > 0$.  This establishes the claim except when $l = O_\eps(1)$.

Finally, we consider the case when $l = O_\eps(1)$.  Let $\eps' > 0$ be chosen later, then by expressing each $\phi_\kappa$ as an average of truncated free solutions as before, and applying Lemma \ref{lemfree} and Corollary \ref{bcl-2}, we know that for each $\kappa$ we can divide $I$ into $O_{\eps',l}(1)$ intervals $J$ such that
$$ \sup_{x_0,\omega} \| P_{0,\kappa} \nabla_{t,x} \phi_\kappa \|_{L^2_t L^\infty_x(T_{x_0,\omega,0} \cap (J \times \R^2))} \lesssim \eps'$$
on each $J$.  By applying Corollary \ref{bcl-2} again we conclude that we can divide $I$ into $O_{\eps',l}(1)$ intervals $J$ such that
$$ \sup_{x_0,\omega} \| P_{0,\kappa} \nabla_{t,x} \phi \|_{L^2_t L^\infty_x(T_{x_0,\omega,0} \cap (J \times \R^2))} \lesssim_l \eps'$$
on each $J$.  Taking $\eps'$ sufficiently small depending on $\eps,l$ we obtain the claim.

\section{$L^p$ and exceptional sets}

Minkowski's integral inequality implies that if a family $(f_y)_{y \in Y}$ of functions on a measure space $X$ is uniformly bounded in $L^p(X)$, then (assuming some mild measurability conditions) any average $\int_Y f_y\ d\nu(y)$ of the $f_y$, where $\nu$ is a probability measure on $Y$, is also bounded in $L^p(X)$.  

Now suppose that we additionally know that each $f_y$ is small in $L^p$ outside of an exceptional set $F_y$.  If the exceptional set $F_y$ was independent of $y$, then another appeal to Minkowski's inequality would tell us that the average $\int_Y f_y\ d\nu(y)$ is also small in $L^p$ outside of the exceptional set.  For our applications, however, the exceptional set $F_y$ will depend on $y$.  Fortunately, one still has a substitute for this fact in this setting, as long as one stays away from the endpoints $p=1,p=\infty$:

\begin{lemma}[Minkowski's inequality with exceptional sets]\label{minkex}  Let $(X,\mu)$ be a measure space, let $1 < p < \infty$, let $0 < \eps \leq 1$, let $(Y,\nu)$ be a probability space, and for each $y \in Y$ let $f_y \in L^p(X)$ and $F_y \subset X$ be a measurable set such that
$$ \| f_y \|_{L^p(X)} \leq A$$
and
$$ \| f_y \|_{L^p(X \backslash F_y)} \leq \eps A.$$
Assume also that the function $(x,y) \mapsto f_y(x)$ is measurable on $X \times Y$, and the set $\{ (x,y): x \in F_y \}$ is similarly measurable in $X \times Y$. Then for any $k \geq 1$, there exists $y_1,\ldots,y_k \in Y$ such that
$$ \| \int_Y f_y\ d\nu(y) \|_{L^p(X \backslash (F_{y_1} \cup \ldots \cup F_{y_k}))} \leq \eps A + (p-1)^{1/p'} k^{-1/p'} A$$
where $p' = p/(p-1)$ is the dual exponent of $p$.   In particular, if $k$ is sufficiently large depending on $\eps,p$, one has
$$ \| \int_Y f_y\ d\nu(y) \|_{L^p(X \backslash (F_{y_1} \cup \ldots \cup F_{y_k}))} \leq 2\eps A$$
\end{lemma}

\begin{proof}  We can normalise $A := 1$.
We split
$$ \int_Y f_y\ d\nu(y) = \int_Y f_y 1_{F_y}\ d\nu(y) + \int_Y f_y 1_{X \backslash F_y}\ d\nu(y).$$
From Minkowski's inequality we have
$$ \| \int_Y f_y 1_{X \backslash F_y}\ d\nu(y) \|_{L^p(X)} \leq \eps$$
so it suffices to show that
$$ \| \int_Y f_y 1_{F_y} \ d\nu(y) \|_{L^p(X \backslash (F_{y_1} \cup \ldots \cup F_{y_k}))} \leq (p-1)^{(p-1)/p} k^{-(p-1)/p}.$$

We use the probabilistic method.  Select $y_1,\ldots,y_k \in Y$ independently and uniformly at random using the probability measure $\nu$.  It suffices to show that
$$ \E \| \int_Y f_y 1_{F_y}\ d\nu(y) \|_{L^p(X \backslash (F_{y_1} \cup \ldots \cup F_{y_k}))}^p \leq (p-1)^{p-1} k^{-(p-1)}.$$
By the Fubini-Tonelli theorem, we can rewrite the left-hand side as
$$\int_X (\E 1_{X \backslash (F_{y_1} \cup \ldots \cup F_{y_k})}(x)) |\int_Y f_y 1_{F_y}(x)\ d\nu(y)|^p d\mu(x).$$
For each $x$, let $E_x \subset Y$ be the set $E_x := \{ y \in Y: x \in F_y\}$.  Then by independence, we see that
$$ \E 1_{X \backslash (F_{y_1} \cup \ldots \cup F_{y_k})}(x) = (1 - \nu(E_x))^k;$$
meanwhile, by H\"older's inequality one has
\begin{align*}
|\int_Y f_y 1_{F_y}(x)\ d\nu(y)|^p &= |\int_{E_x} f_y\ d\nu(y)|^p
&\leq \nu(E_x)^{p-1} \int_Y |f_y|^p\ d\nu(y).
\end{align*}
Also, from calculus we see that the function $t \mapsto (1-t)^k t^{p-1}$ for $0 \leq t \leq 1$ is maximised when $t = \frac{p-1}{p-1+k}$, where it takes the value of $\frac{k^k (p-1)^{p-1}}{(p-1+k)^{p-1+k}} \leq (\frac{p-1}{k})^{p-1}$.  Combining all these estimates together we obtain
$$\int_X \int_Y (\frac{p-1}{k})^{p-1} |f_y|^p \ d\nu(y) d\mu(x)$$
and the claim follows from another application of the Fubini-Tonelli theorem.
\end{proof}

\section{Dispersion of the free wave equation}

Let $\phi$ be a solution to the free wave equation on $\R^{1+2}$.  Observe that the energy
$$ \E(\phi) = \E(\phi[t]) := \frac{1}{2} \int_{\R^2} |\nabla_{t,x} \phi(t,x)|^2\ dx$$
is independent of time.

The purpose of this section is to establish the following local energy dispersion estimate away from a few light rays.

\begin{theorem}[Energy dispersion away from light rays]\label{edisp}  Let $\phi$ be a solution to the free wave equation on $\R^{1+2}$.  Let $0 < \mu < 1$.  Then there exists a collection $(T_\beta)_{\beta \in B}$ of (infinite) tubes of frequency parameter $1$ (i.e. thickness $1$) and cardinality at most $O( \mu^{-O(1)} )$ such that
$$ \int_{|x-x_0| \leq 1} |\nabla_{t,x} \phi(t_0,x)|^2 1_{\R^{1+2} \backslash \bigcup_{\beta \in B} T_\beta}(t_0,x)\ dx \lesssim \mu \E(\phi).$$
for all $(t_0,x_0) \in \R^{1+2}$.
\end{theorem}

We now prove this theorem.  We may normalise $\E(\phi)$ to equal $1$, and we may assume $\mu$ to be small (e.g. $\mu \leq 1/100$).  We now perform the following greedy algorithm to remove pieces of concentrated energy from $\phi$:

\begin{itemize}
\item Step 0.  Initialise $\phi' := \phi$.
\item Step 1.  If there are no points $(t_0,x_0)$ in spacetime where $\int_{|x-x_0| \leq 1} |\nabla_{t,x} \phi'(t_0,x)|^2\ dx \leq \mu$ then {\tt STOP}.
\item Step 2.  Otherwise, let $\psi$ be the solution to the free wave equation with $\psi[t_0] := 1_{|x-x_0| \leq 1} \phi'[t_0]$; observe from computing the energy at $t_0$ that $\E(\phi'-\psi) \leq \E(\phi') - \mu$.  Now replace $\phi'$ by $\phi'-\psi$ and return to step 1.
\end{itemize}

As the energy decreases by $\mu$ at each stage of the iteration, this procedure must halt after at most $O(1/\mu)$ steps, and produces a decomposition
$$ \phi = \phi_0 + \sum_{j=1}^J \psi_j$$
where $J = O(1/\mu)$, $\phi_0, \psi_j$ solve the free wave equation, $\phi_0$ is dispersed in the sense that
$$ \int_{|x-x_0| \leq 1} |\nabla_{t,x} \phi_0(t_0,x)|^2 1_{\R^{1+2} \backslash \bigcup_{\beta \in B} T_\beta}(t_0,x)\ dx \leq \mu $$
for all $t_0,x_0$, and for each $1 \leq j \leq J$ there exists $(t_j,x_j)$ such that $\psi_j[t_j]$ is supported in the disk $\{ x \in \R^2: |x-x_0| \leq 1 \}$.

By the triangle inequality, it thus suffices to find, for each $j$, a collection $(T_\beta)_{\beta \in B}$ of (infinite) tubes of frequency parameter $1$ and cardinality $O(\mu^{-O(1)})$ such that
$$ \int_{|x-x_0| \leq 1} |\nabla_{t,x} \psi_j(t_0,x)|^2 1_{\R^{1+2} \backslash \bigcup_{\beta \in B} T_\beta}(t_0,x)\ dx \leq \mu^2$$
for all $t_0,x_0$.

Fix $j$.  By spacetime translation invariance we may take $t_j=x_j=0$.  

When $|t_0| \leq \mu^{-100}$, $\psi_j$ is supported in a region of diameter $O(\mu^{-100})$ by finite speed of propagation, which can be easily covered by $O(\mu^{-O(1)})$ tubes, so by time reversal symmetry we only need to consider the region $t_0 > \mu^{-100}$.  From the fundamental solution we see that we only need to consider $x_0$ in the range $||x_0| - t_0| \leq \mu^{-10}$ (say).

Call $(t_0,x_0)$ a \emph{concentration point} if $t_0 \geq \mu^{-100}$, $||x_0| - t_0| \leq \mu^{-10}$, and
$$  \int_{|x-x_0| \leq 1} |\nabla_{t,x} \psi_j(t_0,x)|^2\ dx > \mu^2.$$
Using standard Fourier integral operator propagation of singularities theory (or $TT^*$ and stationary phase), we see that if $(t_0,x_0)$ is a concentration point, then there is energy concentration in frequency space along a sector parallel to $x_0$, in the sense that
\begin{equation}\label{mox}
 \int_{\Gamma_{t_0,x_0}} |\hat{\nabla_{t,x} \phi}(0,\xi)|^2\ d\xi \gtrsim \mu^{O(1)}.
 \end{equation}
where
$$ \Gamma_{t_0,x_0} := \{ \xi \in \R^2: |\xi| \gtrsim \mu^{O(1)} t_0; \angle \xi, \pm x_0 \lesssim \mu^{O(1)} / t_0 \}.$$
Applying the greedy algorithm (starting with the maximal value of $t_0$ and working one's way down to locate a maximal disjoint collection of $\Gamma_{t_0,x_0}$), we see that the set of all $(t_0,x_0)$ which obey \eqref{mox} can be covered by $O(\mu^{-O(1)})$ tubes of thickness $O(\mu^{-O(1)})$, and the claim follows.

\section{A refined bilinear $L^2$ estimate}\label{inverse-sec}

In \cite{tao:inverse}, the following refined bilinear $L^2$ estimate was established:

\begin{theorem}\label{tao-thm}\cite{tao:inverse}  Let $0 < \mu < 1$ and let $\phi$ a solution to the free wave equation on $\R^{1+2}$ with spacetime Fourier transform supported on the region $\{ (|\xi|, \xi): |\xi| \sim 1; \angle \xi, e_1 \leq \pi/8 \}$.  Then there exists a collection $(T_\beta)_{\beta \in B}$ of (infinite) tubes of frequency parameter $1$ (i.e. thickness $1$) and cardinality at most $O( \mu^{-O(1)} )$ such that
$$
\| \phi \psi \|_{L^2_{t,x}(\R^{1+2} \backslash \bigcup_{\beta \in B} T_\beta)} \lesssim \mu 2^{-k} \| \phi[0]\|_{\dot H^1 \times L^2(\R^2)} \| \psi[0]\|_{\dot H^1 \times L^2(\R^2)}$$
whenever $k \geq 0$ and $\psi$ is a solution to the free wave equation with spacetime Fourier transform supported on the region $\{ (|\xi|, \xi): |\xi| \sim 2^k; \angle \xi, -e_1 \leq \pi/8 \}$.
\end{theorem}

The point here is the gain of $\mu$, as otherwise this estimate follows from standard bilinear $L^2$ estimates on $\R^{1+2}$.

Theorem \ref{tao-thm} implies a null form analogue:

\begin{theorem}\label{tao-thm2}  Let $0 < \mu < 1$ and let $\phi$ a solution to the free wave equation on $\R^{1+2}$ with spatial Fourier transform supported on the region $\{ |\xi| \sim 1 \}$.  Then there exists a collection $(T_\beta)_{\beta \in B}$ of tubes of frequency parameter $1$ and cardinality at most $O( \mu^{-O(1)} )$ such that
$$
\| \partial^\alpha \phi \partial_\alpha \psi \|_{L^2_{t,x}(\R^{1+2} \backslash \bigcup_{\beta \in B} T_\beta)} \lesssim \mu \| \phi[0]\|_{\dot H^1 \times L^2(\R^2)} \| \psi[0]\|_{\dot H^1 \times L^2(\R^2)}$$
whenever $k \geq 0$ and $\psi$ is a solution to the free wave equation with spatial Fourier transform supported on the region $\{ \xi: |\xi| \sim 2^k \}$.
\end{theorem}

\begin{proof} The spacetime Fourier transforms of $\phi,\psi$ are supported on the double light cone $\{ (\tau,\xi): |\tau|=|\xi| \}$.  Decomposing the double light cone into upper and lower components, and using symmetry, we may assume that $\phi$ has Fourier transform on the cone $\{ (|\xi|, \xi): |\xi| \sim 1 \}$ and $\psi$ has Fourier transform on the cone $\{ (\pm |\xi|, \xi): |\xi| \sim 2^k \}$ for some sign $\pm$.

Let us suppose first that $\pm = +$; we will treat the case $\pm=-$ later.  We then use an angular Whitney decomposition to partition the null form $\partial^\alpha \phi \partial_\alpha \psi$ into components $\sum_{\kappa, \kappa'} \partial^\alpha \phi_\kappa \partial_\alpha \psi_{\kappa'}$ where the spatial Fourier transforms of $\phi_\kappa, \psi_{\kappa'}$ are supported on sectors $\Gamma_\kappa$, $\Gamma_{\kappa'}$ respectively of angular width and separation $\sim 2^{-l}$ for some $l \geq 0$.  Standard null form estimates (e.g. \cite{kman.selberg}) then give
\begin{equation}\label{paka}
\| \partial^\alpha \phi_\kappa \partial_\alpha \psi_{\kappa'} \|_{L^2_{t,x}(\R^{1+2})} \lesssim 2^{-cl} \| \phi_\kappa[0]\|_{\dot H^1 \times L^2(\R^2)} \| \psi_{\kappa'}[0]\|_{\dot H^1 \times L^2(\R^2)}
\end{equation}
for some absolute constant $c>0$ (in fact one can take $c=1/4$); summing this using Cauchy-Schwarz, we see that the contribution of all the sectors with $2^{-l} \ll \mu^C$ for some large absolute constant $C$ are acceptable, so we only need to consider the $O(\mu^{-O(1)})$ contributions in which $2^{-l} = \mu^{O(1)}$.  But for each of these contributions one can use a Lorentz transformed version of Theorem \ref{tao-thm} (discarding the null structure by treating each component $\alpha$ of the null form separately) to obtain the claim for some collection of tubes $(T_\beta)_{\beta \in B}$ depending on $\kappa,\kappa'$ (and with $\mu$ replaced by $\mu^{C'}$ for some suitably large $C'$); taking the union of all these tubes and using the triangle inequality we obtain the claim.

The claim for $\pm=-$ is similar, except that one works with $-\Gamma_{\kappa'}$ instead of $\Gamma_{\kappa'}$, and uses the trivial fact that $\Phi \Psi$ and $\Phi \overline{\Psi}$ have the same $L^2_{t,x}$ norm for any $\Phi,\Psi$.
\end{proof}

The main purpose of this appendix is to extend Theorem \ref{tao-thm2} further, to the case where $\phi, \psi$ solve an inhomogeneous wave equation.  More precisely, we will show

\begin{theorem}\label{tao-thm3}  Let $0 < \mu < 1$ and let $I$ be an interval containing $0$.  Let $\phi, \psi: I \times \R^2 \to \C$ be a Schwartz function with spatial Fourier transforms supported on the regions $\{ |\xi| \sim 1\}$, $\{ |\xi| \sim 2^k\}$ respectively for some $k \geq 0$.
Then there exists a collection $(T_\beta)_{\beta \in B}$ of tubes of frequency parameter $1$ and cardinality at most $O( \mu^{-O(1)}  )$ such that
\begin{equation}\label{joy}
\| \partial^\alpha \phi \partial_\alpha \psi \|_{L^2_{t,x}(I \times \R^2 \backslash \bigcup_{\beta \in B} T_\beta)} \lesssim \mu \| \phi\|_{S^\strong_0(I \times \R^2)} \| \psi \|_{S^\strong_k(I \times \R^2)},
\end{equation}
where the implied constants can depend on $\delta_{-1}$.
\end{theorem}

We now prove this theorem.  We may normalise
$$
\| \phi[0]\|_{\dot H^1 \times L^2(\R^2)} + \|\Box \phi \|_{N_0(I \times \R^2)} = \| \psi[0]\|_{\dot H^1 \times L^2(\R^2)} + \|\Box \psi \|_{N_k(I \times \R^2)} = 1,
$$
and then by \eqref{phis-plus} the right-hand side of \eqref{joy} is $1$.

By Duhamel's formula, on $I \times \R^2$, $\phi$ is an average of a solution to the free wave equation with $\dot H^1 \times L^2(\R^2)$ norm (or energy) $O(1)$, and the fundamental solution $\Box^{-1}$ applied to very strong $N[O(1)]$ atoms as defined in Definition \ref{atom-def2}.  Using Lemma \ref{minkex}, we may assume that $\phi$ is just one of these expressions in the average, i.e. either a solution to the free wave equation of energy $O(1)$, or $\Box^{-1}$ of a very strong $N[O(1)]$ atom.  Similarly, we may assume that $\psi$ is either a solution to the free wave equation of energy $O(1)$, or $\Box^{-1}$ of a strong $N[k+O(1)]$ atom.  Our task is then to find at most $O(\mu^{-O(1)})$ tubes of frequency parameter $1$ such that
\begin{equation}\label{mold}
\| \partial^\alpha \phi \partial_\alpha \psi \|_{L^2_{t,x}(I \times \R^2 \backslash \bigcup_{\beta \in B} T_\beta)} \lesssim \mu.
\end{equation}
If $\phi$ is $\Box^{-1}$ of an $L^1_t L^2_x$ atom, then it can be expressed (by Duhamel's formula) as an average of solutions to the free wave equation of energy $O(1)$, localised to half-spaces in spacetime; furthermore, first derivatives of $\phi$ can be similarly expressed.  By applying Lemma \ref{minkex} again we see that the $L^1_t L^2_x$ atom case can be reduced to the free wave equation case, and similarly for $\psi$.  

Next, we eliminate the case when $\phi$ is $\Box^{-1}$ of a strong null frame atom, reducing it to the free case.  Suppose that $\phi$ is $\Box^{-1}$ of a strong null frame atom $F$ with frequency $\sim 1$ and angle $2^{-l}$ for some $l \geq 10$.  Let $C>0$ be a large constant to be chosen later.  Interpolating the condition \eqref{sukf-2}, one can split the strong null frame atom into $O(\mu^C)$ times an ordinary null frame atom $F$ of the same frequency and angle, plus an expression $\sum_{\kappa \in K_l} \Box^{-1} F_\kappa$ where $F_\kappa$ is as in Definition \ref{atom-def2} but with
$$ \sum_{\kappa \in K_l} \| F_\kappa \|_{NFA[\kappa]} \lesssim \mu^{-O_C(1)}$$
instead of \eqref{sukf-2} or \eqref{sukf} (note that we allow implied constants to depend on $\delta_{-1}$).  

To deal with the contribution of $\sum_{\kappa \in K_l} \Box^{-1} F_\kappa$, it suffices by Lemma \ref{minkex} to find, for each $\kappa$, a collection of tubes $B_\kappa$ of cardinality $O(\mu^{-O_C(1)})$ such that
$$
\| \partial^\alpha (\Box^{-1} F_\kappa) \partial_\alpha \psi \|_{L^2_{t,x}(I \times \R^2 \backslash \bigcup_{\beta \in B} T_\beta)} \lesssim \mu^{C'} \| F_\kappa \|_{NFA[\kappa]}.$$
for some large $C'$ (depending on $C$).  But one can express $\Box^{-1} F_\kappa$ as $\|F_\kappa\|_{NFA[\kappa]}$ times an average of free waves of energy $O(1)$, localised to null half-spaces (see \cite[Section 11]{tao:wavemap2}), so by yet another application of Lemma \ref{minkex}, we see that the contribution of this case reduces to that of a free wave.

Now we look at the contribution of $\Box^{-1} F$.  From the energy estimate in \cite[Equation (27)]{tao:wavemap2}, we have
$$ \| \Box^{-1} F \|_{S_0(\R^{1+2})} \lesssim \mu^C; \quad \| \psi \|_{S_k(\R^{1+2})} \lesssim 1$$
and hence by the null form estimate in \cite[Equation (30)]{tao:wavemap2}, one has
$$ \| P_{k'}( \partial^\alpha (\Box^{-1} F) \partial_\alpha \psi ) \|_{N_k(\R^{1+2})} \lesssim \chi_{k=k'}^\delta \mu^C$$
for all $k'$ and some $\delta>0$.  Since $N_k$ controls the $\dot X^{0,-1/2,\infty}$ norm, we conclude that 
$$ \| Q_{<j}( \partial^\alpha (\Box^{-1} F) \partial_\alpha \psi ) \|_{N_k(\R^{1+2})} \lesssim \mu$$
for some $j \sim C \log \frac{1}{\mu}$.  To deal with the remaining $Q_{\geq j}$ contribution, we see from the geometry of the cone that 
$$ Q_{\geq j}( \partial^\alpha Q_{<j-10} (\Box^{-1} F) \partial_\alpha Q_{<j-10} \psi ) = 0$$
so it suffices to estimate the terms involving either $Q_{\geq j-10} \Box^{-1} F$ or $Q_{\geq j-10} \psi$.  But from the energy estimate in  \cite[Equation (27)]{tao:wavemap2} followed by \cite[Equation (84)]{tao:wavemap2}, one has
$$ \| \nabla_{t,x} Q_{\geq j-10} \Box^{-1} F \|_{L^2_t L^\infty_x(\R^{1+2})}, \| \nabla_{t,x} \Box^{-1} F \|_{L^\infty_t L^\infty_x(\R^{1+2})}, \| \nabla_{t,x} Q_{< j-10} \Box^{-1} F \|_{L^\infty_t L^\infty_x(\R^{1+2})} \lesssim \mu$$
and
$$ \| \nabla_{t,x} Q_{\geq j-10} \psi \|_{L^2_t L^2_x(\R^{1+2})}, \| \nabla_{t,x} \psi \|_{L^\infty_t L^\infty_x(\R^{1+2})}, \| \nabla_{t,x} Q_{< j-10} \psi \|_{L^\infty_t L^2_x(\R^{1+2})} \lesssim 1$$
if $C$ is large enough, so these contributions are acceptable by H\"older's inequality.

The same arguments also allow us to reduce the case when $\psi$ is $\Box^{-1}$ of a strong null frame atom to that where it is a free wave of energy $O(1)$.

To summarise so far, we have reduced to the case where $\phi$ is a free wave of energy $O(1)$ or $\Box^{-1}$ of a $\dot X^{0,-1/2,1}$ atom, and $\psi$ is either a free wave of energy $O(1)$, or $\Box^{-1}$ of a $\dot X^{0,-1/2,1}$ atom.  

If $\phi,\psi$ are both free waves, the claim follows from Theorem \ref{tao-thm2}, so suppose now that $\phi$ is a free wave and $\psi$ is $\Box^{-1}$ of a $\dot X^{0,-1/2,1}$ atom at some modulation $j$, thus $\psi$ has spacetime Fourier support on the region $\{ |\xi| \sim 2^k; ||\tau|-|\xi|| \sim 2^j \}$ and $\nabla_{t,x} \psi$ has an $L^2_{t,x}$ norm of $O( 2^{-j/2} )$.

We apply Theorem \ref{tao-thm2} to create a family of tubes $(T_\beta)_{\beta \in B}$ with the stated properties.  We then partition the time axis $\R$ into intervals $J$ of length $2^{-j}$.  By a Fourier expansion in time (using the Poisson summation formula and Plancherel's theorem), we see that 
$$ (\sum_J (2^{-j/2} \| \nabla_{t,x} \psi \|_{L^2_t L^2_x(J \times \R^2)} + 2^{j/2} \| \Box \psi \|_{L^2_t L^2_x(J \times \R^2)} )^2)^{1/2} \lesssim 1.$$
and thus by H\"older's inequality
$$ (\sum_J (\| \nabla_{t,x} \psi \|_{L^\infty_t L^2_x(J \times \R^2)} + \| \Box \psi \|_{L^1_t L^2_x(J \times \R^2)} )^2)^{1/2} \lesssim 1.$$
On the other hand, from Theorem \ref{tao-thm2}, Duhamel's formula, and Minkowski's inequality, one has
$$ 
\| \partial^\alpha \phi \partial_\alpha \psi \|_{L^2_{t,x}(J \times \R^2 \backslash \bigcup_{\beta \in B} T_\beta)} \lesssim 
\mu (\| \nabla_{t,x} \psi \|_{L^\infty_t L^2_x(J \times \R^2)} + \| \Box \psi \|_{L^1_t L^2_x(J \times \R^2)} ).$$
Square summing in $J$, we obtain the claim.

Now suppose that $\phi$ is $\Box^{-1}$ of a $\dot X^{0,-1/2,1}$ atom at some modulation $j'$, and that $\psi$ is a free wave.  We partition the time axis into intervals $J'$ of length $2^{-j'}$.  As before, the quantities
$$ c'_{J'} := \| \nabla_{t,x} \phi \|_{L^\infty_t L^2_x(J' \times \R^2)} + \| \Box \phi \|_{L^1_t L^2_x(J' \times \R^2)}$$
have an $\ell^2$ norm of $O(1)$.  On each $J'$, one can use the preceding arguments to find a collection $B'_{J'}$ of tubes $T_{\beta'}$ with cardinality $O(\mu^{-O(1)})$ such that
$$ 
\| \partial^\alpha \phi \partial_\alpha \psi \|_{L^2_{t,x}(J' \times \R^2 \backslash \bigcup_{\beta' \in B'_{J'}} T_{\beta'})} \lesssim \mu c'_{J'}.$$
Unfortunately, we cannot square sum in $J'$ to conclude because the $B'_{J'}$ vary in $J'$.  To remedy this, we take advantage of the free nature of $\psi$, and use Theorem \ref{edisp} to find another collection $B$ of tubes $T_\beta$ of cardinality $O(\mu^{-O_C(1)})$, where $C$ is a large constant to be chosen later, such that
$$ \int_{|x-x_0| \leq 1} |\nabla_{t,x} \psi(t_0,x)|^2 1_{\R^{1+2} \backslash \bigcup_{\beta \in B} T_\beta}(t_0,x)\ dx \lesssim \mu^C.$$
for all $(t_0,x_0) \in \R^{1+2}$.

Let us now consider the quantity
\begin{equation}\label{conquat}
\| \partial^\alpha \phi \partial_\alpha \psi \|_{L^2_{t,x}((J' \times \R^2 \cap T_{\beta'}) \backslash \bigcup_{\beta \in B} T_\beta)}
\end{equation}
for some $J'$ and some $\beta' \in B'_{J'}$.  By construction, $\nabla_{t,x} \psi$ has a $L^\infty_t L^2_x$ norm of $O(\mu^C)$ on this domain.  If, on the other hand, $\phi$ which is transverse to $T_{\beta'}$ in the sense that the spatial Fourier transform is supported on a sector making an angle of at least $\mu^{cC}$ with the direction of $T_{\beta'}$ is of size $O( \mu^{cC} )$ if $c > 0$ is a small enough absolute constant, null energy estimates and Bernstein's inequality show that the $L^2_t L^\infty_x$ norm on $T_{\beta'}$ is $O(\mu^{-O(cC)})$, so the net contribution to \eqref{conquat} is $O( \mu^{cC} )$ if $c$ is small enough and $C$ large enough.  Summing over $\beta'$ and then square-summing over $J'$ we see that the contribution to \eqref{mold} here is acceptable for $C$ large enough.

It remains to consider the contribution of
$$
\| \partial^\alpha (P_{T_{\beta'}}\phi) \partial_\alpha \psi \|_{L^2_{t,x}((J' \times \R^2 \cap T_{\beta'}) \backslash \bigcup_{\beta \in B} T_\beta)}$$
where $P_{T_{\beta'}}$ is the spatial Fourier projection to a sector of angular width $O( \mu^{cC} )$ centred in the direction of $T_{\beta'}$.

From Bernstein's inequality we now see that $\nabla_{t,x} P_{T_{\beta'}}\phi$ has an $L^\infty_{t,x}(T_{\beta'})$ norm of $O(\mu^{cC/2} c_{J'})$.  Now suppose that $\psi$ is transverse to $T_{\beta'}$ in the sense that the spatial Fourier transform is supported on a sector making an angle of at least $\mu^{c'C}$ with the direction of $T_{\beta'}$ is of size $O( \mu^{c'C} )$ if $c' > 0$ is an absolute constant much smaller than $c$.  Then null energy estimates show that $\nabla_{t,x} \psi$ has an $L^2_{t,x}(T_{\beta'})$ norm of $O( \mu^{-O(c' C)} )$, so by arguing as before, the net contribution to \eqref{conquat} and hence to \eqref{mold} is acceptable if $c'$ is small enough and $C$ large enough.  Thus we may replace $\psi$ by $P'_{T_{\beta'}} \psi$, where $P'_{T_{\beta'}}$ is defined just as $P_{T_{\beta'}}$ but with $c$ replaced by $c'$.  But now we have a parallel interaction\footnote{Strictly speaking, due to the double nature of the light cone, there will also be some interactions here that are transverse to $T_{\beta'}$, but these can be dealt with exactly as with the previous transverse interactions encountered in this case.}, and standard null form estimates (cf. \eqref{paka}) allow one to again extract a bound of the form $O( \mu^{c'' C} c_{J'} )$ for this contribution to \eqref{conquat}.  All contributions to \eqref{mold} are now acceptable, which concludes the treatment of this case.

Finally, suppose that $\phi$ is $\Box^{-1}$ of a $\dot X^{0,-1/2,1}$ atom at some modulation $j'$, and $\psi$ is $\Box^{-1}$ of a $\dot X^{0,-1/2,1}$ atom at some modulation $j$.  We assume initially that $j' \geq j$.  We then partition the time axis into intervals $J$ of length $2^{-j}$, and then partition further into intervals $J'$ of length $2^{-j'}$ (assuming for sake of notation that $j, j'$ differ by an integer amount, though it is easy to remove this requirement).  As before, the quantities
$$ c'_{J'} := \| \nabla_{t,x} \phi \|_{L^\infty_t L^2_x(J' \times \R^2)} + \| \Box \phi \|_{L^1_t L^2_x(J' \times \R^2)}$$
and
$$ c_{J} := \| \nabla_{t,x} \psi \|_{L^\infty_t L^2_x(J \times \R^2)} + \| \Box \psi \|_{L^1_t L^2_x(J \times \R^2)}$$
both have $\ell^2$ norms of $O(1)$.  Using standard null form estimates, we conclude that
$$ \| \partial^\alpha \phi \partial_\alpha \psi \|_{L^2_{t,x}(J \times \R^2)}^2 \lesssim c_J^2 (\sum_{J' \subset J} (c'_{J'})^2).$$
Thus, the net contribution to \eqref{mold} of all the $J$ for which either $c_J^2 \lesssim \mu$ or $\sum_{J' \subset J} (c'_{J'})^2 \lesssim \mu$ will be acceptable.  By Markov's inequality, there are only $O(1/\mu)$ intervals $J$ remaining.  On each of these intervals, $\psi$ behaves essentially like a free wave and one can repeat the previous arguments (with $\mu$ replaced by $\mu^2$); summing, we obtain the claim.  The proof of Theorem \ref{tao-thm3} is now complete.

\end{document}